\documentclass{gtpart}

\usepackage{enumitem}
\usepackage{euscript}
\usepackage[all]{xy}
\usepackage{graphicx}
\usepackage{tikz}
\usetikzlibrary{arrows}

\newcommand{\PrintMathFonts}{%
  \count255=0
  \loop\ifnum\count255<16
    (\the\count255:~\fontname\textfont\count255)
    \advance\count255 by 1
 \repeat}

\DeclareMathOperator{\conju}{\mathsf{conj}}
\newcommand{\Aut}{\operatorname{Aut}}
\def\BbK{\mathbb{K}}
\def\G{\mathbb{G}}
\def\CP{\mathbb{CP}}
\def\cG{\mathcal{G}}
\def\cL{\mathcal{L}}
\def\acts{\quad\rotatebox[origin=c]{-90}{$\circlearrowright$}\quad}
\def\rar{\ar[r]}
\def\dar{\ar[d]}
\def\drar{\ar[dr]}
\def\hookar{\ar@{^{(}->}}
\def\spec{\mathsf{Spec}}

\def\dbdg{D^b_{\mathsf{dg}}}
\def\val{\mathrm{val}}

\def\ctens{\widehat{\otimes}}

\def\nov{r}
\def\defa{\mathsf{a}}
\def\cV{\mathcal{V}}
\def\snc{\mathsf{snc}}
\def\sh{\mathsf{sh}}
\def\co{\mathsf{co}}
\def\Rpuncomp{\wt{\underline{R}}}
\def\Runcomp{\wt{R}}
\def\Rpunc{\underline{R}}
\def\fmuncomp{\wt{\fm}}
\def\fuk{\EuF}
\def\bc{\mathsf{bc}}
\def\dm{d}
\def\dmm{e}
\def\dmmm{f}
\def\bff{\mathbf{f}}

\def\lmod{\text{-mod}}
\def\bimod{\text{-mod-}}
\def\im{\mathrm{im}}
\def\id{\mathrm{id}}
\def\ind{\mathrm{ind}}

\newsavebox{\lbox}
\newlength{\lwidth}
\newlength{\lboxwidth}

\newcommand{\laurent}[1]{%
\sbox{\lbox}{\ensuremath{#1}}%
\settowidth{\lwidth}{\usebox{\lbox}}%
\sbox{\lbox}{\ensuremath{\left(\usebox{\lbox}\right)}}%
\settowidth{\lboxwidth}{\usebox{\lbox}}%
\addtolength{\lboxwidth}{-\lwidth}%
\left(\hspace{-0.3\lboxwidth}%
\usebox{\lbox}%
\hspace{-0.3\lboxwidth}\right)%
}

\def\Q{\mathbb{Q}}
\def\R{\mathbb{R}}
\def\C{\mathbb{C}}
\def\N{\mathbb{N}}
\def\Z{\mathbb{Z}}
\def\T{\mathbb{T}}
\def\BbK{\mathbb{K}}

\def\Ts{\mathsf{T}}

\newcommand{\iii}{\mathbf{i}}

\newcommand{\frakg}{\mathfrak{g}}
\newcommand{\fm}{\mathfrak{m}}

\newcommand{\ff}{\mathfrak{f}}

\newcommand{\power}[1]{[\![ #1 ]\!]}
\newcommand{\laurents}[1]{(\!( #1 )\!)}
\newcommand{\wt}[1]{\tilde{#1}}
\def\cM{\mathcal{M}}
\def\G{\mathbb{G}}

\def\mca{\mathfrak{a}}
\def\mcb{\mathfrak{b}}
\def\mkah{\cM_{K\ddot{a}h}}
\def\mcpx{\cM_{cpx}}

\def\and{\, \& \,}
\def\r{\mathrm{r}}
\def\l{\ell}
\def\cM{\mathcal{M}}
\def\jmu{i}
\def\AAmp{Amp_{\text{nice}}}
\def\Nef{\mathsf{N}}

\newcounter{mainthm}
\newtheorem{main}[mainthm]{Theorem}

\newtheorem{thm}{Theorem}[section]
\newtheorem{lem}[thm]{Lemma}
\newtheorem{prop}[thm]{Proposition}
\newtheorem{cor}[thm]{Corollary}

\theoremstyle{definition}
\newtheorem{defn}[thm]{Definition}

\theoremstyle{remark}
\newtheorem{rmk}[thm]{Remark}

\newtheorem{example}[thm]{Example}

\newcommand\Hom{\operatorname{\mathsf{Hom}}}
\renewcommand\det{\operatorname{\mathsf{det}}}

\DeclareMathOperator{\proj}{\mathsf{Proj}}

\DeclareMathOperator{\HH}{\mathsf{HH}}

\newcommand{\EuA}{\EuScript{A}}
\newcommand{\EuB}{\EuScript{B}}
\newcommand{\EuC}{\EuScript{C}}
\newcommand{\EuD}{\EuScript{D}}

\newcommand{\EuF}{\EuScript{F}}
\newcommand{\EuG}{\EuScript{G}}
\newcommand{\EuH}{\EuScript{H}}

\newcommand{\EuJ}{\EuScript{J}}

\newcommand{\EuL}{\EuScript{L}}
\newcommand{\EuM}{\EuScript{M}}
\newcommand{\EuN}{\EuScript{N}}
\newcommand{\EuO}{\EuScript{O}}

\newcommand{\EuR}{\EuScript{R}}

\newcommand{\EuV}{\EuScript{V}}

\newcommand{\EuY}{\EuScript{Y}}

\newtheorem{ass}{Assumption}[section]
\newtheorem{deflem}[thm]{Definition--Lemma}
\numberwithin{equation}{section}
\newtheorem{sit}{Situation}
\newtheorem{corstar}[thm]{Corollary}
\newtheorem{propstar}[thm]{Proposition}
\newtheorem{lemstar}[thm]{Lemma}
\newtheorem{thmstar}[thm]{Theorem}

\begin{document}

\title{Versality of the relative Fukaya category}

\author{Nick Sheridan}

\address{Nick Sheridan, School of Mathematics, University of Edinburgh, Edinburgh EH9 3FD, UK}

\begin{abstract}
\textsc{Abstract:} Seidel introduced the notion of a Fukaya category `relative to an ample divisor', explained that it is a deformation of the Fukaya category of the affine variety that is the complement of the divisor, and showed how the relevant deformation theory is controlled by the symplectic cohomology of the complement. 
We elaborate on Seidel's definition of the relative Fukaya category, and give a criterion under which the deformation is versal.
\end{abstract}

\maketitle

\tableofcontents

\section{Introduction}
\label{sec:int}

\subsection{Mirror symmetry philosophy}
\label{subsec:msphil}

We start with some generalities about mirror symmetry, to explain the context for our main results. 
We will be imprecise and sweep technicalities under the rug, with the aim of conveying the important ideas. 
Some of the statements we make in this \S \ref{subsec:msphil} may even be literally wrong, but the statements from \S \ref{subsec:notation} onwards are precise.

Mirror symmetry predicts the existence of `mirror pairs' $(X,Y)$ of Calabi--Yau K\"{a}hler manifolds.
It predicts that certain symplectic invariants (the `$A$-model') of $X$ are equivalent to certain algebraic invariants (the `$B$-model') of $Y$, and vice versa. 
The symplectic invariants of $X$ depend only on the K\"{a}hler form $\omega$, and one defines the `K\"{a}hler moduli space' $\mkah(X)$ to be the moduli space of complexified K\"ahler structures $\omega + iB$ on $X$. 
It should be a complex manifold, and its tangent space should be isomorphic to $H^{1,1}(X;\C)$. 
The algebraic invariants of $Y$ depend only on the complex structure $J$, and one defines the `complex moduli space' $\mcpx(Y)$ to be the moduli space of complex structures on $Y$. 
It should be a complex manifold, and its tangent space should be isomorphic to $H^1(TY)$ (the Bogomolov--Tian--Todorov theorem). 

One can also consider `extended' versions of these moduli spaces:
\[ \mkah(X) \subset \mkah^{symp}(X) \subset \mkah^{ext}(X)\]
with tangent spaces $H^{1,1}(X;\C) \subset H^2(X;\C) \subset H^\bullet(X;\C)$; and
\[ \mcpx(Y) \subset \mcpx^{nc}(Y) \subset \mcpx^{ext}(Y)\]
with tangent spaces $H^1(TY) \subset HH^2(Y) \subset HH^\bullet(Y)$ (where $HH^\bullet(Y) \simeq H^\bullet(\wedge\!^\bullet TY)$). 
$\mkah^{symp}(X)$ corresponds to the moduli space of symplectic (but not necessarily K\"ahler) forms, and $\mkah^{ext}(X)$ to the space of `extended symplectic classes' (or `bulk deformation classes'). 
$\mcpx^{nc}(Y)$ corresponds to the space of non-commutative deformations of $Y$, and $\mcpx^{ext}(Y)$ to the space of `extended deformations' \cite{Barannikov1998}.

Mirror symmetry, then, predicts that there is an isomorphism 
$\psi : \mkah^{ext}(X) \to \mcpx^{ext}(Y)$, called the \emph{mirror map}, such that the symplectic invariants of $(X, \omega_p)$ are equivalent to the algebraic invariants of $(Y, J_{\psi(p)})$. 
The mirror map should take $\mkah \subset \mkah^{symp} \subset \mkah^{ext}$ to $\mcpx \subset \mcpx^{nc} \subset \mcpx^{ext}$.
In this paper we will be concerned with the spaces $\mkah$ and $\mcpx$, rather than the extended spaces.

The symplectic invariant of $(X,\omega_p)$ that we will consider is the \emph{Fukaya category}, denoted $\fuk(X,\omega_p)$. 
For the purposes of this \S \ref{subsec:msphil} we will think of it as a $\C$-linear category.
Its deformation theory should be controlled by a dgla, namely the \emph{Hochschild cochain complex} $CC^\bullet(\fuk(X,\omega_p))$. 
Since the Fukaya category depends on the extended symplectic class $\omega_p$, there should be a map from a formal neighbourhood of $p \in \mkah^{ext}(X)$ (which is isomorphic to a formal neighbourhood of $0 \in H^\bullet(X;\C)$) to the space of formal deformations of the category $\fuk(X,\omega_p)$. 

Indeed, this map is realized by an $L_\infty$ homomorphism 
\begin{equation} \EuC\EuO: H^\bullet(X) \dashrightarrow CC^\bullet(\fuk(X,\omega_p)),\end{equation}
where $H^\bullet(X)$ is equipped with the trivial dgla structure (see, e.g., \cite{FO3}). 
This $L_\infty$ homomorphism induces a map between the corresponding moduli spaces of solutions to the Maurer--Cartan equation for the dgla (see \cite{Manetti1999}), which gives a map
\begin{equation}
\label{eqn:bulkformmap}
\mbox{(form. nbhd. of $0 \in H^\bullet(X;\C)$)} \to \mbox{(form. def. of $\fuk(X,\omega_p)$)}.
\end{equation}
The leading term of this $L_\infty$ homomorphism is called the \emph{closed--open map}: in nice situations (see e.g. \cite{Ganatra2015}) it is a quasi-isomorphism, which implies that the map \eqref{eqn:bulkformmap} is in fact an isomorphism. 
This is what we mean by \emph{versality}: all formal deformations of $\fuk(X,\omega_p)$ come from geometry, i.e., by deforming the extended symplectic class $\omega_p$. 

If we restrict our attention to $\mkah^{symp}(X)$, and to $\Z$-graded deformations of the Fukaya category, then \eqref{eqn:bulkformmap} restricts to a map
\begin{equation}
\label{eqn:formmap}
\mbox{(form. nbhd. of $0 \in H^2(X;\C)$)} \to \mbox{(graded form. def. of $\fuk(X,\omega_p)$)}
\end{equation} 
which is an isomorphism if the closed--open map is a quasi-isomorphism. 
However there is not such an obvious way of singling out the deformations of the Fukaya category that correspond to $\mkah(X)$.

\begin{rmk}
The relationship between Hochschild cohomology and deformation theory of categories is subtler than we are implying. There may, for example, exist deformations which `kill' all objects in the category. 
We are ignoring such issues for the purposes of \S \ref{subsec:msphil}. 
\end{rmk}

The complex invariant of $(Y,J_q)$ that we will consider is a DG enhancement of the bounded derived category of coherent sheaves, denoted $\dbdg Coh(Y,J_q)$. Since the derived category depends on the extended complex structure $J_q$, there should be a map from a formal neighbourhood of $q \in \mcpx^{ext}(Y)$ (which is isomorphic to a formal neighbourhood of $0 \in HH^\bullet(Y)$) to the space of formal deformations of the category $\dbdg Coh(Y,J_q)$. 

Indeed, there is an $L_\infty$ quasi-isomorphism 
\begin{equation} H^\bullet(\wedge\!^\bullet TY) \simeq CC^\bullet(\dbdg Coh(Y)),\end{equation}
where the dgla structure on $H^\bullet(\wedge\!^\bullet TY)$ is trivial (see \cite{Kontsevich2003,Barannikov1998}). 
It induces an isomorphism between the corresponding moduli spaces of solutions to the Maurer--Cartan equation for the dgla:
\begin{equation} \mbox{(form. nbhd. of $0 \in HH^\bullet(Y)$)} \simeq \mbox{(form. def. of $\dbdg Coh(Y)$).}\end{equation}
As in the symplectic case, restricting to $\mcpx^{nc}(Y)$ and to $\Z$-graded deformations, we obtain an isomorphism
\begin{equation} \mbox{(form. nbhd. of $0 \in HH^2(Y)$)} \simeq \mbox{(graded form. def. of $\dbdg Coh(Y)$),}\end{equation}
but it is not so obvious how to single out the deformations corresponding to $\mcpx(Y)$.

Now suppose that $\fuk(X,\omega_p) \simeq \dbdg Coh(Y,J_q)$ for some $p \in \mkah^{symp}(X)$, $q \in \mcpx^{nc}(Y)$: it follows that the spaces of graded formal deformations of the two categories match up, and hence that there is a formal diffeomorphism (i.e., a mirror map)
\begin{equation} \psi: \mbox{(form. nbhd. of $p \in \mkah^{symp}(X)$)} \to \mbox{(form. nbhd. of $q \in \mcpx^{nc}(Y)$}\end{equation}
such that 
\begin{equation} \fuk(X,\omega_s) \simeq \dbdg Coh(Y,J_{\psi(s)}).\end{equation}
One expects that the mirror map should match up the subspaces $\mkah(X)$ and $\mcpx(Y)$. 
Thus:
\begin{quote}
In `nice' (i.e., versal) situations, homological mirror symmetry at a point in (extended) K\"{a}hler/complex moduli space should imply homological mirror symmetry for all nearby points. 
\end{quote}

This paper is about an approach to proving homological mirror symmetry that was initiated by Seidel \cite{Seidel2002,Seidel:HMSquartic}.
It has much in common with the above argument, but rather than starting at smooth points $p \in \mkah(X)$, $q \in \mcpx(Y)$, one starts at points in the boundary of a natural compactification of the respective moduli spaces: $p \in \partial \overline{\cM}_{K\ddot{a}h}(X)$ is a `large-volume limit point', and $q \in \partial \overline{\cM}_{cpx}(Y)$ is a `large complex structure limit point' (or equivalently a point with `maximally unipotent monodromy'). 

More precisely, Seidel introduces an auxiliary piece of data: a simple normal crossings ($\snc$) divisor $D \subset X$. 
We will call the data of a compact complex manifold $X$ equipped with a $\snc$ divisor $D$ a \emph{$\snc$ pair}. 
We introduce the notion of a \emph{relative K\"ahler form} on a $\snc$ pair, which is a K\"ahler form on $X$ equipped with a K\"ahler potential on $X \setminus D$, having a prescribed form near $D$ (see \S \ref{subsec:logpair}). 
A relative K\"ahler form $\omega$ has a well-defined relative K\"ahler class $[\omega] \in H^2(X,X \setminus D;\C)$.
We define $\cM_{K\ddot{a}h}(X,D)$, the `moduli space of complexified relative K\"ahler forms'. 
The tangent space to this moduli space at a smooth point should be $H^2(X,X \setminus D;\C)$. 
There is a natural map $\cM_{K\ddot{a}h}(X,D) \to \cM_{K\ddot{a}h}(X)$ which forgets the K\"ahler potential, and whose differential is the map $H^2(X,X \setminus D;\C) \to H^{1,1}(X;\C)$ (note that the image of the map $H^2(X,X \setminus D;\C) \to H^2(X;\C)$ is spanned by classes Poincar\'e dual to the components of $D$, which lie in $H^{1,1}(X;\C)$).

Seidel's idea \cite{Seidel2002} is that one can construct a natural compactification $\overline{\cM}_{K\ddot{a}h}(X,D)$ by including relative K\"ahler forms which are infinite currents along some components of $D$. 
The large-volume limit point $p \in \partial \overline{\cM}_{K\ddot{a}h}(X,D)$ is the deepest point in the natural stratification of this compactification: namely, the point where the relative K\"ahler form is infinite along all of $D$. 
This means that the divisor $D$ bas been stretched until it is `infinitely far away', and in fact the Fukaya category that one should associate to this point $p$ is the \emph{affine Fukaya category}, $\fuk(X \setminus D)$. 
Its objects are compact exact Lagrangian branes $L \subset X \setminus D$, and its $A_\infty$ structure maps count pseudoholomorphic discs in $X \setminus D$. 

To interpolate between the affine Fukaya category $\fuk(X \setminus D)$ corresponding to the large-volume limit point $p \in \partial \overline{\cM}_{K\ddot{a}h}(X,D)$ and the Fukaya categories $\fuk(X,\omega_q)$ for points $q \in \mkah(X,D)$ close to $p$, Seidel introduces the \emph{relative Fukaya category} $\fuk(X,D)$, which lives over $\overline{\cM}_{K\ddot{a}h}(X,D)$. 
We elaborate on Seidel's definition (see \S \ref{subsec:relfukint} for details). 

To be more precise, $\fuk(X,D)$ will be a category defined over some $\C$-algebra $R$, and we can think of $\spec (R)$ as a formal neighbourhood of $p \in \overline{\cM}_{K\ddot{a}h}(X,D)$.
Thus one can think of the relative Fukaya category as a family of $\C$-linear categories, living over a formal neighbourhood of the deepest stratum of the boundary of $\overline{\cM}_{K\ddot{a}h}(X,D)$.

To prove homological mirror symmetry for $X$ and $Y$, then, one starts by choosing an appropriate $D \subset X$ (corresponding to $p \in \partial \overline{\cM}_{K\ddot{a}h}(X,D)$) and maximal degeneration of $Y$ to some singular variety $Y_0$ (corresponding to $q \in \partial \overline{\cM}_{cpx}(Y)$). 
One then proves homological mirror symmetry over these boundary points $p$ and $q$, by identifying $\fuk(X \setminus D)$ with $\dbdg Coh(Y_0)$ (more precisely, with the subcategory generated by vector bundles).
One then hopes that a general versality result, of the kind sketched above, implies that homological mirror symmetry holds over a formal neighbourhood of $p$ in $\overline{\cM}_{K\ddot{a}h}(X,D)$, up to some formal mirror map.
This paper is about establishing conditions when such a versality result holds.

Versality at a large-volume limit point in $\overline{\cM}_{K\ddot{a}h}(X,D)$ is more complicated than at a smooth point in $\mkah(X,D)$. 
The deformation theory of $\fuk(X \setminus D)$ is controlled by $CC^\bullet(\fuk(X \setminus D))$. 
As before, there is a map from a formal neighbourhood of $p \in \overline{\cM}_{K\ddot{a}h}(X,D)$ to the space of formal deformations of $\fuk(X \setminus D)$; i.e., into the moduli space of formal solutions to the Maurer--Cartan equation in $CC^\bullet(\fuk(X \setminus D))$. 
This map is realized by an $L_\infty$ homomorphism
\begin{equation}
\label{eqn:colinf}
 SC^\bullet(X \setminus D) \dashrightarrow CC^\bullet(\fuk(X \setminus D)),
\end{equation}
where the left-hand side is the (cochain complex underlying the) \emph{symplectic cohomology} of $X \setminus D$ (see, e.g., \cite{Seidel:biased,Seidel2002}), equipped with its $L_\infty$ structure (see, e.g., \cite{Fabert2013,Fabert2019}). 
The leading term in this map, once again, is the closed--open map $\EuC\EuO$, and again, in nice situations, this map is a quasi-isomorphism. 
However, in contrast with the previous case, in general the $L_\infty$ structure on symplectic cohomology is neither formal nor abelian.

Symplectic cohomology has generators corresponding to Reeb orbits linking the components of $D$, as well as generators from the interior of $X \setminus D$. 
The basic meridian loop around a component $D_p$ of $D$ corresponds to a deformation direction in $\mkah(X,D)$: namely, by deforming $[\omega]$ in the direction of $PD(D_p)$ (where $PD$ denotes Poincar\'e duality). 
However, in general the space of deformations of $\fuk(X \setminus D)$ will be larger than $\overline{\cM}_{K\ddot{a}h}(X,D)$: for example, one can deform the category by turning on a non-exact symplectic form on $X \setminus D$ (which may not extend to a form on $X$). 
If we are going to prove a versality result, we will have to rule out these extra deformations by restricting the deformations we consider.

To do this, we first restrict ourselves to \emph{graded} deformations over $\overline{\cM}_{K\ddot{a}h}(X,D)$.
We also assume that $(X,D)$ is either \emph{Calabi--Yau}, \emph{positive} or \emph{semi-positive}.
This means roughly that there exists an effective anticanonical divisor supported on $D$ which is respectively torsion, ample or nef. 
Grading and positivity rule out the possibility of extra deformations altogether (with a minor exception that appears when the minimal Chern number of $X$ is $1$: see Remark \ref{rmk:chern1}). 
So our versality result is particularly simple in the positive case.

Calabi--Yau-ness implies that the only extra deformations come from $H^2(X \setminus D)$, which correspond to deformations of the symplectic form on $X \setminus D$.
Ruling out deformations in these directions is the most important step in proving a versality result in this case, and we do not know of a general way to do it.
One ad-hoc way is to exploit the symmetry of $X$ (as was done by Seidel in \cite{Seidel:HMSquartic}). 
We generalize this argument, to allow symmetries of $X$ which are anti-symplectic as well as symplectic.

This slight generalization has surprisingly broad applications. 
For example, it allows us to prove a versality result in the case of generalized Greene--Plesser mirrors. 
This is the crucial step in the proof of homological mirror symmetry in this case in \cite{SS}. 
It should similarly be possible to prove a versality result for Batyrev mirror Calabi--Yau hypersurfaces in toric varieties (at least when the hypersurfaces have dimension $\ge 3$), and therefore dramatically reduce the amount of work needed to prove homological mirror symmetry in that case also.

We remark that the semi-positive case can also be treated using our methods, combining aspects of the positive and Calabi--Yau cases.

In the rest of this introduction, we give more precise statements of our main results.

\subsection{Standing notation and conventions}
\label{subsec:notation}

Let $\Bbbk$ be a field of characteristic zero. 
Let $\Lambda$ be the universal Novikov field over $\Bbbk$:
\begin{equation} \Lambda := \left\{ \sum_{j=0}^\infty c_j \cdot q^{\lambda_j}: c_j \in \Bbbk, \lambda_j \in \R, \lim_{j \to \infty} \lambda_j = +\infty\right\}.\end{equation}

All of our algebraic objects will be $\G$-graded, where $\G = \{\Z \to G \to \Z/2\}$ is a \emph{grading datum} (see \cite{Sheridan:CY}): i.e., $G$ is an abelian group, and $\Z \to G \to \Z/2$ are homomorphisms of abelian groups whose composition is surjective. 
$\G$-graded objects (vector spaces, algebras\dots) are $G$-graded objects: when we say that the degree of the $A_\infty$ structure map $\mu^s$ is $2-s$, it really means that the degree is the image of $2-s$ in $G$. 
Koszul signs in all formulae come from the map $G \to \Z/2$.

In particular, for any symplectic manifold $M$, we have the grading datum $\G(M)$ defined in \cite{Sheridan:CY}: $G = H_1(\cG M)$, where $\cG M$ denotes the total space of the Lagrangian Grassmannian of $M$. 
The map $\Z \to G$ is induced by the inclusion of a fibre of the Lagrangian Grassmannian: $\Z \simeq H_1(\cG_x M) \to H_1(\cG M)$. 
The map $G \to \Z/2$ is given by the first Stiefel-Whitney class of the tautological bundle over $\cG M$. 

\subsection{The relative Fukaya category}
\label{subsec:relfukint}

Let $(X,D)$ be a $\snc$ pair, and $\omega$ a relative K\"ahler form on $(X,D)$. 
We will call this data $(X,D,\omega)$ a \emph{relative K\"{a}hler manifold}. 
The first version of the Fukaya category we consider is the \emph{affine Fukaya category} $\fuk(X \setminus D)$: it is the Fukaya category of the convex exact symplectic manifold $X \setminus D$ (it does not depend on the choice of relative K\"{a}hler form $\omega$ up to quasi-equivalence, hence the absence of $\omega$ from the notation). 
Its objects are compact exact Lagrangian branes $L \subset X \setminus D$.
It is $\G(X \setminus D)$-graded and $\Bbbk$-linear.

Next, we consider a version of the \emph{relative Fukaya category} $\fuk(X,D)$, following \cite{Seidel2002}. 
We also expect that this category does not depend on the relative K\"{a}hler form $\omega$, and on certain other auxiliary choices, although this is less straightforward than for the affine Fukaya category. 
For the purposes of this introduction we will write $\fuk(X,D)$, but in the rest of the paper we will emphasize the extra choices that the relative Fukaya category depends on where appropriate.
 
$\fuk(X,D)$ is $\G(X \setminus D)$-graded and $R$-linear, where $R$ is a $\Bbbk$-algebra that we will discuss. 
The objects of the relative Fukaya category are the same as those of the affine Fukaya category. 
The generators of the morphism spaces are also the same: they are intersection points between Lagrangians. 
The $A_\infty$ structure maps $\mu^*$ should be a sum of terms corresponding to pseudoholomorphic discs $u: \mathbb{D} \to X$, with boundary on the Lagrangian branes, weighted by certain coefficients in $R$.

The coefficient ring $R$ should be some Novikov-type ring, defined so that the weighted sums of discs make sense. 
We will take $R$ to be the completion of some subring $\Runcomp \subset \Bbbk[H_2(X,X \setminus D)]$: a pseudoholomorphic disc $u$ is then weighted by $\nov^{[u]} \in R$, where $[u] \in H_2(X,X \setminus D)$ is simply the homology class of the disc $u$, and the completion includes all infinite sums of discs that may arise.
Note that if the smooth irreducible components of the divisor $D$ are $\{D_p\}_{p \in P}$, then we have a natural isomorphism $H_2(X,X \setminus D) \simeq \Z^P$. 
Thus we can identify $\Bbbk[H_2(X,X \setminus D)] \simeq \Bbbk[\nov_1^{\pm 1},\nov_2^{\pm 1},\ldots]$, and $\nov^{[u]}$ gets identified with $\nov_1^{u \cdot D_1} \cdot \nov_2^{u \cdot D_2}\cdot \ldots$.

The subring $\Runcomp \subset \Bbbk[H_2(X,X \setminus D)]$ should be the span of certain monomials, including all monomials $\nov^{[u]}$ such that $[u]$ may arise as the homology class of a pseudoholomorphic disc. 
One might at first try taking $\Runcomp$ to be the span of all monomials $\nov^{[u]}$ such that $[u] \cdot D_p \ge 0$ for all $p \in P$, since any smooth pseudoholomorphic disc $u$ will have non-negative intersection number with each complex divisor $D_p$. 
This would give rise to a coefficient ring $R = \Bbbk\power{\nov_1,\nov_2,\ldots}$.

However, this approach has a fatal flaw. 
The moduli spaces of pseudoholomorphic discs defining the structure maps of the Fukaya category may include some \emph{nodal} discs (which would be counted `virtually'): these may have components which are holomorphic spheres contained inside the divisors $D_p$, and having negative intersection number with $D_p$. 
Thus, it may be possible for the nodal disc to have negative overall intersection number with some $D_p$.
These nodal discs would then have to be counted with a weight $\nov^{[u]} \notin \Runcomp$, so our definition would not make sense.
  
In order to resolve this problem, one must either rule out the appearance of holomorphic sphere components inside the components $D_p$, or one must arrange that any sphere inside $D_p$ has non-negative intersection number with it. 
The only practical way to achieve either of these is to impose some positivity assumption on the normal bundle $ND_p$ to each component $D_p$ of $D$. 
In the first case, such positivity rules out bubbling of holomorphic spheres inside $D_p$ from one-dimensional moduli spaces, for index reasons. 
In the second, if $u$ is a holomorphic sphere inside $D_p$, then $u \cdot D_p = c_1(ND_p)(u) \ge 0$, so one does not have the problem of nodal discs whose intersection number with $D_p$ is negative.

This was the approach taken in \cite{Sheridan:CY}, in the case where $X$ was a Fermat hypersurface in projective space and $D$ was the union of coordinate hyperplane sections (each of which is clearly ample, hence has positive normal bundle).
However, there are many interesting cases of  mirror symmetry which come with a natural choice of divisor $D \subset X$, and in the vast majority of these cases $D$ has many non-ample components $D_p$. 
One would like to define a version of the relative Fukaya category that works in these cases too. 

We have
\[ H^2(X,X \setminus D;\R) \simeq \R\left\langle PD(D_p) \right\rangle,\]
where `$PD$' denotes Poincar\'e duality. 
We will identify $H^2(X,X \setminus D;\R)$ with the group of $\R$-divisors supported on $D$, by identifying the class $\sum_p \ell_p \cdot PD(D_p)$ with the divisor $\sum_p \ell_p \cdot D_p$. 

\begin{defn}
\label{defn:nefampxd}
We define $Nef(X,D) \subset H^2(X,X \setminus D;\R)$ to be the closed cone of effective nef divisors supported on $D$. 
We define $Amp(X,D) \subset H^2(X,X \setminus D;\R)$ to be the open cone of ample divisors $\sum_p \ell_p \cdot D_p$ with $\ell_p > 0$ for all $p$.
\end{defn}

\begin{rmk}
Any relative K\"ahler form $\omega$ has a cohomology class $[\omega] \in H^2(X,X \setminus D;\R)$. 
A cohomology class can be realized by a relative K\"ahler form if and only if it lies in $Amp(X,D)$ (see Lemma \ref{lem:ampkahl}). 
Since we assume that $(X,D)$ admits a relative K\"ahler form, it follows that $Amp(X,D)$ is non-empty. 
It follows that $Amp(X,D)$ is the interior of $Nef(X,D)$, and that the closure of $Amp(X,D)$ is $Nef(X,D)$, by a theorem of Kleiman \cite{Kleiman1966}, see \cite[Theorem 1.4.23]{Lazarsfeld2004}. 
\end{rmk}

The relative Fukaya category should be thought of as living over the `space of all relative K\"ahler forms $\omega$'. 
In practice, it will be convenient for us to consider a `smaller' object: the relative Fukaya category living over `the space of all relative K\"ahler forms $\omega$ whose cohomology class lies in the interior of a fixed convex sub-cone $\Nef \subset Nef(X,D)$'. 
The reason for this will be explained in Remark \ref{rmk:needsys}.

\begin{defn}
\label{defn:NEXD}
Let $\Nef \subset Nef(X,D)$ be a convex sub-cone with non-empty interior. 
We define $NE(\Nef)_\R \subset H_2(X,X \setminus D;\R)$ to be the dual cone to $\Nef$, and $NE(\Nef) := NE(\Nef)_\R \cap H_2(X,X \setminus D)$.
When $\Nef = Nef(X,D)$ is as big as possible, we define $NE(X,D) := NE(Nef(X,D))$.
\end{defn}

Any $E$ in the interior of $Nef(X,D)$ corresponds to an effective ample divisor, and hence can be perturbed to a nearby smooth one by Bertini's theorem. 
It follows that $u \cdot E \ge 0$ for any holomorphic disc $u$ with boundary in $X \setminus D$ (including the nodal ones). 
Therefore $[u] \in NE(X,D) \subset NE(\Nef)$ for all holomorphic discs $u$.

\begin{example}
\label{eg:Dsmooth}
Suppose that $D \subset X$ is a smooth ample divisor. 
In this case $H_2(X,X \setminus D) \simeq \Z$, and $NE(X,D) \simeq \Z_{\ge 0}$. 
\end{example}

Since we assume that $\Nef$ has non-empty interior, $NE(\Nef)$ is a strongly convex cone (i.e., does not contain a line). 
It follows that
\begin{equation} \Runcomp(\Nef) := \Bbbk[NE(\Nef)]\end{equation}
has a unique toric maximal ideal $\fmuncomp \subset \Runcomp(\Nef)$. 
We define $R(\Nef)$ to be the $\fmuncomp$-adic completion of $\Runcomp(\Nef)$, and $\fm \subset R(\Nef)$ the corresponding maximal ideal. 
We abbreviate $R(\Nef)$ by $R$ where confusion is unlikely.
Thus, we arrive at a definition of the relative Fukaya category $\fuk(X,D,\Nef)$: it is $\G(X \setminus D)$-graded, $R(\Nef)$-linear, and pseudoholomorphic discs $u$ are weighted by $\nov^{[u]} \in R(\Nef)$.
We define $\fuk(X,D) := \fuk(X,D,Nef(X,D))$, which is linear over $R(X,D) := R(Nef(X,D))$.

It is important to remark that we do not give a general construction of the relative Fukaya category in this paper. 
However, in \S \ref{subsec:relfuk} we state some of its desired properties as Assumptions \ref{ass:relfuk1}, \ref{ass:relfukco} and \ref{ass:relfukambco}, and sketch justifications for them; our main results are proved under these assumptions. 

The reason we take this approach is that the general construction of the Fukaya category requires a more sophisticated technical framework than the classical pseudoholomorphic curve theory we use in this paper, and there are a number of different such technical frameworks available. 
Assumptions \ref{ass:relfuk1}--\ref{ass:relfukambco} should hold independently of the technical framework. 
Thus, anyone who wants to apply our results in their own work need only implement the sketched justifications of our assumptions within their preferred technical framework.

\begin{rmk}
\label{rmk:assint}
A construction of the relative Fukaya category satisfying Assumptions \ref{ass:relfuk1}--\ref{ass:relfukambco} will be implemented in the work-in-preparation \cite{Perutz2015a}, using Cieliebak--Mohnke's `stabilizing divisors' framework \cite{Cieliebak2007}.
\end{rmk}

\begin{defn}
\label{defn:mkah}
We define the affine $\Bbbk$-scheme
\[ \overline{\cM}_{K\ddot{a}h}(X,D,\Nef) := \spec(R(\Nef)).\]
Thus $\fuk(X,D,\Nef)$ is a category `defined over $\overline{\cM}_{K\ddot{a}h}(X,D,\Nef)$'.
Similarly, $\fuk(X,D)$ is a category defined over $\overline{\cM}_{K\ddot{a}h}(X,D)$. 
It is clear that $R(X,D) \subset R(\Nef)$, so we have a morphism $\overline{\cM}_{K\ddot{a}h}(X,D,\Nef) \to \overline{\cM}_{K\ddot{a}h}(X,D)$.
\end{defn}

\begin{rmk}
Definition \ref{defn:mkah} contradicts what we said in \S \ref{subsec:msphil}, which was that $\spec(R)$ would only be a formal neighbourhood of a large-volume limit point of the compactified relative K\"ahler moduli space. 
So perhaps we should use a different notation for $\spec(R)$.
However, since we have no use in this paper for the rest of the relative K\"ahler moduli space, and $\spec(R)$ will appear repeatedly, we prefer to reserve the simpler notation for it. 
We hope no confusion will result.
\end{rmk}

\begin{defn}
We define $\Runcomp_0 \subset \Runcomp$ to be the degree-$0$ component of $\Runcomp$, and similarly $R_0 \subset R$ to be its completion.
\end{defn}

\begin{defn}
\label{defn:Rcl}
We define $\Runcomp_{cl} \subset \Runcomp$ to be the direct sum of graded pieces whose grading lies in the image of $\Z \to H_1(\cG(X \setminus D))$; and we define $R_{cl} \subset R$ to be its completion. 
If $c_1(TX)|_{X \setminus D}$ is torsion, then the map $\Z \to H_1(\cG(X \setminus D))$ is injective (see Lemma \ref{lem:c1tors}), and $R_{cl}$ retains a natural $\Z$-grading.
\end{defn}

\begin{example}
If $D$ is smooth, then $\Runcomp_{cl} \simeq \Bbbk[q]$ (compare Example \ref{eg:Dsmooth}). 
If $|q| = 0$ then $R_{cl} \simeq \Bbbk\power{q}$. 
If $|q| \neq 0$ then $R_{cl} \simeq \Bbbk[q]$, because we take the completion in the graded sense (i.e., we take the completion in each graded piece independently).
\end{example}

\begin{rmk}
The long exact sequence in homology for the pair $(X,X \setminus D)$ shows that we can identify $\Runcomp_{cl}$ with $\Bbbk[H_2(X)/H_2(X \setminus D)]$ (the $\Z$-grading from Definition \ref{defn:Rcl} is given by $2c_1$).
Thus $R_{cl}$ is a natural coefficient ring for defining the closed Gromov--Witten invariants. 
Of course this motivates the notation $NE(X,D)$, since in classical Gromov--Witten theory one takes the coefficient ring to be the completion of the group ring of $NE(X)$, the cone of effective curve classes. 
Indeed one could call $NE(X,D)$ the `cone of effective relative curve classes'.
We remark that in the same way that we find it convenient to replace $NE(X,D)$ by $NE(\Nef)$, in Gromov--Witten theory it is also sometimes useful to replace $NE(X)$ by a larger convex cone which is finitely-generated, see e.g. \cite{Gross2015}.
\end{rmk}

It turns out to be very helpful to us if the coefficient ring $R(\Nef)$ is `non-negatively graded' in a certain sense.

\begin{defn}\label{defn:Npos}
We say that $\Nef$ is:
\begin{itemize}
\item \emph{positive} if $c_1(TX) \in H^2(X;\R)$ lies in the image of the interior of $\Nef$
(it follows that the anticanonical bundle of $X$ is ample, i.e., that $X$ is Fano);
\item \emph{Calabi--Yau} if $c_1(TX)$ is torsion, i.e., vanishes in $H^2(X;\R)$ (it follows that some tensor power of the canonical bundle is trivial);
\item \emph{semi-positive} if $c_1(TX) \in H^2(X;\R)$ lies in the image of $\Nef$
(it follows that the anticanonical bundle of $X$ is nef: some authors say that $X$ is `semi-Fano' in this case).
\end{itemize}
We say that $(X,D)$ is positive/Calabi--Yau/semi-positive if $Nef(X,D)$ is. 
Note that Calabi--Yau $\implies$ semi-positive.
\end{defn}

\begin{example}
\label{eg:smoothpos}
Suppose that $D \subset X$ is smooth and ample. 
Then $(X,D)$ is positive/Calabi--Yau/semi-positive if and only if we have $c_1(TX) = \kappa \cdot PD(D)$ in $H^2(X;\Q)$, where $\kappa \in \Q$ is positive/zero/non-negative.
\end{example}

\begin{rmk}
If, in the setting of Example \ref{eg:smoothpos}, $\kappa - 1$ is negative/zero/positive then we have $K_X + D > 0$/$K_X+D=0$/$K_X+D<0$, which implies that the affine variety $X \setminus D$ is of log general type/log Calabi--Yau/log Fano. 
This condition will come up in Proposition \ref{prop:defgenint}.
\end{rmk}

\begin{rmk}
\label{rmk:gradpositivity}
If $\Nef$ is positive, Calabi--Yau or semi-positive, then it is clear that $c_1(TX)$ lies in 
\[ \im\left(H^2(X,X \setminus D;\Q) \to H^2(X;\Q)\right) = \ker\left(H^2(X;\Q) \to H^2(X \setminus D;\Q)\right) \]
by the long exact sequence in cohomology for the pair $(X,X \setminus D)$, and hence that $c_1(TX)|_{X \setminus D}$ is torsion. 
It follows that we have a $\Z$-grading on $R_{cl}$, as remarked in Definition \ref{defn:Rcl}. 
It is simple to check that, if $\Nef$ is positive/Calabi--Yau/semi-positive, then the maximal ideal $\fm_{cl} \subset R(\Nef)_{cl}$ is concentrated in positive/zero/non-negative degrees.
\end{rmk}

Now we should also be able to associate to the symplectic manifold $(X,\omega)$ the \emph{absolute Fukaya category} $\fuk(X,\omega)$ (see \cite{FO3}). 
Its objects are arbitrary Lagrangian branes in $X$, and it is $\Lambda$-linear where $\Lambda$ is the universal Novikov field. 
It is independent of the choice of K\"ahler potential for the relative K\"ahler form $\omega$.

Let us explain how the absolute Fukaya category should be related to the relative Fukaya category. 
Suppose that $[\omega] \in H^2(X,X \setminus D;\R)$ is the cohomology class of a relative K\"ahler form, and it lies in the interior of $\Nef \subset Nef(X,D)$. 
Then there is a well-defined adically continuous $\Bbbk$-algebra homomorphism
\begin{align*}
p(\omega)^*: R(\Nef) & \to  \Lambda, \quad \text{defined by}\\
p(\omega)^*(r^u) &= q^{\omega(u)}.
\end{align*}
More explicitly, if $[\omega] = \sum_p c_p \cdot PD(D_p)$, then this map sends $\nov_p \mapsto q^{c_p}$. 
Note that the map will not converge unless $[\omega]$ lies in the interior of $\Nef$.

In other words, there is a $\Lambda$-point $p(\omega)$ of $\overline{\cM}_{K\ddot{a}h}(X,D,\Nef)$. 
We can take the fibre of $\fuk(X,D,\Nef)$ over $p(\omega)$, which is a $\Lambda$-linear $A_\infty$ category which we denote by $\fuk(X,D,\Nef)_{p(\omega)}$. 
Explicitly,
\[ \fuk(X,D,\Nef)_{p(\omega)} := \fuk(X,D,\Nef) \otimes_R \Lambda,\]
where $\Lambda$ is regarded as an $R$-algebra via $p(\omega)^*$: so we simply use the homomorphism $p(\omega)^*$ to change coefficients from $R$ to $\Lambda$ in our $A_\infty$ category.

There should be an embedding of this fibre of the relative Fukaya category into the absolute Fukaya category:
\begin{equation} 
\label{eqn:relabsembed}
\fuk(X,D,\Nef)_{p(\omega)} \hookrightarrow \fuk(X,\omega).\end{equation}

\begin{rmk}
\label{rmk:assint2}
As already remarked, we do not give a general construction of the relative Fukaya category in this paper, so we do not prove the existence of the embedding \eqref{eqn:relabsembed}. 
Rather, we state it as Assumption \ref{ass:relabs} (more precisely, we state a slightly more elaborate version taking bounding cochains into account), and again sketch its justification. 
We remark that the relative and absolute Fukaya categories can be constructed using classical pseudoholomorphic curve theory, and shown to satisfy Assumptions \ref{ass:relfuk1}--\ref{ass:relabs}, in the following situations:
\begin{itemize}
\item $(X,\omega)$ is positively monotone (i.e., $c_1(TX) = \kappa \cdot [\omega]$ in $H^2(X;\R)$ for some $\kappa > 0$). This is implemented in \cite{Sheridan:Fano}, up to minor differences in conventions.
\item $(X,\omega)$ is Calabi--Yau (i.e., $c_1(TX)$ is torsion), and we restrict our attention to \emph{strictly unobstructed} Lagrangians: i.e., pairs $(L,J_L)$ where $L$ is a Lagrangian brane and $J_L$ is an almost-complex structure such that $L$ bounds no $J_L$-holomorphic discs and intersects no $J_L$-holomorphic spheres. 
The constructions of \cite{Sheridan:Fano} go through in this case, with minor modifications to rule out unstable sphere and disc bubbling as in \cite{Seidel:HMSquartic,Seidel:Flux}).
\end{itemize}
We remark that the condition of strict unobstructedness is generic in $J_L$ for $\dim_\C(X) \le 2$, so every Lagrangian brane $L$ can be upgraded to a strictly unobstructed Lagrangian $(L,J_L)$ in that case (compare \cite{Seidel:HMSquartic}). 
In higher dimensions it may not be practical to verify that a given $L$ can be equipped with such a $J_L$. 
\end{rmk}

\begin{rmk}
In general, one needs virtual perturbation theory to define the absolute Fukaya category (see \cite{FO3}).
\end{rmk}

In fact, the relative Fukaya category $\fuk(X,D,\Nef)$ should contain information about the absolute Fukaya category $\fuk(X,\omega')$ for all relative K\"{a}hler forms $\omega'$ whose cohomology class lies in the interior of $\Nef$ (and in particular, $\fuk(X,D)$ contains information about the absolute Fukaya category of $X$ equipped with any relative K\"ahler form). 
This may seem a little strange at first, since submanifolds which are Lagrangian for the symplectic form $\omega$ need not be Lagrangian for $\omega'$: however, linear interpolation between the two relative K\"{a}hler forms induces a deformation of the corresponding Liouville domains, so exact Lagrangians can be deformed with the symplectic form to remain Lagrangian.

\subsection{Versality} 
\label{subsec:versint}

The relative Fukaya category $\fuk(X,D,\Nef)$ is a \emph{deformation} of the affine Fukaya category $\fuk(X \setminus D)$ over $R$. 
The main aim of this paper is to study this deformation problem, and establish general criteria under which it is \emph{versal}.
Let us explain what we mean by versality purely in algebraic terms.

Let $R$ be a $\Bbbk$-algebra arising from the following construction: we are given a convex cone $\Nef \subset \R^P$ with non-empty interior, for some finite set $P$. 
We let $\Runcomp := \Bbbk[\Nef^\vee \cap \Z^P]$ (if $\Nef$ is rational polyhedral, this will be the $\Bbbk$-algebra of functions on the corresponding affine toric variety), and we equip it with a $\G$-grading. 
This algebra has a unique toric maximal idea $\fmuncomp \subset \Runcomp$, corresponding to the vertex of the cone $\Nef^\vee$, and we define $R$ to be the $\G$-graded $\fmuncomp$-adic completion of $\Runcomp$. 
Of course, the coefficient ring $R$ of the relative Fukaya category, as sketched in the previous section, is an example of such an algebra.

Our arguments require certain technical hypotheses on $R$. 
If $R$ satisfies these hypotheses, we say that $R$ is `nice' (Definition \ref{defn:Rnice}).
We will explain presently how to ensure that the coefficient ring of the relative Fukaya category is nice.

Suppose that $\EuA$ is a $\G$-graded $A_\infty$ category over $\Bbbk$, with $A_\infty$ structure maps $\mu^*_0$. 
A $\G$-graded deformation of $\EuA$ over $R$ is a $\G$-graded $R$-linear (possibly curved) $A_\infty$ category $\EuA_R = (\EuA \ctens R, \mu^*_R)$, with $\mu^*_R \equiv \mu^*_0 \text{ (mod $\fm$)}$. 
The hat over the tensor product denotes the $\fm$-adic completion. 
Note that all morphism spaces in $\EuA_R$ are free $R$-modules. 
We denote the set of such deformations by $\mathfrak{A}_R(\EuA)$. 
 
To such a deformation $\EuA_R$, there are associated \emph{first-order deformation classes} $\defa_p \in \HH^\bullet(\EuA)$. 
Namely, if we expand the $A_\infty$ structure maps as
\begin{equation}
\label{eqn:mustarr} 
\mu^*_R = \sum_{u \in \Nef^\vee \cap \Z^P} \nov^u \cdot \mu^*_u,
\end{equation} 
then the $A_\infty$ relations show that the Hochschild cochains $\mu^*_{u_p}$ are closed, where $u_p$ are the generators of $\Z^P$ (this requires $R$ to be nice). 
We define $\defa_p := [\mu^*_{u_p}]$: its grading is such that $\nov_p \defa_p \in \HH^2\left(\EuA,\EuA \otimes \fmuncomp\right)$, where $\nov_p := \nov^{u_p}$.
 
We consider $\Aut(R)$, the group of $\G$-graded $\Bbbk$-algebra automorphisms of $R$ that respect the $\fm$-adic filtration. 
Note that $\Aut(R)$ acts on the set of deformations of $\EuA$ over $R$: one applies $\psi^* \in \Aut(R)$ to each of the monomials $\nov^u \in R$ appearing in the expansion \eqref{eqn:mustarr}.

Our versality results will only apply to a certain subset $\mathfrak{A}'_R(\EuA) \subset \mathfrak{A}_R(\EuA)$, consisting of those deformations whose first-order deformation classes $\nov_p \defa_p$ span $\HH^2(\EuA,\EuA \otimes \fmuncomp/\fmuncomp^2)$.

\begin{defn}
\label{defn:vers1int}
We say that $\EuA_R$ is an $R$-\emph{complete} deformation of $\EuA$ if, for any $\EuB_R \in \mathfrak{A}'_R(\EuA)$, there exist
\begin{itemize}
\item An automorphism $\psi^* \in \Aut(R)$;
\item A (possibly curved) $A_\infty$ functor
\begin{equation} F^*: \EuB_R \dashrightarrow \psi^* \EuA_R,\end{equation}
\end{itemize}
such that $F^* \equiv \id \mbox{ (mod $\fm$)}$. 

If the automorphism $\psi^*$ is uniquely determined for each $\EuB_R$, then we say that $\EuA_R$ is \emph{$R$-universal}; if the map $\psi^*:\fm/\fm^2 \to \fm/\fm^2$ is uniquely determined for each $\EuB_R$, we say that $\EuA_R$ is \emph{$R$-versal}.
\end{defn}

\begin{rmk}
An $R$-versal deformation $\EuA_R$ need not be versal in the usual sense: i.e., there may exist deformations of $\EuA$ over another ring $S$ which are not the pullback of $\EuA_R$ under any map $R \to S$. 
\end{rmk}

\begin{rmk}
If $\Nef \subset Nef(X,D)$ is positive in the sense of Definition \ref{defn:Npos}, then each generator $r_p \in R(\Nef)$ is the unique monomial in its graded piece of $R(\Nef)$. 
As a consequence, any graded automorphism $\psi^* \in \Aut(R(\Nef))$ necessarily sends each $r_p$ to a non-zero multiple of itself; thus $\psi^*$ is uniquely determined by $\psi^*: \fm/\fm^2 \to \fm/\fm^2$.
In particular the `mirror maps' appearing in our versality theorems in these cases are determined by first-order data.
\end{rmk}

The categories (and functors) that we consider may be \emph{curved}, i.e., the $A_\infty$ structure map $\mu^0_L \in hom^2_{\EuA_R}(L,L)$ may not vanish, for certain objects $L$. 
However, we always assume that the curvature $\mu^0_L$ is of order $\fm$ (for the relative Fukaya category, this is ensured by exactness of the Lagrangians in $X \setminus D$).
There is a standard way of producing a non-curved $A_\infty$ category $\EuA_R^\bc$ from the curved $A_\infty$ category $\EuA_R$: the objects of $\EuA_R^\bc$ are pairs $(L,\mca)$, where $L$ is an object of $\EuA_R$ and $\mca$ is a \emph{bounding cochain} for $L$ (see \cite{FO3}). 
Bounding cochains are required to be of order $\fm$.

The existence of a curved $A_\infty$ homomorphism $F^*$ satisfying $F^* \equiv \id \mbox{ (mod $\fm$)}$ implies the existence of a non-curved $A_\infty$ isomorphism $ \EuB_R^{\bc} \dashrightarrow \psi^*\EuA_R^{\bc}$ (see \S \ref{subsec:bc}).
In particular, $R$-versality tells us that $\EuA_R^\bc$ is uniquely determined by $\EuA$ together with the first-order deformation classes, up to a formal change of variables in the coefficient ring $R$.

Now we explain how to apply this deformation theory to the relative Fukaya category. 
First, we must ensure that the coefficient ring is nice.  
We give a set of technical hypotheses on the cone $\Nef$ (see Definition \ref{defn:xdenice}), and prove that they imply that $R(\Nef)$ is nice.
We show that a relative K\"ahler class $\omega$  admits a compatible system of divisors which is nice if and only its cohomology class lies in a certain open convex sub-cone $\AAmp(X,D) \subset Amp(X,D)$ (Lemma \ref{lem:niceAAmp}).
This means that our results can only give information about the absolute Fukaya category of $X$ equipped with a relative K\"ahler form whose cohomology class lies in $\AAmp(X,D)$. 

\begin{rmk}
\label{rmk:needsys}
Typically, $Nef(X,D)$ is not nice. 
This is the main reason that we need to introduce the cone $\Nef$ into the story.
\end{rmk}

An important source of examples is the following. 
Let $Y$ be a toric variety, and $D' \subset Y$ the toric boundary divisor. 
Let $X \subset Y$ be a smooth subvariety, and $D:= D' \cap X$ the (transverse) intersection of $X$ with $D'$. 
We say that a relative K\"ahler form on $(X,D)$ is \emph{ambient} if it is restricted from a relative K\"ahler form on the ambient toric variety. 

In general, we only know how to use our versality results to get information about the absolute Fukaya of $X$ equipped with ambient K\"ahler forms.
If each component of $D'$ intersects $X$ either in an empty or a connected divisor, we prove that all ambient relative K\"ahler classes lie in $\AAmp(X,D)$ (Corollary \ref{cor:toricEabc}). 
We don't know how to show that any non-ambient classes lie in $\AAmp(X,D)$, so we don't know how to apply our versality results to them.
 
\begin{rmk}
If there's a component of $D'$ which intersects $X$ in a disconnected divisor, then we don't know anything about $\AAmp(X,D)$: for all we know it may even be empty. 
However all is not lost: we define the \emph{ambient relative Fukaya category} $\fuk_{amb}(X,D,\Nef)$, a modified version which only contains information about the absolute Fukaya category of $X$ equipped with ambient relative K\"ahler forms. 
It is defined over a quotient $R_{amb}(\Nef)$ of $R(\Nef)$, which corresponds to a sub-scheme $\overline{\cM}_{amb\text{-}K\ddot{a}h}(X,D,\Nef) \subset \overline{\cM}_{K\ddot{a}h}(X,D,\Nef)$, the subspace of relative K\"ahler forms whose cohomology class lies in the image of the restriction map $H^2(Y,Y \setminus D') \to H^2(X,X \setminus D)$ (note that this map is surjective if and only if each component of $D'$ intersects $X$ in an empty or a connected divisor). 
We define a condition on $\Nef$ called \emph{amb-niceness} (Definition \ref{defn:ambnice}) which ensures that $R_{amb}(\Nef)$ is nice, and prove that the cohomology class of any ambient relative K\"ahler form lies in the interior of an amb-nice cone $\Nef$ (Lemma \ref{lem:ambniceAAmp}).
We prove versality results for the ambient relative Fukaya category parallel to those we prove for the ordinary relative Fukaya category (see Theorem \ref{thm:eqversalityamb} and Corollary \ref{cor:geomhyp2amb}). 
We will not mention these variations on our results further in this introduction. 
\end{rmk}

\begin{rmk}
The hypothesis of `niceness' of $\Nef$ is significantly stronger than it needs to be for $R(\Nef)$ to be nice. 
That is because it serves other important purposes in this paper: it is used to help verify the hypothesis of the versality criterion by ruling out deformation directions of the relative Fukaya category (see, e.g., Lemma \ref{lem:defgen}), and also to rule out certain types of sphere bubbling in moduli spaces (see, e.g., Lemma \ref{lem:coalg}). 
It would have been possible to introduce a series of related hypotheses on $\Nef$, adapted to each of these intended purposes, but it would have been unwieldy and resulted in no extra generality in the case of subvarieties of toric varieties.
\end{rmk}

\begin{rmk}
It is also true that $R(X,D) = R(\Nef)$ is nice if $D$ is smooth, although $\Nef$ will never be nice in this case according to our definition. 
We will treat the `$D$ smooth' case separately throughout.
\end{rmk}

We are now ready to state our first main versality theorem. 
Let $\EuA_R \subset \fuk(X,D,\Nef)$  be a full subcategory of the relative Fukaya category, and $\EuA \subset \fuk(X \setminus D)$ the corresponding full subcategory of the affine Fukaya category.
The versality criterion has to do with $\HH^2\left(\EuA,\EuA \otimes \fmuncomp\right)$, which is where the deformation classes $\nov_p \defa_p$ of $\EuA_R$ live. 

\begin{main}[= Theorem \ref{thm:versality}]
\label{thm:1}
Suppose that:
\begin{itemize}
\item $\Nef$  is nice, or that $D$ is smooth and $(X,D)$ semi-positive;
\item The classes $\nov_p \defa_p$ span $\HH^2\left(\EuA,\EuA \otimes \fmuncomp\right)$ as $\Runcomp_0$-module.
\end{itemize}
Then $\EuA_R$ is an $R$-complete deformation of $\EuA$. 

If furthermore $\defa_p \neq 0$ for all $p$, then $\EuA_R$ is $R$-versal.
\end{main}

Theorem \ref{thm:1} is proved using a standard order-by-order deformation theory argument for the $\fm$-adic filtration on $R$ (compare \cite[Lemma 3.5]{Seidel:HMSquartic}). 
The basic idea is that $\HH^2(\EuA,\EuA \otimes \fmuncomp)$ is the obstruction space for the deformation problem, and since it is spanned by the first-order deformation classes, any higher-order obstructions can be deformed away by modifying the automorphism $\psi^*$.

Of course, if Theorem \ref{thm:1} is to be at all useful, its hypotheses should be satisfied reasonably often. 
We now give a geometric criterion for the hypotheses of our versality theorem to be satisfied, which uses the notion of semi-positivity and also the \emph{symplectic cohomology} of the complement of $D$, $SH^\bullet(X \setminus D;\Bbbk)$ (see, e.g., \cite{Seidel:biased,Viterbo1999}).

\begin{rmk}
\label{rmk:caution}
The cautious reader should be advised that the proofs of the geometric criteria we are about to state are not presented in the same level of detail as the rest of the proofs in the paper. 
They are some of the most important motivations and ideas in this paper: they tell us when we should hope to be able to apply Theorem \ref{thm:1}, and motivate a generalization (Theorem \ref{thm:2}) which is of much wider applicability. 
For our immediately intended applications however (see, e.g., \cite{SS}), it is easier to verify the hypotheses of Theorems \ref{thm:1} and \ref{thm:2} directly, so the geometric criteria are unnecessary; furthermore, detailed proofs would involve lengthy technicalities (in particular: modifications of certain results from the works-in-preparation \cite{Borman2015,Borman2016}; modifications of certain results from \cite{Ganatra2016a}, which uses a somewhat different technical setup from the present paper; and a treatment of the Morse--Bott approach to symplectic cohomology in the present context). 
Therefore we have decided simply to sketch the proofs. 
The results whose proofs are only sketched are confined to \S \ref{sec:SH} and the corresponding results in this \S \ref{sec:int}.
\end{rmk}

The symplectic cohomology $SH^\bullet(X \setminus D;\Bbbk)$  is a $\G$-graded $\Bbbk$-algebra (the $\G$-grading  is explained in \S \ref{subsec:SH}). 
It is the cohomology of a cochain complex $SC^\bullet(X \setminus D;\Bbbk)$. 
There is a subcomplex $C^\bullet(X \setminus D;\Bbbk) \subset SC^\bullet(X \setminus D;\Bbbk)$ whose cohomology is $H^\bullet(X \setminus D;\Bbbk)$, and we denote the quotient complex by $SC^\bullet_+(X \setminus D;\Bbbk)$. 
This results in a long exact sequence
\[ \ldots \to  H^\bullet(X \setminus D;\Bbbk) \xrightarrow{i} SH^\bullet(X \setminus D;\Bbbk) \to SH^\bullet_+(X \setminus D;\Bbbk) \to \ldots\]
(see \cite{Viterbo1999} and \cite{BEE}). 
The map $i$ is a graded $\Bbbk$-algebra homomorphism.
We will also consider the \emph{closed--open string map}
\begin{equation} \EuC\EuO: SH^\bullet(X \setminus D;\Bbbk) \to \HH^\bullet(\fuk(X \setminus D)).\end{equation}
 (see \cite{Seidel2002,Ganatra2013} for the definition: it counts pseudoholomorphic discs with an interior puncture, asymptotic to an orbit which is a generator of symplectic cohomology). 

Under certain assumptions on $(X,D)$ (including when it admits a nice cone $\Nef$, or when $D$ is smooth and $(X,D)$ is semi-positive), one can define certain classes $\mathsf{d}_p \in SH^\bullet(X \setminus D;\Bbbk)$ corresponding to basic orbits going around the components $D_p$ of $D$ (compare \cite{Ganatra2016a}). 
Their degree is such that $\nov_p \mathsf{d}_p \in SH^2(X \setminus D;\fm)$, and furthermore the first-order deformation classes of the relative Fukaya category are equal to $\EuC\EuO(\mathsf{d}_p)$.

\begin{propstar}[= Lemma \ref{lem:defgen}]
\label{prop:defgenint} 
Suppose that either:
\begin{itemize}
\item $\Nef$ is nice and semi-positive; or
\item $D$ is smooth, $(X,D)$ is semi-positive, and the constant $\kappa \in \Q$ from Example \ref{eg:smoothpos} satisfies $0 \le \kappa < 1$.
\end{itemize}
Then the images of the classes $\nov_p \mathsf{d}_p$ are a $\Runcomp_0$-basis for $SH^2_+\left(X \setminus D; \fmuncomp\right)$.
\end{propstar}

This allows us to give a geometric criterion for when the hypothesis of Theorem \ref{thm:1} is satisfied. 
The idea is to assume that $\EuC\EuO$ is surjective, so that when the hypotheses of Proposition \ref{prop:defgenint} are satisfied, $\HH^2\left(\EuA,\EuA \otimes \fmuncomp\right)$ is spanned by the image of the deformation classes together with classes in the image of $H^2(X \setminus D;\fmuncomp)$. 
The latter classes correspond to unwanted `non-geometric' directions in which $\EuA_R$ could be deforming at higher order, over which we have no control: so if we are going to apply Theorem \ref{thm:1}, we must assume that they do not exist.

\begin{corstar}[= Corollary \ref{cor:geomhyp}]
\label{cor:geomhypint}
Consider the situation of Theorem \ref{thm:1}. 
Suppose furthermore that:
\begin{itemize}
\item Either $\Nef$ is nice and  semi-positive, or $D$ is smooth and $(X,D)$ is semi-positive with $0 \le \kappa < 1$;
\item $H^2\left(X \setminus D;\fmuncomp\right) =0$;
\item The map
\begin{equation}
\label{eqn:cohopesurji}
\EuC \EuO \otimes \fmuncomp :SH^2\left(X \setminus D;\fmuncomp\right) \to \HH^2\left(\EuA,\EuA \otimes \fmuncomp \right)\end{equation}
is surjective (e.g., $\EuC\EuO$ is surjective).
\end{itemize}
Then $\EuA_R$ is an $R$-complete deformation of $\EuA$. 

If furthermore \eqref{eqn:cohopesurji} is an isomorphism (e.g., if $\EuC\EuO$ is an isomorphism), then $\EuA_R$ is $R$-versal.
\end{corstar}
 
We now consider the hypothesis $H^2\left(X \setminus D;\fmuncomp\right) =0$. 
The following easy Lemma gives two situations where this can be achieved:

\begin{lem}[= Lemma \ref{lem:nodeg2}]
\label{lem:nodeg2int}
Suppose that one of the following holds:
\begin{itemize}
\item $\Nef$ is Calabi--Yau, and $H^2(X\setminus D;\Q) = 0$;
\item $\Nef$ is positive, and $X$ has minimal Chern number $\ge 2$.
\end{itemize}
Then $H^2\left(X \setminus D;\fmuncomp\right) = 0$.
\end{lem}

\begin{rmk}
\label{rmk:chern1}
If $\Nef$ is positive and $X$ has minimal Chern number $1$, then $H^2(X \setminus D;\fmuncomp)$ may contain classes $\nov^v \cdot e$ where $v$ is the class of an effective curve with Chern number $1$, and $e$ is the identity. 
For example, such classes exist for Fano-index-$1$ hypersurfaces in projective space, which were considered in \cite{Sheridan:Fano}. 
There it was discovered that the resulting extra deformation directions $\nov^v \cdot e$ corresponded to the appearance of curvature terms $\mu^0$, which count holomorphic discs passing through $D$ with boundary on a single Lagrangian. 
So in this case the ambiguity introduced by the extra deformation directions coming from Chern-number-$1$ spheres was rather slight, it essentially amounted to determining which eigenvalue of quantum cup product with $c_1$ our Lagrangians lay in. 
This is likely to be a general phenomenon (because the map $\EuC\EuO$ is unital).
\end{rmk}

\begin{rmk}
\label{rmk:gentype}
The case that $\Nef$ is not semi-positive is also interesting (it is central to the very interesting area of homological mirror symmetry for varieties of general type \cite{Kapustin2009}).
One may still have a hope that the classes $\nov_p \mathsf{d}_p$ span $SH^2_+\left(X \setminus D; \fmuncomp\right)$, but it is much less likely. 
That is because $\fmuncomp_{cl}$ will include classes in \emph{negative} degree: so although the indices of orbits tend to increase as they wrap more times around the divisor, this can be compensated by the negative-degree classes in $\fmuncomp_{cl}$ to bring the degree back down to $2$, so we expect there to be a large number of non-geometric `directions' for $\EuA$ to deform in. 
One can still hope to prove versality results in the presence of such non-geometric deformation directions, but they will be of a different nature to those we prove in this paper (see, e.g., \cite{Seidel:g2}).  
We will not consider generalizations in this direction in this paper.
\end{rmk}

In the case when $\Nef$ is Calabi--Yau and $H^2(X \setminus D;\Q) \neq 0$, we are stuck.
Of course, this includes many interesting cases, so we would like to improve our results to cover more of them.

\begin{rmk}
One might hope to prove a weaker result: namely that there exists a bulk deformation of the relative Fukaya category, so that the bulk-deformed $\EuA_R$ gets matched up with $\EuB_R$ (compare \cite{Seidel2016}). 
Not only would this approach introduce additional complications, it would not work in general: observe that bulk deformation classes live in $H^2(X)$, whereas our unwanted deformation directions live in $H^2(X \setminus D)$: so unless the restriction map $H^2(X) \to H^2(X \setminus D)$ is surjective, there will remain possible `directions' for the Fukaya category to deform in, which cannot be killed even by the introduction of a bulk class. 
On the mirror side, $X \setminus D$ corresponds to the singular central fibre $Y_0$ of a maximally unipotent degeneration. In general $Y_0$ may admit deformations which are not smoothings: for example, by `re-gluing' along the divisors where the components of $Y_0$ meet (compare \cite[Theorem 5.10]{Friedman1983} in the K3 case). 
It is natural to expect that these deformations correspond to the `extra' deformations of $\fuk(X \setminus D)$ coming from classes in $H^2(X \setminus D)$. 
We thank Helge Ruddat for this observation.
\end{rmk}

\subsection{Versality in the presence of a group action}
\label{subsec:versgpint}

In \cite{Seidel:HMSquartic}, Seidel studied the deformations of $\fuk(X \setminus D)$ in the case that $X$ is the Fermat quartic hypersurface in $\CP^3$, and $D$ is the union of coordinate hyperplanes. 
In this case $X$ is Calabi--Yau and $H^2(X \setminus D;\Q)$ is non-zero, so we have some unwanted directions in our deformation problem. 
Seidel circumvented this issue by exploiting symmetry: there is an obvious action of $\Gamma := \Z/4$ on $(X,D)$ by cyclically permuting the homogeneous coordinates of $\CP^3$, and it respects the restriction of the Fubini--Study symplectic form and acts on the arrangement of vanishing cycles in $X \setminus D$ that Seidel considers. 
In particular, the endomorphism algebra of this arrangement of vanishing cycles in the relative Fukaya category admits an action of $\Gamma$. 
It follows that the relative Fukaya category can only deform in the direction of $H^2(X \setminus D)^\Gamma$: but a simple computation shows that $H^2(X \setminus D)^\Gamma \simeq 0$. 
Thus our unwanted directions are ruled out, by symmetry (a similar strategy was subsequently used in \cite{Sheridan:CY}).

This is a convenient way to eliminate the unwanted deformation directions in $H^2(X \setminus D)$, but the existence of a sufficiently large symmetry group is the exception rather than the norm.  
We generalize Seidel's technique, by considering symmetry groups that act on $(X,D)$ by a combination of K\"ahler isometries and anti-isometries. 
Whereas a K\"ahler isometry induces an autoequivalence of the relative Fukaya category, a K\"ahler anti-isometry induces a \emph{duality} of the Fukaya category \cite{Castano2010}.

\begin{defn}
\label{defn:sgngrp}
A \emph{signed group} $(\Gamma,\sigma)$ consists of a group $\Gamma$, together with a homomorphism $\sigma: \Gamma \to \Z/2$. 
A homomorphism of signed groups is a homomorphism of groups, respecting the map to $\Z/2$.
\end{defn}

\begin{defn}
An \emph{action} of $(\Gamma,\sigma)$ on an $A_\infty$ algebra $\EuA$ is an action of $\Gamma$ on the underlying vector space of $\EuA$, which is a strict $A_\infty$ isomorphism 
\begin{equation} \gamma: \EuA \to \begin{cases}
					\EuA & \mbox{ if $\sigma(\gamma) = 0$} \\
					\EuA^{op} & \mbox{ if $\sigma(\gamma) = 1$,}
				\end{cases} \end{equation}
where $\EuA^{op}$ denotes the opposite $A_\infty$ algebra (see \S \ref{subsec:ainfversgp} and \S \ref{subsec:sgp} for further details). 
\end{defn}

Given an action of $(\Gamma,\sigma)$ on $\EuA$ and $R$, we say that $\EuA_R$ is a $(\Gamma,\sigma)$-equivariant deformation of $\EuA$ over $R$ if the induced action of $\Gamma$ on the underlying $R$-module of $\EuA_R$ continues to define an action of $(\Gamma,\sigma)$ on $\EuA_R$.  
We denote the set of such equivariant deformations by $\mathfrak{A}_R(\EuA)^\Gamma \subset \mathfrak{A}_R(\EuA)$. 
We define $\mathfrak{A}'_R(\EuA)^\Gamma \subset \mathfrak{A}_R(\EuA)^\Gamma$ to be the subset of equivariant deformations such that $\HH^2(\EuA,\EuA \otimes \fmuncomp/\fmuncomp^2)^\Gamma$ is spanned by the first-order deformation classes $\nov_p \defa_p$. 

\begin{defn}
\label{defn:eqversainf}
We say that $\EuA_R$ is an $R$-\emph{complete} equivariant deformation of $\EuA$ if, for any $\EuB_R \in \mathfrak{A}'_R(\EuA)^\Gamma$, there exist $\psi^*$ and $F^*$ as in Definition \ref{defn:vers1int}. 
We also define the notions of $R$-universal and $R$-versal equivariant deformations as there.
\end{defn}

\begin{defn}
An action of $(\Gamma,\sigma)$ on a relative K\"{a}hler manifold $(X,D,\omega)$ is an action of $\Gamma$ on $X$ by diffeomorphisms, such that
\begin{itemize}
\item $\gamma$ acts holomorphically if $\sigma(\gamma) = 0$, and anti-holomorphically if $\sigma(\gamma) = 1$;
\item $\gamma$ preserves the K\"{a}hler potential.
\end{itemize}
Observe that if $\Gamma$ is finite and acts on $(X,D)$ by biholomorphisms and anti-biholomorphisms, respecting the cohomology class of the K\"ahler form, one can average to obtain a relative K\"ahler form and potential which are preserved by $\Gamma$.
\end{defn}

It turns out that we will also need to equip $(X,D,\omega)$ with a morphism of grading data $p: \G \to \Z/4$. 
An action of $(\Gamma,\sigma)$ on $(X,D,\omega)$, respecting the morphism $p$ and the cone $\Nef$, induces an action of $\Gamma$ on the coefficient ring $R(\Nef)$ (see \S \ref{subsec:verseq}). 

\begin{lem}[= Lemma \ref{lem:signgrpactrel}]
\label{lem:intsigngrpactrel}
Suppose that $(\Gamma,\sigma)$ acts on $(X,D,\omega)$, respecting the morphism of grading data $p: \G \to \Z/4$ and the cone $\Nef$.
Suppose that $\EuA_R \subset \fuk(X,D,\Nef)$ is a full subcategory, and $\EuA \subset \fuk(X \setminus D)$ is the corresponding full subcategory. 
Suppose that $\Gamma$ acts freely on the set of unanchored branes underlying $Ob(\EuC)$. 
Then $(\Gamma,\sigma)$ acts on $\EuA$ up to shifts, and $\EuA_R$ is a $(\Gamma,\sigma)$-equivariant deformation of $\EuA$ over $R$.
\end{lem}

\begin{main}[= Theorem \ref{thm:eqversality}]
\label{thm:2}
Consider the situation of Lemma \ref{lem:intsigngrpactrel}. 
Suppose furthermore that
\begin{itemize}
\item $\Nef$ is nice, or $D$ is smooth and $(X,D)$ semi-positive;
\item $\Gamma$ is finite;
\item $\HH^2\left(\EuA,\EuA \otimes \fmuncomp\right)^\Gamma$ is contained in the $\Runcomp_0$-span of the classes $\nov_p \defa_p$. 
\end{itemize}
Then $\EuA_R$ is an $R$-complete $(\Gamma,\sigma)$-equivariant deformation of $\EuA$.

If furthermore $\defa_p \neq 0$ for all $p$, then $\EuA_R$ is $R$-versal.
\end{main}

In this situation, we have an analogue of Corollary \ref{cor:geomhypint}. 
The action of $(\Gamma,\sigma)$ on $(X,D,\omega)$ induces an action of $\Gamma$ on $H^\bullet(X \setminus D)$:
\begin{equation} \label{eqn:gammahact}
\gamma \cdot \alpha := (-1)^{\sigma(\gamma)} \cdot \gamma^* \alpha.\end{equation} 

\begin{propstar}[= Corollary \ref{cor:geomhyp2}]
\label{prop:geomhyp2int}
Consider the situation of Theorem \ref{thm:2}. 
Suppose furthermore that:
\begin{itemize}
\item Either $\Nef$ is nice and semi-positive, or $D$ is smooth and $(X,D)$ is semi-positive with $0 \le \kappa < 1$;
\item $H^2\left(X \setminus D;\fmuncomp\right)^\Gamma =0$;
\item The map
\begin{equation}
\label{eqn:cosurjagint}
\EuC \EuO \otimes \fmuncomp :SH^2\left(X \setminus D;\fmuncomp\right) \to \HH^2\left(\EuA,\EuA \otimes \fmuncomp \right)\end{equation}
is surjective (e.g., $\EuC\EuO$ is surjective).
\end{itemize}
Then $\EuA_R$ is an $R$-complete $(\Gamma,\sigma)$-equivariant deformation of $\EuA$.

If \eqref{eqn:cosurjagint} is furthermore an isomorphism, then $\EuA_R$ is $R$-versal. 
\end{propstar}

\begin{rmk}
When $\Gamma$ acts purely by symplectomorphisms (i.e., $\sigma \equiv 0$), Proposition \ref{prop:geomhyp2int} is rather intuitive. 
However the welter of sign computations that go into the proof of Proposition \ref{prop:geomhyp2int} may leave the reader without an intuition for where the sign in \eqref{eqn:gammahact} `comes from'. 
It comes from the same sign in the equation 
\[ \gamma^* \omega = (-1)^{\sigma(\gamma)} \omega.\]
Since this is the behaviour of the symplectic form $\omega$ under the action of $\gamma$, it should also be the behaviour of the deformation classes of the relative Fukaya category.
\end{rmk}

\subsection{The case of a smooth anti-canonical divisor}

One situation of significant interest in mirror symmetry is when $D \subset X$ is a smooth anti-canonical divisor. 
This corresponds to the case that $D$ is smooth, $(X,D)$ is positive, and $\kappa = 1$: so it just fails to be covered by the results of the previous section. 
Indeed we do not have a versality result in this case, but one can prove a variation which we feel it is worthwhile to state here.

Let $\EuA_R \subset \fuk(X,D)$ be a full subcategory, and $\EuA \subset \fuk(X \setminus D)$ the corresponding full subcategory. 

\begin{thmstar}
\label{thm:anticanvers}
Suppose that:
\begin{itemize}
\item $D$ is smooth and $(X,D)$ is positive with $\kappa = 1$;
\item The map
\begin{equation}
\label{eqn:cosurjagi}
\EuC \EuO \otimes \fmuncomp :SH^2\left(X \setminus D;\fmuncomp\right) \to \HH^2\left(\EuA,\EuA \otimes \fmuncomp \right)\end{equation}
is surjective (e.g., $\EuC\EuO$ is surjective);
\item $\EuB_R \in \mathfrak{A}_R(\EuA)$ is another deformation of $\EuA$ over $R$ with the same first-order deformation classes as $\EuA_R$.
\end{itemize}
Then there is a (possibly curved) $A_\infty$ homomorphism $F^*: \EuB_R \dashrightarrow \EuA_R$ with $F^* \equiv \id\text{ (mod $\fm$)}$.
\end{thmstar}

We will not give the proof, since it requires making some (minor) modifications to the statement and proof of Lemma \ref{lem:vers}. 
The crucial observation is that $SH^2(X \setminus D;\fmuncomp^2) = 0$, which follows from the first assumption by similar arguments to those given in the proofs of Lemmas \ref{lem:defgen} and \ref{lem:nodeg2}.

\subsection{Fixing the mirror map}

Our versality results fit into a natural strategy for proving cases of homological mirror symmetry, which we outline in \S \ref{subsec:hms}. 
The versality results are well-adapted to proving the existence of a mirror map $\Psi: \mathcal{M}_{K\ddot{a}h}(X) \xrightarrow{\sim} \mathcal{M}_{cpx}(Y)$ such that $D^\pi\fuk(X,\omega_p) \simeq \dbdg Coh(Y_{\Psi(p)})$. 

However, our versality results do not allow us to compute the mirror map $\Psi$.
Without a formula for $\Psi$ we cannot say precisely which complex structure $\Psi(p)$ corresponds to the K\"ahler form $\omega_p$. 
The solution is to use the results of \cite{Ganatra2015}, which give a criterion for homological mirror symmetry to imply Hodge-theoretic mirror symmetry; the mirror map can then be characterized uniquely using the notion of `flat coordinates'. 
We prove the precise statements in Appendix \ref{sec:hodgems}. 

\paragraph{Acknowledgments:} It is a pleasure to thank the following people: Mohammed Abouzaid, Strom Borman, John Lesieutre and Helge Ruddat for helpful discussions; Paul Seidel for suggesting the crucial idea of systems of divisors; Ivan Smith for many helpful discussions and suggestions; and Siu-Cheong Lau for suggesting that anti-symplectic symmetries could help constrain the deformation theory of the Fukaya category.

The author was partially supported by a Sloan Research Fellowship, and also by the National Science Foundation through Grant number DMS-1310604 and under agreement number DMS-1128155.
Any opinions, findings and conclusions or recommendations expressed in this material are those of the authors and do not necessarily reflect the views of the National Science Foundation. 
The author would also like to acknowledge support from Princeton University and the Institute for Advanced Study, and hospitality from the Instituto Superior T\'ecnico and ETH Z\"urich.

\section{Versality}
\label{sec:vers}

In this section, we fix a field $\Bbbk$ of characteristic $0$, and a grading datum $\G = \{ \Z \to G \to \Z/2\}$.

\subsection{Coefficient ring}
\label{subsec:ringR}

Let $P$ be a finite set, and $f: \Z^P \to G$ be a homomorphism of abelian groups, landing in the \emph{even} part of $G$. 
Let $\Nef \subset \R^P$ be a strongly convex cone with vertex at the origin.
We introduce the following $\G$-graded $\Bbbk$-algebras:
\begin{align}
\Rpuncomp & := \Bbbk\left[\Z^P\right] \quad \text{ equipped with the $\G$-grading induced by $f$,} \\
\Runcomp & :=  \Bbbk\left[\Nef^\vee \cap \Z^P\right] \subset \Rpuncomp.
\end{align}
The ring $\Runcomp$ has a unique maximal toric ideal $\fmuncomp \subset \Runcomp$, corresponding to the vertex of the cone $\Nef^\vee$ at the origin. 
We introduce the following $\G$-graded $\Bbbk$-algebras:
\begin{align}
 R&:= \text{ $\fmuncomp$-adic completion of $\Runcomp$, in $\G$-graded sense, }\\
\Rpunc & := R \otimes_{\Runcomp} \Rpuncomp. \label{eqn:rpunc}
\end{align}
We also define the ideal $\fm \subset R$ to be the completion of $\fmuncomp$.
For $u \in \Z^P$, we will denote the corresponding ring elements by $\nov^u$.
We will denote the $p$th basis element of $\Z^P$ by $u_p$, and $\nov_p := \nov^{u_p}$.

We will denote by $\Rpuncomp_0, \Runcomp_0, \ldots$ the parts in degree $0 \in G$.
Observe that $(R_0,\fm_0)$ is a complete local ring. 
We denote the units of this local ring by
\begin{equation} R_0^* = R_0 \setminus \fm_0.\end{equation}

\begin{defn}
$\Aut(R)$ denotes the group of $\G$-graded $\Bbbk$-algebra maps $\psi^*:R \to R$, which respect the $\fm$-adic filtration.
\end{defn}

\begin{defn}
\label{defn:Rauts}
We define a map
\begin{align}(R_0^*)^P & \to \Aut(R),\\
(\psi_p)_{p \in P} & \mapsto \psi^*,
\end{align}
where
\begin{equation} 
\label{eqn:impsistar}
\psi^*\left(\prod_{p \in P} \nov_p^{a_p} \right) := \prod_{p \in P} \nov_p^{a_p} \cdot \prod_{p \in P} \psi_p^{a_p}.\end{equation}
Observe that, because the $\psi_p$ are units, the negative powers of $\psi_p$ make sense, and $\psi^*$ respects the $\fm$-adic filtration (hence converges).
\end{defn}

\begin{defn}
\label{defn:Rnice}
We say that $R$ is \emph{nice} if the following two conditions hold for all $p \in P$:
\begin{itemize}
\item $\nov_p$ lies in $\fm$, and generates the corresponding graded piece $\fmuncomp_{f(u_p)}$ as an $\Runcomp_0$-module;
\item $\nov_p \notin \fm^2$.
\end{itemize}
\end{defn}

\begin{lem}
\label{lem:typeAaut}
If $R$ is nice, then the map of Definition \ref{defn:Rauts} is a bijection.
\end{lem}
\begin{proof}
First we show that the map is injective. 
If $(\psi_p)$ and $(\psi'_p)$ map to the same automorphism $\psi^*$, then
\[ \psi^*(r_p) = r_p \cdot \psi_p = r_p \cdot \psi'_p\]
so $\psi_p = \psi'_p$ for all $p$. 

Next we show it is surjective. 
Observe that because $R$ is nice, $\nov_p$ generates $\fm_{f(u_p)}$ as an $R_0$-module for all $p$. 
Therefore for any $\psi^* \in \Aut(R)$, we have $\psi^*(r_p) = r_p \cdot \psi_p$ for some $\psi_p \in R_0$, because $\psi^*$ is $\G$-graded and respects the $\fm$-adic filtration.
Since $\psi^*$ is an isomorphism, $\psi^*|_{\fm_{f(u_p)}}$ is surjective. 
Because $\fm_{f(u_p)}$ is isomorphic to $R_0 \cdot \nov_p$, it follows that $\psi_p$ is a unit. 
It is now a simple matter to show that $\psi^*$ has the form \eqref{eqn:impsistar}: so $\psi^*$ is the image of $(\psi_p)$.
\end{proof}

\begin{lem}
\label{lem:nicerpdif}
If $R$ is nice, then the generators $\nov_p$ have distinct degrees.
\end{lem}
\begin{proof}
If $R$ is nice and $\deg(\nov_p) = \deg(\nov_q)$, then both $\nov_p$ and $\nov_q$ generate the same graded piece of $\fmuncomp$ as an $\Runcomp_0$-module. 
It follows that both $\nov_p \nov_q^{-1}$ and $\nov_q \nov_p^{-1}$ lie in $\Runcomp_0$, and hence that both $u_p - u_q$ and $u_q - u_p$ lie in $\Nef^\vee$. 
It follows by strong convexity of $\Nef$ that $u_p = u_q$, and hence $p=q$.
\end{proof}

\subsection{Deformation theory}

We recall the basics of deformation theory via differential graded Lie algebras, following \cite{Manetti1999}.
Let $(\frakg,\partial,[\cdot,\cdot])$ be a differential ($\G$-)graded Lie algebra over $\Bbbk$ (we'll call this a `d$\G$la'). 
Let $(R,\fm)$ be as in the previous section.

\begin{defn}
We define the set of solutions to the Maurer--Cartan equation over $R$:
\begin{equation} MC_\frakg(R) := \left\{ \alpha \in (\frakg \ctens \fm)_1: \partial \alpha + \frac{1}{2}[\alpha,\alpha] = 0\right\}.\end{equation}
Here, `$\frakg \ctens \fm$' denotes the $\fm$-adic completion of $\frakg \otimes \fm$, in the $\G$-graded sense. 
It is a filtered pronilpotent d$\G$la, and $(\frakg \ctens \fm)_1$ is the part in degree $1 \in G$.
\end{defn}

There is a Lie algebra homomorphism from $(\frakg \ctens \fm)_0$ to the space of affine vector fields on $(\frakg \ctens \fm)_1$, sending $\gamma$ to the vector field $\alpha \mapsto -\partial \gamma + [\gamma,\alpha]$. 
Such vector fields are tangent to the Maurer--Cartan set.
The Lie algebra $(\frakg \ctens \fm)_0$ is filtered pronilpotent, hence can be exponentiated to a prounipotent group, which acts on $(\frakg \ctens \fm)_1$, preserving the Maurer--Cartan set. 
Explicitly, $\alpha_t = e^{t\gamma} \cdot \alpha$ is the unique solution of the differential equation
\begin{equation}
\label{eqn:odemc} 
\frac{d}{dt} \alpha_t = -\partial \gamma + [\gamma,\alpha_t] 
\end{equation}
with initial condition $\alpha_0 =\alpha$. 
 
\begin{defn}
Given $\alpha, \beta \in MC_\frakg(R)$, we say $\alpha$ and $\beta$ are \emph{gauge equivalent} if $\beta = e^\gamma \cdot \alpha$ for some $\gamma \in (\frakg \ctens \fm)_0$. 
We write `$\alpha \sim \beta$'.
\end{defn}

\begin{defn}
For any $\psi^* \in \Aut(R)$, there is an induced isomorphism of filtered pronilpotent d$\G$la,
\begin{equation} \id \otimes \psi^*: \frakg \ctens \fm \to \frakg \ctens \fm.\end{equation}
In particular, this isomorphism preserves solutions to the Maurer--Cartan equation: we denote the corresponding action of $\Aut(R)$ on $MC_\frakg(R)$ by $\alpha \mapsto \psi^*\alpha$.
\end{defn}

\subsection{Versality}

Let $\alpha \in MC_\frakg(R)$. 
We can write
\begin{equation} \alpha = \sum_{u \in \Nef^\vee \cap \Z^P} \nov^u \cdot \alpha_u,\end{equation}
where $\alpha_u \in \frakg_{1-f(u)}$. 
The Maurer--Cartan equation says that $\partial \alpha \equiv 0 \mbox{ (mod $\fm^2$)}$. 
Because $\fm$ is a monomial ideal, it follows that $\partial \alpha_u = 0$ for all $r^u \in \fm \setminus \fm^2$. 
So we have classes $\nov^u \cdot [\alpha_u] \in H^1(\frakg \otimes \fmuncomp)$ for all such $u$.

If $R$ is nice, then $\nov_p \in \fm \setminus \fm^2$ so we have classes 
\begin{equation}\nov_p \defa_p := \nov^{u_p}\cdot [\alpha_{u_p}] \in H^{1}(\frakg \otimes \fmuncomp).\end{equation}
These are called the \emph{first-order deformation classes}. 

\begin{defn}
We denote by $MC'_\frakg(R) \subset MC_\frakg(R)$ the subset of Maurer--Cartan elements whose first-order deformation classes $\nov_p\defa_p$ span $H^1(\frakg \otimes \fmuncomp/\fmuncomp^2)$.
\end{defn}

\begin{defn}
\label{defn:MCRvers}
We say that $\alpha \in MC_\frakg(R)$ is:
\begin{itemize}
\item \emph{$R$-complete} if for any $\beta \in MC'_\frakg(R)$, there exists $\psi^* \in \Aut(R)$ such that $\beta \sim \psi^* \alpha$;
\item \emph{$R$-universal} if it is $R$-complete, and furthermore the automorphism $\psi^*$ is uniquely determined for each $\beta$;
\item \emph{$R$-versal} if it is $R$-complete, and furthermore the map $\psi^*: \fm/\fm^2 \to \fm/\fm^2$ is uniquely determined for each $\beta$.
\end{itemize}
\end{defn}

\begin{rmk}
The restriction to $\beta \in MC'_\frakg(R) \subset MC_\frakg(R)$ is not usually part of the definition of completeness/universality/versality, but we don't know how to prove our results without this restriction (and the more general version would be no more useful in the applications we envision). 
The difficulty is with constructing endomorphisms of the ring $R$: it is easier to construct automorphisms via Definition \ref{defn:Rauts}.
\end{rmk}

\begin{lem}
\label{lem:vers}
Let $R$ be nice, and $\alpha \in MC_\frakg(R)$ have first-order deformation classes $\nov_p\defa_p$. Then:
\begin{itemize}
\item If $\nov_p\defa_p$ span $H^1(\frakg \otimes \fmuncomp)$ as $\Runcomp_0$-module, then $\alpha$ is $R$-complete;
\item If furthermore $\defa_p \neq 0$ for all $p$, then $\alpha$ is $R$-versal.
\end{itemize}
\end{lem}
\begin{proof}
Let $\beta \in MC'_\frakg(R)$. We construct:
\begin{itemize}
\item $\psi_p \in R_0$, $\psi_p \equiv 1 \mbox{ (mod $\fm$)}$ for all $p$, and 
\item $\gamma \in (\frakg \ctens \fm)_0$,
\end{itemize}
such that
\begin{equation}
\label{eqn:gauge}
\epsilon :=  e^\gamma \cdot \beta - \psi^* \alpha
\end{equation}
vanishes (where $\psi^* \in \Aut(R)$ corresponds to the $\psi_p$ in accordance with Definition \ref{defn:Rauts}). 

We denote 
\begin{align}
\psi_p^k &:= \psi_p \mbox{ (mod $\fm^{k}$)},\\ 
\gamma^k & := \gamma \mbox{ (mod $\fm^{k+1}$)}, \\
\epsilon^k & := \epsilon \mbox{ (mod $\fm^{k+1}$)}.
\end{align}
We prove, inductively in $k$, that it is possible to choose $\psi_p^k$ and $\gamma^k$ so as to make $\epsilon^k=0$.
The induction starts at $k=0$: setting $\psi_p^0 = 0$ and $\gamma^0 = 0$ obviously yields $\epsilon^0 = 0$.

Now suppose that $\epsilon^{k-1}=0$. 
Note that
\begin{eqnarray*}
[e^\gamma \cdot \beta - \psi^* \alpha,e^\gamma\cdot \beta + \psi^* \alpha] &=&[e^\gamma \cdot \beta,e^\gamma \cdot \beta] - [\psi^*\alpha,\psi^* \alpha ] \\
&=& -2 \partial (e^\gamma \cdot \beta - \psi^* \alpha) \\
&=& -2\partial (\epsilon)
\end{eqnarray*}
because $e^\gamma \cdot \beta$ and $\psi^*\alpha$ are solutions to the Maurer--Cartan equation. 
The left-hand side is a bracket of an element of order $\fm^{k}$ (by the inductive hypothesis) with one of order $\fm$; so the right-hand side is of order $\fm^{k+1}$. 
In particular, $\epsilon^k$ is closed: so it defines a cohomology class
\begin{equation}\label{eqn:defpsip} \left[\epsilon^k\right] \in H^1(\frakg \otimes \fm^k/\fm^{k+1}).\end{equation}

We have
\begin{equation}
\label{eqn:varphipk}
\left[\epsilon^k\right] = \sum_p \varphi_{p}^{k} \cdot \nov_p \defa_p \quad \text{in $H^1(\frakg \otimes \fm^k/\fm^{k+1})$}\end{equation}
for some $\varphi_{p}^{k} \in \fm^{k-1}/\fm^k$, by the hypothesis that the classes $\nov_p \defa_p$ span $H^1\left(\frakg \otimes \fmuncomp\right)$.
We then modify $\psi_{p}$:
\begin{equation} 
\label{eqn:psich}
\psi_p\mapsto \psi_p+ \varphi_{p}^{k}.\end{equation}
This modification has the effect of changing
\begin{equation} \epsilon^k \mapsto \epsilon^k - \sum_p \varphi_{p}^{k} \cdot \nov_p \beta_{u_p},\end{equation} 
so the new $\epsilon^k$ is nullhomologous.

We now have $\epsilon^k = \partial c^k$: modifying $\gamma^k \mapsto \gamma^k +c^k$ now has the effect of changing \begin{equation}\epsilon^k \mapsto \epsilon^k - \partial c^k = 0.\end{equation}
Therefore, with these choices of $\psi_p^k$ and $\gamma^k$, we have $\epsilon^k = 0$, and the inductive step is complete.

It is clear that the corrections we make to $\psi_p$ and $\gamma$ are of successively higher orders in the $\fm$-adic filtration, so the construction converges. 

It remains to show that we can choose $\psi^1_p \neq 0$, so that $\psi_p \in R_0^*$ for all $p$ (this is necessary for the automorphism $\psi^*$ to be well-defined in accordance with Definition \ref{defn:Rauts}). 
If the first-order deformation classes of $\beta$ are $\nov_p \mathsf{b}_p$, we require that
\[ \sum_p \nov_p \mathsf{b}_p = \sum_p \psi_p^1 \cdot \nov_p \defa_p.\]
in accordance with \eqref{eqn:defpsip}.
It follows by Lemma \ref{lem:nicerpdif} that $\mathsf{b}_p = \psi_p^1 \cdot \defa_p$ for all $p$. 

It also follows by Lemma \ref{lem:nicerpdif}, together with the assumption that $\beta \in MC'_\frakg(R)$, that $\mathsf{b}_p$ span their respective graded pieces of $H^\bullet(\frakg)$. 
If the graded piece has rank $1$, then $\psi_p^1$ is uniquely determined and non-zero; if the graded piece has rank $0$, we can choose $\psi_p^1$ to take any non-zero value. 
Thus we may choose $\psi_p \in R_0^*$ as required. 

Furthermore, if $\nov_p\defa_p \neq 0$ for all $p$, then these graded pieces have rank $1$ for all $p$, so $\psi_p^1$ is uniquely determined for all $p$. 
It is easy to check that the values of $\psi_p^1$ determine $\psi^*: \fm/\fm^2 \to \fm/\fm^2$ uniquely, as required.
\end{proof} 

\subsection{Versality with a group action}
\label{subsec:versgp}

We will also use a generalization of Lemma \ref{lem:vers}, in the presence of a group action. 
Let $\Gamma$ be a finite group, simultaneously acting on the grading datum $\G$ by morphisms of grading data, on the d$\G$la $\frakg$ by morphisms of d$\G$la relative to the action of $\Gamma$ on $\G$, on the set $P$, and on the ring $R$ by $\G$-graded automorphisms of $R$ relative to the action of $\Gamma$ on $\G$, of the special form
\[ g \cdot r_p = \pm r_{g \cdot p} \mbox{ for all $g, p$}.\]
 
When we say that the action of $\Gamma$ is `$\G$-graded relative to the action on $\G$', we mean that $g$ sends $\frakg_y  \mapsto \frakg_{g \cdot y}$, and respects the differential and bracket on $\frakg$; and similarly $g$ sends $R_y \mapsto R_{g \cdot y}$.
It follows that $\Gamma$ acts on $\frakg \otimes \fm$ and its completion $\frakg \ctens \fm$, preserving $MC_\frakg(R) \subset (\frakg \ctens \fm)_1$. 
We denote $MC_\frakg(R)^\Gamma \subset MC_\frakg(R)$ the subset of $\Gamma$-invariant Maurer--Cartan elements. 
We define $MC'_\frakg(R)^\Gamma \subset MC_\frakg(R)^\Gamma$ to be the subset of $\Gamma$-invariant Maurer--Cartan elements such that the span of the first-order deformation classes contains $H^1(\frakg \otimes \fmuncomp/\fmuncomp^2)^\Gamma$.

We define notions of $R$-completeness, $R$-universality and $R$-versality by analogy with Definition \ref{defn:MCRvers}, replacing `$MC'_\frakg(R)$' by `$MC'_\frakg(R)^\Gamma$'. 
We have the following version of Lemma \ref{lem:vers}:

\begin{lem}
\label{lem:eqvers}
Let $R$ be nice, and $\alpha \in MC_\frakg(R)^\Gamma$ have first-order deformation classes $\nov_p\defa_p$. Then:
\begin{itemize}
\item If $H^1(\frakg \otimes \fmuncomp)^\Gamma$ is contained in the $\Runcomp_0$-span of $\nov_p\defa_p$, then $\alpha$ is $R$-complete;
\item If furthermore $\defa_p \neq 0$ for all $p$, then $\alpha$ is $R$-versal.
\end{itemize}
\end{lem}
\begin{proof}
The proof is parallel to that of Lemma \ref{lem:vers}: we seek to inductively construct $\psi_p$ and $\gamma$ as there, now with the additional properties that $\psi_{g \cdot p} = g \cdot \psi_p$ ($\Gamma$-equivariance of $\psi^*$) and $g \cdot \gamma = \gamma$ for $g \in \Gamma$.
This is achieved by `averaging' at each inductive step. 

The main changes required to the argument are these: after obtaining \eqref{eqn:varphipk}, we average both sides to make them $\Gamma$-invariant. 
This doesn't change $[\epsilon^k]$, which is $\Gamma$-invariant by the inductive assumption; on the right-hand side, it follows from the special form of the action of $\Gamma$ on $R$ that 
\[ g \cdot (\nov_p \defa_p) = \nov_{g \cdot p} \defa_{g \cdot p},\]
which means that the effect of the averaging is to modify $\varphi_p^k$ so that $g \cdot \varphi_p^k = \varphi_{g \cdot p}^k$. 
So the substitution \eqref{eqn:psich} preserves $\Gamma$-equivariance of $\psi^*$.

The second part of the inductive step is more straightforward: we simply average $c^k$ to make it $\Gamma$-invariant.
\end{proof}

\subsection{Application to $A_\infty$ structures without a group action}
\label{subsec:defainf}

We review the deformation theory of $A_\infty$ algebras (compare \cite[\S 3]{Seidel:HMSquartic}). 

The Hochschild differential and the Gerstenhaber bracket equip the Hochschild cochain complex $\frakg := CC^\bullet(\EuA)[1]$ with the structure of a d$\G$la. 
Its cohomology is $H^\bullet(\frakg) \simeq \HH^\bullet(\EuA)[1]$, the Hochschild cohomology of $\EuA$.

An element $\alpha \in \left(\frakg \ctens \fm\right)_1$ is a solution to the Maurer--Cartan equation if and only if $\EuA_R = \left(A \ctens R,\mu^* + \alpha\right)$ defines a deformation of $\EuA$ over $R$. 
Thus we have an isomorphism $MC_\frakg(R) \simeq \mathfrak{A}_R(\EuA)$. 
This isomorphism identifies $\mathfrak{A}'_R(\EuA) \subset \mathfrak{A}_R(\EuA)$ with $MC'_\frakg(R) \subset MC_\frakg(R)$.

Now suppose the Maurer--Cartan elements $\alpha$ and $\beta$ correspond to deformations $\EuA_R$ and $\EuB_R$ of $\EuA$. 
If $\alpha$ and $\beta$ are gauge equivalent, then there is a (possibly curved) $A_\infty$ homomorphism $F^*: \EuB_R \dashrightarrow \EuA_R$ with $F^* \equiv \id\text{ (mod $\fm$)}$, defined by universal formulae in $\gamma$ (see e.g. \cite[Equation (3.11)]{Seidel:g2}). 

To be explicit about the formulae for $F^*$, suppose that $\alpha =  e^\gamma \cdot \beta$. 
Observe that $\mu^*_t := \mu^* +  e^{t\gamma} \cdot \beta$ is the unique solution to the differential equation
\[ \frac{d}{dt} \mu^*_t = [\gamma,\mu^*_t] \]
with initial condition $\mu^*_0 = \mu^* + \beta$ (see \eqref{eqn:odemc}). 
There is then an $A_\infty$ isomorphism $F_t^*:\left(A \ctens R,\mu^*_0\right) \dashrightarrow \left(A \ctens R,\mu^*_t\right)$ for all $t$: it is the unique solution to the differential equation
\[ \frac{d}{dt} F_t^*(\ldots) = \sum \gamma(F_t^*(\ldots),\ldots,F_t^*(\ldots)) \]
with initial condition $F_0^* = \id$. 
To be precise, one must prove that the solution $F_t^*$ exists and is unique, then that it defines an $A_\infty$ homomorphism. 
It is an immediate consequence of the construction that $F^*_t \equiv \id \mbox{ (mod $\fm$)}$, for all $t$. 
The desired $A_\infty$ homomorphism is then $F^* = F^*_1$.

\begin{lem}
\label{lem:versainf}
Suppose that $R$ is nice, and $\EuA_R \in \mathfrak{A}_R(\EuA)$ is a deformation whose first-order deformation classes $\nov_p \defa_p$ span $\HH^2\left(\EuA,\EuA \otimes \fmuncomp\right)$ as $\Runcomp_0$-module. 
Then $\EuA_R$ is $R$-complete. 

If furthermore $\defa_p \neq 0$ for all $p$, then $\EuA_R$ is $R$-versal.
\end{lem}
\begin{proof}
Follows directly from Lemma \ref{lem:vers} and the previous discussion.
\end{proof}

\subsection{Application to $A_\infty$ structures with a group action}
\label{subsec:ainfversgp}

Lemma \ref{lem:versainf} can be generalized in the presence of a group action, and this generalization proves very useful in our intended applications. 
Let $\Gamma$ be a finite group simultaneously acting on the grading datum $\G$ and the $\G$-graded vector space $A$. 
Then $\Gamma$ acts on the vector space $CC^\bullet(A)$, by the formula
\begin{equation} (\gamma \cdot \eta)^k(a_1,\ldots,a_k) := \gamma^{-1} \cdot \eta^k(\gamma \cdot a_1,\ldots,\gamma \cdot a_k). \end{equation}
If this action preserves $\mu^* \in CC^2(A)$, then we say that $\Gamma$ acts on the $A_\infty$ algebra $(A,\mu^*)$. 
In this case $\Gamma$ acts on the d$\G$la $\frakg := CC^\bullet(\EuA)$.
If $\Gamma$ also acts simultaneously on $R$ and $\alpha \in MC_{\frakg}(R)^\Gamma$, we say that the corresponding deformation $\EuA_R$ is $\Gamma$-equivariant. 

In fact we will want to consider a generalization of a group action, which we call a \emph{signed group action}. 
It combines the usual notion of a group action with the notion of a duality (see \cite{Castano2010}). 
See \S \ref{subsec:sgp} for more details on what follows, including the generalization from algebras to categories.
 
First, we define the following $\Z/2$-action on the vector space $CC^\bullet(A)$:
\begin{align}
op: CC^\bullet(A) & \to CC^\bullet(A^{op})\\
\eta & \mapsto \eta_{op},\quad \mbox{ where}\\
\eta_{op}^k(a_1,\ldots,a_k) & := (-1)^\dagger \cdot \eta(a_k,\ldots,a_1)
\end{align}
with the sign given by
\begin{equation} \dagger := \sum_{1 \le i<j \le s} (|a_i| + 1) \cdot (|a_j|+1).\end{equation}
This map preserves the Gerstenhaber bracket.

Now let $(\Gamma,\sigma)$ be a finite signed group. 
If $\Gamma$ acts simultaneously on the grading datum $\G$ and the $\G$-graded vector space $A$, we define a graded action of $\Gamma$ on the vector space $CC^\bullet(A)$ by 
\begin{equation} \gamma \cdot \eta^k(a_1,\ldots,a_k) = \begin{cases}
									\gamma^{-1} \cdot \eta^k(\gamma \cdot a_1,\ldots,\gamma \cdot a_k) & \mbox{ if $\sigma(\gamma) = 0$} \\
									\gamma^{-1} \cdot \eta_{op}^k(\gamma \cdot a_1,\ldots,\gamma \cdot a_k) & \mbox{ if $\sigma(\gamma) = 1$.}
									\end{cases}\end{equation}
If this action preserves $\mu^* \in CC^2(A)$, then we say that $(\Gamma,\sigma)$ acts on $\EuA = (A,\mu^*)$. 
In this case $\Gamma$ acts on $\frakg := CC^\bullet(\EuA)$. 
If $\Gamma$ also acts simultaneously on $R$, and $\alpha \in MC_{\frakg}(R)^\Gamma$, we say that the corresponding deformation $\EuA_R$ is $(\Gamma,\sigma)$-equivariant. 
We denote the set of $(\Gamma,\sigma)$-equivariant deformations by $\mathfrak{A}_R(\EuA)^\Gamma \simeq MC_\frakg(R)^\Gamma$. 
We define the subset $\mathfrak{A}'_R(\EuA)^\Gamma \subset \mathfrak{A}_R(\EuA)^\Gamma$ corresponding to $MC'_\frakg(R)^\Gamma \subset MC_\frakg(R)^\Gamma$, as in \S \ref{subsec:versgpint}.

We recall the notions of $R$-complete/$R$-universal/$R$-versal equivariant deformations from Definition \ref{defn:eqversainf}.

\begin{lem}
\label{lem:eqversainf}
Suppose that $R$ is nice, and the action of $\Gamma$ on it is of the special type described in \S \ref{subsec:versgp}; and that $\EuA_R \in \mathfrak{A}_R(\EuA)^\Gamma$ is an equivariant deformation such that $\HH^2(\EuA,\EuA \otimes \fmuncomp)^\Gamma$ is contained in the $\Runcomp_0$-span of the first-order deformation classes $\nov_p \defa_p$. 
Then $\EuA_R$ is $R$-complete. 

If furthermore $\defa_p \neq 0$ for all $p$, then $\EuA_R$ is $R$-versal.
\end{lem}
\begin{proof}
Follows directly from Lemma \ref{lem:eqvers}.
\end{proof}

\subsection{Bounding cochains}
\label{subsec:bc}

The fact that the $A_\infty$ homomorphisms constructed in Lemmas \ref{lem:versainf} and \ref{lem:eqversainf} are possibly curved may appear disconcerting (this issue was avoided in \cite{Seidel:HMSquartic} by working with \emph{truncated} Hochschild cohomology). 
However it has a very natural interpretation, which we now explain.

If $\EuA_R = (A \ctens R,\mu^*)$ is an $R$-linear (possibly curved) $A_\infty$ algebra, we recall that a \emph{bounding cochain} is an element $\mca \in (A \ctens \fm)_1$ satisfying a (different) Maurer--Cartan equation:
\begin{equation} \sum_{s \ge 0} \mu^s(\mca,\ldots,\mca) = 0.\end{equation}
The infinite sum makes sense because it converges in the $\fm$-adic topology. 
A bounding cochain allows us to deform the (possibly curved) $A_\infty$ algebra $\EuA_R = (A \ctens R,\mu^*)$ to a non-curved $A_\infty$ algebra $(\EuA_R,\mca) = (A \ctens R,\mu^*_{\mca})$, where
\begin{equation} 
\label{eqn:mudefbc}\mu^s_{\mca}(a_1,\ldots,a_s) := \sum_{i_0,\ldots,i_s} \mu^*(\underbrace{\mca,\ldots,\mca}_{i_0},a_1,\underbrace{\mca,\ldots,\mca}_{i_1},a_2,\ldots,a_s,\underbrace{\mca,\ldots,\mca}_{i_s}). \end{equation}
The deformed $A_\infty$ structure remains $R$-linear and $\G$-graded.

We have the following trivial observation:

\begin{lem}
\label{lem:curvmc}
Suppose that $\EuA_R$ and $\EuB_R$ are (possibly curved) $A_\infty$ algebras, deformations of $\EuA$ and $\EuB$ respectively over $R$, and suppose their curvature is of order $\fm$: i.e., $\mu^0 \in (A \otimes \fm)_2$, $\eta^0 \in (B \otimes \fm)_2$. 
Suppose that $F^*$ is a curved $A_\infty$ homomorphism from $\EuA$ to $\EuB$, whose curvature however is of order $\fm$, i.e., $F^0 \in (B \otimes \fm)_1$. 
If $\mca \in (A \otimes \fm)_1$ is a bounding cochain for $\EuA_R$, then
\begin{equation} \mcb:= \sum_{s \ge 0} F^s(\mca,\ldots,\mca)\end{equation}
is a bounding cochain for $\EuB_R$, and 
\begin{equation}
F^s_{\mca}(a_1,\ldots,a_s) := \sum_{i_0,\ldots,i_s} F^s(\underbrace{\mca,\ldots,\mca}_{i_0},a_1,\underbrace{\mca,\ldots,\mca}_{i_1},a_2,\ldots,a_s,\underbrace{\mca,\ldots,\mca}_{i_s})
\end{equation}
for $s \ge 1$ define a non-curved $A_\infty$ homomorphism from $(\EuA_R,\mca)$ to $(\EuB_R,\mcb)$. 
If $F^1$ is an isomorphism of $R$-modules, then so is $F^1_\mca$.
\end{lem}
\begin{proof}
The fact that $F^s_{\mca}$ defines a non-curved $A_\infty$ homomorphism is standard. One proves that $F^1_{\mca}$ is an isomorphism by a spectral sequence comparison argument, for the spectral sequence induced by the $\fm$-adic filtration. 
\end{proof}

We mention two generalizations of Lemma \ref{lem:curvmc}, which are proved in exactly the same way.
First, if $\EuA_R$, $\EuB_R$ and $F^*$ are strictly unital and $\mca$ is a \emph{weak} bounding cochain, then $\mcb$ is also a weak bounding cochain with the same disc potential $\mathfrak{P}(\mcb) = \mathfrak{P}(\mca)$ (see \cite[Definition 3.8.40]{FO3}).  
Second, if $\EuA_R$ and $\EuB_R$ are categories and $F^*$ is a curved $A_\infty$ functor, then there is an induced non-curved $A_\infty$ functor $\EuA_R^{\bc} \dashrightarrow \EuB_R^\bc$. 

\section{Geometric setup}
\label{sec:geomset}

Now we give the geometric data on which our constructions will depend. 
First, let us clarify one convention: we will write $H_\bullet(X) := H_\bullet(X;\Z)$.

\subsection{$\snc$ pairs}
\label{subsec:logpair}

\begin{defn}
A \emph{$\snc$ pair} $(X,D)$ consists of:
\begin{itemize}
\item a smooth compact complex manifold $X$; and
\item a $\snc$ divisor $D \subset X$.
\end{itemize}
We will denote $D=\cup_{p \in P} D_p$, where $D_p$ are the irreducible components of $D$.
\end{defn}

\begin{defn}
\label{defn:relkahl}
A \emph{relative K\"{a}hler form} on $(X,D)$ consists of:
\begin{itemize}
\item A K\"{a}hler form $\omega$ on $X$; and
\item A plurisubharmonic function $h: X \setminus D \to \R$ such that $\omega|_{X \setminus D} = -dd^c h$.
\end{itemize}
We require that $h$ has a particular form near $D$: in local holomorphic coordinates $(z_1,\ldots,z_n)$ such that $D= \{z_1\ldots z_k = 0\}$, we require that
\begin{equation}
\label{eqn:hnearD}
 h = -\psi -\sum_{p=1}^k \ell_p \cdot \log |z_i|
\end{equation}
where $\psi$ is smooth, and each $\ell_p > 0$ is a positive real number called the \emph{linking number} with $D_p$.
\end{defn}

If a $\snc$ pair admits a relative K\"{a}hler form, we say that the $\snc$ pair is \emph{K\"{a}hler}. 
We will call the data $(X,D,\omega=-dd^ch)$ a \emph{relative K\"{a}hler manifold}. 
We will denote a relative K\"{a}hler manifold by $(X,D,\omega)$ where no confusion can arise. 
Observe that the space of relative K\"{a}hler forms is open and convex.

Because the linking numbers are positive, $h$ is proper and bounded below. 
It follows that any regular sublevel set $\{h \le c\}$ is a Liouville domain, with Liouville one-form $\alpha := -d^ch$.

The special form of $h$ near $D$ implies that the critical points of $h$ form a compact subset of $X \setminus D$ (see \cite[Lemma 4.3]{Seidel:biased}). 
Therefore we can choose $c$ to be larger than all critical values of $h$. 
This implies that $\{h \le c \} \subset X \setminus D$ is a deformation retract. 
We will denote a relative K\"{a}hler manifold $(X,D,\omega)$, together with a choice of such $c$, by $(X,D,\omega,c)$. 
We will still call this a relative K\"{a}hler manifold, to avoid excessive terminology.

For any relative K\"{a}hler manifold $(X,D,\omega,c)$, we define
\begin{equation} [\omega] \in H^2(X,X\setminus D;\R) \end{equation}
by 
\begin{equation} [\omega](u) := \omega(u) - \alpha(\partial u)\end{equation}
for any smooth representative $u$. 
By Stokes' theorem (compare \cite[Lemma 3.12]{Sheridan:CY}), we have
\begin{equation}
\label{eqn:links}
 [\omega] = \sum_p \ell_p \cdot PD(D_p),
\end{equation}
where 
\begin{equation} PD: H_{2n-2}(D) \xrightarrow{\simeq} H^2(X,X \setminus D;\R)\end{equation}
denotes Poincar\'{e} duality.

We now recall Definition \ref{defn:nefampxd}: we let
\begin{equation} Nef(X,D) \subset H^2(X,X \setminus D;\R) \simeq \R^P\end{equation} 
be the closed convex cone spanned by classes $\sum_p \ell_p \cdot PD(D_p)$ with $\ell_p \in \N_0$, such that $\sum_p \ell_p \cdot D_p$ is nef. 
In other words, `the cone of effective nef divisors supported on $D$'. 
We also let $Amp(X,D) \subset H^2(X,X \setminus D;\R)$ be the open convex cone spanned by classes $\sum_p \ell_p \cdot PD(D_p)$ with $\ell_p \in \N_{>0}$, such that $\sum_p \ell_p \cdot D_p$ is ample.  

We now recall the relationship between ample divisors and K\"{a}hler forms, following \cite[Chapter I.2]{Griffiths1978}.

\begin{lem}
\label{lem:ampkahl}
A class in $H^2(X,X \setminus D;\R)$ is the cohomology class of a relative K\"{a}hler form if and only if it lies in $Amp(X,D)$.
\end{lem}
\begin{proof}
Both the set of cohomology classes of relative K\"{a}hler forms and $Amp(X,D)$ are open, convex and conical subsets of $H^2(X,X \setminus D;\R)$. 
Therefore it suffices to prove that their intersections with the integer lattice $H^2(X,X \setminus D)$ coincide.
In other words, it suffices to show that for any $\ell_p \in \N_{>0}$,
\[\sum_p \ell_p \cdot D_p \mbox{ is ample} \quad \iff \quad \sum_p \ell_p \cdot PD(D_p) \text{ is the cohomology class of a relative K\"ahler form.}\]

Let $\mathcal{L} := \mathcal{O}(\sum_p \ell_p \cdot D_p)$, and $\chi \in \Gamma(\mathcal{L})$ the section defining $\sum_p \ell_p \cdot D_p$.
If $\mathcal{L}$ is ample, it is positive, so we can equip it with a metric $\| \cdot \|$ whose curvature form is $\Theta = - 2\pi \iii \omega$, where $\omega$ is a K\"{a}hler form. 
We have 
\begin{equation}\Theta = 2 \pi \iii \cdot dd^c \log \left( \| \chi \|^2 \right) \end{equation}
on the region where the right-hand side is defined, i.e., on $X \setminus D$ (see \cite[p. 142]{Griffiths1978}). 
Therefore, $h:=\log \|\chi\|^2$ is a K\"{a}hler potential for $\omega$ on $X \setminus D$. 
One easily computes that $h$ has the prescribed form \eqref{eqn:hnearD} near $D$ (because $\chi$ vanishes to order $\ell_p$ along $D_p$). 
It follows that $\omega=-dd^ch$ is a relative K\"{a}hler form with linking numbers $\ell_p$, so its cohomology class is $\sum_p \ell_p \cdot PD(D_p)$ by \eqref{eqn:links}.

Conversely, suppose that $\sum_p \ell_p \cdot PD(D_p)$ is the cohomology class of a relative K\"{a}hler form.
Then $c_1(\mathcal{L}) = \sum_p \ell_p \cdot PD(D_p) \in H^2(X)$ is represented by a K\"{a}hler form, so $\mathcal{L}$ is a positive line bundle (see \cite[p. 148]{Griffiths1978}). 
Therefore $\mathcal{L}$ is ample by the Kodaira embedding theorem. 
\end{proof}

\begin{cor}
A $\snc$ pair $(X,D)$ is K\"{a}hler if and only if $Amp(X,D) \neq \emptyset$.
\end{cor}

In particular, if $(X,D)$ is K\"ahler, then $Amp(X,D)$ is the interior of $Nef(X,D)$ and $Nef(X,D)$ is the closure of $Amp(X,D)$, by Kleiman's theorem \cite{Kleiman1966} (see \cite[Theorem 1.4.23]{Lazarsfeld2004}).

\subsection{Systems of divisors}
\label{subsec:sysdiv}

\begin{defn}
\label{defn:sysofdiv}
Let $(X,D)$ be an $n$-dimensional $\snc$ pair, and $U \subset X$ an open set such that $D \subset U$ and $X \setminus U \subset X \setminus D$ are deformation retracts. 
A \emph{system of divisors} in $U$ is a $\snc$  divisor $E \subset U$ such that, if $E = \cup_{q \in Q} E_q$ are the irreducible components of $E$, then:
\begin{itemize}
\item The classes $PD(E_q)$ span $H^2(X,X\setminus D;\Q)$, where `$PD$' denotes Poincar\'{e} duality:
\begin{equation} PD: H_{2n-2}(U) \to H^2(X,X \setminus U) \simeq H^2(X,X \setminus D).\end{equation}
\item For each $q \in Q$, there exist at least $n+1$ distinct $r \in Q$ so that $E_r$ is linearly equivalent to $E_q$.
\end{itemize}
We define $\Nef(E) \subset H^2(X,X \setminus D;\R)$ to be the convex cone spanned by the classes $PD(E_q)$. 
\end{defn}

\begin{defn}
We define $sAmp(X,D) \subset Nef(X,D)$, the cone generated by semi-ample effective divisors supported on $D$, i.e., those for which some positive tensor power of the corresponding line bundle is basepoint-free. 
\end{defn}

\begin{rmk}
\label{rmk:bpfree}
Observe that the second condition on a system of divisors implies that each $E_q$ is basepoint-free. 
It follows that $\Nef(E) \subset sAmp(X,D) \subset Nef(X,D)$.
\end{rmk}

\begin{lem}
\label{lem:sysexist}
Let $\Nef \subset sAmp(X,D)$ be a rational polyhedral cone of full dimension. 
Then for any $U$ as in Definition \ref{defn:sysofdiv}, there exists a system of divisors $E$ in $U$ with $\Nef(E) = \Nef$.
\end{lem}
\begin{proof}
Let $\Nef$ be spanned by classes $a_q$. 
We may assume $a_q = \sum_p \ell_p^q \cdot D_p$ is basepoint-free, and that $\ell_p^q \in \N_{0}$.
Let $\chi_q \in \Gamma(\mathcal{L}_q)$ be a defining section of the corresponding divisor $a_q$. 
For a generic perturbation of $\chi_q$, the zero locus is smooth by Bertini's theorem; and for a sufficiently small perturbation, the zero locus remains inside $U$. 
Now if we take sufficiently many distinct such smoothings of each $E_q$, their union will be a system of divisors (again by Bertini, we can make all of them transverse). 
By construction we have $\Nef(E) = \Nef$.
\end{proof}

\subsection{Coefficient ring}
\label{subsec:coeffring}

Let $(X,D)$ be a $\snc$ pair. 
We have a canonical isomorphism
\begin{align} 
\label{eqn:h2zp} H_2(X,X \setminus D) & \xrightarrow{\simeq} \Z^P,\\
u & \mapsto (u \cdot D_1,\ldots,u \cdot D_p).
\end{align}
Similarly, there is an isomorphism $H_2(\cG X,\cG(X \setminus D)) \simeq \Z^P$. 
It follows that the map
\[ H_2(\cG X,\cG(X \setminus D)) \to H_2(X,X \setminus D)\]
is an isomorphism.

We use this observation to define a group homomorphism 
\begin{equation}
\label{eqn:h2g}
f:  H_2(X,X \setminus D) \to H_1(\cG(X \setminus D))\end{equation}
as the composition
\[ H_2(X,X \setminus D) \simeq H_2(\cG X,\cG(X \setminus D)) \xrightarrow{\partial} H_1(\cG(X \setminus D)) \xrightarrow{-\id} H_1(\cG(X \setminus D)).\]
Explicitly, given a map $u:(\Sigma,\partial \Sigma) \to (X,X \setminus D)$ representing a class in $H_2(X,X \setminus D)$, with $\partial \Sigma \neq \emptyset$, we can choose a lift $\tilde{u}: (\Sigma,\partial \Sigma) \to (\cG X,\cG(X \setminus D))$. 
Then $\tilde{u}|_{\partial \Sigma}$ defines a class in $H_1(\cG(X \setminus D))$, which is defined to be $-f([u])$.

We consider the grading datum $\G := \{\Z \to H_1(\cG (X \setminus D)) \to \Z/2\}$ associated to the symplectic manifold $X \setminus D$.
The homomorphism $f$ equips the $\Bbbk$-algebra $\Rpuncomp(X,D) := \Bbbk[H_2(X,X \setminus D)]$ with a $\G$-grading, as in \S \ref{subsec:ringR}.
The significance of this grading has to do with the relationship between the Maslov index and index theory of Cauchy--Riemann operators, which we briefly recall for the reader's convenience. 

Let $E\to \Sigma$ be a complex vector bundle over a compact Riemann surface $\Sigma$ with boundary $\partial \Sigma$, and $F \subset E|_{\partial \Sigma}$ a totally real subbundle over the boundary. 
Then the Riemann--Roch theorem \cite[Theorem 3.1.10]{mcduffsalamon} states that the associated Cauchy--Riemann operator $D_F$ has Fredholm index
\[ i(D_F) = n \chi(\Sigma) + \mu(E,F),\]
where $\mu(E,F)$ denotes the `boundary Maslov index'. 
For example, if $u:(\Sigma,\partial \Sigma) \to (X,L)$ is a surface with boundary on a Lagrangian $L \subset X$, then $\mu(u^*TX,(\partial u)^* TL)$ is the usual Maslov index.

\begin{lem}
\label{lem:masgrad}
Let $(\Sigma,\partial \Sigma)$ be a compact Riemann surface with boundary, and $u: (\Sigma,\partial \Sigma) \to (X,X \setminus D)$ a smooth map. 
Let $\rho: \partial \Sigma \to \cG(X \setminus D)$ be a lift of $u|_{\partial \Sigma}$ to the Lagrangian Grassmannian. 
Then we have 
\[ \mu(u^*TX, \rho) = [\rho] + \deg(\nov^u)\]
in $H_1(\cG(X \setminus D))$.
\end{lem}
\begin{proof}
This follows by the argument of \cite[Lemma 3.19]{Sheridan:CY} (which unfortunately had a sign error in it).
\end{proof}

\begin{defn}[compare Definition \ref{defn:NEXD}]
Let $(X,D)$ be a $\snc$ pair, and $\Nef \subset Nef(X,D) \subset H^2(X,X \setminus D;\R)$ a full-dimensional convex sub-cone.
We define the strongly convex cone $NE(\Nef)_\R := \Nef^\vee \subset H_2(X,X \setminus D;\R)$, and the sub-monoid 
\[ NE(\Nef) := \Nef^\vee \cap H_2(X,X \setminus D) \subset H_2(X,X \setminus D).\]
When $(X,D)$ is K\"ahler, $Nef(X,D)$ is full-dimensional, so we can define
\[ NE(X,D) := NE(Nef(X,D)).\]
\end{defn}

\begin{defn}
\label{defn:rxde}
We define $\G$-graded $\Bbbk$-algebras 
\begin{align*} 
 \Rpuncomp(X,D) & := \Bbbk[H_2(X,X \setminus D)], \\
\Runcomp(\Nef) & := \Bbbk[NE(\Nef)] \subset \Rpuncomp(X,D).
\end{align*}
Because $NE(\Nef)_\R$ is strongly convex, there is a unique toric maximal ideal $\fmuncomp \subset \Runcomp$,
and we define $R(\Nef)$ to be the $\fmuncomp$-adic completion of $\Runcomp(\Nef)$, in the $\G$-graded sense.  
We define $R(X,D) := R(Nef(X,D))$.
We abbreviate $R(X,D)$ or $R(\Nef)$ by $R$ when the meaning is clear from the context.
\end{defn}

\begin{rmk}
\label{rmk:pospow}
We observe that $\N_0^P \subset NE(X,D) \subset NE(\Nef)$, because each generator $u_p$ can be represented by a holomorphic disc, which meets any divisor $E$ supported on (or near) $D$ positively. 
Thus we have a subalgebra $\Bbbk \power{r_1,\ldots,r_p} \subset R(X,D) \subset R(\Nef)$. 
\end{rmk}

\begin{lem}
\label{lem:c1tors}
Suppose that $c_1(TX)|_{X \setminus D}$ is torsion. 
Then the map $\Z \to H_1(\cG(X \setminus D))$ is injective.
\end{lem}
\begin{proof}
If $k c_1(TX)|_{X \setminus D} = 0$, then the complex line bundle $\left(\Omega^{n,0}(X)\right)^{\otimes k}|_{X \setminus D}$ is trivial. 
Hence it admits a non-vanishing section $\eta$, which defines a squared phase map 
\begin{align}
\cG(X \setminus D) & \to  S^1 \\
\label{eqn:sqph} \Lambda & \mapsto  \arg\left(\eta\left((v_1\wedge \ldots \wedge v_n)^{\otimes k} \right)^2\right)
\end{align}
(compare \cite[\S 11j]{Seidel:FCPLT}). 
The induced map on $H_1$ sends $H_1(\cG(X \setminus D)) \to H_1(S^1)$, so that the composition
\begin{equation} \Z \simeq H_1(\cG_x(X \setminus D)) \to H_1(\cG(X \setminus D)) \to H_1(S^1) \simeq \Z\end{equation}
is multiplication by $k$: in particular, the map $\Z \to H_1(\cG(X \setminus D))$ is injective.
\end{proof}

\begin{defn}[= Definition \ref{defn:Rcl}]
\label{defn:rcl}
We denote by $R_{cl} \subset R$ the subalgebra whose grading lies in the image of $\Z$.
We similarly define $\Rpuncomp_{cl}$, $\Runcomp_{cl}$, and let $NE(\Nef)_{cl} \subset NE(\Nef)$ be the corresponding sub-monoid.
These subalgebras retain a grading in the image of $\Z$, which of course is isomorphic to a quotient of $\Z$. 
The quotient is trivial if $c_1(TX)|_{X \setminus D}$ is torsion, by Lemma \ref{lem:c1tors}. 
\end{defn}

\begin{rmk}
\label{rmk:gradcl}
Observe that 
\[ \Z \to H_1(\cG(X \setminus D)) \to H_1(X \setminus D) \to 0\]
 is exact (since it is the abelianization of the exact sequence of homotopy groups of a fibration and abelianization is right-exact), so $R_{cl}$ is identified with the subalgebra of $R$ consisting of elements whose grading maps to $0 \in H_1(X \setminus D)$.
\end{rmk}

\begin{rmk}
\label{rmk:usualnov}
The image of the grading of $\nov^u$ in $H_1(X \setminus D)$ is easily seen to be $-\partial u$ for any $u \in H_2(X,X \setminus D)$. Therefore, by the long exact sequence in homology for the pair $(X,X \setminus D)$, we have an isomorphism
\[ \Rpuncomp_{cl}(X,D) \simeq \Bbbk[H_2(X)/H_2(X \setminus D)].\]
The degree of $\nov^u$ is $2c_1(u)$, by Lemma \ref{lem:masgrad}.
\end{rmk}

\begin{defn}
We denote by $R_0 \subset R_{cl}$ the degree-$0$ component. 
We define $\Rpuncomp_0$, $\Runcomp_0$, and $NE_0$ similarly.
\end{defn}

\subsection{Nice cones}
\label{subsec:nicecone}

Let $(X,D)$ be a K\"ahler $\snc$ pair with $D = \cup_{p \in P} D_p$.

\begin{defn}
For $K \subset P$, we denote $\bar{K} := P \setminus K$. 
We define
\begin{align*}
D_K &:= \bigcap_{p \in K} D_p  \\
D^K &:= \bigcup_{p \in \bar{K}} D_p. 
\end{align*}
We define $D_\emptyset := X$.
\end{defn}

\begin{defn}
Let $\Nef \subset Nef(X,D)$ be a full-dimensional convex sub-cone. 
We define 
\[ \Nef^K :=  \Nef \cap \R^{\bar{K}} \subset \R^{\bar{K}}.\]
\end{defn}

\begin{defn}
\label{defn:xdenice}
We say that $\Nef$ is \emph{nice} if the following holds for all $K$ such that $D_K \neq \emptyset$:
\begin{itemize}
\item $\Nef^K \subset \R^{\bar{K}}$ is full-dimensional;
\item the image of the map
\[ \Nef \hookrightarrow H^2(X,X \setminus D;\R) \to H^2(X;\R)\]
is equal to the image of the subset $\Nef^K \subset \Nef$.
\end{itemize}
\end{defn}

\begin{lem}
\label{lem:typea}
Let $\Nef$ be nice and $D_K \neq \emptyset$. 
If $u \in \Z^K$ and $v \in NE(\Nef)$ have the same image under the boundary map
\[ \partial: H_2(X,X \setminus D) \to H_1(X \setminus D),\] 
then $v-u \in NE(\Nef)_{cl}$.
\end{lem}
\begin{proof}
We have 
\begin{align*}
v \cdot a & \ge 0 \quad \text{for all $a \in \Nef$} \\
\implies (v-u) \cdot a & \ge 0 \quad \text{for all $a \in \Nef^K$.}
\end{align*}
By assumption, $\partial(v-u) = 0$ in $H_1(X \setminus D)$, so $v-u \in \im(H_2(X))$ by the long exact sequence in homology for the pair $(X,X \setminus D)$. 
It follows that $(v-u) \cdot a$ only depends on the image of $a$ in $H^2(X;\R)$. 
Since $\Nef$ is nice, it follows that $(v-u) \cdot a \ge 0$ for all $a \in \Nef$.
Therefore, $v-u \in NE(\Nef)$. 
Applying Remark \ref{rmk:gradcl}, we see that $v-u \in NE(\Nef)_{cl}$ as required.
\end{proof}

\begin{lem}
\label{lem:extremal}
Suppose that $\Nef$ is nice, and $D_K \neq \emptyset$. 
Then $\R^K$ intersects $NE(\Nef)_\R$ in an extremal face.
\end{lem}
\begin{proof}
Suppose $u \in \R^K \cap NE(\Nef)_\R$ and consider the convex cone spanned by vectors $u' - u$ with $u' \in NE(\Nef)_\R$. 
We denote it by $\mathsf{C} := \mathsf{Cone}(NE(\Nef)_\R - u)$, following \cite{Cox2011}. 
We have
\begin{align*}
v \in \mathsf{C} &\iff (u + \delta \cdot v) \cdot a \ge 0 \text{ for all $\delta>0$ sufficiently small, for all $a \in \Nef$} \\
& \implies v \cdot a \ge 0 \text{ for all $a \in \Nef^K$} \\
& \iff v \in \R^K \times \left(\text{convex cone in $\R^{\bar{K}}$ dual to $\Nef^K$}\right).
\end{align*}
Since $\Nef^K$ is full-dimensional, the cone on the final line is strongly convex. 
It follows that $\R^K \cap NE(\Nef)_\R$ is an extremal face of $NE(\Nef)_\R$ as claimed.
\end{proof}

\begin{lem}
\label{lem:Rnice}
If $\Nef$ is nice, then $R(\Nef)$ is nice in the sense of Definition \ref{defn:Rnice}. 
\end{lem}
\begin{proof}
First we must show that $\nov_p$ generates its graded piece of $\fmuncomp$ over $\Runcomp_0$. 
We start by observing that $\nov_p \in \fmuncomp$ by Remark \ref{rmk:pospow}. 
If $\nov^v \in \fmuncomp$ has the same degree as $\nov_p$, then $v$ and $u_p$ have the same image in $H_1(X \setminus D)$. 
We apply Lemma \ref{lem:typea} with $K = \{p\}$, to conclude that $v-u_p \in NE_{cl}$. 
It follows that $\nov^{v-u_p} \in \Runcomp_{cl}$. 
In fact, since $\nov^v$ and $\nov_p$ have the same degree, we have $\nov^{v-u_p} \in \Runcomp_0$. 
So $\nov^v$ is contained in the $\Runcomp_0$-span of $\nov_p$ as required.

Second we must show that $\nov_p \notin \fm^2$. 
Suppose, to the contrary, that $u_p = v+w$ where $v, w \in NE(\Nef)$ are both non-zero. 
Because $D_p \neq \emptyset$, $u_p$ lies on an extremal ray of $NE(\Nef)_\R$ by Lemma \ref{lem:extremal}. 
It follows that $v$ and $w$ must lie in the $\Z$-span of $u_p$. 
Because $NE(\Nef)_\R$ is strongly convex and contains $u_p$, we must have $\Z\cdot u_p \cap NE(\Nef) = \Z_{\ge 0} \cdot u_p$, so in fact $v$ and $w$ are positive multiplies of $u_p$ (as they are assumed to be non-zero). 
But then their sum can not be equal to $u_p$, so we reach a contradiction: $\nov_p \notin \fm^2$.
\end{proof}

\begin{rmk}
\label{rmk:smoothb}
Observe that if $(X,D)$ is K\"ahler and $D$ is smooth, then $R(X,D) \simeq R(\Nef) \simeq \Bbbk\power{\nov}$ is obviously nice.
\end{rmk}

\begin{lem}
\label{lem:newtypec} 
Suppose that $\Nef$ is nice, and $D_K \neq \emptyset$.
If $u \in \Z^K$ can be written as $u = v+w$ where $v \in NE(\Nef)$ and $w \in NE(\Nef)_{cl}$, then $w = 0$.
\end{lem}
\begin{proof}
By Lemma \ref{lem:extremal}, both $v$ and $w$ must lie in $\Z^K$. 
It follows that $w \cdot D_p= 0$ for all $p \in \bar{K}$.  
Because $w \in NE_{cl}$, it is the image of a class $\tilde{w} \in H_2(X)$, by the long exact sequence in homology for the pair $(X,X \setminus D)$, and $\tilde{w} \cdot D_p = 0$ for all $p \in \bar{K}$. 

Now we claim that the classes $D_p$ with $p \in \bar{K}$ span the images of the other classes $D_p$ in $H^2(X;\R)$. 
To prove this, we first observe that $\Nef$ has non-empty interior, and therefore spans $H^2(X,X \setminus D;\R)$. 
Its image in $H^2(X;\R)$ is equal to the image of $\Nef^K$ because $\Nef$ is nice. 
It follows that the image of $\R^{\bar{K}}$ in $H^2(X;\R)$ spans the image of $H^2(X,X \setminus D;\R)$, as claimed.

It follows that $\tilde{w} \cdot D_p = 0$ for all $p$, not just for $p \in \bar{K}$, and hence that $w=0$.
\end{proof}

Now we address the question of which relative K\"ahler forms admit compatible systems of divisors which are nice. 
Observe that for all $K$ there exist maps
\[ Amp\left(X,D^K\right) \hookrightarrow H^2\left(X,X \setminus D^K;\R\right) \to H^2(X;\R).\]

\begin{defn}
We define an open convex cone $Amp'(X,D) \subset H^2(X;\R)$ by
\[ Amp'(X,D) := \bigcap_{K: D_K \neq \emptyset}\im \left(Amp\left(X,D^K\right)\right).\]
We define an open convex sub-cone $\AAmp(X,D) \subset Amp(X,D)$ to be the interior of the cone
\[ \left\{a \in Amp(X,D): \im(a)  \in Amp'(X,D) \right\}.\]
\end{defn}

\begin{lem}
\label{lem:niceAAmp}
Let $(X,D)$ be a $\snc$ pair, and $a \in \AAmp(X,D)$. 
Then there exists a nice cone $\Nef$ containing $a$ in its interior. 
In fact, $\Nef$ can be chosen to be rational polyhedral and contained in $sAmp(X,D)$.
\end{lem}
\begin{proof}
Observe that $\AAmp(X,D)$ is open, so we can choose a finite set of rational classes $a_q \in \AAmp(X,D)$ such that $a$ is contained in the interior of their convex hull. 
It follows that $\im(a_q) \in Amp'(X,D)$ for all $q$. 
Therefore, for each $q$ and each $K \subset P$ such that $D_K \neq \emptyset$, the pre-image of $\im(a_q)$ in $H^2(X,X \setminus D^K ;\R)$ is an affine subspace whose intersection with the open convex cone $Amp(X,D^K)$ is non-empty. 
Therefore we can choose a finite set of rational classes $a_i(K,q)$ in this intersection, whose convex hull is full-dimensional in the affine subspace. 

We define $\Nef \subset \R^P$ to be the cone spanned by the union of the classes $a_q \in \Q^P$ with the classes $a_i(K,q) \in \Q^{\bar{K}} \subset \Q^P$.
By construction, $\Nef^K$ contains the cone spanned by the classes $a_i(K,q)$, which is full-dimensional.
Furthermore, the images of the classes $a_i(K,q) \in \Nef^K$ in $H^2(X;\R)$ are equal to the images of the classes $a_q$ by construction, so the image of $\Nef^K$ is equal to the image of $\Nef$. 
It follows that $\Nef$ is nice, it clearly contains $a$ in its interior by construction, and it is clearly rational polyhedral. 
Finally we observe that $a_i(K,q) \in Amp(X,D^K) \subset sAmp(X,D^K) \subset sAmp(X,D)$, so $\Nef \subset sAmp(X,D)$.
\end{proof}

In fact, we have a converse to Lemma \ref{lem:niceAAmp}:

\begin{lem}
\label{lem:nicemeansnice}
Let $\Nef \subset Nef(X,D)$ be a nice cone. 
Then the interior of $\Nef$ is contained in $\AAmp(X,D)$.
\end{lem}
\begin{proof}
The interior of $\Nef$ gets mapped to the interior of the image of $\Nef$ in $\im(H^2(X,X \setminus D;\R) \to H^2(X;\R))$. 
This is equal to the interior of the image of $\Nef^K$ in the same space for all $D_K \neq \emptyset$, because $\Nef$ is nice.
We now use the fact that, if $p: V \to W$ is a surjective linear map between finite-dimensional real vector spaces and $B \subset V$ an open convex set, then $p(\bar{B})^\circ = p(B)$ (we leave the proof to the reader). %$x \in p(\bar{B})^\circ \Rightarrow \thereexists y_i \in p(\bar{B})$ s.t. $x \in conv(y_i)^\circ$. Let $y_i = p(z_i), z_i \in \bar{B}$. Let $w_i \in B$, $|z_i-w_i| <\epsilon$. Then $|p(z_i)-p(w_i)|<\epsilon \Rightarrow x \in conv(p(w_i)) \Rightarrow x \in p(conv(w_i)) \subset p(B)$.
Specifically, we apply this to the interior of $\Nef^K$ (whose closure is $\Nef^K$, by the assumption that $\Nef^K$ is full-dimensional). 
It follows that the interior of $\Nef$ gets mapped to the image of the interior of $\Nef^K$, which is contained in $Amp(X,D^K)$.
So the interior of $\Nef$ gets mapped to $Amp'(X,D)$. 
It is also contained in $Amp(X,D)$, so it is contained in $\AAmp(X,D)$.
\end{proof}

\begin{defn}
We say that $(X,D)$ is nice if $\AAmp(X,D) \neq \emptyset$. 
\end{defn}

\subsection{Toric varieties}
\label{subsec:toricsys}

Let $N \simeq \Z^n$ be a lattice and $\Delta \subset N_\R$ a complete, non-singular fan.
Denote by $X = X(\Delta)$ the corresponding smooth, compact toric variety.
It admits an action of the algebraic torus $\T := N \otimes \C^*$.
Let $P := \Delta(1)$ denote the set of one-dimensional rays in $\Delta$.
Each $p \in P$ corresponds to a toric divisor $D_p \subset X$. 
The toric boundary divisor $D$ is the union of the $D_p$. 
It is a $\snc$  divisor. 
We call such $(X,D)$ a \emph{toric $\snc$ pair}.

\begin{lem}
\label{lem:torica}
If $(X,D)$ is a toric $\snc$ pair, then
\[ \AAmp(X,D) = Amp(X,D).\]
\end{lem}
\begin{proof}
Follows from Lemma \ref{lem:examp} below.
\end{proof}

Before introducing the Lemma, we recall some basic facts about line bundles on toric varieties (see, e.g., \cite[\S 3.4]{Fulton1993}, whose notation we follow). 

Let $\underline{a} \in \Z^P$. 
We consider the $\T$-equivariant divisor $D_{\underline{a}} := \sum_p a_p D_p$. 
Recall that such $\T$-equivariant divisors are in one-to-one correspondence with piecewise-linear functions
\begin{equation} \psi_{\underline{a}}: N_\R \to \R,\end{equation}
linear with integer slope on each cone of $\Delta$. 
Namely, $D_{\underline{a}}$ corresponds to the unique such function such that
\begin{equation} \psi_{\underline{a}}(v_p) = -a_p,\end{equation}
where $v_p \in N$ is the primitive generator of the ray $p \in \Delta(1)$. 

Denote $M:=\Hom(N,\Z)$. 
Recall that $\underline{a}$ determines a rational convex polytope 
\begin{equation} P_{\underline{a}} := \{u: \langle u, v_p \rangle \ge -a_p \mbox{ for all $p \in \Delta(1)$}\} \subset M_\R,\end{equation}
and that there is a natural correspondence between global sections of the corresponding line bundle and lattice points of the polytope:
\begin{equation} \Gamma(\EuO(D_{\underline{a}})) \simeq \bigoplus_{u \in P_{\underline{a}} \cap M} \C \cdot \chi^u.\end{equation}
Furthermore, $D_{\underline{a}}$ is semi-ample if and only if it is nef, if and only if the function $\psi_{\underline{a}}$ is convex; and it is ample if and only if $\psi_{\underline{a}}$ is strictly convex.

\begin{lem}
\label{lem:examp}
Suppose that $D_{\underline{a}}$ is a semi-ample $\T$-equivariant divisor.
Then for any $K$ such that $D_K \neq \emptyset$, there exists an effective $\T$-equivariant divisor supported on $D^K$ which is linearly equivalent to $k \cdot D_{\underline{a}}$, for some $k \in \N_{>0}$.
\end{lem}
\begin{proof}
Because $D_{\underline{a}}$ is semi-ample, $\psi_{\underline{a}}$ is convex. 
It follows that the facets of $P_{\underline{a}}$ indexed by $K$ intersect in a common face. 
This face contains a rational point, hence there is a lattice point $m \in P_{k\cdot \underline{a}} \cap M$, for some $k \in \N_{>0}$. 
Let $\chi^{m} \in \Gamma(\EuO(D_{k\cdot \underline{a}}))$ be the corresponding section.

The corresponding effective $\T$-equivariant divisor $\left \{ \chi^{m} = 0 \right\}$ is linearly equivalent to $D_{k \cdot \underline{a}} = k\cdot D_{\underline{a}}$ by construction. 
It is equal to $\sum_p b_p D_p$, where $b_p$ is the affine distance from $m$ to the facet of $P_{k\cdot \underline{a}}$ corresponding to $p$. 
In particular,  $b_p = 0$ for $p \in K$ by construction, so the divisor is supported on $D^K$ as required.
\end{proof}

\subsection{Sub-$\snc$ pairs}
\label{subsec:restdiv}

\begin{defn}
Let $(Y,D')$ be a $\snc$ pair, and $\iota: X \hookrightarrow Y$ the inclusion of a closed complex submanifold, transverse to all strata of $D'$. 
Then $D := D' \cap X$ is a $\snc$ divisor of $X$, so $(X,D)$ is a $\snc$ pair. 
In this situation we say that $(X,D) \subset (Y,D')$ is a \emph{sub-$\snc$ pair}.
\end{defn}

Note that if $(X,D) \subset (Y,D')$ is a sub-$\snc$ pair, then any relative K\"{a}hler form on $(Y,D')$ restricts to one on $(X,D)$. 

If $D = \cup_{p \in P} D_p$ and $D' = \cup_{p \in P'} D'_p$, then we have a map of sets $i: P \to P'$, defined by the property that $D_p$ is a connected component of $D'_{i(p)} \cap X$. 
This is related to the map $\iota_*:  H_2(X,X \setminus D) \to H_2(Y,Y \setminus D') $ as follows: if $u_p$ is a basis element of $H_2(X,X \setminus D)$, then $\iota_*(u_p) = u_{i(p)}$. 
The dual map is
\begin{align}
\label{eqn:istar} \iota^*: H^2(Y,Y \setminus D') & \to H^2(X,X \setminus D) \\
\label{eqn:istarform} \iota^*PD(D'_p) & = \sum_{q \in i^{-1}(p)} PD(D_q).
\end{align}

\begin{lem}
\label{lem:subampprime}
Let $(X,D) \subset (Y,D')$ be a sub-$\snc$ pair, such that $D'_p \cap X$ is either connected or empty for all $p$. 
Then \eqref{eqn:istar} maps $\AAmp(Y,D')$ to $\AAmp(X,D)$.
\end{lem}
\begin{proof}
Note that if $D_K \neq \emptyset$, then $D'_{i(K)} = \iota(D_K) \neq \emptyset$. 
It follows that 
\[ Amp'(Y,D') \subset \bigcap_{K:D_K \neq \emptyset} \im\left(Amp\left(Y,(D')^{i(K)}\right)\right).\]

Now it is clear that ampleness is preserved by restriction. 
It follows that $\iota^*$ maps  $Amp\left(Y,(D')^J\right) \to Amp\left(X,D^K\right)$ for all $J$, where $K = i^{-1}(J)$. 
By our hypothesis that $D'_p \cap X$ is either connected or empty for all $p$, we have $i^{-1}(i(K)) = K$, so $\iota^*$ maps $Amp\left(Y,(D')^{i(K)}\right) \to Amp\left(X,D^K\right)$ for any $K$. 
Hence it maps $Amp'(Y,D') \to Amp'(X,D)$.

Finally, observe that  $\iota^*$ maps $Amp(Y,D') \to Amp(X,D)$. 
Combining these observations, the result follows.
\end{proof}

\begin{cor}
\label{cor:toricEabc}
Let $(Y,D')$ be a toric $\snc$ pair, and $(X,D) \subset (Y,D')$ a sub-$\snc$ pair such that $D'_p \cap X$ is connected or empty for all $p$. 
If $\omega$ is a relative K\"ahler form on $(X,D)$ which is the restriction of a relative K\"ahler form from $(Y,D')$, then there exists a  system of divisors $E$ such that $\Nef(E)$ is nice and contains $[\omega]$ in its interior.
\end{cor}
\begin{proof}
Follows by Lemmas \ref{lem:torica}, \ref{lem:subampprime}, \ref{lem:niceAAmp} and \ref{lem:sysexist}.
\end{proof}

\subsection{The ambient coefficient ring}
\label{subsec:Ramb}

In this section we continue to consider a sub-$\snc$ pair $(X,D) \subset (Y,D')$ as in the previous section, but we do not assume that $D'_p \cap X$ is connected or empty for all $p$. 
We prove analogues of the results from the previous section adapted to this case. 
This requires us to introduce an alternative coefficient ring, which we call the \emph{ambient ring}; but first we need to introduce the appropriate grading datum. 

Let $M$ be a manifold and $\EuV$ a complex vector bundle on $M$. 
We define $\EuR(\EuV) \to M$ to be the fibre bundle of totally real subspaces of $\EuV$.
We define $\G(\EuV)$ to be the grading datum $\{\Z \to H_1(\EuR(\EuV)) \to \Z/2\}$, where the map $\Z \to H_1(\EuR(\EuV))$ is induced by the inclusion of a fibre $\EuR(\EuV_m) \hookrightarrow \EuR(\EuV)$, and the map $H_1(\EuR(\EuV)) \to \Z/2$ is given by the first Stiefel--Whitney class of the tautological vector bundle over $\EuR(\EuV)$. 
In particular, we have $\G(X) \simeq \G(TX)$ for any K\"ahler manifold $X$ (compare \cite[Lemma 3.16]{Sheridan:CY}). 

Observe that there is a map $\EuR(\EuV) \to \EuR(\wedge^{top} \EuV)$. 
It sends a totally real subspace $V \subset \EuV$ to the subspace spanned by $v_1 \wedge \ldots \wedge v_n$, where $\{v_1,\ldots,v_n\}$ is a real basis for $V$. 
This map induces an isomorphism on $H_0$ and $H_1$ of the respective fibres, and therefore induces an isomorphism on $H_1$ of the total space, by a comparison argument for the Serre spectral sequences of the respective fibre bundles. 
So we have a natural isomorphism
\[ \G(\EuV) \simeq \G(\wedge^{top} \EuV).\]

We also observe that if $M \subset N$ is a submanifold and $\EuV$ a complex vector bundle on $N$, then we have $\EuR(\EuV|_M) \subset \EuR(\EuV)$. 
The inclusion induces a morphism of grading data $\G(\EuV|_M) \to \G(\EuV)$.

\begin{defn}
Let $(X,D) \subset (Y,D')$ be a sub-$\snc$ pair.
Assume that $\EuV \to Y \setminus D'$ is a holomorphic vector bundle, $s \in \Gamma(\EuV)$ a holomorphic section, and $X \setminus D = \{s=0\}$.
We define
\[ \G_{amb} := \G\left(\EuR\left(TY \oplus \EuV^\vee\right)\right).\]
\end{defn}

We claim that there is a natural morphism of grading data $\bff: \G(X \setminus D) \to \G_{amb}$. 
Indeed, there is a canonical isomorphism
\[ \wedge^{top}(TX)|_{X \setminus D} \simeq \wedge^{top}(TY \oplus \EuV^\vee)\]
induced by the normal bundle sequence 
\[ 0 \to TX \to TY \to \EuN_{X / Y} \to 0\]
restricted to $X \setminus D$, together with the isomorphism $\EuN_{X / Y}|_{X \setminus D} \simeq \EuV|_{X \setminus D}$.
Thus we obtain a morphism of grading data
\begin{align*}
\G(X \setminus D) & \simeq \G(T(X \setminus D)) \\
& \simeq \G((TY \oplus \EuV^\vee)|_{X \setminus D}) \\
& \to \G((TY \oplus \EuV^\vee)|_{Y \setminus D'}) \\ 
& =: \G_{amb}.
\end{align*}
We denote this morphism by $\bff$.

Now we define the appropriate coefficient ring.
We define $P_{amb} \subset P'$ to be the image of $i:P \to P'$. 
Note that we have a natural isomorphism
\[\im\left(\iota_*: H_2(X,X \setminus D) \to H_2(Y,Y \setminus D')\right) \simeq \Z^{P_{amb}}.\]
It is easy to verify that the map
\[ H_2(X,X \setminus D) \xrightarrow{\eqref{eqn:h2g}} \G(X \setminus D) \xrightarrow{\bff} \G_{amb} \]
factors through $\iota_*$. 
This allows us to equip
\[ \Rpuncomp_{amb} := \Bbbk\left[\Z^{P_{amb}} \right]\]
with a $\G_{amb}$-grading, so that we have a homomorphism of $\G_{amb}$-graded $\Bbbk$-algebras
\[ \wt{\Phi}^*: \bff_* \Rpuncomp(X,D) \to \Rpuncomp_{amb}\]
induced by $\iota_*$.

We also have a natural inclusion $\R^{P_{amb}} \subset H^2(Y,Y \setminus D;\R)$, and the composition
\[ \R^{P_{amb}}  \hookrightarrow H^2(Y,Y \setminus D';\R) \xrightarrow{\iota^*} H^2(X,X \setminus D;\R)\]
is injective.

Now let $\Nef \subset Nef(X,D)$ be a full-dimensional convex sub-cone. 
We define
\begin{align*}
\Nef_{amb}  &:= \Nef \cap \R^{P_{amb}}\\
 NE_{amb}(\Nef) &:= \Nef_{amb}^\vee \cap \Z^{P_{amb}}\\
 \Runcomp_{amb}(\Nef) &:= \Bbbk[NE_{amb}(\Nef)].
\end{align*}
We will assume that $\Nef_{amb}$ intersects the interior of $\Nef$. 
This implies that $\Nef_{amb}$ has non-empty interior, and therefore is full-dimensional. 
It follows that $\Runcomp_{amb}(\Nef)$ has a unique toric maximal ideal $\fmuncomp_{amb} \subset \Runcomp_{amb}$, so we can define $R_{amb}(\Nef)$ to be the completion of $\Runcomp_{amb}(\Nef)$ with respect to the $\fmuncomp_{amb}$-adic filtration, in the $\G_{amb}$-graded sense. 

\begin{lem}
\label{lem:Phidef}
Suppose that $\Nef_{amb}$ intersects the interior of $\Nef$. 
Then $\wt{\Phi}^*(\fmuncomp) \subset \fmuncomp_{amb}$.
\end{lem}
\begin{proof}
Observe that
\[ NE(\Nef)_\R \cap \ker(\iota_*) = \mathsf{M}^\vee,\]
where $\mathsf{M}$ is the image of $\Nef$ in $H^2(X,X \setminus D;\R)/\R^{P_{amb}}$. 
Let $a$ be a point in the intersection of the interior of $\Nef$ with $\Nef_{amb}$. 
Then a neighbourhood of $a$ is contained in $\Nef$, so a neighbourhood of the image of $a$ is contained in $\mathsf{M}$. 
The image of $a$ is the origin, so $\mathsf{M}$ contains a neighbourhood of the origin. 
But $\mathsf{M}$ is also conical, so it is all of $H^2(X,X \setminus D;\R)/\R^{P_{amb}}$. 
It follows that $\mathsf{M}^{\vee} = \{0\}$.

Now suppose that $\nov^u \in \fmuncomp$. 
We then have $u \in NE(\Nef) \setminus \{0\}$, so it follows from the above that $\iota_*(u) \neq 0$, hence $\wt{\Phi}^*(\nov^u) = \nov^{\iota_* (u)} \in \fmuncomp_{amb}$. 
\end{proof}

In particular, $\wt{\Phi}^*$ extends to the completions, so in this situation we have a map
\begin{equation}
\label{eqn:Phi}
 \Phi^*: \bff_* R(\Nef) \to R_{amb}(\Nef).
\end{equation}

Now we introduce the appropriate notions of positivity in this context. 
We consider the pre-image of $c_1(TY \oplus \cV^\vee) \in H^2(Y;\R)$ in $H^2(Y,Y \setminus D';\R)$, and the pre-image of $\R^{P_{amb}}$ in $H^2(Y,Y \setminus D';\R)$.

\begin{defn}
We say that $\Nef$ is \emph{amb-positive}/\emph{amb-Calabi--Yau}/\emph{amb-semi-positive} if the pre-image of $c_1(TY \oplus \cV^\vee)$ intersects the pre-image of the interior of $\Nef_{amb} \subset \R^{P_{amb}}$/the pre-image of $0 \in \R^{P_{amb}}$/the pre-image of $\Nef_{amb} \subset \R^{P_{amb}}$.
\end{defn}

If $\Nef$ is amb-positive, amb-Calabi--Yau or amb-semi-positive, then $c_1(TY \oplus \cV^\vee)|_{Y \setminus D'}$ is torsion, so the map $\Z \to \G_{amb}$ is injective (compare Lemma \ref{lem:c1tors}). 
We define $R_{amb}(\Nef)_{cl}$ to be the part of $R_{amb}(\Nef)$ whose grading lies in the image of $\Z$.
It is a $\Z$-graded algebra. 
If $\Nef$ is amb-positive/amb-Calabi--Yau/amb-semi-positive, then the $\Z$-graded part of the maximal ideal $\fm \cap R_{amb}(\Nef)_{cl}$ is graded in positive/zero/non-negative degrees.

\begin{defn}
Let $K \subset P_{amb}$. 
We define
\[\Nef_{amb}^K  := \Nef^{i^{-1}(K)} \cap \R^{P_{amb}} \subset \R^{\bar{K}}.\]
\end{defn}

\begin{defn}
\label{defn:ambnice}
We say that $\Nef$ is \emph{amb-nice} if the following holds for all $K \subset P_{amb}$ such that $D'_K \cap X \neq \emptyset$:
\begin{itemize}
\item $\Nef_{amb}^K$ has non-empty interior in $\R^{\bar{K}}$;
\item the image of the map
\[ \Nef_{amb} \hookrightarrow \R^{P_{amb}} \hookrightarrow H^2(Y,Y \setminus D';\R) \to H^2(Y;\R)\]
is equal to the image of the subset $\Nef_{amb}^K \subset \Nef_{amb}$.
\end{itemize}
\end{defn}

\begin{lem}
\label{lem:ambtypea}
Let $\Nef$ be amb-nice and $D'_K \cap X \neq \emptyset$. 
If $u \in \Z^K$ and $v \in NE_{amb}(\Nef)$ have the same image under the boundary map
\[ \partial: H_2(Y,Y \setminus D) \to H_1(Y \setminus D),\] 
then $v-u \in NE_{amb}(\Nef)_{cl}$.
\end{lem}
\begin{proof}
Follows that of Lemma \ref{lem:typea}. 
\end{proof}

\begin{lem}
\label{lem:ambextremal}
Suppose that $\Nef$ is amb-nice, and $D'_K \cap X \neq \emptyset$. 
Then $\R^K$ intersects $NE_{amb}(\Nef)_\R$ in an extremal face.
\end{lem}
\begin{proof}
Follows that of Lemma \ref{lem:extremal}.
\end{proof}

\begin{lem}
\label{lem:Rambnice}
If $\Nef$ is amb-nice, then $R_{amb}(\Nef)$ is nice in the sense of Definition \ref{defn:Rnice}. 
\end{lem}
\begin{proof}
Follows that of Lemma \ref{lem:Rnice}.
\end{proof}

\begin{lem}
\label{lem:ambnewtypec} 
Suppose that $\Nef$ is amb-nice, and $D'_K \cap X \neq \emptyset$.
If $u \in \Z^K$ can be written as $u = v+w$ where $v \in NE_{amb}(\Nef)$ and $w \in NE_{amb}(\Nef)_{cl}$, then $w = 0$.
\end{lem}
\begin{proof}
Follows that of Lemma \ref{lem:newtypec}.
\end{proof}

\begin{lem}
\label{lem:ambniceAAmp}
Let $a \in \AAmp(Y,D')$. 
Then there exists an amb-nice cone $\Nef$ containing $\iota^*a$ in its interior. 
In fact, $\Nef$ can be chosen to be rational polyhedral and contained in $sAmp(X,D)$. 
Furthermore, $\Nef_{amb}$ intersects the interior of $\Nef$.
\end{lem}
\begin{proof}
Observe that there exists a nice cone $\Nef' \subset Nef(Y,D')$ containing $a$ in its interior, by Lemma \ref{lem:niceAAmp}. 
Furthermore we may assume $\Nef'$ to be rational polyhedral and contained in $sAmp(Y,D')$. 
It follows that $\iota^*\Nef' \subset \R^{P_{amb}} \subset H^2(X,X \setminus D;\R)$ is rational polyhedral and contained in $sAmp(X,D)$. 
We now choose a finite set of rational elements $a_q \in Amp(X,D)$ whose convex hull has non-empty interior, which intersects $\R^{P_{amb}}$ in a subset of the interior of $\iota^*\Nef'$. 
By taking the cone over $\iota^*\Nef'$ together with these elements, we obtain a rational polyhedral cone $\Nef \subset sAmp(X,D)$ of full dimension, such that $\Nef_{amb} = \iota^*\Nef'$. 
It follows from the fact that $\Nef'$ is nice, that $\Nef$ is amb-nice.

Note that $\Nef_{amb}$ intersects the interior of $\Nef$, because $\iota^* a$ lies in both.
\end{proof}

\begin{defn}
We say that $(X,D) \subset (Y,D')$ is amb-nice if $(Y,D')$ is nice.
\end{defn}

\section{The affine Fukaya category and the map $\co$}\label{sec:affco}

Let $(X,D)$ be a K\"{a}hler $\snc$ pair.
In this section we sketch the construction of the affine Fukaya category $\fuk(X \setminus D)$, following \cite{Seidel:FCPLT}. 
We then define a $\G$-graded $\Bbbk$-vector space $\sh^\bullet(X,D)$, which should be thought of as a substitute for a certain subspace of the symplectic cohomology $SH^\bullet(X \setminus D)$ (similar ideas and constructions appear in \cite{Ganatra2016a}). 
Under the assumption that $(X,D)$ is nice (or amb-nice), we construct a map of $\G$-graded $\Bbbk$-vector spaces
\[ \co: \sh^\bullet(X,D) \to \HH^\bullet(\fuk(X \setminus D)),\]
which should be thought of as a substitute for the closed--open string map
\begin{equation}\EuC \EuO: SH^\bullet(X \setminus D) \to \HH^\bullet(\fuk(X \setminus D))\end{equation}
(see \cite{Seidel2002,Ganatra2013}). 

We prove that $\fuk(X \setminus D)$ and $\co$ are invariants of the $\snc$ pair $(X,D)$, i.e., that they do not depend on the choices of relative K\"{a}hler form and other data that are used in their construction (up to quasi-equivalence, in the case of the affine Fukaya category). 
Of course this is expected from the analogy with symplectic cohomology.

We then prove some properties of $\co$: we show that it respects algebra structures in a very restrictive sense (by analogy with \cite[Proposition 5.3]{Ganatra2013}), and explain how it behaves with respect to branched covers of $\snc$ pairs.

\subsection{Almost-complex structures}

Let $(X,D,\omega,c)$ be a relative K\"{a}hler manifold, equipped with a system of divisors $E \subset  \{h > c\}$ (note: by convention, $\{h > c\}$ includes $D = \{h = +\infty\}$). 

\begin{defn}
We say that an almost-complex structure $J$ on $X$ is \emph{adapted} to $E$ if 
\begin{itemize}
\item It is $\omega$-tame.
\item It makes each component $E_q$ of $E$ into an almost-complex submanifold.
\item $\alpha = -dh \circ J$ in a neighbourhood of $\{h=c\}$. 
\end{itemize}
\end{defn}

We denote the integrable complex structure on $X$ by $J_0$. 
It is clearly adapted to any system of divisors.

\begin{lem}
\label{lem:posint}
Suppose that $J$ is adapted to $E$.
Let $(\Sigma,\partial \Sigma)$ be a compact Riemann surface with boundary, and
\begin{equation} u: (\Sigma,\partial \Sigma) \to (X,\{h \le c\}) \end{equation}
a $J$-holomorphic map.
Then $[u] \in NE(\Nef(E))$.
\end{lem}
\begin{proof}
We must show that $u \cdot E_q \ge 0$ for all $q$. 
If $u$ is not contained in $E_q$, then $u \cdot E_q \ge 0$ by positivity of intersection (see \cite[Exercise 2.6.1]{mcduffsalamon}, \cite[Proposition 7.1]{Cieliebak2007}). 
If $u$ is contained in $E_q$, $\Sigma$ must have empty boundary, so $u$ defines a class in $H_2(X)$. 
Because $E$ is a system of divisors, there exists another component $E_r$ of $E$, transverse to $E_q$ and cohomologous to it in $X$. 
In particular, $[u] \cdot E_q = [u] \cdot E_r$; so if $u$ is not contained in $E_r$ then we are done, by the previous argument. 
If $u$ is contained in $E_q \cap E_r$, we can find another $E_s$ transverse to $E_q \cap E_r$, and so on; we ultimately reach the conclusion that $u$ must be contained in a $0$-dimensional set, hence represent the zero homology class, which indeed has vanishing (in particular, non-negative) intersection number with all $E_q$.
\end{proof}

\begin{lem}
\label{lem:maxprin}
Suppose that $J$ is adapted to $E$. 
Let $(\Sigma,\partial \Sigma)$ be a compact Riemann surface with boundary, and
\begin{equation} u:(\Sigma,\partial \Sigma) \to (X,\{h \le c\}) \end{equation}
a $J$-holomorphic map, such that $[u] = 0$ in $H_2(X,X \setminus D)$. 
Then $u$ is contained in $\{h \le c\}$.
\end{lem}
\begin{proof}
This follows from the `integrated maximum principle' of \cite[Lemma 7.2]{Abouzaid2007} (see also \cite[Lemma 7.5]{Seidel:FCPLT}). 
Namely, suppose $\Sigma$ intersects $\{h=c\}$ transversely (if it doesn't, we can perturb $c$ to make it so).
Let $u':(\Sigma',\partial \Sigma') \to (\{h>c\},\{h=c\})$ be the part of $\Sigma$ mapping to $\{h\ge c\}$. 
We have $[u'] = [u] = 0$ in $H_2(X,X \setminus D)$, by assumption; hence, $[\omega](u') = 0$. 
Then, by Equation \eqref{eqn:links},
\begin{equation} \omega(u') = \alpha(\partial u').\end{equation}
But now, $\omega(u') \ge 0$ because $J$ is $\omega$-tame, and $\alpha(\partial u') \le 0$ because $\alpha = -dh \circ J$ along $\{h=c\}$. 
So $\omega(u') = 0$, hence $u'$ is constant, so $u$ is contained in $\{h \le c \}$.
\end{proof}

\begin{lem}
\label{lem:nodisc}
Suppose that $J$ is adapted to $E$.
Let $L \subset \{h < c\}$ be a closed, exact Lagrangian submanifold (i.e., $\alpha|_L$ is exact), and $u:(\Sigma,\partial \Sigma) \to (X,L)$ a non-constant $J$-holomorphic curve with boundary on $L$.
Then $u$ intersects $E$ (and in particular, $u \cdot E_q > 0$ for some $q$, by positivity of intersection).
\end{lem}
\begin{proof}
We prove the contrapositive: suppose that $u$ does not intersect $E$, so $u \cdot E_q = 0$ for all $q$. 
Because $E$ is a system of divisors, the classes $E_q$ span $H_{2n-2}(D)$, so it follows that $[u] = 0$, and hence $u \subset \{h \le c\}$ by Lemma \ref{lem:maxprin}. 
Because $L$ is exact, this implies that $u$ has zero energy and hence is constant.
\end{proof}

\begin{lem} 
\label{lem:c1loss}
Suppose that $J$ is adapted to $E$, and $K \subset Q$ is some subset.
If $\Sigma$ is a closed Riemann surface, and $u: \Sigma \to E_K$ is $J$-holomorphic, then
\begin{equation} c_1(TE_K)(u) \le c_1(TX)(u).\end{equation}
\end{lem}
\begin{proof}
By adjunction, we have
\[ c_1(TX)(u) = c_1(TE_K)(u) + \sum_{k \in K} E_k \cdot u.\]
The result now follows, as $E_k \cdot u \ge 0$ for each $k$ by Lemma \ref{lem:posint}.
 \end{proof}

\subsection{The affine Fukaya category}
\label{subsec:exfuk}

Let $(X,D,\omega,c)$ be a relative K\"{a}hler manifold. 
We recall the construction of the exact Fukaya category of the Liouville domain $\{h \le c\}$, following \cite{Seidel:FCPLT}, with modifications as in \cite{Sheridan:CY}. 
It is a $\Bbbk$-linear, $\G$-graded $A_\infty$ category, where $\G := \{\Z \to H_1(\cG(X \setminus D)) \to \Z/2\}$ is the grading datum associated to $X \setminus D$. 

The objects are closed, exact, anchored Lagrangian submanifolds $L \subset \{h < c\}$, equipped with a Pin structure. 
Recall that the `anchored' part of the terminology means that each Lagrangian comes with a lift
\begin{equation} \xymatrix{ & \wt{\cG}(X \setminus D) \ar[d] \\
L\ar@{-->}[ru] \ar[r] & \cG(X \setminus D) },\end{equation}
where $\wt{\cG}(X \setminus D)$ denotes the universal abelian cover of the (total space of the) Lagrangian Grassmannian $\cG(X \setminus D)$.

\begin{defn}
\label{defn:perts}
Throughout this section, we will choose perturbation data for our pseudoholomorphic curve equations coming from the set $\EuH \times \EuJ$, where
\begin{equation} \EuH := \{H \in C^{\infty}(X): H|_{\{h\ge c\}} = 0\} \end{equation}
is the set of Hamiltonian functions that vanish on $\{h\ge c\}$, and $\EuJ$ is the set of $\omega$-tame almost-complex structures which are equal to the integrable complex structure $J_0$ on $\{h \ge c\}$.
In particular, for any system of divisors $E \subset \{h>c\}$, these almost-complex structures are adapted to $E$.
\end{defn}

For each pair of objects $(L_0,L_1)$, we choose Floer data: 
\begin{equation} H_{01}: [0,1] \to \EuH \mbox{ and } J_{01}: [0,1] \to \EuJ.\end{equation} 
We choose $H_{01}$ so that the time-1 Hamiltonian flow of the corresponding Hamiltonian vector field $X_{H_{01}}$ makes $L_0$ transverse to $L_1$. 
Thus we have some finite number of Hamiltonian chords $y$ from $L_0$ to $L_1$; associated to each of these is a $\G$-graded orientation line $o_y$, to which we associate the normalization $|o_y|$, which is a $\G$-graded one-dimensional $\Bbbk$-vector space (the definition is given in \S \ref{subsec:ainfsign}). 
We define
\begin{equation} hom^\bullet(L_0,L_1) := \bigoplus_y |o_y| .\end{equation}
It is a $\G$-graded $\Bbbk$-vector space.

To define the $A_\infty$ structure maps 
\begin{equation} \mu^s:  hom^\bullet(L_0,L_1) \otimes \ldots \otimes hom^\bullet(L_{s-1},L_s) \to hom^\bullet(L_0,L_s)[2-s],\end{equation}
we consider the moduli spaces $\mathcal{R}(\mathbf{L})$ of holomorphic discs with boundary punctures $(\zeta_0,\ldots,\zeta_s)$ in order, modulo biholomorphism.
We label the boundary components by the tuple $\mathbf{L} := (L_0,\ldots,L_s)$.  
We make a universal choice of strip-like ends $\epsilon_\zeta$ for these moduli spaces, where $\zeta_0$ is outgoing and the rest of the $\zeta_i$ are incoming. 
We make a consistent universal choice of perturbation data, chosen from the set $\EuH \times \EuJ$ (see \cite[\S 9i]{Seidel:FCPLT}). 
We label the boundary punctures by a tuple $\mathbf{y} = (y_0,\ldots,y_s)$ of chords $y_i$ from $L_{i-1}$ to $L_i$, and $y_0$ from $L_s$ to $L_0$, and we consider the moduli space $\mathcal{M}(\mathbf{y})$ of solutions $(r,u)$ to the corresponding pseudoholomorphic curve equation
\begin{equation}
\label{eqn:pseudo}
\left\{ \begin{array}{l}
u: S_r  \to  X\\
{[}u{]} = 0 \mbox{ in $H_2(X,X \setminus D)$} \\
u(z) \in L_C \mbox{ for $z \in C \subset \partial S_r$} \\
u(\epsilon_\zeta(s,t)) \to y_\zeta(t) \mbox{ for $\zeta$ a boundary puncture of $S_r$}\\
(du - Y)^{0,1} = 0.
\end{array} \right.
\end{equation}
Here, $S_r$ is the Riemann surface corresponding to $r \in \mathcal{R}(\mathbf{L})$, and $Y \in \Hom_\R(TS_r,TX)$ is determined by the Hamiltonian part of the perturbation datum, and the complex structure on the bundle $\Hom_\R(TS_r,TX)$ (with respect to which the `$0,1$' component of $du - Y$ is taken) is determined by the complex structure on the Riemann surface $S_r$, together with the almost-complex structure part of the perturbation datum.

For generic choice of Floer and perturbation data, the moduli spaces $\mathcal{M}(\mathbf{y})$ are smooth manifolds of the expected dimension. 
Furthermore, they admit Gromov compactifications $\overline{\mathcal{M}}(\mathbf{y})$ by broken discs. 
Unstable disc and sphere bubbles are ruled out by exactness (compare Lemma \ref{lem:nodisc}), so the boundary strata are smooth manifolds of the expected dimension.
In particular, the zero-dimensional component $\mathcal{M}(\mathbf{y})_0$ is a finite set of points. 
Each determines an isomorphism
\begin{equation}
\label{eqn:oropeq}
o(y_1) \otimes \ldots \otimes o(y_s) \to o(y_0)[2-s]
\end{equation}
in accordance with \S \ref{subsec:ainfsign}.
The sum of the induced maps on normalizations $|o_{y_i}|$ is the corresponding matrix coefficient of the map $\mu^s \in CC^2(\fuk(\{h \le c\}))$. 

The one-dimensional component $\overline{\mathcal{M}}(\mathbf{y})_1$ has the structure of a compact topological one-manifold with boundary, by a gluing theorem. 
Its boundary points are in one-to-one correspondence with the terms in the Gerstenhaber product $(\mu \circ \mu)^*(y_1,\ldots,y_s)$, so the sum of these terms is zero: i.e., the maps $\mu^s$ define the structure of a $\Bbbk$-linear, $\G$-graded $A_\infty$ category on $\fuk(\{h \le c\})$ (the signs are worked out in \cite{Seidel:FCPLT}, we give an alternative formulation in \S \ref{subsec:ainfsign} which is adapted for our later purposes).
 
Because we only considered discs $u$ with $[u] = 0$ in $H_2(X,X \setminus D)$, Lemma \ref{lem:maxprin} implies that all of these discs are contained inside $\{h \le c\}$.
Therefore, the definition we have given coincides with the exact Fukaya category of the Liouville domain $\{h \le c\}$ as defined in \cite{Seidel:FCPLT}.

\subsection{The space $\sh^\bullet(X,D)$}
\label{subsec:shxd}

We define a homomorphism
\begin{align}
H_2(X,X \setminus D) & \to H_1(\cG(X \setminus D)), \\
u & \mapsto 2u \cdot D - f(u), 
\end{align}
where $f$ is the homomorphism from \eqref{eqn:h2g}. 
This equips $\Bbbk[H_2(X,X \setminus D)]$ with a (new) $\G$-grading.

We recall that $H_2(X,X \setminus D) \simeq \Z^P$, where $P$ indexes the irreducible components $D_p$ of $D$. 
The sub-monoid $\N_0^P \subset \Z^P$ corresponds to the classes $u \in H_2(X,X \setminus D)$ such that $u \cdot D_p \ge 0$ for all $p$. 
We define a sub-monoid $M \subset \N_0^P$ to be generated by all $u$ such that $D_{K(u)} = \emptyset$, where
\begin{equation} K(u) := \{p: u \cdot D_p \neq 0\} \subset P.\end{equation}

\begin{defn}
We define the $\G$-graded $\Bbbk$-vector space 
\[ \sh^\bullet(X,D) := \Bbbk\left[\N_0^P\right]/\Bbbk[M].\]
It has a $\Bbbk$-basis $\{z^u\}_{u \in \mathsf{pos}_1}$, where 
\[\mathsf{pos}_1 := \N_0^P \setminus M\]
 is the set of classes $[u] \in H_2(X,X \setminus D)$ that can be represented by a disc $u$ that is disjoint from $D$ except possibly for a single point in its interior, at which it intersects all components of $D$ non-negatively.
\end{defn}

\begin{rmk}
\label{rmk:degrpzp}
It follows easily from the definition that the generators $z_p:=z^{u_p} \in \mathsf{sh}^\bullet(X \setminus D)$ and $\nov_p :=\nov^{u_p} \in R(X,D)$ are graded so that $\nov_p z_p$ has degree $2$ for all $p$.
\end{rmk}

\subsection{The map $\mathsf{co}$}
\label{subsec:co}

Let $\mathbf{L} := (L_0,\ldots,L_s)$ be a tuple of objects of $\fuk(\{h \le c\})$. 
We consider the moduli space $\mathcal{R}_1(\mathbf{L})$ of holomorphic discs with boundary punctures $\zeta_0,\ldots,\zeta_s$ in order, together with an internal marked point $q$, modulo biholomorphism. 
We label the boundary components by the Lagrangians $L_i$. 
As in \S \ref{subsec:exfuk}, we make a universal choice of strip-like ends $\epsilon_\zeta$, outgoing at $\zeta_0$ and incoming at all other $\zeta_i$. 
We make a universal choice of perturbation data for these moduli spaces, chosen from $\EuH \times \EuJ$ as before. 
The Deligne--Mumford compactification of $\mathcal{R}_1(\mathbf{L})$ consists of trees of disc bubbles, each corresponding to an element of some $\mathcal{R}(\mathbf{L}')$ or $\mathcal{R}_1(\mathbf{L}')$. 
We require our choice of perturbation data to be consistent with respect to this compactification. 

Let $v \in \mathsf{pos}_1$.
We label the boundary punctures by a tuple $\mathbf{y} = (y_0,\ldots,y_s)$ of chords $y_i$ from $L_{i-1}$ to $L_i$, and $y_0$ from $L_s$ to $L_0$, and we consider the moduli space $\mathcal{M}(\mathbf{y},v)$ of solutions $(r,u)$ to the corresponding pseudoholomorphic curve equation
\begin{equation}
\label{eqn:defpseudo}
\left\{ \begin{array}{l}
u: S_r  \to  X\\
{[}u{]} = v \mbox{ in $H_2(X,X \setminus D)$} \\
u(z) \in L_C \mbox{ for $z \in C \subset \partial S_r$} \\
u(\epsilon_\zeta(s,t)) \to y_\zeta(t) \mbox{ for $\zeta$ a boundary puncture of $S_r$}\\
\mbox{$u$ is tangent at $q$ to $D_p$ to order $v\cdot D_p - 1$, for all $p$}\\
(du - Y)^{0,1} = 0.
\end{array} \right.
\end{equation}
Note: by convention, `$u$ is tangent at $q$ to $D_p$ to order $-1$' is an empty condition. 

The multiplicity of any intersection of $u$ with a component $D_{p'}$ of $D$ is positive. 
So the fact that $[u] = v$ in $H_2(X,X \setminus D)$ implies that $u$ does not meet $D$ anywhere other than at $q$, where it intersects $D_p$ with multiplicity $v \cdot D_p$ for all $p$.
The moduli space $\mathcal{M}(\mathbf{y},v)$ is a smooth manifold of the expected dimension for generic choice of perturbation data (compare \cite[Lemma 6.7]{Cieliebak2007}), and it admits a Gromov  compactification $\overline{\mathcal{M}}(\mathbf{y},v)$. 

This compactification could, a priori, include some holomorphic sphere bubbles, for which we have not considered questions of transversality (the sphere bubbles could even be contained inside the region $\{h \ge c\}$ on which the perturbation data are fixed, for example). 

\begin{lem}
\label{lem:nospheres}
Suppose that $(X,D)$ is nice or amb-nice. 
Then $\overline{\mathcal{M}}(\mathbf{y},v)$ does not include any sphere bubbles.
\end{lem}
\begin{proof}
First we consider the case that $(X,D)$ is nice.
By Lemmas \ref{lem:niceAAmp} and \ref{lem:sysexist}, we can choose a system of divisors $E \subset \{h > c\}$ such that $\Nef(E)$ is nice. 
The almost-complex structures in our perturbation data are adapted to $E$ by construction. 
Thus, all solutions to the pseudoholomorphic curve equation have non-negative intersection number with each component $E_q$ of $E$, by Lemma \ref{lem:posint}. 

Thus, if a sphere bubbled off, we would have $v = w + w'$ where $w \in NE(\Nef(E))$ and $w' \in NE(\Nef(E))_{cl}$ is the homology class of the sphere bubble. 
But $v \in \mathsf{pos}_1$, so $D_{K(v)} \neq \emptyset$, so $w' = 0$ by Lemma \ref{lem:newtypec}. 
So the sphere bubble is constant, and hence stable. 
But the disc only has a single internal marked point $q$, so no tree of stable sphere bubbles can bubble off from it.

In the case that $(X,D) \subset (Y,D')$ is amb-nice, Lemmas \ref{lem:ambniceAAmp} and \ref{lem:sysexist} show that we can choose a system of divisors $E$ such that $\Nef(E)$ is amb-nice and $\Nef(E)_{amb}$ intersects the interior of $\Nef(E)$. 
The argument above then shows that if $w' \in H_2(X,X \setminus D)$ is the homology class of a sphere bubble, then the image of $w'$ in $H_2(Y,Y \setminus D')$ vanishes. 
It follows that $w' = 0$ by Lemma \ref{lem:Phidef}.
\end{proof}

The compactification also does not include any unstable disc bubbles, by Lemma \ref{lem:nodisc}. 
It follows that the boundary strata of $\overline{\mathcal{M}}_1(\mathbf{y},v)$ are smooth manifolds of the expected dimension.
In particular, the zero-dimensional component $\mathcal{M}_1(\mathbf{y},v)_0$ is a finite set of points, each of which determines an isomorphism
\begin{equation} o(y_1) \otimes \ldots \otimes o(y_s) \to o(y_0)[|z^v|+1-s]\end{equation}
(the degree shift can be computed using Lemma \ref{lem:masgrad}, compare \cite[\S 4.4]{Sheridan:CY}).
The sum of these is the corresponding matrix coefficient of the class $\co(z^v) \in CC^{|z^v|}(\fuk(\{h \le c\}))$. 
Observe that $\co$ preserves $\G$-gradings.

The one-dimensional component $\overline{\mathcal{M}}_1(\mathbf{y},v)_1$ has the structure of a compact, oriented, topological one-manifold with boundary, by a gluing theorem; its boundary points are in one-to-one correspondence with terms in the Gerstenhaber bracket $[\mu^*,\co(z^v)](y_1,\ldots,y_s)$, so the sum of these terms is zero (the sign computation is parallel to \S \ref{subsec:ainfsign}). 
In particular, $\co(z^v)$ is a Hochschild cocycle, hence defines a class in Hochschild cohomology. 

Recall that the affine Fukaya category $\fuk(\{h \le c\})$ is independent of the choices of perturbation data made in its construction, up to quasi-equivalence; this can be proved using the `double category' trick (see \cite[\S 10a]{Seidel:FCPLT}). 
By the same trick, one can prove that the classes $\co(z^v) \in \HH^\bullet(\fuk(\{h \le c\}))$ are independent of the choices of perturbation data made in their construction. 

\begin{rmk}
\label{rmk:Dsmoothco}
When $D$ is smooth and $(X,D)$ is semi-positive, the classes $\co(z_p)$ can be defined in a similar way, but one needs to be a little bit more careful to rule out sphere bubbling. 
Namely, Lemma \ref{lem:nospheres} is no longer true, and one must use the techniques of \cite{Hofer1995} together with the assumption of semi-positivity.
Briefly, one works with almost-complex structures which are allowed to vary along $D$, but which still make $D$ an almost-complex manifold. 
Then it is true that for a generic choice of such almost-complex structure, no spheres will bubble off from a one-dimensional moduli space (we need to allow the almost-complex structure to vary along $D$ to ensure regularity of simple holomorphic spheres contained in $D$). 
\end{rmk}

\subsection{Independence of relative K\"{a}hler form}
\label{subsec:indsymp}

Because the space of relative K\"{a}hler forms is convex, the Liouville domain $\{h \le c\}$ is an invariant of $(X,D)$ up to Liouville deformation (see \cite[Lemma 4.4]{Seidel:biased}). 
Therefore $\fuk(\{h \le c\})$ is independent of the choice of relative K\"{a}hler form up to quasi-equivalence, by \cite[Corollary 10.10]{Seidel:FCPLT}. 
In other words, it is an invariant of the $\snc$ pair $(X,D)$ (in fact, of the affine variety $X \setminus D$). 
For that reason, we are justified in referring to it as the `affine Fukaya category', and denoting it  `$\fuk(X \setminus D)$'.

We will now prove the analogous result for the map $\co$: i.e., that it does not depend on the choice of relative K\"{a}hler form, and hence is an invariant of the $\snc$ pair $(X,D)$.

\begin{lem}
\label{lem:indepd+}
Let $(X,D,\omega,c)$ be a relative K\"{a}hler manifold. 
The exact Fukaya category $\fuk(\{h \le c\})$ does not depend on the choice of $c$ up to quasi-equivalence, and similarly for the classes $\co(z^v)$.
\end{lem}
\begin{proof}
Suppose $c_0 < c_1$. 
We have an inclusion $\{h \le c_0\} \subset \{h \le c_1\}$, so the perturbation data admitted in the definition of $\fuk(\{h \le c_0\})$ and its $\co(z^v)$ is a subset of that admitted in the definition of $\fuk(\{h \le c_1\})$ and its $\co(z^v)$. 
Therefore, the inclusion induces a full strict embedding $\fuk(\{h \le c_0\}) \hookrightarrow \fuk(\{h \le c_1\})$, and furthermore the classes $\co(z^v)$ for $\fuk(\{h \le c_0\})$ are the restrictions of those for $\fuk(\{h \le c_1\})$ on the cochain level. 
We now observe that this embedding is a quasi-equivalence by \cite[Corollary 10.6]{Seidel:FCPLT}.
\end{proof}

\begin{lem}
\label{lem:indepfam}
Let $(X,D)$ be a $\snc$ pair. 
Let $(X,D,\omega_0,c_0)$ and $(X,D,\omega_1,c_1)$ be two relative K\"{a}hler structures.
Then we have a quasi-equivalence
\begin{equation} \fuk(\{h_0 \le c_0\}) \simeq \fuk(\{h_1 \le c_1\}),\end{equation}
and the induced map of Hochschild cohomology groups identifies the classes $\co(z^v)$.
\end{lem}
\begin{proof}
The relative K\"{a}hler forms can be connected by a path $\omega_t = -dd^c h_t$ parametrized by $t \in [0,1]$, where
\begin{equation} h_t := t \cdot h_1 + (1-t) \cdot h_0.\end{equation}
We aim to prove that these relative K\"{a}hler manifolds have quasi-equivalent affine Fukaya categories and equal classes $\co(z^v)$, for all $t \in [0,1]$. 
Clearly, it suffices to prove the result for $t \in (-\epsilon,\epsilon)$, where $\epsilon$ can be arbitrarily small, then iterate.

By the Moser trick, there exists a vector field $\tilde{v}_t$ on $X \setminus D$ whose flow $\tilde{\psi}_t$ defines an exact symplectomorphism 
\begin{equation} \tilde{\psi}_t: (U, \omega = d\alpha_0) \to \left(\tilde{\psi}_t(U),\omega_t = d\alpha_t\right) \end{equation}
(where $\alpha_t := -d^ch_t$), for any open set $U$ such that the flow is defined. 
However the vector field $\tilde{v}_t$ is not complete, so the flow is not defined everywhere. 

To fix this, let $\delta>0$ and $\beta: \R \to \R$ be a smooth function such that
\begin{align}
\beta(h) &= \begin{cases}
			1 & \text{ for $h \le c_0+\delta$} \\
			0 & \text{ for $h \ge c_0+2\delta$}.
		\end{cases}
\end{align}
We define a new vector field $v_t := (\beta \circ h_0) \cdot \tilde{v}_t$, whose flow $\psi_t$ is well-defined for all times $t$. 
It has the following properties:
\begin{itemize}
\item $\psi_t|_{\{h_0 \ge c_0+2\delta\}} = \id$ for all $t$, and
\item $\psi_t(x) = \tilde{\psi}_t(x)$ for all $x$ such that $h_0(\tilde{\psi}_\tau(x)) \le c_0+\delta$ for all $\tau$ between $0$ and $t$.
\end{itemize}
We can choose $\epsilon>0$ such that $\tilde{\psi}_t(\{h_0 \le c_0\}) \subset \{h_0 \le c_0 + \delta\}$ for all $t \in (-\epsilon,\epsilon)$. 
Then $\psi_t|_{\{h_0 \le c_0\}} = \tilde{\psi}_t|_{\{h_0 \le c_0\}}$ is an exact symplectomorphism onto its image, for all $t \in (-\epsilon,\epsilon)$. 
Shrinking $\epsilon$ if necessary, we may also assume that $\{h_0 \le c_0 + 2\delta\} \subset \{h_t \le c_0 + 3\delta\}$ for $t \in (\epsilon,\epsilon)$. 
Then $\psi_t$ defines a map 
\begin{equation}
\label{eqn:psiobj}
 Ob(\fuk(\{h_0 \le c_0\})) \to Ob(\fuk(\{h_t \le c_0+3\delta\})),
\end{equation}
for all $t \in (-\epsilon,\epsilon)$.

In order to define $\fuk(\{h_0 \le c_0\})$ and the classes $\co(z^v)$, we must choose perturbation data $(H,J)$ for the various moduli spaces of pseudoholomorphic discs. 
Given this choice, we choose perturbation data for $\fuk(\{h_t \le c_0+3\delta\})$ by pushing forward by $\psi_t$. 
These pushed-forward perturbation data are valid for sufficiently small $t$, because
\begin{itemize}
\item $(\psi_t)_* J$ is equal to the integrable complex structure $J_0$ on $\{h_0 \ge c_0+2\delta\}\supset \{h_t \ge c_0 + 3\delta\}$, seeing as $J|_{\{h_0 \ge c_0\}} = J_0$ and $\psi_t|_{\{h_0 \ge c_0 + 2\delta\}} = \id$.
\item For sufficiently small $t$, $(\psi_t)_* J$ is $\omega_t$-tame, because tameness is an open condition.
\item $(\psi_t)_* H = 0$ on $\psi_t(\{h_0 \ge c_0\}) \supset \{h_t \ge c_0 + 3\delta\}$.
\end{itemize}
Furthermore, the obvious map $u \mapsto \psi_t \circ u$ gives a well-defined map between the moduli space of pseudoholomorphic curves defined using $(H,J)$ and that using the pushed-forward perturbation data. 
Recall that the symplectic form enters into the pseudoholomorphic curve equation, as it is used to turn the Hamiltonian component $H$ of the perturbation datum into a Hamiltonian vector field $X_H$; so the fact that $\psi_t$ does not respect the symplectic forms outside of $\{h_0 \le c_0\}$ may at first appear problematic. 
However, our assumption that $H|_{\{h_0 \ge c_0\}} = 0$ ensures that $X_H = 0$ regardless of the choice of symplectic form on this region. 
Hence the map $u \mapsto \psi_t \circ u$ respects solutions to the pseudoholomorphic curve equation. 
Because $\psi_t$ is invertible, in fact this map is an isomorphism of moduli spaces.

It follows that the map on objects \eqref{eqn:psiobj} extends to a full strict embedding $\fuk(\{h_0 \le c_0\}) \subset \fuk(\{h_t \le c_0+3\delta\})$, which respects the classes $\co(z^v)$ on the cochain level, because the two sides are defined by counting isomorphic moduli spaces. 
The fact that the embedding is a quasi-equivalence follows by \cite[Corollary 10.10]{Seidel:FCPLT}.
\end{proof}

\subsection{$\co$ is `almost an algebra homomorphism'}

We now prove a result which is not used elsewhere in this paper, but which is very useful for computing the map $\mathsf{co}$.

\begin{lem}
\label{lem:coalg}
Suppose that $(X,D)$ is nice or amb-nice, and let $\{v,w,v+w\} \subset \mathsf{pos}_1$. 
Then
\begin{equation} \mathsf{co}(z^v) \cup \mathsf{co}(z^w) = \mathsf{co}(z^{v+w}),\end{equation}
where $\cup$ denotes the Yoneda product on $\HH^\bullet(\fuk(X \setminus D))$.
\end{lem}
\begin{proof}
The proof follows that of \cite[Theorem 5.12]{Sheridan:CY} and \cite[Lemma 4.2]{Pomerleano2012}, compare also \cite[Proposition 5.3]{Ganatra2013}.
Given classes $v, w \in \mathsf{pos}_1$, we consider a new moduli space, $\mathcal{M}_2(\mathbf{y},v,w)$. 
The moduli space $\mathcal{R}_2(\mathbf{L})$ is defined to be the moduli space of holomorphic discs, equipped with boundary punctures $(\zeta_0,\ldots,\zeta_s)$ and internal marked points $q_1,q_2$, admitting a biholomorphism to the unit disc $\mathbb{D} \subset \C$ that sends 
\begin{eqnarray*}
\zeta_0 & \mapsto &-\iii,\\
q_1 & \mapsto & t, \\
q_2 & \mapsto & -t,
\end{eqnarray*}
for some $t \in (0,1) \subset \mathbb{D}$; these discs are considered modulo biholomorphism.

We label the boundary components by a tuple of objects $\mathbf{L}$, and the boundary punctures by a tuple of chords $\mathbf{y}$. 
We make a choice of perturbation data on $\mathcal{R}_2(\mathbf{L})$ from the set $\EuH \times \EuJ$ of Definition \ref{defn:perts}, consistent with respect to the Deligne--Mumford compactification.
We then define $\mathcal{M}_2(\mathbf{y},v,w)$ to be the moduli space of pairs $(r,u)$ where $r \in \mathcal{R}_2(\mathbf{L})$ and $u$ is a solution to
\begin{equation}
\left\{ \begin{array}{l}
u: S_r  \to  X\\
{[}u{]} = v+w \mbox{ in $H_2(X,X \setminus D)$} \\
u(z) \in L_C \mbox{ for $z \in C \subset \partial S_r$} \\
u(\epsilon_\zeta(s,t)) \to y_\zeta(t) \mbox{ for $\zeta$ a boundary puncture of $S_r$}\\
\mbox{$u$ is tangent at $q_1$ to $D_p$, to order $v\cdot D_p - 1$, for all $p$}\\
\mbox{$u$ is tangent at $q_2$ to $D_p$, to order $w\cdot D_p - 1$, for all $p$} \\
(du - Y)^{0,1} = 0.
\end{array} \right.
\end{equation}

As before, $\mathcal{M}_2(\mathbf{y},v,w)$ is a smooth moduli space of the expected dimension, for generic choice of perturbation data. 
It admits a Gromov compactification, $\overline{\mathcal{M}}_2(\mathbf{y},v,w)$; the same argument as that given for Lemma \ref{lem:nospheres} implies that the Gromov compactification does not contain any non-constant sphere bubbles, by the hypothesis that $v+w \in \mathsf{pos}_1$. 

A constant sphere bubble must be stable, by definition of the Gromov compactification. 
The only stable sphere that appears in the Deligne--Mumford compactification of $\mathcal{R}_2(\mathbf{L})$ is the boundary component at $t=0$, when $q_1$ and $q_2$ bubble off together in a sphere. 
These strata of $\overline{\mathcal{M}}_2(\mathbf{y},v,w)$ are isomorphic to strata of $\overline{\mathcal{M}}_1(\mathbf{y},v+w)$, via the map which forgets the constant sphere containing $q_1$ and $q_2$. 
To see why, observe that if $(r,u) \in \mathcal{M}_2(\mathbf{y},v,w)$, then $u$ intersects $D$ only at $q_1$ and $q_2$; therefore, if a constant sphere bubble containing $q_1$ and $q_2$ bubbles off in the Gromov topology, then the tree of discs that remains after deleting the sphere can only intersect $D$ at the point where the sphere was attached. 
Furthermore, the homology class of the tree of discs is $v+w$; so it must intersect $D_p$ with multiplicity $(v+w) \cdot D_p$ at that point, for all $p$.
The correspondence is clearly bijective. 

It follows that all strata of $\overline{\mathcal{M}}_2(\mathbf{y},v,w)$ are smooth manifolds of the expected dimension; in particular, the zero-dimensional component $\mathcal{M}_2(\mathbf{y},v,w)_0$ is a finite set of points. 
Summing the contributions of these points defines an element $H(v,w) \in CC^\bullet(\fuk(X \setminus D))$.

Now we consider the one-dimensional component of the moduli space, $\overline{\mathcal{M}}_2(\mathbf{y},v,w)$. 
It is a compact one-manifold with boundary, by standard gluing theorems.
The signed sum of its boundary points is $0$, and each boundary point corresponds to a term on the left-hand side of the equation
\begin{equation} \mathsf{co}\left(z^{v+w}\right) - \mathsf{co}\left(z^v\right) \cup \mathsf{co}\left(z^w\right) - [\mu^*,H(v,w)] = 0.\end{equation}
Boundary points at $t=0$ correspond to elements of $\mathcal{M}_1(\mathbf{y},v+w)$ by the above argument, and hence contribute the term $\mathsf{co}(z^{v+w})$. 
Boundary points at $t=1$ contribute the term $\mathsf{co}(z^v) \cup \mathsf{co}(z^w)$, and boundary points for $0<t<1$ contribute the term $[\mu^*,H(v,w)]$ (see \cite[\S 5.5]{Ganatra2013} for details in an analogous situation).

It follows that 
\begin{equation}\mathsf{co}(z^{v+w}) = \mathsf{co}(z^v) \cup \mathsf{co}(z^w)\end{equation}
on the level of cohomology, as required.
\end{proof}

\begin{rmk}
Lemma \ref{lem:coalg} should be thought of as an analogue of the fact that $\EuC \EuO$ is an algebra homomorphism. 
To make $\mathsf{co}$ into an honest algebra homomorphism, one would need to enlarge $\mathsf{sh}^\bullet(X \setminus D)$, and also deform the algebra structure by some relative Gromov-Witten invariants (compare \cite{Ganatra2016a}); however, for generators $z^v, z^w$ as in the Lemma, the only curves that contribute to the product are constant.
\end{rmk}

\subsection{Branched covers}\label{subsec:brcov}

\begin{defn}
Let $(X,D)$ and $\left(\wt{X},\wt{D}\right)$ be $\snc$ pairs. 
We say that a map $\phi: \wt{X} \to X$ is a \emph{branched cover of $\snc$ pairs} if:
\begin{itemize}
\item $\phi$ is holomorphic;
\item $\phi^{-1}(D) = \wt{D}$, and $\phi|_{\wt{X} \setminus \wt{D}}: \wt{X} \setminus \wt{D} \to X \setminus D$ is an unbranched cover.
\end{itemize}
\end{defn}

\begin{defn}
Let $(X,D,\omega,c)$ and $(\wt{X},\wt{D},\tilde{\omega},\tilde{c})$ be relative K\"{a}hler manifolds. 
We say that a map $\phi: \wt{X} \to X$ is a \emph{branched cover of relative K\"{a}hler manifolds} if:
\begin{itemize}
\item $\phi$ is a branched cover of $\snc$ pairs;
\item $c = \tilde{c}$, and $\phi^*h = \tilde{h}$ on $\{\tilde{h} \le \tilde{c}\}.$
\end{itemize}
In particular, $\phi|_{\{\tilde{h} \le c\}}$ is an unbranched cover of Liouville domains. 
\end{defn}

\begin{lem}
\label{lem:covex}
(compare \cite[Proposition 10]{Auroux2000}) 
Let:
\begin{itemize}
\item $(X,D,\omega,c)$ be a relative K\"{a}hler manifold;
\item $\left(\wt{X},\wt{D}\right)$ be a K\"ahler $\snc$ pair; and
\item $\phi: \wt{X} \to X$ be a branched cover of $\snc$ pairs.
\end{itemize}
Then we can equip $\left(\wt{X},\wt{D}\right)$ with a relative K\"{a}hler form so that $\phi$ becomes a branched cover of relative K\"{a}hler manifolds.
\end{lem}
\begin{proof}
Define $\tilde{h}' := \phi^* h$, and $\tilde{\omega}' := \phi^* \omega$. 
These would define a structure of relative  K\"{a}hler manifold on $\left(\wt{X},\wt{D}\right)$, except that $\tilde{\omega}'$ is not a K\"{a}hler form. 
Namely, we have $\tilde{\omega}'(v,Jv) \ge 0$ for all $v$ (by the corresponding property of $\omega$, because $\phi$ is holomorphic); however, we only have $\tilde{\omega}'(v,Jv) >0$ for $v \notin \ker(\phi_*)$, whereas we want it to be true for all $v \neq 0$. 
We know that $(\ker(\phi_*) \setminus 0) \subset T\wt{X}|_{\wt{D}}$, because $\phi|_{\wt{X} \setminus \wt{D}}$ is a covering map. 
Thus, $\tilde{\omega}'$ is positive on $\wt{X} \setminus \wt{D}$, but may only be non-negative along $\wt{D}$. 
So we need to alter $\tilde{\omega}'$ in a neighbourhood of $\wt{D}$. 

By assumption, $\left(\wt{X},\wt{D}\right)$ supports a relative K\"{a}hler form $\tilde{\omega}'' = -dd^c \tilde{h}''$. 
We set
\begin{equation} \tilde{h} := \tilde{h}' + \epsilon \cdot \rho \cdot \tilde{h}'',\end{equation}
where $\rho: \wt{X} \to \R$ is a smooth function such that $\rho = 0$ in a neighbourhood of $\{\tilde{h}' \le c\}$ and $\rho = 1$ in a neighbourhood of $\wt{D}$, and $\epsilon > 0$. 
It is clear that the resulting $\tilde{\omega} := -dd^c\tilde{h}$ extends smoothly across $\wt{D}$. 

Now $\tilde{\omega} = \tilde{\omega}' + \epsilon \cdot \tilde{\omega}''$ in a neighbourhood of $\wt{D}$, so it is positive there for any $\epsilon > 0$.
On the remaining compact region, $\tilde{\omega}' = -dd^c \wt{h}'$ is positive, but $-dd^c(\epsilon \cdot \rho \cdot \wt{h}'')$ is indefinite; however by compactness, we can choose $\epsilon > 0$ sufficiently small that the sum is positive. 

It follows that $\tilde{\omega} = -dd^c\tilde{h}$ equips $\left(\wt{X},\wt{D}\right)$ with the structure of a relative K\"{a}hler manifold, and $\phi$ is a branched cover of relative K\"{a}hler manifolds by construction.
\end{proof}

\subsection{Behaviour of $\co$ with respect to branched covers}\label{subsec:cobr}

Let $\left(\wt{X},\wt{D}\right)$ and $(X,D)$ be $\snc$ pairs, and $\phi: \wt{X} \to X$ a branched cover of $\snc$ pairs. 
Recall that the map $\phi|_{\wt{X} \setminus \wt{D}}$ induces a map of grading data,
\begin{equation} \mathbf{p}: H_1\left(\cG\left(\wt{X} \setminus \wt{D}\right)\right) \to H_1\left(\cG(X \setminus D)\right)\end{equation}
(see \cite[Lemma 3.29]{Sheridan:CY}). 

Consider the map
\begin{equation} \phi_*: H_2\left(\wt{X}, \wt{X} \setminus \wt{D}\right) \to H_2(X,X \setminus D).\end{equation}
Any $u \in \mathsf{pos}_1( \wt{X},\wt{D})$ can be represented by a small holomorphic disc $u: \mathbb{D} \to \wt{X}$, which only hits $\wt{D}$ at the origin. 
It follows that $\phi \circ u$ is holomorphic, and therefore intersects each component of $D$ non-negatively; and furthermore it intersects $D$ only at the origin, so $[\phi \circ u] \in \mathsf{pos}_1(X,D)$. 
Thus we have a map
\begin{equation} \phi_* : \mathsf{pos}_1\left(\wt{X},\wt{D}\right) \to \mathsf{pos}_1(X,D).\end{equation}

However, the resulting map
\begin{align}
\sh^\bullet\left(\wt{X},\wt{D}\right) & \to \sh^\bullet(X,D) \\
z^u & \mapsto z^{\phi_* u}
\end{align}
need \emph{not} respect $\G$-gradings. 

To see how the grading changes under this map, we start by giving an alternative description of the $\G$-grading of $\sh^\bullet(X,D)$. 
Consider the line bundle $\mathcal{O}(K_X + D)$ of `logarithmic volume forms', i.e., meromorphic $(n,0)$ forms which take the local form
\begin{equation} \eta = \psi \cdot d\log z_1 \wedge \ldots \wedge d\log z_k \wedge dz_{k+1} \wedge\ldots \wedge dz_n\end{equation}
in a local coordinate system in which $D = \{z_1\ldots z_k = 0\}$, where $\psi$ is holomorphic.

Let $\eta$ be a meromorphic section of $\mathcal{O}(K_X+D)$ whose divisor of zeroes $D_\eta$ is supported on $D$. 
There is an associated squared phase map
\[ \alpha_\eta:\cG (X \setminus D)  \to  S^1\]
defined as in \eqref{eqn:sqph}, and the induced map on $H_1$ defines a morphism of grading data
\begin{align}
p_\eta:H_1(\cG (X \setminus D)) &\to \Z.
\end{align}

\begin{lem}
\label{lem:altgrad}
Let $\eta$ be as above, and let $u \in \mathsf{pos}_1$ be represented by a smooth disc $u:\mathbb{D} \to X$. 
Let $\rho: S^1 \to \cG(X \setminus D)$ be a lift of $\partial u$ such that
\begin{equation}
\label{eqn:kxdrho}
 p_\eta([\rho]) = 2u \cdot D_\eta.
\end{equation}
Then the grading of $z^u \in \sh^\bullet(X,D)$ is equal to  $[\rho]$.
\end{lem}
\begin{proof}
First we observe that \eqref{eqn:kxdrho} uniquely determines $[\rho]$. 
Then we check that the class $\rho$ does not depend on the choice of $\eta$: indeed, multiplying $\eta$ by a meromorphic function whose zeroes and poles lie along $D$ has the effect of adding the same constant to both sides of \eqref{eqn:kxdrho}. 

Thus we may choose $\eta$ to be a non-vanishing holomorphic $(n,0)$-form in a neighbourhood of $u$. 
By the definition of the homomorphism $f$ from \eqref{eqn:h2g}, $f(u)$ represents the boundary of a disc in $\cG X$: so the squared phase map $\alpha_\eta$ applied to this lift extends over the interior of $u$. 
It follows that $p_\eta(f(u)) = 0$, and hence that $p_\eta([\rho]) = 2u \cdot D$ where $[\rho]$ is the grading of $z^u$. 

Finally, because we regard $\eta$ as a section of $\mathcal{O}(K_X+D)$, its divisor of zeroes is $D_\eta = D$ when restricted to a neighbourhood of $u$. 
In particular $u \cdot D_\eta = u \cdot D$, so $p_\eta([\rho]) = 2u \cdot D_\eta$ as required.
\end{proof}

\begin{lem}
\label{lem:logram}
For any $u \in \mathsf{pos}_1\left(\wt{X},\wt{D}\right)$, we have
\begin{equation}
\label{eqn:logram}
 \mathbf{p}\left( \mathsf{deg}\left(z^u\right)\right) = \mathsf{deg}\left(z^{\phi_* u}\right) + 2 R_\phi \cdot u
\end{equation}
where $R_\phi$ is the \emph{logarithmic ramification divisor}  (see \cite[\S 11.4]{Iitaka1982}). 
\end{lem}
\begin{proof}
Let $\rho: S^1 \to \cG\left(\wt{X} \setminus \wt{D} \right)$ be a lift of $\partial u$ so that $\deg(z^u) = [\rho]$, and let $\rho': S^1 \to \cG(X \setminus D)$ be a lift of $\partial(\phi \circ u)$ so that $\deg(z^{\phi_* u}) = [\rho']$. 
Observe that $\phi_* \rho$ and $\rho'$ are lifts of the same loop $\partial (\phi \circ u)$ in $X \setminus D$, so we have
\[ i = \phi_* \rho - \rho' \]
in $\G$, for some $i \in \Z$. 

Let $\eta$ be a section of $\mathcal{O}(K_X + D)$ which has no zeroes or poles in a neighbourhood of $\phi \circ u$. 
Restricting to that neighbourhood, we have $D_\eta = 0$ and $D_{\phi^* \eta} = R_\phi$ (the latter is the definition of the logarithmic ramification divisor). 
We then have
\begin{align}
i &= p_\eta(\phi_* \rho) - p_\eta(\rho') \quad \text{ because $p_\eta$ is a splitting of the grading datum} \\
&= p_{\phi^* \eta}(\rho) - p_\eta(\rho') \\
&= 2R_\phi \cdot u - 0 \quad \text{ by applying Lemma \ref{lem:altgrad} to both terms},
\end{align}
which completes the proof.
\end{proof}

\begin{rmk}
Observe that the logarithmic ramification divisor is effective, so $R_\phi \cdot u \ge 0$ for any $u \in \mathsf{pos}_1$.
\end{rmk}

\begin{defn}
We define a map of $\Bbbk$-vector spaces,
\begin{eqnarray*} 
\phi^*: \mathbf{p}^* \mathsf{sh}^\bullet(X,D) & \to & \mathsf{sh}^\bullet\left(\wt{X},\wt{D}\right)\\
\phi^*(z^u) &:=& \sum_{\substack{v \in \mathsf{pos}_1\left(\wt{X},\wt{D}\right): \\ \phi_* v = u, \\
R_\phi \cdot v = 0}} z^v.
\end{eqnarray*}
It follows from Lemma \ref{lem:logram} that the map is $\G\left(\wt{X} \setminus \wt{D} \right)$-graded.
\end{defn}

Now suppose that $(X,D)$ and $\left(\wt{X},\wt{D}\right)$ are both nice, the covering group of $\phi$ is abelian, and the induced map $\phi_*: H_1\left(\wt{X} \setminus \wt{D} \right) \to H_1(X \setminus D)$ is injective. 
Equip $(X,D)$ and $\left(\wt{X},\wt{D}\right)$ with relative K\"{a}hler structures so that $\phi$ is a branched cover of relative K\"{a}hler manifolds. 
In this situation, one can make choices of perturbation data such that there is a full strict embedding of $A_\infty$ categories,
\begin{equation} 
\label{eqn:fukembeds}
\mathbf{p}^* \fuk(X \setminus D) \hookrightarrow \fuk\left(\wt{X} \setminus \wt{D} \right). \end{equation}
For example, by combining our assumptions that the covering group of $\phi$ is abelian and the induced map on $H_1$ is injective, we can show that the universal abelian covers of $\EuG(X \setminus D)$ and $\EuG(\wt{X} \setminus \wt{D})$ are isomorphic; a choice of isomorphism determinues the functor \eqref{eqn:fukembeds} on the level of objects. 
See \cite[Proposition 3.10]{Sheridan:CY} for details. 

It follows that we have a map
\[ \wt{\co}: \sh^\bullet\left( \wt{X},\wt{D} \right) \to \HH^\bullet\left(\mathbf{p}^* \fuk(X \setminus D)\right).\]
We also have an inclusion 
\begin{equation}
\label{eqn:eqinc}
 \mathbf{p}^* \HH^\bullet(\fuk(X \setminus D)) \simeq \HH^\bullet(\mathbf{p}^* \fuk(X \setminus D))^\Gamma \hookrightarrow \HH^\bullet(\mathbf{p}^* \fuk(X \setminus D))
 \end{equation}
where $\Gamma$ is the covering group (the first identification arises by \cite[Remark 2.66]{Sheridan:CY}).

\begin{lem} 
\label{lem:branchedfuks}
Let $\phi: \wt{X} \to X$ be a branched cover of $\snc$ pairs. Suppose that both $(\wt{X},\wt{D})$ and $(X,D)$ are nice, the covering group of $\phi$ is abelian, and the induced map $H_1(\wt{X} \setminus \wt{D}) \to H_1(X \setminus D)$ is injective: it follows that we have an embedding \eqref{eqn:fukembeds}, and the maps $\co$ and $\wt{\co}$ are defined. 
Then the diagram
\begin{equation} \xymatrix{ \mathsf{sh}^\bullet\left(\wt{X},\wt{D}\right) \ar[r]^-{\wt{\mathsf{co}}}  & \HH^\bullet(\mathbf{p}^* \fuk(X\setminus D))  \\
\mathbf{p}^* \mathsf{sh}^\bullet(X,D) \ar[u]_-{\phi^*} \ar[r]_-{\mathbf{p}^* \mathsf{co}} & \mathbf{p}^* \HH^\bullet(\fuk(X \setminus D)) \ar[u]^-{\eqref{eqn:eqinc}}}  \end{equation}
commutes. 
\end{lem}
\begin{proof}
The proof follows that of \cite[Lemma 4.55]{Sheridan:CY}. 
First we observe that all elements of the diagram are independent of the choice of relative K\"{a}hler structure. 
So we are free to choose relative K\"{a}hler structures so that $\phi$ is a branched cover of relative K\"{a}hler manifolds, by Lemma \ref{lem:covex}.

Consider the moduli spaces $\mathcal{M}_1(\mathbf{y},v)$ involved in the definition of $\wt{\mathsf{co}}(z^v)$, and the moduli spaces $\mathcal{M}_1(\phi(\mathbf{y}),u)$ involved in the definition of $\mathsf{co}(z^u)$. 
We can choose perturbation data for the former moduli spaces by pulling back the perturbation data from the latter via $\phi$: the map
\begin{align}
\phi^*: \EuH(X,D) \times \EuJ(X,D) & \to  \EuH\left(\wt{X},\wt{D}\right) \times \EuJ\left(\wt{X},\wt{D}\right) \\
(H,J) & \mapsto  (\phi^*H,\phi^*J)
\end{align}
is well-defined by our choice of perturbation data in Definition \ref{defn:perts}, together with the fact that $\phi$ is a holomorphic map.

We can achieve regularity of $\mathcal{M}_1(\phi(\mathbf{y}),u)$ by a generic choice of perturbation data; and we can also achieve regularity of $\mathcal{M}_1(\mathbf{y},v)$, with the pulled-back data, by a generic choice of perturbation data downstairs. 
This follows from \cite[Lemma 6.7]{Cieliebak2007} together with the fact that the covering group $\Gamma$ acts freely on the region $\{\tilde{h} < c\}$ where we need to perturb the data.
Hence, for generic choice of perturbation data, we can ensure that both moduli spaces are regular.

With this choice of pulled-back perturbation data, there is a well-defined map
\begin{align}
\label{eqn:liftingorb} \bigsqcup_{v:\phi_* v = u} \mathcal{M}_1(\mathbf{y},v) & \to  \mathcal{M}_1(\phi(\mathbf{y}),u),\text{ sending} \\
w & \mapsto  \phi \circ w.
\end{align}
The argument is similar to that in the proof of Lemma \ref{lem:indepfam}, see also \cite[Lemma 4.55]{Sheridan:CY}.
The map is injective, by the homotopy lifting criterion applied to the map from the disc with the internal marked point removed.

Now suppose that $z^u \in \mathbf{p}^*\mathsf{sh}^\bullet(X,D)$, or equivalently, $\partial u \in H_1(X \setminus D)$ lies in the image of the map $\phi_*:H_1\left(\wt{X} \setminus \wt{D} \right) \to H_1(X \setminus D)$.
Then we claim that the map \eqref{eqn:liftingorb} is surjective. 
To see this, suppose that $w$ is an element of the right-hand side. 
Puncture $w$ at the internal marked point: the remainder maps to $X \setminus D$, and admits a lift to $\wt{X} \setminus \wt{D}$ by the homotopy lifting criterion (using our hypothesis on $\partial u$). 
The puncture can then be filled in, by the removable singularity theorem.

The right-hand side of \eqref{eqn:liftingorb}, and each component of the left-hand side, are smooth manifolds of the expected dimension, for generic choice of perturbation data (as explained above).  
The left-hand side may have components of various dimensions. 
The real codimension of $\mathcal{M}_1(\mathbf{y},v)$ in $\mathcal{M}_1(\phi(\mathbf{y}),\phi_*v)$ is $2 R_\phi \cdot v \ge 0$, by \eqref{eqn:logram}. 
Therefore, if $\mathcal{M}_1(\phi(\mathbf{y}),u)$ is $0$-dimensional, then the only non-empty moduli spaces $\mathcal{M}_1(\mathbf{y},v)$ are those with $R_\phi \cdot v = 0$. 
In this case, \eqref{eqn:liftingorb} becomes a bijection
\begin{equation} \bigsqcup_{v:\phi_* v = u, R_\phi \cdot v = 0} \mathcal{M}_1(\mathbf{y},v)  \xrightarrow{\simeq}  \mathcal{M}_1(\phi(\mathbf{y}),u).\end{equation}

The Hochschild cochain $\mathbf{p}^* \mathsf{co}(z^u)$ is defined by counting the points of the right-hand moduli space, and $\wt{\mathsf{co}} (\phi^*(z^u))$ is defined by counting the points of the left-hand moduli space. 
The commutativity of the diagram now follows on the cochain level, because the maps count points in isomorphic moduli spaces.
\end{proof}

\subsection{Branched covers of sub-$\snc$ pairs}

Let $(X,D) \subset (Y,D')$ be a sub-$\snc$ pair, where $X \setminus D$ is defined by the vanishing of a section $s$ of the holomorphic vector bundle $\EuV$ as in \S \ref{subsec:Ramb}.
We define $\fuk_{amb}(X \setminus D)$ exactly as we defined $\fuk(X \setminus D)$, except that Lagrangians are anchored with respect to a different Maslov cover: they come equipped with a lift to the universal abelian cover of $\EuR(\wedge^{top}(TY \oplus \EuV^\vee))$ (which lives over $Y \setminus D'$), rather than the universal abelian cover of $\EuR(\wedge^{top}(TX))$ (which lives over $X \setminus D$). 
In particular there is a strict embedding
\[ \bff_* \fuk(X \setminus D) \hookrightarrow \fuk_{amb}(X \setminus D),\]
where $\bff: \G(X \setminus D) \to \G_{amb}$ is the morphism of grading data defined in \S \ref{subsec:Ramb}. 

Let $(\wt{X},\wt{D}) \subset (\wt{Y},\wt{D}')$ be another sub-$\snc$ pair, and $\phi_Y: \wt{Y} \to Y$ a branched cover of $\snc$ pairs inducing a branched cover of sub-$\snc$ pairs $\phi: \wt{X} \to X$. 
Then $\wt{X} \setminus \wt{D}$ is defined by the vanishing of a section $\tilde{s}$ of $\tilde{\EuV} := \phi_Y^* \EuV$.
There is an induced morphism of grading data 
\[\mathbf{p}_{amb}: \G_{amb}(\wt{X}) \to \G_{amb}(X).\]

\begin{lem}
If the covering group of $\phi_Y: \wt{Y} \setminus \wt{D}' \to Y \setminus D'$ is abelian, and the induced map $H_1(\wt{Y} \setminus \wt{D}') \to H_1(Y \setminus D')$ is injective, then there is a full strict embedding of $A_\infty$ categories
\begin{equation}
\label{eqn:eqincamb}
\mathbf{p}_{amb}^* \fuk_{amb}(X \setminus D) \hookrightarrow \fuk_{amb}(\wt{X} \setminus \wt{D}).
\end{equation}
\end{lem}
\begin{proof}
The proof follows that of \cite[Proposition 3.10]{Sheridan:CY}: the main point is to check that the hypotheses imply the universal abelian covers of $\EuR(\wedge^{top}(TY \oplus \EuV^\vee))$ and $\EuR(\wedge^{top}(T \wt{Y} \oplus \tilde{\EuV}^\vee))$ are isomorphic.
\end{proof}

Now we want to state the analogue of Lemma \ref{lem:branchedfuks} in this setting.
We define $\sh^\bullet_{amb}(X,D) := \bff_* \sh^\bullet(X,D)$. 

\begin{lem}
Suppose that $\phi_Y: \wt{Y} \to Y$ is a branched cover of $\snc$ pairs inducing a branched cover of sub-$\snc$ pairs $\phi: \wt{X} \to X$. 
Assume that both $(\wt{X},\wt{D})$ and $(X,D)$ are amb-nice, the covering group of $\phi_Y$ is abelian, and the induced map $H_1(\wt{Y} \setminus \wt{D}') \to H_1(Y \setminus D)$ is injective.
Then the diagram
\begin{equation} \xymatrix{ \mathsf{sh}_{amb}^\bullet\left(\wt{X},\wt{D}\right) \ar[r]^-{\wt{\mathsf{co}}}  & \HH^\bullet(\mathbf{p}_{amb}^* \fuk_{amb}(X\setminus D))  \\
\mathbf{p}_{amb}^* \mathsf{sh}_{amb}^\bullet(X,D) \ar[u]_-{\phi^*} \ar[r]_-{\mathbf{p}_{amb}^* \mathsf{co}} & \mathbf{p}_{amb}^* \HH^\bullet(\fuk_{amb}(X \setminus D)) \ar[u]^-{\eqref{eqn:eqincamb}}}  \end{equation}
commutes.
\end{lem}
\begin{proof}
Identical to that of Lemma \ref{lem:branchedfuks}.
\end{proof}

\subsection{Behaviour of $\co$ with respect to products}
\label{subsec:prods}

There is a natural notion of product of $\snc$ pairs: 
\[(X_1,D_1) \times (X_2,D_2) := (X_1 \times X_2, D_1 \times X_2 \cup X_1 \times D_2).\] 
It follows from work of Amorim \cite{Amorim2017a} that, if $(X,D) = (X_1,D_1) \times (X_2,D_2)$, then we have a fully faithful embedding
\[ \fuk(X_1 \setminus D_1) \otimes_\infty \fuk(X_2 \setminus D_2) \hookrightarrow \fuk(X \setminus D)\]
which, on the level of objects, sends the tensor product of Lagrangians $L_1 \otimes L_2$ to the product Lagrangian $L_1 \times L_2$. 
The tensor product of $A_\infty$ categories is a rather tricky notion: see \cite{Amorim2016} for the definition of `$\otimes_\infty$'.

An alternative approach has been suggested by Ma'u following \cite{Mau2015}, and implemented by Ganatra in \cite{Ganatra2013} using the quilt formalism of Ma'u--Wehrheim--Woodward. 
One observes that the tensor product of DG categories is a straightforward notion (see e.g. \cite[\S 2.3]{Keller2006}). 
We can replace any $A_\infty$ category $\EuC$ with a quasi-equivalent DG category $\EuC^{\mathsf{dg}}$, namely the image of the Yoneda embedding $\EuY:\EuC \dashrightarrow \EuC \lmod$ (cf. \cite[Corollary 2.13]{Seidel:FCPLT}). 
Thus, one could simply define the tensor product of $A_\infty$ categories $\EuC$ and $\EuD$ to be the DG category $\EuC^{\mathsf{dg}} \otimes \EuD^{\mathsf{dg}}$. 
For our purposes, it will be convenient to choose a slightly different model: we define $\EuC \tilde{\otimes} \EuD$ to be the image of the quasi-embedding
\[ \EuC^{\mathsf{dg}} \otimes \EuD^{\mathsf{dg}} \dashrightarrow \EuC \bimod \EuD^{op},\]
where the right-hand side is a category of $A_\infty$  bimodules (see \cite[Proposition 2.13]{Ganatra2013} for the proof that the natural $A_\infty$ functor is a quasi-embedding). We have:

\begin{prop}[\cite{Ganatra2013}, Propositions 9.2, 9.3, 9.4 and 9.5]\label{prop:kunn}
Let $(X,D)$ be the product of $\snc$ pairs $(X_1,D_1)$ and $(X_2,D_2)$, let us abbreviate $\fuk_i := \fuk(X_i \setminus D_i)$, and let $\fuk_{prod} \subset \fuk(X \setminus D)$ be the subcategory whose objects are product Lagrangians $L_1 \times L_2$. 
Then we have a quasi-equivalence
\[ \mathbf{M}: \fuk_{prod} \dashrightarrow \fuk_1 \tilde{\otimes} \fuk_2\]
which, on the level of objects, sends
\[ L_1 \times L_2 \mapsto \EuY(L_1) \otimes \EuY(L_2).\]
\end{prop}

Now we explain how the map $\co$ behaves with respect to products. 
First we must explain how Hochschild cohomology behaves under tensor product. 
For any two proper unital DG categories $\EuM$ and $\EuN$, the \emph{Alexander--Whitney map} defines an isomorphism
\begin{equation}
\label{eqn:aw}
AW: \HH^\bullet(\EuM) \otimes \HH^\bullet(\EuN) \xrightarrow{\sim} \HH^\bullet(\EuM \otimes \EuN)
\end{equation} 
whose inverse is the \emph{Eilenberg--Zilber map} (see, e.g., \cite[Proposition 9.4.1]{Weibel1994}). 
It follows immediately that there is a corresponding isomorphism when $\EuM$ and $\EuN$ are arbitrary proper $A_\infty$ categories. 

\begin{prop}
In the setting of Proposition \ref{prop:kunn}, we have a commutative diagram
\begin{equation}
\label{eqn:coprodsh}
\xymatrix{\sh^\bullet(X_1,D_1) \ar@{^{(}->}[rr]^-{\iota} \ar[d]^\co && \sh^\bullet(X,D) \ar[d]^-\co  \\
 \HH^\bullet(\fuk_1) \ar[r]_-{AW(-,e)} & \HH^\bullet(\fuk_1 \tilde{\otimes} \fuk_2)  & \HH^\bullet(\fuk_{prod}) \ar@{<->}[l]_-{\sim}^-{\mathbf{M}},}
 \end{equation}
where $\iota: \sh^\bullet(X_1,D_1) \hookrightarrow \sh^\bullet(X,D)$ is the map induced by $H_2(X_1,X_1 \setminus D_1) \to H_2(X,X \setminus D)$, and the `$e$' appearing in `$AW(-,e)$' is the unit $e \in \HH^0(\fuk_2)$.
\end{prop}
\begin{proof}
The proof is a straightforward extension of that of Proposition \ref{prop:kunn}, using the relationship between Hochschild cochains and first-order deformations of $A_\infty$ categories. 
Observe that any first-order deformation of $\fuk_1$ induces a first-order deformation of $\fuk_1 \tilde{\otimes} \fuk_2$. 
This results in a chain map
\begin{equation}
\label{eqn:ch1}
CC^\bullet(\fuk_1) \to CC^\bullet(\fuk_1 \tilde{\otimes} \fuk_2)
\end{equation}
which induces the map $AW(-,e)$ on the level of cohomology (as one can easily verify using the chain-level formula for $AW$, which can be found in \cite[\S 3]{Le2014}). 

Thus, the class $\co(y^u)$ determines a first-order deformation of $\fuk_1\tilde{\otimes} \fuk_2$ via \eqref{eqn:ch1}.
Similarly, the class $\co(\iota(y^u))$ defines a first-order deformation of $\fuk_{prod}$. 
Ganatra's construction of the $A_\infty$ functor $ \mathbf{M}$ extends straightforwardly to an $A_\infty$ functor between these two first-order deformations. 
The fact that this extension is an $A_\infty$ functor is equivalent to the commutativity of the following diagram:
\[ \xymatrix{ \sh^\bullet(X_1,D_1) \ar@{^{(}->}[rrr]^-\iota \ar[d]^-\co &&& \sh^\bullet(X,D) \ar[d]^-\co \\
CC^\bullet(\fuk_1) \ar[r]^-{\eqref{eqn:ch1}} & CC^\bullet(\fuk_1 \tilde{\otimes} \fuk_2)\ar[r]^-{\sim}_-{\mathbf{M}_*} & CC^\bullet(\fuk_{prod},\mathbf{M}^*(\fuk_1 \tilde{\otimes} \fuk_2)) & CC^\bullet(\fuk_{prod}) \ar[l]_-{\sim}^-{\mathbf{M}^*},}\]
up to a homotopy given by the first-order term in the deformation of $\mathbf{M}$. 
This implies the commutativity of the diagram on the level of cohomology, as required.
\end{proof}

\section{The relative Fukaya category}
\label{sec:relfuk}

In this section, we obtain versality results for the relative Fukaya category. 
We start by listing our assumptions on the relative Fukaya category in \S \ref{subsec:relfuk}, and sketching their justifications. 
We recall from Remark \ref{rmk:assint} that a version of the relative Fukaya category satisfying these assumptions will be constructed in work-in-preparation \cite{Perutz2015a}.

Then in \S \ref{subsec:versnoeq} we give a criterion for the relative Fukaya category to be an $R$-versal deformation of the affine Fukaya category, and in \S \ref{subsec:verseq} we give a generalization in the presence of a group action.

Finally we state our assumption about the relationship between the relative and absolute Fukaya categories in \S \ref{subsec:relabsol}, and sketch its justification. 
We recall from Remark \ref{rmk:assint2} that versions of the relative  and absolute Fukaya categories satisfying all of the relevant assumptions can be constructed using classical pseudoholomorphic curve theory when $(X,\omega)$ is positively monotone, or Calabi--Yau and of complex dimension $\le 2$.

\subsection{The relative Fukaya category}
\label{subsec:relfuk}

Let $(X,D,\omega)$ be a relative K\"ahler manifold, and $\Nef \subset Nef(X,D)$ a convex sub-cone containing $[\omega]$ in its interior. 
Let $\G$ be the associated grading datum and $R := R(\Nef)$ the $\G$-graded ring of Definition \ref{defn:rxde}. 

\begin{ass}
\label{ass:relfuk1}
The relative Fukaya category $\fuk(X,D,\Nef)$ is a $\G$-graded deformation of $\fuk(X \setminus D)$ over $R(\Nef)$. 
\end{ass}

We define $\fuk(X,D) := \fuk(X,D,Nef(X,D))$.

Now suppose that $\Nef$ is nice, or that $D$ is smooth and $(X,D)$ is semi-positive.

\begin{ass}
\label{ass:relfukco}
The first-order deformation classes of $\fuk(X,D,\Nef)$ are $\co(z_p)$.
\end{ass}

See \S \ref{sec:vers} for the definition of the first-order deformation classes, which relies on the the fact that the coefficient ring is nice (see Lemma \ref{lem:Rnice} and Remark \ref{rmk:smoothb}). 
See \S \ref{subsec:co} for the definition of the classes $\co(z_p)$.

\begin{rmk}
It follows from Assumption \ref{ass:relfuk1} that $\fuk(X,D,\Nef)$ is a (possibly curved) $\G$-graded $R$-linear $A_\infty$ category; its objects are the same as those of the affine Fukaya category $\fuk(X \setminus D)$; its morphism spaces are the free $R$-modules with the same generators as in the affine Fukaya category; and its curvature  is of order $\fm$.
\end{rmk}

\begin{rmk}
If $\Nef = \Nef(E)$ for some system of divisors $E$ (or more generally, if $\Nef$ is contained in such a cone), then one possible construction of the relative Fukaya category $\fuk(X,D,\Nef)$ uses the intersections of discs with the components of $E$ as stabilizing marked points in the domain, as in \cite{Cieliebak2007}. 
This was carried out in \cite{Sheridan:CY}, under the restrictive hypothesis that each component of $D$ is ample; the more general version will be carried out in \cite{Perutz2015a}. 
This is technically the simplest definition that works in general, and suffices to compute Gromov--Witten invariants via the relative Fukaya category (see \cite{Ganatra2015}). 
\end{rmk}

\begin{rmk}
We do not address the question of the dependence of $\fuk(X,D,\Nef)$ on the choice of the relative K\"{a}hler form $\omega$, because our applications (e.g., \cite{SS}) do not require any invariance-type result. 
\end{rmk}

Now let us provide partial justification for Assumption \ref{ass:relfuk1}, in the case that $\Nef = \Nef(E)$ for some system of divisors $E$.
In order to define the pseudoholomorphic curves whose counts define the structure coefficients of the category, one must specify the class of perturbation data one uses. 
We must choose a larger class of perturbation data than in Definition \ref{defn:perts}, to avoid problems with transversality. 
Following \cite{Seidel:FCPLT}, perturbation data consist of two pieces: a Hamiltonian piece and an almost-complex structure piece.
We choose both of these to respect the system of divisors $E$:
\begin{itemize}
\item We only use Hamiltonians $H$ whose Hamiltonian flow $X_H$ preserves $E$.
\item We only use almost-complex structures $J$ that are $\omega$-tame and make each component $E_q$ of $E$ into an almost-complex submanifold. 
\end{itemize} 
By positivity of intersection (compare Lemma \ref{lem:posint}), these assumptions imply that any pseudoholomorphic curve $u:(\Sigma,\partial \Sigma) \to (X,\{h \le c\})$ satisfies $u \cdot E_q \ge 0$ for all $q$, hence $[u] \in NE(\Nef)$. 

The $A_\infty$ structure maps in the relative Fukaya category are defined by counting pseudoholomorphic curves $u$ in $X$, with boundary conditions on the Lagrangians, in the usual way. 
Each pseudoholomorphic curve $u$ is counted with a weight $\nov^{[u]}$, which lies in the coefficient ring $R(\Nef)$ by our assumptions on the perturbation data. 
The resulting structure maps define a (possibly curved) $\G$-graded, $R(\Nef)$-linear $A_\infty$ category, by the usual argument.

The fact that the $0$th order category coincides with the affine Fukaya category follows by definition. 

Assumption \ref{ass:relfukco} is closely related to \cite[Theorem 6.1]{Auroux2007}: it simply states that the first-order deformation classes and the classes $\nov_p \co(z_p)$ are both defined by counting pseudoholomorphic discs in homology class $u_p$.

More generally, $\fuk(X,D,\Nef)$ may be defined by counting stable $J$-holomorphic discs in some virtual sense, as in \cite{FO3}, where $J$ can be taken to be equal to the integrable complex structure in a neighbourhood of $D$. 
Such a stable disc has homology class $[u] \in NE(\Nef(E))$ for all systems of divisors $E$ by Lemma \ref{lem:posint}. 
It follows that $[u] \in NE(X,D) \subset NE(\Nef)$, so we can define $\fuk(X,D,\Nef)$ over $R(\Nef)$. 

Now we consider the situation of \S \ref{subsec:Ramb}: $(X,D) \subset (Y,D')$ is a sub-$\snc$ pair. 
We assume that $\omega$ is a relative K\"ahler form on $(X,D)$, restricted from $(Y,D')$, and that $\Nef \subset Nef(X,D)$ is a convex sub-cone containing $[\omega]$ in its interior. 
The analogue of Assumption \ref{ass:relfuk1} in this situation says that the ambient relative Fukaya category $\fuk_{amb}(X,D,\Nef)$ is a $\G_{amb}$-graded deformation of $\fuk_{amb}(X \setminus D)$ over $R_{amb}(\Nef)$. 

\begin{rmk}
In this situation, $\Nef_{amb}$ intersects the interior of $\Nef$, so we have the map
\[ \Phi^*: \bff_*R(\Nef) \to R_{amb}(\Nef)\]
of \eqref{eqn:Phi}. 
One expects there to be a strict full embedding
\[ \bff_*\fuk(X,D,\Nef) \otimes_{R(\Nef)} R_{amb}(\Nef) \hookrightarrow \fuk_{amb}(X,D,\Nef),\]
where $R_{amb}(\Nef)$ is regarded as an $R(\Nef)$-algebra via $\Phi^*$. 
\end{rmk}

Now suppose that $\Nef$ is amb-nice.
We have an analogue of Assumption \ref{ass:relfukco} in this case:

\begin{ass}
\label{ass:relfukambco}
The first-order deformation classes of $\fuk_{amb}(X,D,\Nef)$ are
\[ \sum_{q \in i^{-1}(p)} \co(z_q).\]
\end{ass}

The first-order deformation classes can be defined because $R_{amb}(\Nef)$ is nice by Lemma \ref{lem:Rambnice}, and $\co$ can be defined because $(X,D)$ is amb-nice. 
The justification for this assumption is the same as that for Assumption \ref{ass:relfukco}.

\subsection{Versality without equivariance}
\label{subsec:versnoeq}

Let $(X,D,\omega)$ be a relative K\"ahler manifold, and $\Nef \subset Nef(X,D)$ a convex sub-cone containing $[\omega]$ in its interior. 
We use the notation from \S \ref{subsec:relfuk}: $\G$ is the associated grading datum, $R:=R(\Nef)$ is the coefficient ring of the relative Fukaya category.

Let $\EuA_R \subset \fuk(X,D,\Nef)$ be a full subcategory, and $\EuA \subset \fuk(X \setminus D)$ the corresponding full subcategory: so $\EuA_R$ is a $\G$-graded (possibly curved) deformation of $\EuA$ over $R$.
The curvature of this deformation is of order $\fm$.

\begin{thm}[= Theorem \ref{thm:1}]
\label{thm:versality}
Suppose that:
\begin{itemize}
\item $\Nef$  is nice, or that $D$ is smooth and $(X,D)$ semi-positive;
\item The classes $\nov_p \co(z_p)$ span $\HH^2\left(\EuA,\EuA \otimes \fmuncomp\right)$ as $\Runcomp_0$-module.
\end{itemize}
Then $\EuA_R$ is an $R$-complete deformation of $\EuA$. 

If furthermore $\co(z_p) \neq 0$ for all $p$, then $\EuA_R$ is $R$-versal.
\end{thm}
\begin{proof}
First we note that the coefficient ring $R$ is nice (see Lemma \ref{lem:Rnice} and Remark \ref{rmk:smoothb}).
Combining with Assumptions \ref{ass:relfuk1} and \ref{ass:relfukco}, we see that the hypotheses of the versality criterion Lemma \ref{lem:versainf} are verified.
\end{proof}

Now we give an analogue in the case of the ambient relative Fukaya category. 
Let $(X,D) \subset (Y,D')$ be a sub-$\snc$ pair, $\omega$ a relative K\"ahler form on $(X,D)$ restricted from $(Y,D')$, and $\Nef \subset Nef(X,D)$ a convex sub-cone containing $[\omega]$ in its interior.
Let $R := R_{amb}(\Nef)$, and let $\EuA_R \subset \fuk_{amb}(X,D,\Nef)$ be a full subcategory, and $\EuA \subset \fuk_{amb}(X \setminus D)$ the corresponding full subcategory.

\begin{thm}
\label{thm:versalityamb}
Suppose that:
\begin{itemize}
\item $\Nef$  is amb-nice;
\item The classes $\nov_p \sum_{q \in i^{-1}(p)}\co(z_q)$ span $\HH^2\left(\EuA,\EuA \otimes \fmuncomp\right)$ as $\Runcomp_0$-module.
\end{itemize}
Then $\EuA_R$ is an $R$-complete deformation of $\EuA$.

If furthermore $\sum_{q \in i^{-1}(p)}\co(z_q) \neq 0$ for all $p$, then $\EuA_R$ is $R$-versal.
\end{thm}
\begin{proof}
The proof follows that of Theorem \ref{thm:versality}.
\end{proof}

\begin{rmk}
\label{rmk:useless}
In practice, Theorem \ref{thm:versalityamb} is not very useful.
That is because the hypotheses will not often be satisfied: the classes $\nov_p \sum_{q \in i^{-1}(p)} \co(z_q)$ will typically not span the classes $\nov_p \co(z_q)$ for all $q \in i^{-1}(p)$.
The exception is when $|i^{-1}(p)| = 1$ for all $p$, but in that case the fact that $\Nef$ is amb-nice implies that it is nice, so we can apply Theorem \ref{thm:versality}. 
The ambient relative Fukaya category becomes more useful when we introduce equivariance into the story, see Theorem \ref{thm:eqversalityamb}.
\end{rmk}

\subsection{Versality with equivariance}
\label{subsec:verseq}

Let $(X,D)$ be a $\snc$ pair, $\G$ the associated grading datum, and $(\Gamma,\sigma)$ be a signed group.

\begin{defn}
\label{defn:eqrelKahl}
An action of $(\Gamma,\sigma)$ on $(X,D)$ is an action of $\Gamma$ on $X$ by diffeomorphisms, such that
\begin{itemize}
\item $\Gamma$ preserves $D$ as a set;
\item $\gamma \in \Gamma$ acts holomorphically if $\sigma(\gamma) = 0$ and anti-holomorphically if $\sigma(\gamma) = 1$.
\end{itemize}
We say that a relative K\"ahler form $\omega = -dd^ch$ is $(\Gamma,\sigma)$-invariant if $\Gamma$ preserves the K\"ahler potential. 
In this situation we say that $(\Gamma,\sigma)$ acts on the relative K\"ahler manifold $(X,D,\omega)$. 
\end{defn}

Observe that $\gamma^* \omega = (-1)^{\sigma(\gamma)} \omega$ and $\gamma$ preserves $X \setminus D$: it follows that $\Gamma$ acts on $\cG(X \setminus D)$, and hence on $H_1(\cG(X \setminus D))$. 
We modify this to obtain an action $\Gamma$ on the associated grading datum $\G$:
\begin{equation} \gamma \cdot y := (-1)^{\sigma(\gamma)}\cdot \gamma_* y.\end{equation}
Observe that the sign factor is necessary: the diagram
\begin{equation} \xymatrix{ \Z \ar[r] \ar@{=}[d] & H_1(\cG(X \setminus D)) \ar[d]^{\gamma_*} \\
\Z \ar[r] & H_1(\cG(X \setminus D))}\end{equation}
does not commute when $\gamma$ is anti-symplectic, so $\gamma_*$ does not define a map of grading data unless we modify it by multiplying by $(-1)^{\sigma(\gamma)}$.

The constructions that follow will depend further on the choice of a morphism $p: \G \to \Z/4$ of grading data (`$\Z/4$' is shorthand for the grading datum $\{\Z \to \Z/4 \to \Z/2\}$). 
We will require that $\Gamma$ respects $p$: explicitly, this means that
\begin{equation} 
\label{eqn:pgamma}
p(\gamma_* y) = (-1)^{\sigma(\gamma)}\cdot p(y).\end{equation}

\begin{rmk}
In all of our immediately intended applications, $\Gamma$ in fact acts trivially on $\G$. 
In this case \eqref{eqn:pgamma} is trivially satisfied. 
\end{rmk}

We define an action of $\Gamma$ on $H_2(X,X \setminus D)$ by setting
\begin{equation} \gamma \cdot u = (-1)^{\sigma(\gamma)}\cdot \gamma_* u;\end{equation}
this induces an action of $\Gamma$ on $R(X,D)$, sending $\nov^u \mapsto \nov^{\gamma \cdot u}$; and this action is $\G$-graded relative to the action of $\Gamma$ on $\G$ (compare the discussion preceding Lemma \ref{lem:eqvers}). 

However, we need to modify this action of $\Gamma$ on $R$ still further, using the homomorphism $p$. 
Because $R$ is $\G$-graded, it acquires a $\Z/4$-grading via $p$. 
Because $R$ is graded in even degree, this $\Z/4$-grading lies in $2\Z/4 \simeq \Z/2$. 
So if $r \in R$ is of pure degree $p(|r|) \in \Z/4$, there is a well-defined corresponding sign $p(|r|)/2 \in \Z/2$. 
We define an involution on $R$, 
\begin{align}
i_p: R &\to R,\\
i_p(r) &:= (-1)^{p(|r|)/2} \cdot r.
\end{align}
We observe that this involution commutes with the above-defined action of $\Gamma$ on $R$, because the action of $\Gamma$ respects the $\Z/4$-grading.
Thus we can modify the action of $\Gamma$ on $R$:
\begin{equation}
\label{eqn:gamR} \gamma \cdot \nov^u := i_p^{\sigma(\gamma)} \left(\nov^{\gamma \cdot u} \right).\end{equation}
This is the action of $\Gamma$ on $R$ that will concern us.

If the action of $(\Gamma,\sigma)$ preserves $\Nef \subset Nef(X,D)$, then we similarly obtain an action of $\Gamma$ on $R:=R(\Nef)$. 
This is the case, for example, when $\Nef = \Nef(E)$ for some system of divisors $E$ which is preserved as a set by the action of $\Gamma$.

We have proved:

\begin{deflem}
\label{deflem:gammagrad}
An action of $(\Gamma,\sigma)$ on $(X,D,\omega)$ induces an action of $\Gamma$ on $\G$, the grading datum associated to $X \setminus D$, which we denote by $y \mapsto \gamma \cdot y$. 
Let $p: \G \to \Z/4$ be a morphism of grading data, and suppose that $\Gamma$ respects $p$ and $\Nef$. 
Then there is an induced action of $\Gamma$ on $R(\Nef)$, which is $\G$-graded relative to the action of $\Gamma$ on $\G$, and which we denote by $r \mapsto \gamma \cdot r$.
\end{deflem}

\begin{example}
\label{eg:holvol}
Suppose that $c_1(TX)|_{X \setminus D} = 0$. 
As in Lemma \ref{lem:c1tors}, we can choose a holomorphic volume form $\eta$, whose squared phase map defines a morphism of grading data
\begin{align}
p_\eta:\G &\to  \Z.
\end{align}
If we assume that there exist constants $\beta_\gamma \in \C^*$ for all $\gamma \in \Gamma$ such that
\begin{equation} \gamma^* \eta = \begin{cases}
					\beta_\gamma \cdot \eta & \mbox{ if $\sigma(\gamma) = 0$,} \\
					\beta_\gamma \cdot \overline{\eta} & \mbox{ if $\sigma(\gamma) = 1$,}
				\end{cases}
\end{equation}
then we have 
\begin{equation}p_\eta(\gamma \cdot y) = p_\eta(y),\end{equation}
so the reduction of $p_\eta$ modulo $4$ defines a map $p_\eta: \G \to \Z/4$ with the required properties to form part of a $(\Gamma,\sigma)$-action. 
The map $p_\eta$ equips the generator $r_p$ of $R$ with degree $2d_p \in \Z/4$, where $d_p$ is the `order of pole of $\eta$ along $D_p$' (compare \cite[Lemma 3.22]{Sheridan:CY}). 
Therefore the action of $\Gamma$ on $R$ is determined by
\begin{equation} \gamma\cdot \nov_p = \begin{cases} 
					\nov_{\gamma(p)} & \mbox{ if $\sigma(\gamma) = 0$} \\
					(-1)^{d_p} \cdot \nov_{\gamma(p)} & \mbox{ if $\sigma(\gamma) = 1$,}
					\end{cases} \end{equation}
where we assume that $\gamma$ sends $D_p$ to $D_{\gamma(p)}$. 
\end{example}

Now suppose we are given an action of $(\Gamma,\sigma)$ on $(X,D,\omega)$, preserving the morphism $p: \G \to \Z/4$ and the convex sub-cone $\Nef \subset Nef(X,D)$ containing $[\omega]$ in its interior.
In Definition--Lemma \ref{deflem:gammagrad}, we defined a simultaneous action of $\Gamma$ on the associated grading datum $\G$ and coefficient ring $R$ of the relative Fukaya category $\fuk(X,D,\Nef)$.

Recall that the objects of $\fuk(X \setminus D)$ are \emph{anchored branes} $L^\# = (L,\iota,\tilde{\iota},P^\#)$, where $L$ is a smooth compact manifold equipped with a Pin structure $P^\#$, $\iota: L \to X \setminus D$ is an exact Lagrangian embedding, and $\tilde{\iota}: L \to \wt{\cG}(X \setminus D)$ is an anchoring. 
There is an action of $G \oplus \Z/2$ on these objects (where $\G = \{\Z \to G \to \Z/2\}$ is our grading datum): $g \oplus \sigma$ sends
\begin{equation} (L,\iota,\tilde{\iota},P^\#) \mapsto \left(L,\iota,g \circ \tilde{\iota}, P^\# \otimes \lambda(TL)^{\otimes \sigma}\right),\end{equation}
where $g$ acts on $\wt{\cG} X$ by the covering group action. 
We define an \emph{unanchored brane} to be an equivalence class of the quotient of the set of anchored branes by the action of $G \oplus \Z/2$: explicitly, it is given by data $(L,\iota,[P^\#])$ where $[P^\#]$ is a Pin structure on $L$, up to tensoring with $\lambda(TL)$, and $\iota: L \to X \setminus D$ admits an anchoring.

The following Lemma explains how $(\Gamma,\sigma)$ acts on the affine and relative Fukaya categories. 
It is proved in Appendix \ref{sec:signsrel}.

\begin{lem}
\label{lem:signgrpactrel}
Suppose that $(\Gamma,\sigma)$ acts on $(X,D,\omega)$, preserving the morphism of grading data $p: \G \to \Z/4$ and the convex sub-cone $\Nef \subset Nef(X,D)$ containing $[\omega]$ in its interior.
Suppose that $\EuA_R \subset \fuk(X,D,\Nef)$ is a full subcategory, closed under shifts, and $\EuA \subset \fuk(X \setminus D)$ is the corresponding full subcategory. 
Suppose that $\Gamma$ acts freely on the set of unanchored branes underlying $Ob(\EuA)$. 
Then, for appropriate choices of perturbation data in the definition of the relative Fukaya category,
\begin{itemize}
\item $(\Gamma,\sigma)$ acts on $\EuA$ up to shifts (Lemma \ref{lem:AutsigmaFukr}).
\item $\EuA_R$ is a $(\Gamma,\sigma)$-equivariant deformation of $\EuA$ over $R$ (Lemma \ref{lem:autsigrel}).
\end{itemize}
The actions are simultaneous with the action of $\Gamma$ on $\G$ and $R$ from Definition--Lemma \ref{deflem:gammagrad}. 
\end{lem}

\begin{thm}[= Theorem \ref{thm:2}]
\label{thm:eqversality}
Consider the situation of Lemma \ref{lem:signgrpactrel}. Suppose furthermore that
\begin{itemize}
\item $\Nef$ is nice, or $D$ is smooth and $(X,D)$ semi-positive;
\item $\Gamma$ is finite;
\item $\HH^2\left(\EuA,\EuA \otimes \fmuncomp\right)^\Gamma$ is contained in the $\Runcomp_0$-span of the classes $\nov_p \co(z_p)$. 
\end{itemize}
Then $\EuA_R$ is an $R$-complete $(\Gamma,\sigma)$-equivariant deformation of $\EuA$.

If furthermore $\co(z_p) \neq 0$ for all $p$, then $\EuA_R$ is $R$-versal.
\end{thm}
\begin{proof}
Follows from Lemmas \ref{lem:signgrpactrel} and \ref{lem:eqversainf}, together with Assumptions \ref{ass:relfuk1} and \ref{ass:relfukco}.
\end{proof}

\begin{rmk}
It follows by analogous arguments to those in \S \ref{subsec:indsymp} that the $\Gamma$-action on $\HH^\bullet(\EuA)$ is independent of the choice of $\Gamma$-invariant relative K\"ahler form. 
So if we verify the hypotheses of Theorem \ref{thm:eqversality} for one $\Gamma$-invariant relative K\"ahler form, then we verify them for all such.   
\end{rmk}

\subsection{Versality with equivariance for sub-$\snc$ pairs}

Now we consider the situation of \S \ref{subsec:Ramb}. 
Let $(Y,D')$ be a $\snc$ pair equipped with an action of $(\Gamma,\sigma)$, and let $\cV \to Y \setminus D'$ be an equivariant holomorphic vector bundle. 

\begin{rmk}
Let us clarify one potentially confusing point about what we mean by an `equivariant holomorphic vector bundle', when some group elements act anti-holomorphically. 
Recall that a $\Gamma$-equivariant structure on a real vector bundle $\cV$ is a choice of isomorphisms $\gamma^*\cV \simeq \cV$ for all $\gamma$, satisfying the cocycle condition. 
A $(\Gamma,\sigma)$-equivariant structure on a holomorphic vector bundle $\cV$ is a $\Gamma$-equivariant structure on the underlying real vector bundle $\cV$, with the property that the isomorphism $\gamma^*\cV \simeq \cV$ is holomorphic when $\sigma(\gamma)=0$, and anti-holomorphic when $\sigma(\gamma) = 1$. 
\end{rmk}

Let $s \in \Gamma(\cV)$ be an equivariant holomorphic section, and $X=\{s=0\}$. 
Then the action of $(\Gamma,\sigma)$ on $(Y,D')$ restricts to one on $(X,D)$. 
When combined with the action on $\cV$, it induces an action of $\Gamma$ on $\G_{amb}$. 
Suppose that the action of $\Gamma$ preserves a morphism of grading data $p:\G_{amb} \to \Z/4$ and a convex sub-cone $\Nef \subset Nef(X,D)$. 
Then there is a unique action of $\Gamma$ on $R_{amb}(\Nef)$ so that the morphism $\Phi^*: R(\Nef) \to R_{amb}(\Nef)$ of \eqref{eqn:Phi} is $\Gamma$-equivariant. 

If $\Gamma$ acts freely on the set of unanchored Lagrangian branes underlying $\EuA_R \subset \fuk_{amb}(X,D,\Nef)$, then it follows from Lemma \ref{lem:signgrpactrel} that $\EuA_R$ is a $(\Gamma,\sigma)$-equivariant deformation of $\EuA$ over $R_{amb}(\Nef)$, simultaneous with the actions of $\Gamma$ on $\G_{amb}$ and $R_{amb}(\Nef)$ defined above.

\begin{thm}
\label{thm:eqversalityamb}
In this situation, suppose furthermore that:
\begin{itemize}
\item $\Nef$ is amb-nice;
\item $\Gamma$ is finite;
\item $\HH^2\left(\EuA,\EuA \otimes \fmuncomp\right)^\Gamma$ is contained in the $\Runcomp_0$-span of the classes $\nov_p \sum_{q \in i^{-1}(p)}\co(z_q)$. 
\end{itemize}
Then $\EuA_R$ is an $R$-complete $(\Gamma,\sigma)$-equivariant deformation of $\EuA$.

If furthermore $\sum_{q \in i^{-1}(p)}\co(z_q) \neq 0$ for all $p$, then $\EuA_R$ is $R$-versal.
\end{thm}
\begin{proof}
The proof follows that of Theorem \ref{thm:eqversality}.
\end{proof}

\begin{rmk}
Recall that the ambient relative Fukaya category does not give a useful versality result in the absence of equivariance (see Remark \ref{rmk:useless}). 
Observe that Theorem \ref{thm:eqversalityamb} can fix the problem if $\Gamma$ acts transitively on the set of components of $D'_p \cap X$ for all $p$: in this case, although the classes $\nov_p \co(z_q)$ may not be spanned by the first-order deformation classes of the ambient relative Fukaya category, the only linear combination of them which is invariant under $\Gamma$ is their sum, which is equal to the first-order deformation class of the ambient relative Fukaya category.
\end{rmk}

\subsection{Relationship between the absolute and relative Fukaya categories}
\label{subsec:relabsol}

We begin with some recollections about the absolute Fukaya category. 

\begin{defn}
Let $\Lambda$ be the universal Novikov field over the field $\Bbbk$: that is,
\begin{equation} \Lambda := \left\{ \sum_{j=0}^\infty c_j \cdot q^{\lambda_j}: c_j \in \Bbbk, \lambda_j \in \R, \lim_{j \to \infty} \lambda_j = +\infty\right\}.\end{equation}
It is a nonarchimedean complete valued field, where the valuation is given by
\begin{equation} val\left(\sum_{j=0}^\infty c_j \cdot q^{\lambda_j}\right) := \min_{c_j \neq 0}\{\lambda_j\}.
\end{equation}
The subring $\Lambda_{\ge 0} := val^{-1}(\R_{\ge 0})$ is called the universal Novikov ring; it is a local ring with maximal ideal $\Lambda_{>0} := val^{-1}(\R_{>0})$.
\end{defn}

Roughly, one expects there to be a $\Lambda_{\ge 0}$-linear `positive-energy absolute Fukaya category' $\fuk(X,\omega)_{\ge 0}$, whose base-change to $\Lambda$ is the absolute Fukaya category $\fuk(X,\omega)$. 
One can formally enlarge the positive-energy absolute Fukaya category with positive-energy bounding cochains, to obtain a new $\Lambda_{\ge 0}$-linear category $\fuk(X,\omega)^\bc_{\ge 0}$; we denote the base-change of this category to $\Lambda$ by $\fuk(X,\omega)^\bc$. We refer to \cite{FO3} for further details.

\begin{defn}
We define the $\Bbbk$-scheme
\begin{align*} 
\overline{\cM}_{K\ddot{a}h}(X,D,\Nef) &:= \spec(R(\Nef)), 
\end{align*}
and let $0 \in \overline{\cM}_{K\ddot{a}h}(X,D,\Nef)$ be the point corresponding to the maximal ideal $\fm$. 
We define the open subscheme
\begin{align*}
\cM_{K\ddot{a}h}(X,D,\Nef) &:= \spec\left(\Rpunc(\Nef)\right) \quad \text{ (see \eqref{eqn:rpunc}),}
\end{align*}
which is the complement of the divisor $\partial \overline{\cM}_{K\ddot{a}h}(X,D,\Nef) := \cup_p \{\nov_p = 0\}$. 
We similarly define 
\[ \cM_{K\ddot{a}h}(X,D) \subset \overline{\cM}_{K\ddot{a}h}(X,D) := \overline{\cM}_{K\ddot{a}h}(X,D,Nef(X,D)).\]
\end{defn}

\begin{lem}
\label{lem:RLamb}
Let $(X,D)$ be a $\snc$ pair, and $\Nef \subset Nef(X,D)$ a convex sub-cone.
Then any class $\beta$ in the interior of $\Nef$ determines a $\Lambda$-point $\dm(\beta)$ of $\overline{\cM}_{K\ddot{a}h}(X,D,\Nef)$ by the equation
\begin{equation}
\dm(\beta)^*\left(\sum_u c_u \cdot \nov^u\right) := \sum_u c_u \cdot q^{\beta(u)}.
\end{equation} 
\end{lem}
\begin{proof}
The only thing to check is that the right-hand side converges; this follows because $\beta(\fm)$ is strictly positive and bounded away from $0$.
\end{proof}

\begin{defn}
\label{defn:fibainf}
If $\EuC$ is an $R$-linear $A_\infty$ category and $\dm$ is a $\Lambda$-point of $\spec(R)$, we define the \emph{fibre} 
\[ \EuC_\dm := \EuC \otimes_R \Lambda,\]
where $\Lambda$ is regarded as an $R$-algebra via $\dm^*$. 
It is a $\Lambda$-linear $A_\infty$ category.
\end{defn}

\begin{ass}
\label{ass:relabs}
There is an embedding of $\Lambda$-linear $A_\infty$ categories
\begin{equation}
\label{eqn:relemb}
\left(\fuk(X,D,\Nef)^\bc\right)_{\dm([\omega])} \hookrightarrow \fuk(X,\omega)^\bc,
\end{equation}
for any relative K\"ahler form $\omega$ whose cohomology class $[\omega]$ lies in the interior of $\Nef$.
\end{ass}

The embedding should take the form
\begin{eqnarray*}
hom^\bullet_{\fuk(X,D,\Nef)^\bc}(L_0,L_1) \otimes_R \Lambda &\to& hom^\bullet_{\fuk(X,\omega)^\bc}(L_0,L_1)\\
y \otimes 1 & \mapsto & q^{A(y)} \cdot y,
\end{eqnarray*} 
where $y$ is a Hamiltonian chord and $A(y)$ is its action. 
Modulo foundational issues, Assumption \ref{ass:relabs} simply expresses the different ways in which holomorphic curves are weighted in the relative and absolute Fukaya categories ($r^{u} \in R$ versus $q^{\omega(u)} \in \Lambda_{\ge 0}$), using Stokes' theorem (see \cite[\S 8.1]{Sheridan:CY})

\begin{rmk}
Recall that bounding cochains in $\fuk(X,D,\Nef)$ are required to be of order $\fm$, while those in $\fuk(X,\omega)_{\ge 0}$ are required to have positive energy. 
The image of a bounding cochain under the natural embedding \eqref{eqn:relemb} need not have positive energy, because the action $A(y)$ of some of the chords $y$ involved in the bounding cochain may be negative. 
For a fixed subcategory of the relative Fukaya category with finitely many objects, this can be fixed by applying the reverse Liouville flow, which acts by Hamiltonian isotopy on our Lagrangians (because they are exact) and exponentially suppresses the action (see \cite[Lemma 5.2]{Sheridan:Fano}). 
In particular, it suffices to apply the reverse Liouville flow for long enough that all chords have action $A(y)$ strictly greater than $-\ell_p$, for all linking numbers $\ell_p$.
\end{rmk}

In the case of the ambient relative Fukaya category, we define
\[ \overline{\cM}_{amb\text{-}K\ddot{a}h}(X,D,\Nef) := \spec(R_{amb}(\Nef))\]
and similarly $\cM_{amb\text{-}K\ddot{a}h}(X,D,\Nef)$. 
The map $\Phi^*$ of \eqref{eqn:Phi} defines a morphism 
\[ \Phi:\overline{\cM}_{amb\text{-}K\ddot{a}h}(X,D,\Nef) \to \overline{\cM}_{K\ddot{a}h}(X,D,\Nef).\]
The $\Lambda$-point $\dm(\omega)$ lies in the image of this morphism if and only if the cohomology class $[\omega]$ lies in the image of the restriction map $\iota^*: H^2(Y,Y \setminus D';\R) \to H^2(X,X \setminus D;\R)$.
If this is the case, we denote its pre-image by $\dm_{amb}([\omega])$. 
It follows from Assumption \ref{ass:relabs} that, in this case, there is an embedding
\[ \left(\fuk_{amb}(X,D,\Nef)^\bc\right)_{\dm_{amb}([\omega])} \hookrightarrow \fuk(X,\omega)^\bc.\]

\section{Motivation from symplectic cohomology}
\label{sec:SH}

In this section we explain geometric criteria (Corollaries \ref{cor:geomhyp} and \ref{cor:geomhyp2}) under which the hypotheses of our versality theorems (Theorems \ref{thm:1} and \ref{thm:2}) ought to hold.
In practice, it may be easier to establish the hypotheses directly, using motivation from this section to understand when this ought to be possible. 
For that reason, we only sketch the proofs of some of the results in this section (see Remark \ref{rmk:caution}).

\subsection{Symplectic cohomology} 
\label{subsec:SH}

Recall the definition of the symplectic cohomology of the Liouville domain $\{h \le c\}$: we denote it by $SH^\bullet(X \setminus D;\Bbbk)$. 
We briefly explain how to equip it with a $\G(X \setminus D)$-grading. 
Let $\gamma:S^1 \to X \setminus D$ be an orbit of the Hamiltonian used to define symplectic cohomology, and $\rho$ a section of $\gamma^*\EuG(X \setminus D)$. 
Let $Z_- := (-\infty,0] \times S^1$, and consider the map $u:Z_- \to X \setminus D$ given by $u(s,t) = \gamma(t)$. 
We consider the Cauchy--Riemann operator $D_\rho$ associated to the bundle $u^* TX$ over the Riemann surface $Z_-$, with boundary condition over $\partial Z_- = \{0\} \times S^1$ given by $\rho$, and denote its Fredholm index by $i(D_\rho)$. 
The section $\rho$ defines a loop in $\EuG(X \setminus D)$, whose homology class we denote by $[\rho]$; we define the grading of $\gamma$ to be 
\[ |\gamma| := \dim_\C(X) +i(D_\rho) - [\rho] \in H_1(\EuG(X \setminus D)). \]
Standard index theory arguments show that the usual Floer theoretic operations respect this $\G(X \setminus D)$-grading: e.g., the differential has degree $1$.

There is a long exact sequence
\begin{equation}
\label{eqn:sh+}
 \ldots \to H^\bullet(X \setminus D;\Bbbk) \xrightarrow{i} SH^\bullet(X \setminus D;\Bbbk) \to SH^\bullet_+(X \setminus D;\Bbbk) \to \ldots.
\end{equation}
The orbits in the image of $i$ correspond to constant orbits in the interior of $X \setminus D$, whereas the classes in $SH^\bullet_+$ correspond to orbits of the Reeb flow on the contact boundary $\{h=c\}$. 
We have the following control over the gradings of the generators:

\begin{lemstar}
\label{lem:BSinput}
$SH^\bullet_+(X \setminus D;\Bbbk)$ vanishes except in those $\G$-degrees which can be written as $\deg(z^u) + k$ where $u \in \mathsf{pos}_1(X,D)\setminus \{0\}$, $\deg(z^u)$ is the degree of the corresponding element of $\sh^\bullet(X,D)$, and $k$ is an integer with $0 \le k \le 2\dim_\C(X) - u \cdot D$.
\end{lemstar}
\begin{proof}[Proof (sketch)]
Following \cite{Viterbo1999}, we can define symplectic cohomology as a direct limit:
\[SH^\bullet(X \setminus D;\Bbbk) := \varinjlim_k HF^\bullet(H_k;\Bbbk), \]
where $HF^\bullet(H_k;\Bbbk)$ denotes the Hamiltonian Floer cohomology of a sequence of Hamiltonians $\{H_k\}_{k \ge 0}$ with increasing slopes at infinity. 
A specific such sequence of Hamiltonians $\{H_k\}$ is constructed in \cite{Ganatra2016a} (following \cite{McLean2012}), whose orbits can be explicitly enumerated. 
To each orbit can be assigned a class $u \in \mathsf{pos}_1(X,D)$, which is the homology class of the `small cap' of the orbit. 

The orbits with small cap $u=0$ form a sub-complex, whose cohomology is $H^\bullet(X \setminus D;\Bbbk)$. 
We denote the quotient complex by $CF^\bullet_+(H_k;\Bbbk)$, and its cohomology by $HF^\bullet_+(H_k;\Bbbk)$. 
We have 
\[SH^\bullet_+(X \setminus D;\Bbbk) \simeq \varinjlim_k HF^\bullet_+(H_k;\Bbbk).\]

The orbits with class $u$ can be identified with generators of the Morse cohomology of the normal $(S^1)^{K(u)}$-bundle over $D_{K(u)}$ (compare \cite[Lemma 6.8]{McLean2012}), whose dimension is
\begin{align*}
\dim(D_{K(u)}) + |K(u)| &= 2\dim_\C(X) - 2u \cdot D + u \cdot D\\
&= 2\dim_\C(X) - u \cdot D.
\end{align*} 
The degree of such an orbit can be calculated to be equal to the degree of $z^u$ in $\sh^\bullet(X,D)$, plus the degree of the corresponding generator of Morse cohomology. 
The degree of a generator of Morse cohomology is bounded below by $0$ and above by the dimension of the manifold; this yields a bound on the degrees in which the cochain complex $CF^\bullet_+(H_k;\Bbbk)$ is non-zero, which implies the desired degree bound after taking cohomology and passing to the direct limit.
\end{proof}

Now, suppose that $(X,D)$ is nice, or that $D$ is smooth and $(X,D)$ is semi-positive. 
Ganatra and Pomerleano define the `logarithmic PSS map' \cite{Ganatra2016}, which is a homomorphism of $\G$-graded vector spaces
\begin{equation}
\label{eqn:shSH}
PSS_{log}: \sh^\bullet(X,D) \to SH^\bullet(X \setminus D;\Bbbk)
\end{equation}
defined by counting pseudoholomorphic `caps' with a constraint at an internal marked point as in the definition of $\co$, and an asymptotic condition on the cylindrical end given by one of the Hamiltonians $H_k$ (for sufficiently large $k$). 
To ensure that this is a chain map, one must rule out sphere bubbling and breaking along orbits of $H_k$ that lie on $D$. 
Ganatra and Pomerleano place certain restrictions on $D$ to achieve this, but the hypotheses `$(X,D)$ is nice' or `$D$ smooth and $(X,D)$ semi-positive' also suffice. 

A neck-stretching-type argument (similar to those appearing in \cite{Ganatra2016}) shows that the diagram
\begin{equation}
\label{eqn:psslogcomm}
\xymatrix{ \sh^\bullet(X,D) \ar[r]^-{PSS_{log}} \ar[dr]^-{\co} & SH^\bullet(X \setminus D;\Bbbk) \ar[d]^-{\EuC\EuO} \\
& \HH^\bullet(\fuk(X \setminus D))}
\end{equation}
commutes, where $\EuC\EuO$ is the usual closed--open string map (see, e.g., \cite[\S 5.4]{Ganatra2013}). 

\begin{rmk}
This gives an alternative proof of Lemma \ref{lem:indepfam}: the symplectic cohomology, affine Fukaya category and closed--open map are all invariants of the Liouville domain $\{h \le c\}$ up to deformation equivalence, and deforming the relative K\"{a}hler form induces a deformation equivalence of the Liouville domain.
\end{rmk}

\subsection{The closed--open map on constant loops}
\label{subsec:coi}

In this section we will be interested in the composition 
\[ H^\bullet(X \setminus D;\Bbbk) \xrightarrow{i} SH^\bullet(X \setminus D;\Bbbk) \xrightarrow{\EuC \EuO} \HH^\bullet(\fuk(X \setminus D)).\]
The map $\EuC\EuO \circ i$ is defined in an ad-hoc but convenient way in \S \ref{subsec:coisigns} (forgetting that $\EuC\EuO$ and $i$ exist independently).

Now suppose that we are in the situation of Lemma \ref{lem:signgrpactrel}: in particular, $(\Gamma,\sigma)$ acts on $(X,D,\omega)$, preserving the morphism of grading data $p: \G \to \Z/4$. 
We define an action of $\Gamma$ on $H^\bullet(X \setminus D)$, by
\begin{equation} \gamma \cdot \alpha := (-1)^{\sigma(\gamma)} \cdot \gamma^* \alpha.\end{equation}
The following result is proved in \S \ref{subsec:coisigns}:

\begin{lem}[= Lemma \ref{lem:coisigns}]
\label{lem:coisignsnoapp}
The map
\begin{equation} \EuC \EuO \circ i: H^\bullet(X \setminus D) \to \HH^\bullet(\EuA) \end{equation}
is $\Gamma$-equivariant.
\end{lem}

\subsection{The semi-positive case}

In this section, we will assume that some convex sub-cone $\Nef \subset Nef(X,D)$ with non-empty interior has been chosen, and we will denote $\Runcomp := \Runcomp(\Nef)$ and $\fmuncomp \subset \Runcomp$ the toric maximal ideal. 

We will consider the $\G(X \setminus D)$-graded $\Runcomp$-module $SH^\bullet_+(X \setminus D;\fmuncomp)$. 
We will be particularly concerned with the image of the classes $z_p \in \sh^\bullet(X,D)$ under $PSS_{log}$, which we denote by $\mathsf{d}_p \in SH^\bullet(X \setminus D;\Bbbk)$. 
They correspond to Reeb orbits linking the corresponding divisor $D_p$ once.

\begin{lemstar}
\label{lem:defgen} 
Suppose that either:
\begin{itemize}
\item $\Nef$ is nice and semi-positive; or
\item $D$ is smooth, $(X,D)$ is semi-positive, and the constant $\kappa \in \Q$ from Example \ref{eg:smoothpos} satisfies $0 \le \kappa < 1$.
\end{itemize}
Then the images of the classes $\nov_p \mathsf{d}_p$ are a $\Runcomp_0$-basis for $SH^2_+\left(X \setminus D; \fmuncomp\right)$.
\end{lemstar}
\begin{proof}[Proof (sketch)]
It suffices to prove the corresponding result for the Hamiltonian Floer cohomology groups $HF^\bullet_+(H_k;\Bbbk)$ (for $k$ sufficiently large that $\mathsf{d}_p$ are well-defined), then pass to the direct limit. 
Recall that the generators of the cochain complex $CF^\bullet_+(H_k;\Bbbk)$ are orbits $\gamma$ of $H_k$.
As in the sketch proof of Lemma \ref{lem:BSinput}, we denote the `small cap' of $\gamma$ by $u \in \mathsf{pos}_1$, and we denote the index of the corresponding Morse critical point by $j$.

Suppose that $\nov^v\cdot \gamma$ is a generator of $CF^i_+\left(H_k;\fmuncomp\right)$, for some $i \in \Z$.
We must have
\begin{align*}
\deg(\nov^v)+ \deg(z^u) + j &= i \\
\implies \deg(\nov^{v-u}) + 2u\cdot D + j &= i
\end{align*}
(using the fact that $\deg(\nov_p z_p) = 2$ for all $p$ -- see Remark \ref{rmk:degrpzp}). 
It follows that $\deg(\nov^{v-u})$ lies in the image of $\Z$, so $\nov^{v-u} \in \Rpuncomp_{cl}$ and we have
\begin{equation}
\label{eqn:adddeg}
 2c_1(v-u) + 2u \cdot D + j = i.
\end{equation}
We observe that if $u \cdot D = 0$ then $u=0$, which means the class vanishes in $CF^\bullet_+$ by definition, so we assume that $u \cdot D \ge 1$.

\paragraph{Case 1: $\Nef$ is nice and semi-positive.}
Let us first show that the classes $\nov_p \mathsf{d}_p$ span $HF^2_+\left(H_k;\fmuncomp\right)$ over $\Runcomp_0$. 
Let $\nov^v \cdot \gamma$ be a generator of $CF^2_+\left(H_k;\fmuncomp \right)$.
We observe that $\deg(\nov^v)$ and $\deg(\nov^u)$ have the same image in $H_1(X \setminus D)$. 
Furthermore, because $u \in \mathsf{pos}_1$, we have $u \in \Z^K$ where $D_K \neq \emptyset$.
Therefore, it follows from Lemma \ref{lem:typea} that $v-u \in NE_{cl}$.
It follows by our semi-positivity assumption that $2c_1(v-u) \ge 0$ (see Remark \ref{rmk:gradpositivity}). 
Considering \eqref{eqn:adddeg} with $i=2$, it is clear that we must have $c_1(v-u) = 0$, $u \cdot D = 1$, and $j = 0$.

These correspond to classes $\nov^v \cdot \gamma$, where $\gamma$ has small cap $u=u_p$ and Morse index $j=0$. 
These are the orbits corresponding to the minimum of the chosen Morse function on the normal circle bundle to $D_p$ in the sketch proof of Lemma \ref{lem:BSinput} (we may assume the Morse function has a unique local minimum). 
They have degree $\deg(\gamma) = \deg(\mathsf{d}_p)$, and in fact we claim that such an orbit $\gamma$ is the unique generator of $CF^\bullet_+(H_k;\Bbbk)$ with $\deg(\gamma) = \deg(\mathsf{d}_p)$. 

Indeed, the above argument already shows that if $\gamma'$ is a generator with $\deg(\gamma') = \deg(\gamma)$, it must have small cap $u'=u_q$ for some $q$. 
This implies that $\deg(\nov_p) = \deg(\nov_q)$. 
However $R$ is nice by Lemma \ref{lem:Rnice}, so it follows by Lemma \ref{lem:nicerpdif} that $p=q$, and therefore that $\gamma' = \gamma$. 

One can show that the coefficient of $\gamma$ in $\mathsf{d}_p$ is equal to $1$. 
It follows that $\mathsf{d}_p$ spans the corresponding graded piece of $HF^\bullet_+(H_k;\Bbbk)$.
Therefore we have $\nov^v \cdot \gamma = \nov^{v-u_p} \cdot \nov_p \mathsf{d}_p$ where $\nov^{v-u_p} \in \Runcomp_0$. 
So the class $\nov^v \cdot \gamma$ is spanned by $\nov_p \mathsf{d}_p$ over $\Runcomp_0$, as required.

It remains to prove that the classes $\nov_p \mathsf{d}_p$ are independent over $\Runcomp_0$. 
Since the classes $\mathsf{d}_p$ all have different degrees by Lemma \ref{lem:nicerpdif}, it suffices to show that $\mathsf{d}_p \neq 0$ for all $p$. 
This can be shown using the spectral sequence induced by the action filtration on $SH^\bullet_+(X \setminus D)$: we need to show that the generator $\mathsf{d}_p$ cannot have a non-trivial differential mapping to or from it in the spectral sequence. 
If $\partial(\gamma) = \mathsf{d}_p$, then $\deg(\nov_p \cdot \gamma) = 1$: we see from \eqref{eqn:adddeg} that there can be no such classes. 
If $\partial(\mathsf{d}_p) = \gamma$, then $\deg(\nov_p \cdot \gamma) = 3$: a similar argument to above shows that we must have $u \cdot D = 1$ and $j=1$. 
It follows from Lemma \ref{lem:nicerpdif} that the small cap for $\gamma$ must be $u_p$. 
Thus the only differential connecting $\mathsf{d}_p$ and $\gamma$ is the Morse differential: since we assume $\mathsf{d}_p$ corresponds to the unique local minimum of the Morse function, it is closed under the Morse differential. 
Thus $\mathsf{d}_p \neq 0$, as required.

\paragraph{Case 2: $D$ is smooth, $(X,D)$ is semi-positive, and $0 \le \kappa < 1$.}
Let $\nov^v \cdot \gamma$ be a generator of $CF^2_+\left(H_k;\fmuncomp \right)$.
$D$ is smooth, so both $u$ and $v$ live in $H_2(X, X \setminus D) \simeq \Z$. 
In fact, since $\nov^v \in \fmuncomp$, we have $v \ge 1$. 
Since $u \in \mathsf{pos}_1$ is non-zero, we have $u \ge 1$. 
We have $c_1(v-u) = \kappa(v-u)$, so \eqref{eqn:adddeg} yields
\begin{align}
\label{eqn:degdsm}
2 &=  2\kappa(v-u) + 2u + j\\
\nonumber& = 2\kappa v + (2-2\kappa) u + j\\
\nonumber& \ge 2\kappa + 2 - 2\kappa + j \\
\nonumber& \ge 2.
\end{align}
It follows that we must have equality at each stage, which implies that $j=0$ and $u=1$.  
The rest of the proof that $\nov_p \mathsf{d}_p$ spans $SH^2_+(X \setminus D;\Bbbk)$ is as in Case 1. 

It remains to show that $\mathsf{d}_p \neq 0$. 
As in Case 1, we easily rule out the possibility that $\partial(\gamma) = \mathsf{d}_p$. 
If $\partial(\mathsf{d}_p) = \gamma$, we would have $\deg(\nov \gamma) = 3$. 
Returning to \eqref{eqn:adddeg}, we observe that $j$ must be odd and hence $\ge 1$.
Then an inequality similar to \eqref{eqn:degdsm} then shows that we must have $u \cdot D = 1$ and $j=1$. 
The rest of the proof that $\mathsf{d}_p \neq 0$ is as in Case 1.
\end{proof}

\begin{corstar}[= Corollary \ref{cor:geomhypint}]
\label{cor:geomhyp}
Consider the situation of Theorem \ref{thm:versality}. 
Suppose furthermore that:
\begin{itemize}
\item Either $\Nef$ is nice and  semi-positive, or $D$ is smooth and $(X,D)$ is semi-positive with $0 \le \kappa < 1$;
\item $H^2\left(X \setminus D;\fmuncomp\right) =0$;
\item The map
\begin{equation}
\label{eqn:cohopesurj}
\EuC \EuO \otimes \fmuncomp :SH^2\left(X \setminus D;\fmuncomp\right) \to \HH^2\left(\EuA,\EuA \otimes \fmuncomp \right)\end{equation}
is surjective (e.g., $\EuC\EuO$ is surjective).
\end{itemize}
Then $\EuA_R$ is an $R$-complete deformation of $\EuA$. 

If furthermore \eqref{eqn:cohopesurj} is an isomorphism (e.g., if $\EuC\EuO$ is an isomorphism), then $\EuA_R$ is $R$-versal.
\end{corstar}
\begin{proof}
By Lemma \ref{lem:defgen} combined with the second hypothesis and the long exact sequence \eqref{eqn:sh+}, $SH^2(X \setminus D;\fmuncomp)$ is spanned by the classes $\nov_p\mathsf{d}_p$.
By surjectivity of $\EuC\EuO \otimes \fmuncomp$, and the commutativity of the diagram \eqref{eqn:psslogcomm}, it follows that $\HH^2(\EuA,\EuA \otimes \fmuncomp)$ is spanned by the classes $\nov_p \co(z_p)$. 
It follows that $\EuA_R$ satisfies the hypotheses of Theorem \ref{thm:versality}, so it is $R$-complete. 
If $\EuC\EuO \otimes \fmuncomp$ is in fact an isomorphism, then $\co(z_p) \neq 0$, so $\EuA_R$ is $R$-versal, again by Theorem \ref{thm:versality}.
\end{proof}

\begin{lem}[= Lemma \ref{lem:nodeg2int}]
\label{lem:nodeg2}
Suppose that one of the following holds:
\begin{itemize}
\item $\Nef$ is Calabi--Yau, and $H^2(X\setminus D;\Q) = 0$;
\item $\Nef$ is positive, and has minimal Chern number $\ge 2$.
\end{itemize}
Then $H^2\left(X \setminus D;\fmuncomp\right) = 0$.
\end{lem}
\begin{proof}
Let $\nov^v \cdot \alpha$ be a generator of this group. 
If $\Nef$ is Calabi--Yau, then $\nov^v$ has degree $0$, so $\alpha$ has degree $2$ and therefore vanishes (recalling that $\mathrm{char}(\Bbbk) = 0$). 
If $\Nef$ is positive with minimal Chern number $\ge 2$, then $\nov^v \in \fmuncomp_{cl}$ has degree $2c_1(v) \ge 4$ and $\alpha$ has degree $\ge 0$, so the total degree cannot be $2$.
\end{proof}

Similarly, we have:

\begin{corstar}[= Proposition \ref{prop:geomhyp2int}]
\label{cor:geomhyp2}
Consider the situation of Theorem \ref{thm:eqversality}. 
Suppose furthermore that:
\begin{itemize}
\item Either $\Nef$ is nice and semi-positive, or $D$ is smooth and $(X,D)$ is semi-positive with $0 \le \kappa < 1$;
\item $H^2\left(X \setminus D;\fmuncomp\right)^\Gamma =0$;
\item The map
\begin{equation}
\label{eqn:cosurjag}
\EuC \EuO \otimes \fmuncomp :SH^2\left(X \setminus D;\fmuncomp\right) \to \HH^2\left(\EuA,\EuA \otimes \fmuncomp \right)\end{equation}
is surjective (e.g., $\EuC\EuO$ is surjective).
\end{itemize}
Then $\EuA_R$ is an $R$-complete $(\Gamma,\sigma)$-equivariant deformation of $\EuA$.

If \eqref{eqn:cosurjag} is furthermore an isomorphism, then $\EuA_R$ is $R$-versal.
\end{corstar}
\begin{proof}
By Lemma \ref{lem:defgen} combined with the long exact sequence \eqref{eqn:sh+}, $SH^2(X \setminus D;\fmuncomp)/\im(H^2(X \setminus D;\fmuncomp))$ is spanned by the classes $\nov_p\mathsf{d}_p$.
By surjectivity of $\EuC\EuO \otimes \fmuncomp$, and the commutativity of the diagram \eqref{eqn:psslogcomm}, it follows that $\HH^2(\EuA,\EuA \otimes \fmuncomp)$ is spanned by the classes $\nov_p \co(z_p)$ together with $\EuC\EuO \circ i (H^2(X \setminus D;\fmuncomp))$. 
 
By Lemma \ref{lem:coisignsnoapp} together with the second hypothesis, it follows that $\HH^2(\EuA,\EuA \otimes \fmuncomp)^\Gamma$ is contained in the span of the classes $\nov_p \co(z_p)$. 
It follows that $\EuA_R$ satisfies the hypotheses of Theorem \ref{thm:eqversality}, so it is $R$-complete. 
If $\EuC\EuO \otimes \fmuncomp$ is in fact an isomorphism, then $\co(z_p) \neq 0$, so $\EuA_R$ is $R$-versal, again by Theorem \ref{thm:eqversality}.
\end{proof}

We also have an analogue for the ambient relative Fukaya category:

\begin{corstar}
\label{cor:geomhyp2amb}
Consider the situation of Theorem \ref{thm:eqversalityamb}.
Suppose furthermore that:
\begin{itemize}
\item $\Nef$ is amb-nice and amb-semi-positive;
\item $H^2\left(X \setminus D;\fmuncomp\right)^\Gamma =0$;
\item $\Gamma$ acts transitively on the components of $D'_p \cap X$, for all $p$;
\item The map
\begin{equation}
\label{eqn:cosurjag2}
\EuC \EuO \otimes \fmuncomp :SH^2\left(X \setminus D;\fmuncomp\right) \to \HH^2\left(\EuA,\EuA \otimes \fmuncomp \right)\end{equation}
is surjective (e.g., $\EuC\EuO$ is surjective).
\end{itemize}
Then $\EuA_R$ is an $R$-complete $(\Gamma,\sigma)$-equivariant deformation of $\EuA$.

If \eqref{eqn:cosurjag2} is furthermore an isomorphism, then $\EuA_R$ is $R$-versal.
\end{corstar}
\begin{proof}
It follows by the argument of the proof of Lemma \ref{lem:defgen} that $SH^2_+(X \setminus D;\fmuncomp)$ is spanned by the classes $\nov_p \mathsf{d}_q$ for $q \in i^{-1}(p)$. 
Since $\Gamma$ acts transitively on $i^{-1}(p)$ by hypothesis, it follows that $SH^2_+(X \setminus D;\fmuncomp)^\Gamma$ is contained in the span of the images of the classes $\nov_p \sum_{q \in i^{-1}(p)}\mathsf{d}_q$. 
By the second hypothesis combined with the long exact sequence \eqref{eqn:sh+}, it follows that $SH^2(X \setminus D;\fmuncomp)^\Gamma$ is contained in the span of these classes. 
Finally by the surjectivity and $\Gamma$-equivariance of $\EuC\EuO \otimes \fmuncomp$, it follows that $\HH^2(\EuA,\EuA \otimes \fmuncomp)^\Gamma$ is contained in the span of the classes $\Runcomp_0 \cdot \nov_p \sum_{q \in i^{-1}(p)}\co(z_q)$, so the hypotheses of Theorem \ref{thm:eqversalityamb} are satisfied.
\end{proof}
\appendix

\section{Signed group actions on $A_\infty$ categories}
\label{sec:signgrpact}

In this Appendix we fix a grading datum $\G=\{\Z \to G \xrightarrow{\sigma} \Z/2\}$ and a field $\Bbbk$. 
The category of $\G$-graded $\Bbbk$-vector spaces is equipped with the symmetric monoidal structure given by the Koszul sign rule (a $\G$-graded object becomes a $\Z/2$-graded object via $\sigma$).
If $a$ is a homogeneous element of a $\G$-graded $\Bbbk$-vector space of degree $g$, we define $|a| := \sigma(g) \in \Z/2$ and $|a|' := |a| + 1$. 
We abbreviate `$\otimes$' by `$|$', and denote $V' := V[1]$.
If $V$ is a $\G$-graded $\Bbbk$-vector space, $V^* := \Hom^\bullet(V,\Bbbk)$ is its dual.

It will also be convenient, for the purposes of keeping track of Koszul signs, to introduce `trivialized lines', which are $\G$-graded $\Bbbk$-vector spaces $\Ts_g$ equipped with an isomorphism $\Ts \simeq \Bbbk[-g]$ for some $g \in \G$ (so $\Ts_g$ is concentrated in degree $g$). 

\subsection{$A_\infty$ categories}

We follow the conventions for $A_\infty$ categories of \cite{Seidel2008c,Sheridan2015a}. 
We recall them, to fix notation.

A $\G$-graded, $\Bbbk$-linear pre-category $\EuC$ consists of:
\begin{itemize}
\item A set of objects $Ob(\EuC)$;
\item A $\G$-graded $\Bbbk$-vector space $hom^\bullet_\EuC(L_0,L_1)$ for each pair of objects, called a morphism space. We will abbreviate it by $\EuC(L_0,L_1)$ where convenient.
\end{itemize}
We define the shifted Hochschild cochain complex, which is a $\G$-graded vector space
\[ CC^\bullet(\EuC)' := \prod_{L_0,\ldots,L_k} \Hom^\bullet \left(\EuC'(L_0,L_1) | \ldots | \EuC'(L_{k-1},L_k), \EuC'(L_0,L_k)\right)\]
(we are abusing terminology since at this stage it is just a graded vector space, not a complex). 
To keep track of Koszul sign rules, it is convenient to regard a Hochschild cochain $\varphi$ as defining a map 
\begin{equation}
\label{eqn:hochtriv}
 \EuC'(L_0,L_1) | \ldots | \EuC'(L_{k-1},L_k) \to \Ts_\varphi| \EuC'(L_0,L_k),
\end{equation}
where $\Ts_\varphi \simeq \Bbbk[|\varphi|']$.

We define the Gerstenhaber product of $\varphi,\psi \in CC^\bullet(\EuC)'$:
\[(\varphi \circ \psi)^k (a_1,\ldots,a_k) := \sum (-1)^\dagger \varphi^*(a_1,\ldots,\psi^*(a_{i+1},\ldots),a_{j+1},\ldots,a_k)\]
where $\dagger$ is the Koszul sign associated to commuting $\psi$ (with its grading as an element of the shifted Hochschild cochain complex) through $a_1,\ldots,a_i$ (with their gradings as elements of the shifted morphism spaces). 
Equivalently, it is the sum of all compositions of isomorphisms 
\begin{align*} 
\EuC'(L_0,L_1) | \ldots | \EuC'(L_{k-1},L_k) & \xrightarrow{\psi} \EuC'(L_0,L_1)|\ldots|\Ts_\psi|\EuC'(L_i,L_j)|\ldots| \EuC'(L_{k-1},L_k) \\
& \xrightarrow{\varphi} \Ts_\psi|\Ts_\varphi|\EuC'(L_0,L_k) \\
& \to \Ts_{\varphi \circ \psi}| \EuC'(L_0,L_k)
\end{align*}
where the identity isomorphisms and Koszul isomorphisms are implicit, and the final line is the isomorphism $\Ts_\varphi| \Ts_\psi \to \Ts_{\varphi \circ \psi} $ sending $1 | 1 \mapsto 1$.

We define the Gerstenhaber bracket
\[ [\varphi,\psi] := \varphi \circ \psi - (-1)^{|\varphi| \cdot |\psi|} \psi \circ \varphi.\]
It is a graded Lie bracket on $CC^\bullet(\EuC)'$.

A ($\G$-graded, $\Bbbk$-linear) $A_\infty$ category is a pre-category $\EuC$ equipped with an element $\mu \in CC^1(\EuC)'$ satisfying $\mu \circ \mu = 0$, such that $\mu^0 = 0$.
Since $[\mu,\mu] = 0$, the Hochschild differential $\partial := [\mu,-]$ squares to $0$ and $(CC^\bullet(\EuC)',\partial,[-,-])$ becomes a d$\G$la.
 
If $\EuC$ is a pre-category, we define the \emph{opposite} pre-category $\EuC^{op}$. 
It has the same objects as $\EuC$, or more precisely, the objects are in canonical bijection with those of $\EuC$: we denote the object of $\EuC^{op}$ corresponding to $X \in Ob(\EuC)$ by $X^{op}$.
The morphism spaces are
\begin{equation} \EuC^{op}(X^{op},Y^{op}) := \EuC(Y,X);\end{equation}
we denote the isomorphism between the two by $op: \EuC^{op}(X^{op},Y^{op}) \to \EuC(Y,X)$.

We have an identification of Hochschild cochain complexes
\begin{align}
\label{eqn:opcc} op: CC^\bullet(\EuC)' & \to CC^\bullet(\EuC^{op})' 
\end{align}
given by the isomorphisms on morphism spaces and a Koszul isomorphism; explicitly,
\begin{align}
\eta_{op}^k(a_1,\ldots,a_k) & := (-1)^\dagger \eta^k(a_k,\ldots,a_1)
\end{align}
where $\dagger$ is the Koszul sign associated to the reordering of the inputs $a_i$.
This map respects the Gerstenhaber product and bracket. 

It follows that, if $\mu$ is an $A_\infty$ structure on $\EuC$, then $\mu_{op} := op(\mu)$ is an $A_\infty$ structure on $\EuC^{op}$. 
It is immediate that \eqref{eqn:opcc} becomes an isomorphism of d$\G$la.

\begin{rmk}
One can define a category whose objects are $A_\infty$ categories, and whose morphisms are strict $A_\infty$ isomorphisms (of course this is a very restrictive notion, but it suffices for what we want to describe here). 
Then one can define an endofunctor $op$ of this category, which on the level of objects sends $\EuC$ to $\EuC^{op}$. 
There is an obvious natural isomorphism of functors $op \circ op \simeq \id$.
The induced isomorphism 
\[ CC^\bullet(\EuC) \xrightarrow{op} CC^\bullet(\EuC^{op}) \xrightarrow{op} CC^\bullet((\EuC^{op})^{op}) \simeq CC^\bullet(\EuC)\]
is the identity.
\end{rmk}

\subsection{Shifts in $A_\infty$ categories}
\label{subsec:shifts}

We have isomorphisms
\begin{align*}
c_{g_0,g_1}: \Ts_{g_0}|\Ts_{g_1} & \xrightarrow{\sim} \Ts_{g_0+g_1} \\
c_{g_0,g_1}(1|1) & := 1.
\end{align*}
We consider the full subcategory of the category of $\G$-graded $\Bbbk$-vector spaces with objects $\Ts_g$ for $g \in \G$. 
Identifying objects via the isomorphisms $c_{g_0,g_1}$, this subcategory inherits a \emph{strict} symmetric monoidal structure (recall that a symmetric monoidal category is called strict if the associator is an identity morphism). 
We denote this strict symmetric monoidal category by $\mathbf{G}$.

If $L_0$ and $L_1$ are objects of a pre-category $\EuC$, then an isomorphism $L_0 \simeq L_1$ is a choice of isomorphisms $\EuC(L_0,K) \simeq \EuC(L_1,K)$ and $\EuC(K,L_0) \simeq \EuC(K,L_1)$ for all objects $K$, which are compatible in the sense that the isomorphisms between $\EuC(L_0,L_0), \EuC(L_0,L_1), \EuC(L_1,L_0),$ and $\EuC(L_1,L_1)$ agree.

Using these two notions one can define the notion of a \emph{set of rightwards shift functors} for a pre-category $\EuC$ to be a fully faithful pre-functor
\begin{equation}
\label{eqn:rightshiftsfunc}
 \EuC \otimes \mathbf{G} \to \EuC
\end{equation}
which satisfies certain natural compatibilities with the symmetric monoidal structure.
Rather than spell the definition out in these terms, we give an equivalent but more explicit formulation. 

There should be given an action of the abelian group $G$ on the objects of $\EuC$: we denote $g \cdot L$ by $L[g]$ (it should be thought of as the image of the object $L \otimes \Ts_g$ under the functor \eqref{eqn:rightshiftsfunc}). 
There should be given isomorphisms
\[ s_\r^{g_0,g_1}: \EuC(L_0[g_0],L_1[g_1]) \xrightarrow{\sim} \Hom^\bullet(\Ts_{g_0},\Ts_{g_1}) | \EuC(L_0,L_1)\]
which are compatible in the sense that the composition
\begin{align}
\label{eqn:sghcomps} \EuC(L_0[g_0+h_0],L_1[g_1+h_1]) & \xrightarrow{s_\r^{h_0,h_1}} \Hom^\bullet(\Ts_{h_0},\Ts_{h_1})|  \EuC(L_0[g_0],L_1[g_1])\\
 & \xrightarrow{s_\r^{g_0,g_1}}  \Hom^\bullet(\Ts_{h_0},\Ts_{h_1})|\Hom^\bullet(\Ts_{g_0},\Ts_{g_1})| \EuC(L_0,L_1) \nonumber\\
 & \to \Hom^\bullet(\Ts_{g_0}|\Ts_{h_0},\Ts_{g_1}|\Ts_{h_1})| \EuC(L_0,L_1) \quad \text{Koszul isomorphism} \nonumber\\
 & \xrightarrow{c_{g_0,h_0},c_{g_1,h_1}} \Hom^\bullet(\Ts_{g_0+h_0},\Ts_{g_1+h_1})|\EuC(L_0,L_1) \nonumber
 \end{align}
 coincides with $s_\r^{g_0+h_0,g_1+h_1}$. 

Equivalently but even more explicitly, we could `trivialize the lines' so that $s_\r^{g_0,g_1}$ becomes an isomorphism $\EuC(L_0[g_0],L_1[g_1]) \simeq \EuC(L_0,L_1)[g_1-g_0]$ satisfying 
\[ s_\r^{g_0,g_1} \circ s_\r^{h_0,h_1} = (-1)^{|g_1|\cdot|h_0|} s_\r^{g_0+h_0,g_1+h_1}.\] 
However we prefer not to trivialize the lines so that we do not have to write out the Koszul signs explicitly.

The following lemma will be useful when we come to define shift functors on the Fukaya category:

\begin{lem}
\label{lem:shiftsrgsg}
In order to define a set of rightwards shift functors on a pre-category $\EuC$, it suffices to define
\begin{align}
 s_\r^g: \EuC(L_0,L_1[g]) &\to \EuC(L_0,L_1)|\Ts^g \quad \text{and}\\
 s^g: \EuC(L_0[g],L_1[g]) & \to \EuC(L_0,L_1)
 \end{align}
 such that: 
 \begin{itemize}
 \item $s_\r^{g_0} \circ s_\r^{g_1} = s_\r^{g_0+g_1}$, in the sense that the diagram
\begin{equation}
\label{eqn:srsrissr}
\xymatrix{
\EuC(L_0,L_1[g_0+g_1]) \ar[r]^-{s_\r^{g_1}}  \ar[d]^-{s_\r^{g_0+g_1}} & \EuC(L_0,L_1[g_0])|\Ts_{g_1} \ar[d]^-{s_\r^{g_0}} \\
\EuC(L_0,L_1)|\Ts_{g_0+g_1} & \EuC(L_0,L_1)|\Ts_{g_0}|\Ts_{g_1}  \ar[l]^-{c_{g_0,g_1}}}
\end{equation}
commutes;
\item $s^{g_0} \circ s^{g_1} = s^{g_0+g_1}$, in the only meaningful sense;
 \item $s_\r^{g_1} \circ s^{g_0} = (-1)^{|g_0|\cdot|g_1|} \cdot s^{g_0} \circ s_\r^{g_1}$, in the sense that the diagram
 \begin{equation}
 \label{eqn:srcomms}
  \xymatrix{ \EuC(L_0[g_0],L_1[g_0+g_1]) \ar[r]^-{s^{g_0}} \ar[d]^-{s_\r^{g_1}} & \EuC(L_0,L_1[g_1]) \ar[d]^-{s_\r^{g_1}} \\
 \EuC(L_0[g_0],L_1[g_0])|\Ts_{g_1} \ar[r]^-{s^{g_0}} & \EuC(L_0,L_1)|\Ts_{g_1}} 
 \end{equation}
 commutes up to the sign $(-1)^{|g_0|\cdot |g_1|}$.
 \end{itemize}
 \end{lem}
 \begin{proof}
Given a set of shift functors, we set $s_\r^g:= s_\r^{0,g}$ and $s^g|\id := s_\r^{g,g}$. 
The sign in \eqref{eqn:srcomms} arises from the different trivializations of $\Ts_{g_0+g_1}$ induced by $\Ts_{g_0}|\Ts_{g_1} \to \Ts_{g_0+g_1}$ and $\Ts_{g_1}|\Ts_{g_0} \to \Ts_{g_0+g_1}$.

Conversely, given $s_\r^{0,g}$ and $s_\r^{g,g}$ satisfying the correct commutation relations, there is clearly a unique extension to maps $s_\r^{g_0,g_1}$ which are compatible in the sense of \eqref{eqn:sghcomps}.
\end{proof}

We say that a Hochschild cochain $\eta \in CC^\bullet(\EuC)'$ \emph{respects rightwards shifts} if 
\[ (comp | \eta)\circ(s_\r^{g_0,g_1}|\ldots|s_\r^{g_{k-1},g_k}) = \eta,\]
as maps
\[ \EuC'(L_0[g_0],L_1[g_1]) | \ldots |\EuC'(L_{k-1}[g_{k-1}],L_k[g_k]) \to \Ts_\eta|\EuC'(L_0[g_0],L_k[g_k]),\]
where 
\[ comp: \Hom^\bullet(\Ts_{g_0},\Ts_{g_1}) | \ldots | \Hom^\bullet(\Ts_{g_{k-1}},\Ts_{g_k}) \to \Hom^\bullet(\Ts_{g_0},\Ts_{g_k})\]
denotes the composition map, and the Koszul isomorphism is implicit (compare \cite[\S 3k]{Seidel:FCPLT}).

We denote the subspace of Hochschild cochains which respect shifts by $CC^\bullet_\r(\EuC)' \subset CC^\bullet(\EuC)'$. 
We observe that it is closed under the Gerstenhaber product and bracket. 

Now let $\EuC = (\EuC,\mu)$ be an $A_\infty$ category. 
A set of rightwards shift functors for $\EuC$ is defined to be a set of rightwards shift functors for the underlying pre-category, such that $\mu \in CC^1_\r(\EuC)'$. 
The Hochschild differential preserves the Hochschild cochains which respect shifts, so $CC^\bullet_\r(\EuC)' \subset CC^\bullet(\EuC)'$ becomes a sub-d$\G$la. 
In fact it is a quasi-isomorphic sub-d$\G$la: if $\EuD \subset \EuC$ is a full subcategory with the property that every object of $\EuC$ is equal to a shift of precisely one object of $\EuD$, then $\EuD$ is Morita-equivalent to $\EuC$ so the restriction map $CC^\bullet(\EuC) \to CC^\bullet(\EuD)$ is a quasi-isomorphism. 
But the composition $CC^\bullet_\r(\EuC) \hookrightarrow CC^\bullet(\EuC) \to CC^\bullet(\EuD)$ is an isomorphism: hence the inclusion is a quasi-isomorphism.

The following lemma will be useful when we come to prove that the $A_\infty$ structure maps on the Fukaya category respect shifts:

\begin{lem}
\label{lem:mucompatshifts}
Given a set of rightwards shift functors specified by $s_\r^g$ and $s^g$ as in Lemma \ref{lem:shiftsrgsg}, a Hochschild cochain $\eta$ is compatible with rightwards shifts if and only if
\begin{itemize}
\item It satisfies $s^g(\mu^k(a_1,\ldots,a_k)) = \mu^k(s^g(a_1),\ldots,s^g(a_k))$, in the only meaningful sense;
\item It satisfies $s_\r^g(\eta^k(a_1,\ldots,a_k)) = \eta^k(a_1,\ldots,a_{i-1},s_\r^g(a_i),s^g(a_{i+1}),\ldots,s^g(a_k))$ in the sense that the diagram
\[ \xymatrix{ \EuC'(L_0,L_1)| \ldots| \EuC'(L_{i-1},L_i[g])|\ldots|\EuC'(L_{k-1}[g],L_k[g])  \ar[r]^-{\eta} \ar[d]_-{\id|\ldots|\id|s_\r^g|s^g|\ldots|s^g} & \Ts_\eta|\EuC'(L_0,L_k[g]) \ar[d]^-{s_\r^g} \\
\EuC'(L_0,L_1)|\ldots|\EuC'(L_{k-1},L_k)|\Ts_g \ar[r]^-{\eta} & \Ts_\eta|\EuC'(L_0,L_k)|\Ts_g}\]
commutes (with a Koszul isomorphism implicit, on the left vertical arrow).
\end{itemize}
\end{lem}

\begin{rmk}
Note that there is a strict auto-equivalence of $\EuC$ which acts on objects by $L \mapsto L[g]$, and on morphism spaces by $s^{-g}$: by abuse of notation we will denote this functor by $[g]$.
\end{rmk}

Now we define the notion of a set of \emph{leftwards shift functors}: on the pre-category level, these are given by a fully faithful pre-functor
\[ \EuC \otimes \mathbf{G}^{op} \to \EuC,\]
with properties as before. 
Explicitly, we have isomorphisms
\[ s_\l^{g_0,g_1}: \EuC(L_0[-g_0],L_1[-g_1]) \xrightarrow{\sim} \Hom^\bullet(\Ts_{g_1},\Ts_{g_0}) | \EuC(L_0,L_1)\]
satisfying properties analogous to those of $s_\r^{g_0,g_1}$; and this is equivalent to the existence of
\begin{align}
 s_\l^g: \EuC(L_0[-g],L_1) &\to \EuC(L_0,L_1)|\Ts_g \quad \text{and}\\
 s^g: \EuC(L_0[g],L_1[g]) & \to \EuC(L_0,L_1)
 \end{align}
 satisfying appropriate properties. 
 This allows us to define the subcomplex $CC^\bullet_\l(\EuC)' \subset CC^\bullet(\EuC)'$ of Hochschild cochains that respect the leftwards shift functors, and a set of leftwards shift functors on an $A_\infty$ category $\EuC$ is a set of leftwards shift functors on the underlying pre-category such that $\mu \in CC^1_\l(\EuC)'$. 
 
 Now suppose that $p: \G \to \Z/4$ is morphism of grading data. 
We consider the isomorphisms
\begin{align*}
i_g: \Ts_{-g} & \to \Ts_g^* \\
i_g(1) & := (-1)^{\dagger_{p(g)}}\cdot 1^*
\end{align*}
where
\begin{equation}
\label{eqn:dagsgn}
 \dagger_k := \frac{k(k-1)}{2} \quad \text{ for $k \in \Z/4$}
\end{equation}
and $1^*(1) = 1$.
The crucial property of this map is that it makes the diagram
\[ \xymatrix{\Ts_{-g_0}|\Ts_{-g_1} \ar[r]^-{i_{g_0}|i_{g_1}} \ar[dd]^{c_{-g_0,-g_1}} & \Ts_{g_0}^*| \Ts_{g_1}^* \ar[d] \\
& (\Ts_{g_0}|\Ts_{g_1})^*  \\
\Ts_{-g_0-g_1} \ar[r]^-{i_{g_0+g_1}}& \Ts_{g_0+g_1}^* \ar[u]^-{c_{g_0,g_1}^*}}\]
commute.

As a result, there is a $\G$-graded strict isomorphism of $\G$-graded symmetric monoidal categories $\mathbf{G} \simeq \mathbf{G}^{op}$ which sends $\Ts_g \mapsto \Ts_{-g}$ on the level of objects, and on the level of morphisms
sends
\begin{align}
\label{eqn:gopfun}
\Hom^\bullet(\Ts_{g_0},\Ts_{g_1}) & \xrightarrow{\alpha \mapsto \alpha^*} \Hom^{\bullet}(\Ts_{g_1}^*,\Ts_{g_0}^*) \\
& \xrightarrow{i_{g_0},i_{g_1}} \Hom^\bullet(\Ts_{-g_1},\Ts_{-g_0}).\nonumber
\end{align}

Given this isomorphism, we say that a set of rightwards and leftwards shift functors are \emph{compatible} if the functors $\EuC \otimes \mathbf{G} \to \EuC$ and $\EuC \otimes \mathbf{G}^{op} \to \EuC$ defining the rightwards and leftwards shift functors are identified by the isomorphism $\mathbf{G} \simeq \mathbf{G}^{op}$. 
Equivalently, $s_\l^{g_0,g_1}$ is equal to the composition
\begin{align*}
\EuC(L_0[-g_0],L_1[-g_1]) & \xrightarrow{s_\r^{-g_0,-g_1}} \Hom^\bullet(\Ts_{-g_0},\Ts_{-g_1})|\EuC(L_0,L_1) \\
& \xrightarrow{\eqref{eqn:gopfun}} \Hom^\bullet(\Ts_{g_1},\Ts_{g_0})| \EuC(L_0,L_1).
\end{align*}
Given $p$ and a set of rightwards shift functors, there is a unique compatible set of leftwards shift functors. 
If the rightwards shift functors are determined by maps $s_\r^g$ and $s^g$ as in Lemma \ref{lem:shiftsrgsg}, then the leftwards shift functors are determined by the same maps $s^g$, together with the maps
\begin{equation}
\label{eqn:slgfromsrg}
 s_\l^g = (-1)^{\dagger_{p(-g)}} \cdot s^{-g} \circ s_\r^g.
\end{equation}

In this situation, a Hochschild cochain respects rightwards shift functors if and only if it respects the induced leftwards shift functors. 
We denote $CC^\bullet_{\l \r}(\EuC)' := CC^\bullet_\r(\EuC)'  = CC^\bullet_\l(\EuC)' \subset CC^\bullet(\EuC)'$.

\subsection{Group actions}
\label{subsec:gract}

Let $\EuC$ be an $A_\infty$ category. 
We denote by $\mathsf{StrAut}(\EuC)$ the group of strict autoequivalences of $\EuC$. 
Explicitly, a strict autoequivalence consists of an automorphism $\gamma: \G \to \G$ of the grading datum, together with a map on objects $L \mapsto \gamma \cdot L$, and isomorphisms of morphism spaces
\begin{equation}\gamma: hom^g(L_0,L_1) \to hom^{\gamma \cdot g}(\gamma \cdot L_0, \gamma \cdot L_1) \end{equation}
which respect the $A_\infty$ products.

\begin{rmk}
In many of our intended applications, the action of $\gamma$ on $\G$ will be trivial.
\end{rmk}

Suppose that $\EuC$ is endowed with a set of rightwards shift functors. 
The autoequivalence $\gamma$ is said to commute with rightwards shift functors if it satisfies $\gamma \cdot (L[g]) = (\gamma \cdot L)[\gamma \cdot g]$, and the maps on morphism spaces
satisfy 
\begin{align}
\gamma \circ s_\r^{g_0,g_1} &= s_\r^{\gamma \cdot g_0,\gamma \cdot g_1} \circ \gamma,
\end{align}
where the identifications $\Ts_{g_i} \simeq \Ts_{\gamma \cdot g_i}$ sending $1 \mapsto 1$ are implicit.
There is an obvious notion of composition of strict autoequivalences, and it is obvious that if two strict autoequivalences commute with shifts then so does their composition.
We denote the subgroup of strict autoequivalences that commute with rightwards shift functors by $\mathsf{StrAut}_\r(\EuC)$.

Similarly, if $\EuC$ is endowed with a set of leftwards shift functors we define $\mathsf{StrAut}_\l(\EuC)$. 
If it is endowed with compatible rightwards and leftwards shift functors, then a strict autoequivalence commutes with the rightwards shift functors if and only if it commutes with the leftwards shift functors: we say it commutes with shift functors, and denote the corresponding subgroup by $\mathsf{StrAut}_{\l\r}(\EuC)$.

We observe that $\gamma$ commutes with rightwards shifts if and only if it commutes with the maps $s_\r^g$ and $s^g$ as defined in Lemma \ref{lem:shiftsrgsg}.
Thus we see, by Lemma \ref{lem:mucompatshifts}, that the shift functor $[g] \in \mathsf{StrAut}(\EuC)$ commutes with shift functors if and only if $g$ is even. 
It follows that we have a group homomorphism
\begin{equation} G_{ev} \to \mathsf{StrAut}_\r(\EuC)\end{equation}
where $G_{ev} := \sigma^{-1}(0)$ is the even part of the grading datum $\G = \{\Z \to G \xrightarrow{\sigma} \Z/2\}$. 
We will call an element of the quotient $\mathsf{StrAut}_\r(\EuC)/G_{ev}$ an \emph{autoequivalence up to shift}. 
We call $G_{ev}$ the group of \emph{even shifts}.

Any strict autoequivalence $\gamma: \EuC \to \EuC$ acts on $CC^\bullet(\EuC)'$: 
\begin{equation} (\gamma \cdot \eta)^k(a_1,\ldots,a_k) := \gamma^{-1} \cdot \eta^k(\gamma \cdot a_1,\ldots,\gamma \cdot a_k).\end{equation}
If $\gamma$ commutes with rightwards shifts, then this action preserves $CC^\bullet_\r(\EuC)'$. 
Furthermore, the action of the even shifts $G_{ev}$ on $CC^\bullet_\r(\EuC)'$ is trivial, by definition. 
Therefore, we have an action
\begin{equation} 
\label{eqn:autracts}
\mathsf{StrAut}_\r(\EuC)/G_{ev} \acts CC^\bullet_\r(\EuC)'.\end{equation}

This action clearly respects the Gerstenhaber product $\circ$, hence also the Gerstenhaber bracket $[-,-]$. 
It preserves the element $\mu \in CC^1(\EuC)'$ by definition; hence it respects the Hochschild differential $\partial := [\mu^*,-]$. 
Thus $\mathsf{StrAut}_\r(\EuC)/G_{ev}$ acts on $CC_\r^\bullet(\EuC)'$ as a d$\G$la.
Analogous statements apply in the situation of leftwards shifts, or of compatible leftwards and rightwards shifts.

\subsection{Signed group actions}
\label{subsec:sgp}

A set of rightwards shift functors on $\EuC$ induces a set of leftwards shift functors on $\EuC^{op}$: on the level of objects,
 \[L^{op}[-g] := L[g]^{op},\]
and the isomorphisms $s^{\l}_{g_0,g_1}$ on $\EuC^{op}$ are the unique ones that make the diagram
\begin{equation} \xymatrix{ \EuC(L_1[g_1],L_0[g_0]) \ar[d]^-{s_\r^{g_1,g_0}} \ar[r]^-{op}  & \EuC^{op}(L_0^{op}[-g_0],L_1^{op}[-g_1]) \ar[d]^-{s_\l^{g_0,g_1}} \\
  \Hom^\bullet(\Ts_{g_1},\Ts_{g_0})|\EuC(L_1,L_0)\ar[r]^-{op} &  \Hom^\bullet(\Ts_{g_1},\Ts_{g_0})|\EuC^{op}(L_0^{op},L_1^{op}) }\end{equation}
commute.
We observe that $s_\r^g$ corresponds to $s_\l^g$, and $s^g$ to $s^{-g}$ under this correspondence.

Similarly, a set of leftwards shift functors for $\EuC$ induces a set of rightwards shift functors for $\EuC^{op}$. 
Given a morphism of grading data $p: \G \to \Z/4$ and compatible rightwards and leftwards shift functors on $\EuC$, the induced leftwards and rightwards shift functors on $\EuC^{op}$ are also compatible (with respect to the same $p$). 
This can be seen from the symmetry of \eqref{eqn:slgfromsrg} under exchanging $s_\r^g$ with $s_\l^g$ and substituting $s^{-g}$ for $s^g$.
In this situation, we define  $\EuC \coprod \EuC^{op}$ to be the obvious $A_\infty$ category with objects $Ob(\EuC) \coprod Ob(\EuC^{op})$, and morphism spaces $hom^\bullet(X,Y) := 0$ if $X$ and $Y$ come from different categories. 
It is endowed with a set of shift functors and compatible left- and right-shift maps.

\begin{defn}
We define a subgroup 
\begin{equation} \mathsf{StrAut}_{\l\r}^\sigma(\EuC) \subset \mathsf{StrAut}_{\l\r}\left(\EuC \coprod \EuC^{op}\right)\end{equation}
consisting of all autoequivalences $\gamma$ such that
\begin{itemize}
\item The action of $\gamma$ on objects either sends $Ob(\EuC) \to Ob(\EuC)$ and $Ob(\EuC^{op}) \mapsto Ob(\EuC^{op})$ (in which case we say $\gamma$ is \emph{even}), or it swaps $Ob(\EuC)$ with $Ob(\EuC^{op})$ (in which case we say $\gamma$ is \emph{odd}).
\item $\gamma \cdot X^{op} = (\gamma \cdot X)^{op}$.
\item The action of $\gamma$ on morphism spaces satisfies 
\begin{equation} \gamma \circ op = op \circ \gamma.\end{equation}
\end{itemize}
\end{defn}

Thus we have an exact sequence
\begin{equation} 0 \to \mathsf{StrAut}_{\l\r}(\EuC) \to \mathsf{StrAut}^\sigma_{\l\r}(\EuC) \to \Z/2.\end{equation}
As before, we call an element of the quotient $\mathsf{StrAut}^\sigma_\r(\EuC)/G_{ev}$ a \emph{signed autoequivalence up to shift}, and we call $G_{ev}$ the group of \emph{even shifts}.

Now observe that $\mathsf{StrAut}_{\l\r}^\sigma(\EuC)/G_{ev}$ is a subgroup of $\mathsf{StrAut}_{\l\r}(\EuC \coprod \EuC^{op})/G_{ev}$, and hence it acts on the d$\G$la $CC^\bullet_{\l\r}(\EuC \coprod \EuC^{op})'$ as in the previous section.
We observe that 
\begin{equation} CC^\bullet_{\l\r}\left(\EuC \coprod \EuC^{op} \right)' \simeq CC^\bullet_{\l\r}(\EuC)' \oplus CC^\bullet_{\l\r}(\EuC^{op})'.\end{equation}
This allows us to define an action of $\Z/2$ on this d$\G$la, where the non-trivial element of $\Z/2$ maps
\begin{equation} \eta \oplus \xi \mapsto \xi^{op} \oplus \eta^{op}.\end{equation}
This $\Z/2$-action commutes with the action of $\mathsf{StrAut}_{\l\r}^\sigma(\EuC)/G_{ev}$, and the $\Z/2$-invariant part is isomorphic to $CC^\bullet_{\l\r}(\EuC)'$. 
Therefore we obtain an action
\begin{equation}\label{eqn:strautrlacts}
 \mathsf{StrAut}_{\l\r}^\sigma(\EuC)/G_{ev} \acts CC^\bullet_{\l\r}(\EuC)'.\end{equation}

\begin{defn}
\label{defn:sgngrpactainf}
Let  $(\Gamma,\sigma)$ be a signed group. 
An action of $(\Gamma,\sigma)$ on $\EuC$ up to shifts is a signed group homomorphism $\Gamma \to \mathsf{StrAut}^\sigma_{\l\r}(\EuC)/G_{ev}$. 
It is clear from the above that a $(\Gamma,\sigma)$-action up to shifts induces an action of $\Gamma$ on the d$\G$la $CC^\bullet_{\l\r}(\EuC)'$.
\end{defn}

Observe that even elements of $\Gamma$ act by autoequivalences of $\EuC$, and odd elements act by equivalences $\EuC \overset{\sim}{\to} \EuC^{op}$ (also known as \emph{dualities} \cite{Castano2010}). 

\section{Signs in the relative Fukaya category}
\label{sec:signsrel}

We continue with the notational conventions of Appendix \ref{sec:signgrpact}.

\subsection{Orienting Cauchy-Riemann operators}
\label{subsec:orCR}

In this section, we summarize the relationship between Pin structures and orientations of Cauchy-Riemann operators, following \cite[Chapter 11]{Seidel:FCPLT}. 

Associated to a graded real vector space $V$, we have the graded real line (henceforth, `line') $\lambda(V):= \wedge^{top}(V)$.  
More generally, associated to a Fredholm operator $D:X \to Y$, we have the determinant line $\lambda(D) := \det(D)$, which is concentrated in degree $\ind(D)$.

Now, let $\mathbb{D} \subset \C$ be the closed unit disc with its standard complex structure, and identify $\partial \mathbb{D} = S^1$.
Let $V$ be a symplectic vector space with compatible complex structure. 
Let $\EuL (\cG V)$ denote the free loop space of the unoriented Lagrangian Grassmannian $\cG V$, i.e., the space of maps $\rho:S^1 \to \cG V$. 
Its connected component $\EuL_\jmu (\cG V)$ consists of those loops with Maslov index $\mu(\rho) = k \in \Z$.
Any loop $\rho: S^1 \to \cG V$ defines boundary conditions for a Cauchy-Riemann operator on the disc $\mathbb{D}$, which we denote by $D_\rho$. 
There is an associated graded real line bundle over $\EuL (\cG V)$, whose fibre over $\rho$ is $\lambda(D_\rho)$. 

We also have a principal $\Z/2$-bundle $\mathsf{Pin}$ over $\EuL (\cG V)$, whose fibre $\mathsf{Pin}_\rho$ over $\rho$ is the set of isomorphism classes of Pin structures \cite[Section 11i]{Seidel:FCPLT} on the real vector bundle over $S^1$ defined by $\rho$. 
There is an associated line bundle over $\EuL (\cG V)$, which we also denote by $\mathsf{Pin}$ by abuse of notation. 
It is concentrated in degree $0$. 

We have a graded real vector bundle $\tau$ over $\EuL (\cG V)$, whose fibre $\tau_\rho$ over $\rho$ is the vector space $\rho(1)$, concentrated in degree $1$.
We will consider the following two graded line bundles over $\EuL (\cG V)$: $\lambda(\tau)^*$ (concentrated in degree $-n$) and $\lambda(\tau')|\Ts_1$ (concentrated in degree $1$). 
These line bundles are isomorphic (except for the gradings), but appear in different ways in our computations: the former appears in gluing theorems (as in Definition--Lemma \ref{deflem:bdglue}), whereas the latter appears as part of a twisted Pin structure (Definition \ref{defn:twpin}).

\begin{defn}
\label{defn:twpin}
We define a line bundle over $\EuL (\cG V)$,
\begin{equation} \label{eqn:twpindef}
\mathsf{TwPin}_{\rho}:= (\lambda(\tau_{\rho}')|\Ts_1)^{\otimes \mu(\rho)} \otimes \mathsf{Pin}_\rho.
\end{equation}
We call this the bundle of twisted Pin structures, following \cite[Equation (11.32)]{Seidel:FCPLT}. 
Observe that it is in degree $\mu(\rho)$.
\end{defn}

\begin{deflem}
\label{deflem:bdglue}
There exists two isomorphisms of line bundles,
\begin{equation}
\label{eqn:torsors} i_\l,i_\r:\lambda(D_\rho) \otimes \lambda(\tau_\rho)^* \simeq \mathsf{TwPin}(\rho),
\end{equation}
with properties we will now describe.
\end{deflem}

The existence of an isomorphism between the two line bundles follows from a computation of first Stiefel--Whitney classes on the two sides: see \cite[Lemma 11.17]{Seidel:FCPLT}. 
In fact, \cite[Proof of Lemma 11.17]{Seidel:FCPLT} goes further than showing the existence of an isomorphism: it \emph{prescribes} an isomorphism. 
The prescription is given explicitly over $\EuL_0(\cG V)$ and $\EuL_{-1} (\cG V)$. 
It is then extended to the remaining $\EuL_\jmu (\cG V)$, essentially by requiring that it be compatible with boundary gluings. 

The details of this extension are relevant to our purposes here, so let us expand on Seidel's treatment.
It appears to us that there are two sensible conventions for what `compatible with boundary gluings' means. 
Suppose we have two discs $\mathbb{D}_1, \mathbb{D}_2$ with boundary conditions $\rho_1, \rho_2$, and $\rho_1(-1) = \rho_2(1)$.
Then we can form a boundary connect sum, by attaching the point $-1 \in \partial \mathbb{D}_1$ to the point $+1 \in \partial \mathbb{D}_2$: one obtains the loop $\rho_1 \# \rho_2$. 
In this situation:
\begin{itemize}
\item There is a natural isomorphism
\begin{equation}
\label{eqn:glPin}
 \mathsf{Pin}_{\rho_1 \# \rho_2} \simeq \mathsf{Pin}_{\rho_1} \otimes \mathsf{Pin}_{\rho_2},
\end{equation}
given by `gluing' the Pin structures (this requires a lift of the isomorphism $\rho_1(-1) \simeq \rho_2(1)$ to the corresponding principal Pin bundles, but the glued Pin structure does not depend on the choice up to isomorphism). 
\item There are two natural choices of isomorphism
\begin{equation}
\label{eqn:glrho}
(\lambda(\tau_{\rho_1 \# \rho_2}')|\Ts_1)^{\otimes\mu(\rho_1 \# \rho_2)} \simeq (\lambda(\tau_{\rho_1}')|\Ts_1)^{\otimes \mu(\rho_1)} \otimes (\lambda(\tau_{\rho_2}')|\Ts_1)^{\otimes \mu(\rho_2(1))},
\end{equation}
arising from the natural identification $\rho_1(1) \simeq \rho_1 \# \rho_2(1)$, and the two natural ways of identifying $\lambda(\rho_2(1)) \simeq \lambda(\rho_1(-1))$ with $\lambda(\rho_1(1))$, by transporting clockwise (which we label `$\l$') or anticlockwise (which we label `$\r$') around the boundary of $\mathbb{D}_1$.
\item There is a natural isomorphism 
\begin{equation}
\label{eqn:gldet}
 \lambda(D_{\rho_1 \# \rho_2}) \otimes \lambda(\tau_{\rho_2}) \simeq \lambda(D_{\rho_1}) \otimes \lambda(D_{\rho_2})
\end{equation}
induced by the short exact sequence of Fredholm operators
\begin{equation} 0 \to D_{\rho_1 \# \rho_2} \to D_{\rho_1} \oplus D_{\rho_2} \to \tau_{\rho_2} \to 0\end{equation}
coming from a boundary gluing theorem (see \cite[Section 11c]{Seidel:FCPLT}).
\end{itemize}
Each of these isomorphisms (\eqref{eqn:glPin}, the two choices for \eqref{eqn:glrho}, and \eqref{eqn:gldet}) is `associative' under compositions of gluings. 
In the latter case, this follows from the commutativity of the diagram of Fredholm operators
\begin{equation} \xymatrix{ D_{\rho_1 \# \rho_2 \# \rho_3} \ar[r]\ar[d] & D_{\rho_1 \# \rho_2} \oplus D_{\rho_3} \ar[r]\ar[d] & \tau_{\rho_3} \ar[d] \\
D_{\rho_1} \oplus D_{\rho_2 \# \rho_3} \ar[r] \ar[d] & D_{\rho_1} \oplus D_{\rho_2} \oplus D_{\rho_3} \ar[r]\ar[d] & \tau_{\rho_3} \ar[d] \\
\tau_{\rho_2} \ar[r] & \tau_{\rho_2} \ar[r] & 0} \end{equation}
with short exact rows and columns, and the `exact squares' property of the determinant line bundle \cite[(2.27)]{Zinger2013}.

Combining \eqref{eqn:glPin} and the two choices of \eqref{eqn:glrho}, we obtain two choices of isomorphism
\begin{equation}
\label{eqn:glTwpin}
\mathsf{TwPin}(\rho_1 \# \rho_2) \simeq \mathsf{TwPin}(\rho_1) \otimes \mathsf{TwPin}(\rho_2),
\end{equation}
labelled by `$\l$' and `$\r$'.
Using \eqref{eqn:gldet} and the identification $\rho_1 \# \rho_2(1) \simeq \rho_1(1)$, we obtain an isomorphism
\begin{equation}
\label{eqn:gldets}
\lambda(D_{\rho_1 \# \rho_2}) \otimes \lambda(\tau_{\rho_1 \# \rho_2})^* \simeq \left(\lambda(D_{\rho_1}) \otimes \lambda(\tau_{\rho_1})^*\right) \otimes \left(\lambda(D_{\rho_2}) \otimes \lambda(\tau_{\rho_2})^*\right).
\end{equation} 
We declare the isomorphism $i_\l$ (respectively, $i_\r$) from \eqref{eqn:torsors} to be `compatible with gluing' if it is compatible with the isomorphisms \eqref{eqn:glTwpin} labelled $\l$ (respectively, labelled $\r$) and \eqref{eqn:gldets} in the obvious way. 

Given the choice of isomorphism \eqref{eqn:torsors} for $\mu(\rho) = 0,-1$ specified in \cite[Lemma 11.17]{Seidel:FCPLT}, there exists a unique extension of the isomorphism $i_\l$ (respectively, $i_\r$) to all $\EuL_\jmu(\cG(V))$ that is compatible with gluing (existence follows by the `associativity of gluing' of each of the isomorphisms \eqref{eqn:glPin}, \eqref{eqn:glrho}, \eqref{eqn:gldet}, together with the fact that the image of the gluing map in $\EuL_{\mu(\rho_1) + \mu(\rho_2)}(\cG (V))$ is connected). 
The isomorphisms $i_\l$ and $i_\r$ differ by the sign $(-1)^{\mu(\rho)(\mu(\rho)+1)}$.

\subsection{Orientation operators}
\label{subsec:orops}

Let $V$ be as in the previous section.
Let $\Lambda^\#_0, \Lambda^\#_1$ be transverse Lagrangian subspaces $\Lambda_i \in \cG V$, equipped with principal homogeneous $Pin$ spaces $P^{\#}_0, P^{\#}_1$, with isomorphisms $P^{\#}_i \times_{Pin} \R^n \simeq \rho(i)$; and let $k \in \Z$. 

Consider the space of paths $\rho:[0,1] \to \cG V$ from $\Lambda_0$ to $\Lambda_1$, with Maslov index $k$ (i.e., in the notation of \cite[\S 11g]{Seidel:FCPLT}, we consider paths with $I_H(\rho) = k$). 
Let $H$ be the upper half-plane, equipped with its standard complex structure and an outgoing strip-like end. 
Such a path $\rho$ can be turned into boundary conditions along $\partial H$, and there is an associated Cauchy-Riemann operator $D_{H,\rho}$, with determinant $\lambda(D_{H,\rho})$. 
This forms a graded line bundle over the space of paths. 

On the other hand, we have the principal $\Z/2$-bundle $\mathsf{Pin}(\Lambda_0^\#,\Lambda_1^\#)_\rho$ of isomorphism classes of Pin structures on $\rho$, equipped with identifications with $P_0^{\#}$ and $P_1^{\#}$ at the ends. 
Again we denote the associated graded line bundle (concentrated in degree $0$) by the same symbol. 
For any $\Z/2$-torsor $\cL$, there are canonical isomorphisms
\begin{equation}
\label{eqn:tenstors}
\cL \otimes \mathsf{Pin}(\Lambda_0^\#,\Lambda_1^\#)_\rho \simeq \mathsf{Pin}(\Lambda_0^\# \otimes \cL,\Lambda_1^\#)_\rho \simeq \mathsf{Pin}(\Lambda_0^\#,\Lambda_1^\# \otimes \cL)_\rho,
\end{equation}
where we denote $P^\# \otimes \cL := P^\# \times_{\Z/2} \cL$.

The (ungraded) line bundles $\lambda(D_{H,\rho})$ and $\mathsf{Pin}(\Lambda_0^\#,\Lambda_1^\#)_\rho$ are isomorphic (this follows from the argument in \cite[\S 11h]{Seidel:FCPLT}), so their tensor product is trivial. 
Since the path space is connected, we can canonically identify all fibres and thus define the graded line
\begin{equation} o(\Lambda_0^\#,\Lambda_1^\#,k) := \lambda(D_{H,\rho}) \otimes \mathsf{Pin}(\Lambda_0^\#,\Lambda_1^\#)_\rho. \end{equation}
We call this an \emph{orientation line} ($D_{H,\rho}$ is called an \emph{orientation operator}). 
It is concentrated in degree $k$, by \cite[Lemma 11.11]{Seidel:FCPLT}.

There is an obvious isomorphism
\begin{equation}
\label{eqn:sj}
 s^\jmu: o(\Lambda_0^\# \otimes \lambda(\Lambda_0)^{\otimes \jmu},\Lambda_1^\# \otimes \lambda(\Lambda_1)^{\otimes \jmu},k) \simeq  o(\Lambda_0^\#,\Lambda_1^\#,k), \end{equation} 
which is given by identifying $\lambda(\Lambda_0)$ with $\lambda(\Lambda_1)$ by transport along the path $\rho$. 
This will play the role of the isomorphism $s^\jmu = s_\r^{\jmu,\jmu} = s_\l^{-\jmu,-\jmu}$ from Section \ref{subsec:shifts}.
Next we turn to the isomorphisms $s_\r^\jmu$ and $s_\l^\jmu$.

Let $\rho_1:[0,1] \to \cG V$ be a path as above, $\rho_2 \in \EuL_\jmu (\cG V)$, with $\rho_1(1/2) = \rho_2(1)$. 
Then we can perform a boundary connect sum as in the previous section, to obtain a gluing isomorphism
\begin{align}
 \lambda(D_{H,\rho_1 \# \rho_2}) \otimes \lambda(\tau_{\rho_2}) &\simeq \lambda(D_{H,\rho_1}) \otimes \lambda(D_{\rho_2}) \\
\label{eqn:sR1} \implies  \lambda(D_{H,\rho_1 \# \rho_2}) & \simeq \lambda(D_{H,\rho_1}) \otimes \mathsf{TwPin}_{\rho_2} \text{ (applying $i_\r$).}
\end{align}
We also have an obvious isomorphism
\begin{equation} \mathsf{Pin}_{\rho_1 \# \rho_2} \simeq \mathsf{Pin}_{\rho_1} \otimes \mathsf{Pin}_{\rho_2};\end{equation}
combining with \eqref{eqn:sR1} and \eqref{eqn:twpindef} we obtain an isomorphism
\begin{align}
\lambda(D_{H,\rho_1 \# \rho_2}) \otimes \mathsf{Pin}(\rho_1 \# \rho_2) &\simeq \lambda(D_{H,\rho_1}) \otimes \mathsf{Pin}(\rho_1) \otimes (\lambda(\tau'_{\rho_2})|\Ts_1)^{\otimes \jmu} \\
\implies o(\Lambda_0^\#,\Lambda_1^\#,k+\jmu) & \simeq o(\Lambda_0^\#,\Lambda_1^\#,k)\otimes \lambda(\tau'_{\rho_2})^{\otimes i} \otimes \Ts_\jmu \text{ (by definition)}\\
\label{eqn:sR2} \implies o(\Lambda_0^\#, \Lambda_1^\# \otimes \lambda(\Lambda_1)^{\otimes \jmu}, k+\jmu) & \simeq  \Ts_\jmu \otimes o(\Lambda_0^\#,\Lambda_1^\#,k),
\end{align}
where the final isomorphism arises by transporting $\lambda(\rho_2(1)) \simeq \lambda(\rho_1(1/2))$ to $\lambda(\rho_1(1)) = \lambda(\Lambda_1)$ and applying \eqref{eqn:tenstors}.

Thus we have defined an isomorphism
\begin{equation}
\label{eqn:srjmu}
 s_\r^\jmu:  o(\Lambda_0^\#,\Lambda_1^\# \otimes \lambda(\Lambda_1)^{\otimes \jmu},k+\jmu) \xrightarrow{\simeq} \Ts_\jmu | o(\Lambda_0^\#,\Lambda_1^\#,k) ,\end{equation}
where $\Ts_\jmu$ is a trivial line in degree $\jmu$.
By an analogous construction (applying $i_\l$ instead of $i_\r$, and transporting $\lambda(\rho_2(1)) \simeq \lambda(\rho_1(1/2))$ to $\lambda(\rho_1(0)) = \lambda(\Lambda_0)$ before applying the other version of \eqref{eqn:tenstors}), we obtain isomorphisms
\begin{equation} s_\l^\jmu: o(\Lambda_0^\# \otimes \lambda(\Lambda_0)^{\otimes -\jmu},\Lambda_1^\#,k+\jmu) \simeq  \Ts_\jmu |o(\Lambda_0^\#,\Lambda_1^\#,k).\end{equation}

\begin{lem}
\label{lem:shiftsigns}
We have
\begin{align}
s^{\jmu_1} \circ s^{\jmu_2} &= s^{\jmu_1+\jmu_2} \label{eqn:sh1} \\
s_\r^{\jmu_1} \circ s_\r^{\jmu_2} &= s_\r^{\jmu_1+\jmu_2} \label{eqn:sh2} \\
s_\l^{\jmu_1} \circ s_\l^{\jmu_2} &= s_\l^{\jmu_1+\jmu_2} \label{eqn:sh3} \\
s^{\jmu_1} \circ s_\r^{\jmu_2} &= (-1)^{\jmu_1 \cdot \jmu_2}\cdot  s_\r^{\jmu_2} \circ s^{\jmu_1} \label{eqn:sh4} \\
s^{\jmu_1} \circ s_\l^{\jmu_2} &= (-1)^{\jmu_1 \cdot \jmu_2} \cdot  s_\l^{\jmu_2} \circ s^{\jmu_1} \label{eqn:sh5} \\
s_\l^{\jmu} &= (-1)^{\dagger_{-\jmu}} \cdot s^{-\jmu} \circ s_\r^\jmu \label{eqn:sh6} 
\end{align}
where $\dagger_\jmu := \jmu(\jmu-1)/2$ as in \eqref{eqn:dagsgn}.
\end{lem}
\begin{proof}
\eqref{eqn:sh1} is easy. 
\eqref{eqn:sh2} and \eqref{eqn:sh3} follow by associativity of gluing as in the previous section. 
Observe that when we transport $\lambda(\rho_1(1/2))$ to $\lambda(\rho_1(1))$ in the definition of $s_\r^{\jmu}$, we are transporting anticlockwise: that is why we must apply $i_\r$ rather than $i_\l$ in the definition of $s_\r^\jmu$, because otherwise associativity of gluing will not hold.  

To prove \eqref{eqn:sh4}, by \eqref{eqn:sh1}--\eqref{eqn:sh3} it suffices to prove the special case $s^1 \circ s_\r^1 = -s_\r^1 \circ s^1$.
This follows because the identification of $\lambda(\Lambda_0)$ with $\lambda(\Lambda_1)$ along $\rho$ acquires an extra minus sign as it goes around the disc of Maslov index $1$ that we glued on. 
Similarly for \eqref{eqn:sh5}. 

To prove \eqref{eqn:sh6}, by \eqref{eqn:sh1}--\eqref{eqn:sh5} it suffices to prove the special case $\jmu=-1$, which is equivalent to the commutativity of the diagram
\[\xymatrix{ o(\Lambda_0^\# \otimes \lambda(\Lambda_0),\Lambda_1^\#,k-1) \ar[r]^-{s_\r^{-1}} \ar[dr]_-{s_\l^{-1}} & o(\Lambda_0^\# \otimes \lambda(\Lambda_0),\Lambda_1^\# \otimes \lambda(\Lambda_1),k)|\Ts_{-1} \ar[d]^-{s^1} \\
& o(\Lambda_0^\#,\Lambda_1^\#,k)|\Ts_{-1}. }\]
The diagram commutes because both $s_\l^{-1}$ and $s_\r^{-1}$ are defined by gluing a disc of Maslov index $-1$, on which the isomorphisms $i_\l$ and $i_\r$ agree by construction. 
The other possible sign difference arises from the identification of orientation lines of the spaces $\rho_1(1)$ and $\rho_1(1/2)$: but translating $\lambda(\rho_1(1))$ to $\lambda(\rho_1(0))$, then back to $\lambda(\rho_1(1/2))$, is the same as transporting it directly to $\lambda(\rho_1(1/2))$.
\end{proof}

\subsection{Signs in the exact Fukaya category}
\label{subsec:ainfsign}

We recall the sign conventions for the exact Fukaya category, following \cite[\S 12f]{Seidel:FCPLT} with some modifications. 
Let $M$ be a Liouville domain with $\omega = d \alpha$. 
Let 
\begin{equation}\G = \{\Z \to G \xrightarrow{\sigma} \Z/2\}  := \{ \Z \to H_1(\cG M) \xrightarrow{\sigma} \Z/2\}\end{equation}
be the grading datum associated to $M$, where we recall that the sign morphism $\sigma: H_1(\cG M) \to \Z/2$ is given by the first Stiefel--Whitney class of the tautological bundle of the Lagrangian Grassmannian $\cG M$ (see \cite[\S 3.1]{Sheridan:CY}). 

Objects of the exact Fukaya category are \emph{anchored branes} $L^\# = (L,\iota,\tilde{\iota},P^\#)$, where $L$ is a smooth compact manifold equipped with a Pin structure $P^\#$, $\iota: L \to M$ is an exact Lagrangian embedding, and $\tilde{\iota}: L \to \wt{\cG}M$ is an anchoring (i.e., a lift of $\iota$ to the universal abelian cover $\wt{\cG}M$ of the total space of $\cG M$). 
There is an action of $G \oplus \Z/2$ on these objects:
\begin{equation}  (g \oplus \sigma) \cdot (L,\iota,\tilde{\iota},P^\#) := \left(L,\iota,g \circ \tilde{\iota}, P^\# \otimes \lambda(TL)^{\otimes \sigma}\right),\end{equation}
where $g$ acts on $\wt{\cG} M$ by the covering group action. 
We define an \emph{unanchored brane} to be an equivalence class of the quotient of the set of anchored branes by the action of $G \oplus \Z/2$: explicitly, it is given by data $(L,\iota,[P^\#])$ where $[P^\#]$ is a Pin structure on $L$, up to tensoring with $\lambda(TL)$.

We make a universal choice of perturbation data for the exact Fukaya category, which we assume to be invariant under the action of $H_1(\cG M) \oplus \Z/2$ on the objects labelling each boundary component of each disc.
Then for each pair of objects $(L_0^\#,L_1^\#)$, we have a finite set of non-degenerate Hamiltonian chords $y$ from $L_0$ to $L_1$. 
Let  $\tilde{y}: [0,1] \to \wt{\cG} M$ be a lift of $y$ such that $\tilde{y}(0) = \tilde{\iota}_0(y(0))$ and $\tilde{y}(1) = g\cdot\tilde{\iota}_1(y(1))$ for some element $g$ of the covering group (compare \cite[\S 3.3]{Sheridan:CY}). 
Let $k$ be the Maslov index of the corresponding path of Lagrangian subspaces, and define the $\G$-graded line
\[ o(L_0^\#,L_1^\#,y,k) := o(T_{y(0)} L_0^\#,  T_{y(1)}L_1^\#,k)|\Ts_{-g}\]
(using the Hamiltonian flow to identify $T_{y(0)}M$ with $T_{y(1)}M$). 
We will omit the objects $L_i^\#$ from the notation when convenient.
This line is concentrated in degree $\deg(y) := k-g$, which one easily sees to be independent of the choice of $\tilde{y}$.

Note that we have isomorphisms 
\begin{equation}
\label{eqn:orysig}
 o(y,k+2\jmu) \xrightarrow{s_\r^{2\jmu}} \Ts_{2\jmu}| o(T_{y(0)} L_0^\#,  T_{y(1)}L_1^\#,k)|\Ts_{-g} \xrightarrow{c_{2\jmu,-g}|\id} o(y,k)
\end{equation}
for any $\jmu$ (using the canonical trivialization of $\lambda(TL_1)^{\otimes 2\jmu}$ at the first step). 
By \eqref{eqn:sh2} these isomorphisms give us compatible identifications of all $o(y,k)$ with $k$ of the same parity, so we have a well-defined $\G$-graded line $o_\r(y,\sigma)$ for each $\sigma \in \Z/2$. 
Similarly we define $o_\l(y,\sigma)$, where we use the isomorphisms $s_\l^{2\jmu}$ instead of $s_\r^{2\jmu}$ in \eqref{eqn:orysig}.
However we observe that $o_\r(y)$ and $o_\l(y)$ are not canonically isomorphic, because the isomorphisms $s^{-2} \circ s_\r^2$ and $s_\l^2$ are different by \eqref{eqn:sh6}. 

\begin{defn}
We define $o_\r(y) := o_\r(y,\sigma)$, where $\sigma$ is the parity of $\deg(y)$. 
Similarly we define $o_\l(y)$, and we remark that these two lines are not canonically isomorphic.
\end{defn}

Now observe that we can also identify the lines $o(y,k)$ via the isomorphisms $s_\l^{4\jmu}$ and $s_\r^{4\jmu}$. 
Since $s^{-4\jmu} \circ s_\r^{4\jmu} = s_\l^{4\jmu}$ by \eqref{eqn:sh6}, these identifications are compatible, so we have a well-defined $\G$-graded line $o_{\l\r}(y,\eta)$ for each $\eta \in \Z/4$. 
There are obvious identifications
\[ o_\r(y,[\eta]\mbox{ mod 2}) \simeq o_{\l\r}(y,\eta) \simeq o_\l(y,[\eta]\mbox{ mod 2}).\]

\begin{defn}
Suppose we are given a morphism of grading data $p: \G \to \Z/4$. We then define $o_{\l\r}(y) := o_{\l\r}(y,p(\deg(y)))$. 
There are natural identifications
\[ o_\r(y) \simeq o_{\l\r}(y) \simeq o_\l(y).\] 
\end{defn}

The two choices, $\r$ and $\l$, will lead to two different versions of the exact Fukaya category, $\fuk_\r(M)$ and $\fuk_\l(M)$. 
A morphism of grading data $p:\G \to \Z/4$ will determine an isomorphism $\fuk_\r(M) \simeq \fuk_\l(M)$.
We define
\begin{equation} hom^\bullet_{\fuk_\r(M)}(L_0^\#,L_1^\#) := \bigoplus_y |o_\r(y)|_\Bbbk,\end{equation}
and similarly for $\fuk_\l(M)$. 
Here $\Bbbk$ is the coefficient field, and `$|\cdot|_\Bbbk$' denotes the $\Bbbk$-normalization (a $\G$-graded $\Bbbk$-line) or a $\G$-graded real line (see \cite[\S 12f]{Seidel:FCPLT}).

\begin{rmk}
\label{rmk:w2twist}
If we supposed our Lagrangians to be orientable and spin, one might expect from \cite[Remark 5.2.7(a)]{Wehrheim2015} that $\fuk_\l(M) \simeq \fuk_{\r,w_2(TM)}(M)$, where $\fuk_{\r,st}(M)$ denotes the Fukaya category twisted by the relative spin structure $st \in H^2(M;\Z/2)$ (see \cite{FO3}).
We understand this observation is due to Fukaya, we thank Mohammed Abouzaid and Jack Smith for drawing it to our attention.
\end{rmk}

The $A_\infty$ structure maps are defined by counting pseudoholomorphic disks as in Section \ref{subsec:exfuk}: we need to specify how each disc $(r,u)$ determines an isomorphism of orientation lines \eqref{eqn:oropeq}. 
We will write $\mathcal{R} := \mathcal{R}(\mathbf{L})$ and $\mathcal{M} := \mathcal{M}(\mathbf{y})$. 
Let $g_i := \deg(y_i)$, and let us choose orientation operators $D_{y_i}$ for $i=1,\ldots,k$ whose boundary conditions have Maslov index of the same sign as $\sigma(g_i)$: so we have a canonical isomorphism $o_\r(y_i) \simeq \lambda(D_{y_i}) \otimes \mathsf{Pin}(y_i)$.
Next, given a rigid disc $u$, we can glue the associated orientation operators $D_{y_i}$ onto the incoming strip-like ends of the operator $D_u$ to obtain an orientation operator $D_{y_0}$. 
Because $u$ is rigid, an index argument shows that 
\begin{equation}
\label{eqn:mugradings}
 g_0 = 2-s+\sum_{j=1}^k g_j 
\end{equation}
in $G$, and in particular in $\Z/2$ (see \cite[\S 3.3]{Sheridan:CY}\footnote{The corresponding formula in \cite{Sheridan:CY} had an error: it had $s-2$ instead of $2-s$.}). 
It follows that the glued boundary conditions for $D_{y_0}$ have Maslov index of the same sign as $\sigma(g_0)$ also: so we have a canonical isomorphism $o_\r(y_0) \simeq \lambda(D_{y_0}) \otimes \mathsf{Pin}(y_0)$.

A gluing theorem determines an isomorphism
\begin{equation}  \lambda(D_u) \otimes \lambda(D_{y_1}) \otimes \ldots \otimes \lambda(D_{y_k}) \simeq \lambda(D_{y_0}).\end{equation}
Because $(r,u)$ is parametrized regular, we have an isomorphism
\begin{equation}
\label{eqn:orM}
 \lambda(D_u) \otimes \lambda(T_r \mathcal{R}) \simeq \lambda(T_{(r,u)} \mathcal{M}).
\end{equation}
We have an obvious isomorphism
\begin{equation} \mathsf{Pin}(y_1) \otimes \ldots \otimes \mathsf{Pin}(y_k) \simeq \mathsf{Pin}(y_0).\end{equation}
Combining all of these isomorphisms gives us an isomorphism
\begin{equation}
 \lambda(T_r \mathcal{R})^* \otimes  o_\r(y_1) \otimes \ldots \otimes o_\r(y_k) \simeq \lambda(T_{(r,u)} \mathcal{M})^* \otimes o_\r(y_0).\end{equation}
Henceforth in this section, we will write `$|$' instead of `$\otimes$' and `$o_i$' instead of `$o_\r(y_i)$', to save space.

Now, let $\Ts _s,\Ts_t,\Ts _1,\ldots,\Ts _k$ be trivial $\G$-graded lines in degree $1$. 
We define an isomorphism
\begin{equation}
\label{eqn:orTR}
\Ts_s^*| \Ts_1|\ldots|\Ts_k|\Ts_t^* \simeq \lambda(T_r \mathcal{R})
\end{equation}
as follows. 
We regard $\mathcal{R}$ as the moduli space of points $\zeta_1 < \ldots <\zeta_k$ in $\R = \partial H$ (recall $H$ is the upper halfplane), modulo the action of $\mathsf{PSL}(2,\R)$ by translation and scaling.
We identify $\Ts _i$ with the tangent vector corresponding to translating $\zeta_i$ in the positive $\R$-direction, and we identify $\Ts _s$ and $\Ts_t$ with the tangent vectors to $\mathsf{PSL}(2,\R)$ corresponding to positively scaling and positively translating, respectively.

Combining the previously-defined isomorphisms with isomorphisms $o'_i \simeq o_i|\Ts_i^*$ and $o'_0 \simeq o_0 |\Ts_t^*$ (trivializing the lines) gives an isomorphism
\begin{equation}
\label{eqn:MorO}
 o'_1 | \ldots | o'_k \simeq \lambda(T_{(r,u)} \mathcal{M})^* | \Ts_s^*|o'_0 .
\end{equation}
In the case at hand, $(r,u)$ is an isolated point of $\mathcal{M}$, so $\lambda(T_{(r,u)} \mathcal{M}) \simeq \lambda(0)$ admits a canonical trivialization. 
Identifying $\Ts_s^* \simeq \Ts_\mu$, we obtain the isomorphism
\begin{equation} 
\label{eqn:MorO1}
o'_1 | \ldots | o'_k \simeq \Ts_\mu|o'_0,\end{equation}
which defines the contribution of the disc $(r,u)$ to the $A_\infty$ product $\mu^k$ via \eqref{eqn:oropeq} (in accordance with the convention of \eqref{eqn:hochtriv}). 

To show that the $A_\infty$ relations are satisfied, we consider a boundary point of the 1-dimensional moduli space of discs. 
Suppose a disc containing chords $y_{i+1},\ldots,y_{i+j}$ bubbles off at this point. 
We have a gluing map
\begin{equation} (0,\epsilon) \times \mathcal{R}(L_0,\ldots,L_i,L_{i+j},\ldots,L_k) \times \mathcal{R}(L_i,\ldots,L_{i+j}) \to \mathcal{R}(L_0,\ldots,L_k),\end{equation}
which we abbreviate
\begin{equation} (0,\epsilon) \times \mathcal{R}_2 \times \mathcal{R}_1 \to \mathcal{R}_{12}.\end{equation}
The gluing map is locally an isomorphism, so induces an isomorphism of orientation lines
\begin{equation}
\label{eqn:glueR}
 \Ts_n|\lambda(T \mathcal{R}_2) | \lambda(T \mathcal{R}_1) \simeq \lambda(T\mathcal{R}_{12}),
\end{equation}
where $\Ts _n$ is the tangent space to $(0,\epsilon)$, the space of gluing parameters. 
It is trivialized by choosing $1$ to correspond to the direction away from the boundary of $\mathcal{R}_{12}$ (i.e., in the positive direction on the interval).

Recall that we think of $\mathcal{R}$ as parametrizing order tuples of real points, modulo scaling and translation; then the gluing map shrinks the configuration $\mathcal{R}_2$ and inserts it into $\mathcal{R}_1$ at the $(i+1)$st position, which we will denote by $g$ (for `gluing') in what follows. 
Positively translating the points in $\mathcal{R}_2$ is equivalent to positively translating the gluing position; and positively scaling the points in $\mathcal{R}_2$ corresponds to moving away from the boundary of $\mathcal{R}_{12}$. 
Therefore, under the isomorphisms \eqref{eqn:orTR}, the gluing isomorphism gives an isomorphism
\begin{multline}
\label{eqn:gluecoords}
\Ts _n |\left(\Ts _{s_1}^* | \Ts _1 | \ldots | \Ts _i | \Ts _g | \Ts _{i+j+1} | \ldots | \Ts _k | \Ts^* _{t_1} \right) | \left(\Ts^* _{s_2} | \Ts _{i+1} | \ldots | \Ts _{i+j} | \Ts^* _{t_2} \right) \\
\simeq  (\Ts^* _{s} | \Ts _1 | \ldots | \Ts _i | \Ts _{i+1} | \ldots | \Ts _{i+j}  | \Ts _{i+j+1} | \ldots | \Ts _k | \Ts^* _t),
\end{multline} 
via the map which identifies
\begin{eqnarray*}
\Ts _\ell & \simeq & \Ts _{\ell} \text{ for all $\ell \neq g$} \\
\Ts _{t_2}^* | \Ts _g & \simeq & \R[0], \\
\Ts _{s_1} & \simeq & \Ts _s, \\
\Ts^* _{s_2}|\Ts_n & \simeq & \R[0], \\
\Ts^* _{t_1} & \simeq & \Ts^* _t.
\end{eqnarray*}

Now suppose that 
\begin{eqnarray*}
(r_1,u_1) \in \mathcal{M}_1 & :=& \mathcal{M}(y_{i+1},\ldots,y_j), \\
(r_2,u_2) \in \mathcal{M}_2 & := & \mathcal{M}(y_1,\ldots,y_i,y_g,y_{i+j+1},\ldots,y_k)
\end{eqnarray*}
 are pseudoholomorphic discs constituting a boundary point of the one-dimensional component of the moduli space $\mathcal{M}_{12} := \mathcal{M}(y_1,\ldots,y_k)$. 
Near such a boundary point of $\mathcal{M}_{12}$ we have an isomorphism $\Ts _n \simeq \mathcal{M}_{12}$, where $\Ts _n$ is as before (the tangent space to the space of choices of gluing parameter). 
Via \eqref{eqn:MorO}, this gives us an identification
\begin{equation}
\label{eqn:boundor}
 o'_1 | \ldots | o'_k \simeq \lambda(T \mathcal{M}_{12})^*|\Ts_s^* | o'_0  \simeq  \Ts^* _n|\Ts_s^* | o'_0 .
\end{equation}
Each connected component of $\mathcal{M}_{12}$ contributes two isomorphisms \eqref{eqn:boundor} with opposite sign, because the oriented sum of boundary points of a compact 1-manifold is $0$. 

On the other hand, the corresponding term in the composition of $A_\infty$ maps defines an isomorphism
\begin{equation}
\label{eqn:mumu}
 o'_1 | \ldots | o'_k \simeq \Ts_{\mu_2}|\Ts_{\mu_1}|o'_0.
\end{equation}
We claim that the isomorphisms \eqref{eqn:mumu} coincide with the corresponding isomorphisms \eqref{eqn:boundor} after identifying $\Ts^*_n \simeq \Ts_{\mu_1}$ and $\Ts^*_s \simeq \Ts_{\mu_2}$. 
It follows that all terms in the sum of the isomorphisms \eqref{eqn:mumu} cancel, so the $A_\infty$ relations are satisfied.

\begin{rmk}
More precisely, the isomorphisms \eqref{eqn:mumu} and \eqref{eqn:boundor} coincide up to a constant overall sign which depends on the orders in which we apply certain trivializations of trivial lines (in fact this was already true in \eqref{eqn:gluecoords}: there is an overall sign involved in our choice of identification $\Ts^*_{t_2}|\Ts_g \simeq \R[0]$, for example).
Since multiplication by a constant overall sign does not change the fact that the sum is zero, we will be lax about specifying such signs.
\end{rmk}

We now prove the claim. 
By gluing $(r_1,u_1)$ to $(r_2,u_2)$ we get $(r_1 \#_\rho r_2,u_1 \#_\rho u_2) \in \mathcal{M}_{12}$ for $\rho \in (0,\epsilon)$, which gives an isomorphism 
\begin{equation}
\label{eqn:glueM} \Ts_n|\lambda(T\mathcal{M}_1) | \lambda(T\mathcal{M}_2) \to \lambda(T\mathcal{M}_{12}).
\end{equation}
When $\mathcal{M}_1$ and $\mathcal{M}_2$ are isolated points, this gives the identification $\Ts _n \simeq \lambda(T\mathcal{M}_{12})$ mentioned above.

This isomorphism \eqref{eqn:glueM} fits into a commutative diagram with the isomorphisms \eqref{eqn:orM}:
\begin{equation}
\label{eqn:glueMR}
\xymatrix{ \Ts_n|\lambda(T\mathcal{R}_1) | \lambda(D_{u_1}) | \lambda(T\mathcal{R}_2) | \lambda(D_{u_2}) \ar[r] \ar[d] &  \Ts _n |\lambda(T \mathcal{M}_1) | \lambda(T\mathcal{M}_2) \ar[d]  \\
\lambda(T\mathcal{R}_{12}) | \lambda(D_{u_1 \# u_2}) \ar[r]& \lambda(T\mathcal{M}_{12}) ,}
\end{equation}
where the left vertical map is the tensor product of the gluing isomorphism 
\begin{equation} \lambda(D_{u_1}) | \lambda(D_{u_2}) \simeq \lambda(D_{u_1 \# u_2})\end{equation}
with the gluing isomorphism \eqref{eqn:glueR}.

We have a commutative diagram of gluing maps:
\begin{equation}
\label{eqn:glueo} \xymatrix{ \lambda(D_{u_1}) | \lambda(D_{u_2}) | o_1 | \ldots | o_k  \ar[r]\ar[d] & \lambda(D_{u_2}) | o_1 | \ldots | o_g | \ldots | o_k  \ar[r] & o_0 \ar[d] \\
\lambda(D_{u_1 \# u_2}) | o_1| \ldots | o_k  \ar[rr] && o_0 .}
\end{equation}

Tensoring \eqref{eqn:glueMR} with \eqref{eqn:glueR} then combining with \eqref{eqn:glueo}, and using the canonical trivializations of $\lambda(T\mathcal{M}_1)$ and $\lambda(T\mathcal{M}_2)$, we obtain a commutative diagram
\begin{equation}
\label{eqn:gluei}
\xymatrix{\lambda(T\mathcal{R}_1)^* | \lambda(T\mathcal{R}_2)^* | o_1 | \ldots | o_k  \ar[r]\ar[d] &   o_0\ar[d] \\
 \lambda(T \mathcal{M}_{12})|\lambda(T\mathcal{R}_{12})^* | o_1 | \ldots | o_k  \ar[r] &   o_0  .
}
\end{equation}

Now we need to trivialize the $\lambda(T\mathcal{R}_i)$. 
Let us abbreviate $o\Ts^*_i := o_i|\Ts^*_i$ in what follows.
We have a commutative diagram
\begin{equation}
\xymatrix{ \Ts_n|o'_1 | \ldots | o'_k \ar[r] \ar[d] & \Ts_{s_1} | \Ts_{s_2} | o\Ts^*_1 | \ldots | o\Ts^*_{i+j} | \Ts^*_g | o\Ts^*_{i+j+1} | \ldots | o\Ts^*_k | \Ts_{t_2} \ar[d] \\
\Ts_n|o'_1 | \ldots | o'_k \ar[r] & \Ts_n | \Ts_s| o \Ts^*_1| \ldots|o \Ts^*_k | \Ts_t,}
\end{equation}
such that the isomorphism on the right fits into the commutative diagram
\begin{equation}
\xymatrix{ \Ts_{s_1} | \Ts_{s_2} | o \Ts^*_1 | \ldots | o\Ts^*_k | \Ts_{t_2} \ar[r] \ar[d] &  \lambda(T\mathcal{R}_1)^* | \lambda(T\mathcal{R}_2)^* | o_1 | \ldots | o_k  \ar[d]\\
\Ts_n | \Ts_s| o\Ts^*_1| \ldots|o\Ts^*_k | \Ts_t \ar[r] &  \lambda(T \mathcal{M}_{12})|\lambda(T\mathcal{R}_{12})^*  | o_1 | \ldots | o_k }
\end{equation}
with the isomorphism on the left of \eqref{eqn:gluei}, by \eqref{eqn:gluecoords}.

Composing our commutative diagrams horizontally, and tensoring with $\Ts_n^*$, we obtain a commutative diagram
\begin{equation} \xymatrixcolsep{5pc}\xymatrix{o'_1|\ldots|o'_k \ar[r]\ar[d] & \Ts_n^*|o_0 \ar[d] \ar[r] & \Ts_{\mu_1}|\Ts_{\mu_2}|o'_0 \ar[d] \\
o'_1|\ldots|o'_k \ar[r] & \Ts_n^*|o_0  \ar[r] & \Ts_n^*|\Ts_s^*|o'_0,} \end{equation}
in which the top isomorphism coincides with \eqref{eqn:mumu} and the bottom isomorphism coincides with \eqref{eqn:boundor}. 
This completes the proof of the claim.

\subsection{Shifts in the exact Fukaya category}
\label{subsec:fukshifts}

We define a set of shift functors for the two versions of the Fukaya category $\fuk_\r(M)$, $\fuk_\l(M)$ that we have defined. 
Recall that the group $G \oplus \Z/2$ acts on the set of objects of the Fukaya category: in fact it is easy to see that the group acts strictly on (either version of) the Fukaya category. 
On the level of objects, we define the shift of an object to be 
\begin{equation}
\label{eqn:Lshift}
L^\#[g] := (g \oplus \sigma(g)) \cdot L^\#.
\end{equation} 

Observe that, if a chord $y$ lifts to a path of Maslov index $k$ from $\tilde{\iota}(y(0))$ to $g' \cdot \tilde{\iota}(y(1))$, then it also lifts to a path of Maslov index $k$ from $g \cdot \tilde{\iota}(y(0))$ to $g' \cdot g \cdot \tilde{\iota}(y(1))$. 
Thus we obtain an isomorphism
\[o((g \oplus 0)\cdot L_0^\#,(g \oplus 0) \cdot L_1^\#,y,k) \simeq o(L_0^\#,L_1^\#,y,k).\]
Combining with the isomorphism $s^{\sigma(g)}$ of \eqref{eqn:sj} (trivializing $\lambda(\Lambda_i)^{\otimes 2\jmu}$), we obtain an isomorphism
\begin{equation}
o(L_0^\#[g],L_1^\#[g],y,k) \simeq o(L_0^\#,L_1^\#,y,k).
\end{equation} 
It follows by \eqref{eqn:sh4} that these isomorphisms commute with $s_\r^2$, hence give well-defined isomorphisms
\begin{equation}
o_\r(L_0^\#[g],L_1^\#[g],y) \simeq o_\r(L_0^\#,L_1^\#,y),
\end{equation}
whose direct sum is the shift map
\begin{equation}
s^g: hom^\bullet_{\fuk_\r(M)}(L^\#_0[g],L^\#_1[g]) \to hom^\bullet_{\fuk_\r(M)}(L^\#_0,L^\#_1).
\end{equation}

Next we define a set of right-shift maps on $\fuk_\r(M)$. 
For each Hamiltonian chord $y$ we have isomorphisms
\begin{equation}
\label{eqn:srshiftsL}
 s_\r^\jmu: o(T_{y(0)} L_0^\#,  T_{y(1)}L_1^\# \otimes \lambda(TL_1)^{\otimes i},k+i) \to \Ts_i|o(T_{y(0)} L_0^\#,  T_{y(1)}L_1^\#,k).
\end{equation}
If $y$ lifts to a path of Maslov index $k$ from $\tilde{\iota}(y(0)$ to $g \cdot \tilde{\iota}(y(1))$, then it also lifts to a path of Maslov index $k+i$ from $\tilde{\iota}(y(0))$ to $(g+i) \cdot \tilde{\iota}(y(1))$. 
It follows that \eqref{eqn:srshiftsL} defines an isomorphism
\begin{equation}
o(L_0^\#,L_1^\#[i],y,k+i) \simeq \Ts_i|o(L_0^\#,L_1^\#,y,k).
\end{equation}   
Similarly, since we can identify $g' \cdot \tilde{\iota}(y(1)) = (g'-g)\cdot g \cdot \tilde{\iota}(1)$, the isomorphism $c_{g,-g'}$ gives an isomorphism
\begin{equation}
o(L_0^\#,L_1^\#[g],y,k) \simeq \Ts_g|o(L_0^\#,L_1^\#,y,k).
\end{equation}
It follows by \eqref{eqn:sh2} that these isomorphisms commute with $s_\r^2$, hence give well-defined isomorphisms
\[ o_\r(L_0^\#,L_1^\#[g],y) \simeq \Ts_g|o_\r(L_0^\#,L_1^\#,y),\]
whose direct sum is the right-shift map
\begin{equation} s_\r^g: hom^\bullet_{\fuk_\r(M)}(L^\#_0,L^\#_1[g]) \to \Ts_g|hom^\bullet_{\fuk_\r(M)}(L^\#_0,L^\#_1).\end{equation}
Note that, because $s_\l^\jmu$ does not commute with $s_\r^2$, it is \emph{not} possible to define left-shift maps by the same methods. 

To prove that $s^g$ and $s_\r^g$ define a set of rightwards shift functors on the pre-category level, it suffices to check they satisfy the hypotheses of Lemma \ref{lem:shiftsrgsg}. 
These follow immediately from Lemma \ref{lem:shiftsigns}. 
To prove that they define a set of rightwards shift functors on $\fuk_\r(M)$, we must check that $\mu \in CC^1_\r(\fuk_\r(M))'$. 
This is a consequence of associativity of gluing, as in the proof of the $A_\infty$ relations in the previous section. 

Similarly, we define a set of left-shift maps on $\fuk_\l(M)$.

Now, suppose that $M$ comes endowed with a morphism of grading data, $p: \G \to \Z/4$. 
Then we define yet another version of the Fukaya category, $\fuk_{\l\r}(M)$, whose morphism spaces are the direct sum of the orientation lines $o_{\l\r}(y)$.
We observe that \eqref{eqn:mugradings} holds modulo $4$: so when we glue orientation operators $D_{y_1},\ldots,D_{y_k}$ onto $D_u$, we obtain an orientation operator $D_{y_0}$ of the correct index modulo $4$ to define an element of $o_{\l\r}(y_0)$.

$\fuk_{\l\r}(M)$ is isomorphic to $\fuk_\r(M)$, and thereby comes equipped with a set of rightwards shift functors; similarly it is isomorphic to $\fuk_\l(M)$, and thereby comes equipped with a set of leftwards shift functors. 
These leftwards and rightwards shift functors are compatible with respect to $p$, by Lemma \ref{lem:shiftsigns}.

\subsection{Autoequivalences of the exact Fukaya category}

We define $\Aut(M)$ to be the group of symplectomorphisms $\phi: M \to M$. 
$\Aut(M)$ also acts on the grading datum $\G$: namely, $\phi$ acts by $\phi_*: H_1(\cG M) \to H_1(\cG M)$. 
$\Aut(M)$ acts on unanchored Lagrangian branes in the obvious way.
 
We also have the group of \emph{graded symplectomorphisms} \cite{Seidel1999}. 
A graded symplectomorphism is a pair $(\phi,\tilde{\phi})$, where $\phi \in \Aut(M)$ and $\tilde{\phi}: \wt{\cG}M \to \wt{\cG}M$ is a lift of the map $\phi_*: \cG M \to \cG M$. 
The group of graded symplectomorphisms fits into a short exact sequence
\begin{equation} 0 \to G \to \Aut^{gr}(M) \to \Aut(M) \to 0.\end{equation}
Note that this sequence is exact on the right because any symplectomorphism $\phi$ admits a lift $\tilde{\phi}$ to the universal abelian cover of its Lagrangian Grassmannian: this is in contrast to the case for other Maslov covers, where the analogous sequence is not exact on the right, see e.g. \cite[Lemma 2.4]{Seidel1999}.

$\Aut^{gr}(M)$ acts on the objects of $\fuk(M)$: $(\phi,\tilde{\phi})$ sends
\begin{equation} (L,\iota,\tilde{\iota},P^\#) \mapsto (L,\phi \circ \iota,\tilde{\phi} \circ \tilde{\iota},P^\#).\end{equation}
The action commutes with the action of $G \oplus \Z/2$ defined earlier.

There is a natural map $\Aut^{gr}(M) \to \Z/2$, defined as follows: the tautological vector bundle over $\wt{\cG}M$ is orientable because its first Stiefel-Whitney class vanishes, so choose an orientation for it. 
Any graded symplectomorphism induces an isomorphism of the tautological vector bundle covering $\tilde{\phi}$, via $\phi_*$: if this isomorphism respects orientation, we call $(\phi,\tilde{\phi})$ \emph{even}, otherwise we call it \emph{odd}. 
We denote the group of even graded symplectomorphisms by $\Aut^{gr}_{ev}(M)$. 
It fits into a short exact sequence
\begin{equation} 0 \to G_{ev} \to \Aut^{gr}_{ev}(M) \to \Aut(M) \to 0.\end{equation}

\begin{lem}
\label{lem:AutFukr}
Let $\EuA \subset \fuk_\r(M)$ be a full subcategory which is closed under shifts. 
Let $\Gamma$ be a group and $\Gamma \to \Aut(M)$ a group homomorphism, so that the action of $\Gamma$ on the set of unanchored branes underlying $Ob(\EuA)$ is free. 
Then for appropriate choice of perturbation data in $\fuk_\r(M)$, there is an induced action of $\Gamma$ on $\EuA$ up to shifts:
\begin{equation} \Gamma \to \mathsf{StrAut}_\r(\EuA)/G_{ev}.\end{equation}
\end{lem}
\begin{proof}
Any $\gamma \in \Gamma$ admits a lift $\tilde{\gamma} \in \Aut^{gr}_{ev}(M)$, well-defined up to shift by $G_{ev}$.  
The lift $\tilde{\gamma}$ acts strictly on $\fuk_\r(M)$, respecting shifts and right-shift maps: this simply expresses the naturality of the construction of $\fuk_\r(M)$, the shifts and right-shift maps. 
The action of $G_{ev} \subset \Aut^{gr}_{ev}(M)$ coincides with the action of $G_{ev} \subset G$ by shifts (note that the same is not true of the action of $G \subset \Aut^{gr}(M)$: it differs from the action of shifts by the twisting of the Pin structure, $P^\# \mapsto P^\# \otimes \lambda(TL)^{\otimes \sigma(g)}$). 
Therefore $\gamma$ induces a well-defined element of $\mathsf{StrAut}_\r(\fuk_\r(M))/G_{ev}$.

The assumption that $\Gamma$ acts freely on unanchored branes allows us to bypass issues with equivariant transversality (see \cite[\S 8b]{Seidel:HMSquartic}). 
\end{proof}

We also have the following analogue of Lemma \ref{lem:AutFukr}.

\begin{lem}
Let $p: \G \to \Z/4$ be a morphism of grading data.
Let $\EuA \subset \fuk_{\l\r}(M)$ be a full subcategory which is closed under shifts. 
Let $\Gamma$ be a group and $\Gamma \to \Aut(M)$ a group homomorphism, so that the corresponding action of $\Gamma$ on $\G$ respects $p$, and the action of $\Gamma$ on the set of unanchored branes underlying $Ob(\EuA)$ is free. 
Then for appropriate choice of perturbation data in $\fuk_{\l\r}(M)$, there is an induced action of $\Gamma$ on $\EuA$ up to shifts:
\begin{equation} \Gamma \to \mathsf{StrAut}_{\l\r}(\EuA)/G_{ev}.\end{equation}
\end{lem}

\subsection{The closed--open map}
\label{subsec:coisigns}

In this section we recall the definition of the map
\begin{equation} \EuC \EuO \circ i: H^\bullet(M) \to \HH^\bullet(\fuk_\r(M))\end{equation}
from Section \ref{subsec:coi} (more precisely, we give a definition of this map that is well-adapted to our purposes).

Let $\mathbf{L} = (L_0,\ldots,L_k)$ be a tuple of Lagrangians, with an associated set of generators $\mathbf{y} = (y_0,\ldots,y_k)$. 
We consider the moduli space $\mathcal{M}_1(\mathbf{y})$ of pairs $(r,u)$, $r \in \mathcal{R}_1(\mathbf{L})$ is an element of the moduli space of discs with boundary components labelled by $\mathbf{L}$ and an internal marked point $q$ (as in Section \ref{subsec:co}), and $u: S_r \to M$ satisfies the corresponding pseudoholomorphic curve equation, which coincides with \eqref{eqn:defpseudo}, except that we don't impose a condition at the internal marked point $q$.

As in the previous section, there is a natural isomorphism
\begin{equation} \lambda(T_r\mathcal{R}_1)^*  | o_1|\ldots|o_k\simeq \lambda(T_{(r,u)}\mathcal{M})^* | o_0\end{equation}
obtained by gluing. 
We now trivialize $\lambda(T_r\mathcal{R}_1)$: there is a natural map $\ff:\mathcal{R}_1 \to \mathcal{R}$ forgetting the internal marked point $q$, and with fibre $S_r$ over $r$. 
In particular, we have a natural isomorphism $\lambda(T_r\mathcal{R}_1) \simeq \lambda(T_{\ff(r)} \mathcal{R}) \otimes \lambda(T_q S_r)$. 
Observe that $T_qS_r$ has a natural complex structure, so we have a canonical trivialization $\lambda(T_q S_r) \simeq \Ts_2$ given by the complex orientation. 

By analogy with \eqref{eqn:MorO1}, we can trivialize $\lambda(T_{\ff(r)} \mathcal{R})$ and identify $\Ts_2|\Ts_s^* \simeq \Ts_{\EuC\EuO}$ to obtain an isomorphism
\begin{equation}
\label{eqn:COcont}
o_1'|\ldots|o'_k  \simeq \Ts_{\EuC\EuO}| \lambda(T_{(r,u)}\mathcal{M}_1(\mathbf{y}))|o'_0
\end{equation}

Observe that we have the evaluation map $\mathsf{ev}_q: \mathcal{M}_1(\mathbf{y}) \to M$.
Now let $\alpha \in H^\bullet(M)$, and let us represent it by a locally finite pseudocycle $f:V \to M$, together with an orientation of its normal bundle. 
We define $\EuC \EuO \circ i(\alpha)$ by counting (rigid) pairs $(w,v)$ where $w = (r,u) \in \mathcal{M}_1(\mathbf{y})$ and $v \in V$, such that $\mathsf{ev}_q(w) = f(v)$. 
At each such point, $(\mathsf{ev}_q)_*$ defines an isomorphism between $T_{(r,u)}\mathcal{M}_1(\mathbf{y})$ and the normal bundle to the pseudocycle $f$, which is equipped with an orientation. 
Thus we have an isomorphism 
\[\lambda(T_{(r,u)}\mathcal{M}_1(\mathbf{y})) \simeq \Ts_{|\alpha|},\]
which when combined with \eqref{eqn:COcont} and an isomorphism $\Ts_{\EuC\EuO}|\Ts_{|\alpha|} \simeq \Ts_{\EuC\EuO(\alpha)}$ gives the isomorphism
\begin{equation} o'_1|\ldots|o'_k \simeq \Ts_{\EuC\EuO(\alpha)}|o'_0,\end{equation}
which we define to be the contribution of $(r,u,v)$ to $\EuC \EuO \circ i(\alpha)$.

The Hochschild cochain thus defined has degree $|\alpha|$: since $\Ts_{\EuC\EuO}$ is in degree $1$, $\EuC\EuO \circ i(\alpha) \in CC^{|\alpha|-1}(\fuk_\r(M))' = CC^{|\alpha|}(\fuk_\r(M))$. It is in fact a Hochschild cocycle: the argument up to sign is standard, and the sign computation is parallel to the proof of the $A_\infty$ relations in \S \ref{subsec:ainfsign}. 
Furthermore, $\EuC\EuO \circ i(\alpha)$ respects the shift and right-shift maps: the sign computation is parallel to the proof that the $A_\infty$ structure maps respect shifts and right-shifts.
So in fact we have a Hochschild cocycle
\begin{equation} \EuC \EuO \circ i(\alpha) \in CC^\bullet_\r(\fuk_\r(M)).\end{equation} 
The Hochschild cohomology class is in fact independent of the choice of cocycle up to cobordism, by another standard argument: so we have defined a map
\begin{equation}\EuC\EuO \circ i: H^\bullet(M) \to \HH^\bullet(\fuk_\r(M)).\end{equation}

Now let us consider the situation of Lemma \ref{lem:AutFukr}: we have a map $\Gamma \to \mathsf{StrAut}_\r(\EuA)/G_{ev}$, and hence an action of $\Gamma$ on $CC^\bullet_\r(\EuA)$ from \eqref{eqn:autracts}. 
We also have an action of $\Gamma$ on $H^\bullet(M)$: $\gamma$ acts by $(\gamma^{-1})^*$.
By an analogous argument to that given in the proof of Lemma \ref{lem:AutFukr}, we have
\begin{equation} \EuC\EuO \circ i(\gamma \cdot \alpha) = \gamma \cdot \EuC\EuO \circ i(\alpha)\end{equation}
for appropriate choice of perturbation data: so $\EuC \EuO \circ i$ is strictly $\Gamma$-equivariant.

We can make analogous constructions in the `$\l$' and `$\l\r$' cases.

\subsection{The opposite symplectic manifold and the opposite Fukaya category}
\label{subsec:oppfuk}

In this section we recall how the Fukaya category changes when one reverses the sign of the symplectic form. 
Let $M:=(M,\omega = d\alpha)$ be a Liouville domain, and $\overline{M} := (M,-\omega = d(-\alpha))$ the `opposite' Liouville domain. 
Observe that we have a natural isomorphism $\cG M \simeq \cG \overline{M} $; this induces an identification of grading data, $\G(M) \simeq \G(\overline{M})$, via the isomorphism
\begin{equation} H_1(\cG M) \xrightarrow{-\id} H_1(\cG M) \xrightarrow{\simeq} H_1(\cG\overline{M}).\end{equation}
We will be implicitly using this identification in the rest of this section.

\begin{lem}
\label{lem:oppfuk}
For appropriate choice of perturbation data, there exist strict isomorphisms 
\begin{align} 
c^M_\r: \fuk_\r(M) &\to \fuk_\l(\overline{M})^{op} \text{ commuting with rightwards shift functors}\\
c^M_\l: \fuk_\l(M) & \to \fuk_\r(\overline{M})^{op}  \text{ commuting with leftwards shift functors.}
\end{align}
We have 
\begin{equation}\left(c_\l^{\overline{M}}\right)^{op} \circ c_\r^M = \id.\end{equation}

If $M$ comes equipped with a morphism of grading data $p: \G \to \Z/4$, then we have a strict isomorphism
\begin{align}
c^M_{\l\r}: \fuk_{\l\r}(M) &\to \fuk_{\l\r}(\overline{M})^{op} \text{ commuting with rightwards and leftwards shift functors,}
\end{align}
and 
\begin{equation} \left(c_{\l\r}^{\overline{M}}\right)^{op} \circ c_{\l\r}^M = \id.\end{equation}
\end{lem}
\begin{proof}
Let $V=(V,\Omega,J)$ be a symplectic vector space equipped with a compatible complex structure, and $\overline{V}:=(V,-\Omega,-J)$. 
Recall that, given transverse Lagrangian subspaces equipped with Pin structures $\Lambda_0^\#, \Lambda_1^\#$, we defined the orientation operators
\begin{equation} o(\Lambda_0^\#,\Lambda_1^\#,k) := \lambda(D_{H,\rho}) \otimes \mathsf{Pin}(\rho)\end{equation}
for all $k \in \Z$, by choosing a path $\rho: [0,1] \to \cG V$ from $\Lambda_0$ to $\Lambda_1$ with Maslov index $k$, and using it as boundary conditions for a Cauchy-Riemann operator $D_{H,\rho}$ on the upper half-plane $H$. 

We denote by $\overline{\Lambda}_0^\#,\overline{\Lambda}_1^\#$ the same subspaces of $V$, with the same Pin structures, but now regarded as Lagrangian subspaces of $\overline{V}$. 
For any path $\rho$, we define the opposite path $\overline{\rho}(s) := \rho(1-s)$, which has the same Maslov index.
We can obviously arrange for there to be an isomorphism of Cauchy--Riemann operators
\begin{eqnarray*}
D_{H,\rho} & \to & D_{H,\overline{\rho}} \\
u & \mapsto & \overline{u},
\end{eqnarray*}
where $\overline{u}$ denotes conjugation of the domain:
\begin{equation} \overline{u}(s+it) := u(-s+it).\end{equation}
There is also an obvious isomorphism,
\begin{equation} \mathsf{Pin}(\rho) \simeq \mathsf{Pin}(\overline{\rho}).\end{equation}
Combining these, we obtain an isomorphism
\begin{equation}
\label{eqn:oppors}
\conju:  o(\Lambda_0^\#,\Lambda_1^\#,k) \to o(\overline{\Lambda}_1^\#,\overline{\Lambda}_0^\#,k).
\end{equation}

Because conjugation of the domain reverses the boundary orientation, we see that
\begin{eqnarray*}
 \conju \circ s_\r &=& s_\l \circ \conju, \\
 \conju \circ s_\l &=& s_\r \circ \conju, \quad \text{and} \\
 \conju \circ s &=& s \circ \conju . 
\end{eqnarray*}
As a consequence, we have a natural isomorphism
\begin{equation}
\label{eqn:opporssign}
o_\r(\Lambda_0^\#,\Lambda_1^\#,\sigma) \simeq o_\l(\overline{\Lambda}_1^\#,\overline{\Lambda}_0^\#,\sigma),
\end{equation}
and vice-versa.

Now we define the isomorphism $c_\r^M:\fuk_\r(M) \to \fuk_\l(\overline{M})^{op}$. 
On the level of objects, the identification is almost tautological, as $\cG M \simeq \cG \overline{M}$. 
On the level of morphisms, we have the identification
\begin{eqnarray*}
c_\r^M:hom^\bullet_{\fuk_\r(M)}(L_0^\#,L_1^\#) & \to & hom^\bullet_{\fuk_\l(\overline{M})}(L_1^\#,L_0^\#), \text{ sending}\\
o_\r(y) & \overset{- \conju}{\longrightarrow} & o_\l(\overline{y}),
\end{eqnarray*} 
where $\overline{y}(s) :=y(1-s)$ is the reverse Hamiltonian chord to $y$. 
It is already clear that $c_\r^M$ respects shifts and right-shift maps, and that if we define $c_\l^{\overline{M}}$ in the analogous way, then
\begin{equation} \left(c_\l^{\overline{M}}\right)^{op} \circ c_\r^M = \id\end{equation}
because $\conju \circ \conju = \id$.

Now let $\mathcal{M}(\mathbf{y})$ be the moduli space of pseudoholomorphic discs contributing to an $A_\infty$ structure map in $\fuk_\r(M)$. 
By choosing the perturbation data in $\overline{M}$ to preserve the Hamiltonian part and reverse the almost-complex structure part, we ensure that there is an isomorphism
\begin{eqnarray*}
\mathcal{M}(y_0,y_1\ldots,y_k) & \to & \mathcal{M}(\overline{y}_0,\overline{y}_k,\ldots,\overline{y}_1), \\
(r,u) & \mapsto & (\overline{r},\overline{u}).
\end{eqnarray*}

It is clear that the following diagram commutes:
\begin{equation} \xymatrix{ \lambda(T_r\mathcal{R})^*|o_{\r,1}|\ldots|o_{\r,k} \ar[r] \ar[d] & \lambda(T_{\overline{r}}\mathcal{R})^*  | \overline{o}_{\l,k}|\ldots|\overline{o}_{\l,1} \ar[d] \\
\lambda(T_{(r,u)}\mathcal{M}(\mathbf{y}))^* | o_{\r,0} \ar[r] & \lambda(T_{(\overline{r},\overline{u})}\mathcal{M}(\overline{\mathbf{y}}))^* | \overline{o}_{\l,0}.}\end{equation}
Furthermore, the following diagram commutes:
\begin{equation} \xymatrix{\Ts_s^*|\Ts_1|\ldots|\Ts_k|\Ts^*_t \ar[r] \ar[d] & \Ts_s^*|\Ts_k|\ldots|\Ts_1|\Ts_t^* \ar[d]\\
\lambda(T_r\mathcal{R}(L_0,\ldots,L_k)) \ar[r] & \lambda(T_{\overline{r}}\mathcal{R}(L_k,\ldots,L_0)),} \end{equation}
where the top isomorphism sends
\begin{eqnarray}
\label{eqn:Rswap}
\Ts_s & \overset{\id}{\to} & \Ts_s, \\
\label{eqn:Rswap2}\Ts_t & \overset{-\id}{\to} & \Ts_t, \\
\label{eqn:Rswap3}\Ts_i & \overset{-\id}{\to} & \Ts_i \text{ for all $i=1,\ldots,k$.}
\end{eqnarray}
This follows easily from examining the action of complex conjugation $s+it \mapsto -s+it$ on the moduli space $\mathcal{R}$.

As a consequence, the diagram
\begin{equation} \xymatrix{ o'_{\r,1}|\ldots|o'_{\r,k} \ar[r]^{\conju^{\otimes k}} \ar[d]^{\mu^k} & \overline{o}'_{\l,k}|\ldots|\overline{o}_{\l,1} \ar[d]^{\mu^k} \\
\Ts_\mu|o'_{\r,0} \ar[r]^\conju & \Ts_\mu|\overline{o}'_{\l,0}} 
\end{equation}
commutes up to the sign $(-1)^{k+1}$ arising from the difference in trivializations in $\Ts_t$ from \eqref{eqn:Rswap2} and the $\Ts_i$ from \eqref{eqn:Rswap3}. 
In particular, recalling that $c_\r^M = -\conju$, we have
\[ c_\r^M(\mu^k(a_1,\ldots,a_k))  =  \mu_{op}^k(c_\r^M(a_1),\ldots,c_\r^M(a_k))\]
so indeed $c_\r^M$ is a strict isomorphism of $A_\infty$ categories.
\end{proof}

\begin{cor}
Suppose $M$ is equipped with a morphism of grading data $\G \to \Z/4$.
Then Lemma \ref{lem:oppfuk} gives a strict isomorphism
\begin{equation} c^M_{\l\r}: \fuk_{\l\r}(M) \to \fuk_{\l\r}(\overline{M})^{op}\end{equation}
which respects shifts and left- and right-shift maps.
\end{cor}
\begin{proof}
Follows from the fact that the isomorphism 
\begin{equation}c^M_\r:hom^\bullet_{\fuk_\r(M)}(L_0,L_1) \to hom^\bullet_{\fuk_\l(\overline{M})}(L_1,L_0)\end{equation}
respects $\Z/4$-gradings.
\end{proof}

\subsection{Signed group actions on symplectic manifolds and the Fukaya category}

Recall the notion of a \emph{signed group} $(\Gamma,\sigma)$ from Definition \ref{defn:sgngrp}: it is a group $\Gamma$ together with a homomorphism $\sigma: \Gamma \to \Z/2$. 
A morphism of signed groups is a group homomorphism respecting the map to $\Z/2$.

We define $\Aut^\sigma(M)$ to be the signed group of `signed symplectomorphisms': this consists of the even elements, which are symplectic automorphisms of $M$ (elements of $\Aut(M)$), and odd elements, which are anti-symplectic automorphisms of $M$ (symplectomorphisms $\phi: M \to \overline{M}$). 
There is an exact sequence
\begin{equation} 0 \to \Aut(M) \to \Aut^\sigma(M) \to \Z/2.\end{equation}
As before, we have the group of `graded signed symplectomorphisms' which fits into a short exact sequence
\begin{equation} 0 \to G \to \Aut^{\sigma,gr}(M) \to \Aut^\sigma(M) \to 0,\end{equation}
and we also have a subgroup fitting into the short exact sequence
\begin{equation} 0 \to G_{ev} \to \Aut^{\sigma,gr}_{ev}(M) \to \Aut^\sigma(M) \to 0\end{equation}
like before.
 
We define an action of a signed group $(\Gamma,\sigma)$ on $M$ to be a signed group homomorphism $(\Gamma,\sigma) \to \Aut^\sigma(M)$.
Now we give the analogue of Lemma \ref{lem:AutFukr}:

\begin{lem}
\label{lem:AutsigmaFukr}
Suppose $M$ is equipped with a morphism $p: \G \to \Z/4$. 
Let $\EuA \subset \fuk_{\l\r}(M)$ be a full subcategory which is closed under shifts. 
Let $(\Gamma,\sigma)$ act on $M$, so that the corresponding action of $\Gamma$ on $\G$ respects $p$, and the action of $\Gamma$ on the set of unanchored branes underlying $Ob(\EuA)$ is free. 
Then for appropriate choice of perturbation data in $\fuk_{\l\r}(M)$, there is an induced signed action of $(\Gamma,\sigma)$ on $\EuA$ up to shifts.
\end{lem}
\begin{proof}
Let
\begin{equation} \EuA \coprod \overline{\EuA} \subset \fuk_{\l\r}(M) \coprod \fuk_{\l\r}(\overline{M}) \simeq \fuk_{\l\r}\left(M \coprod \overline{M}\right)\end{equation}
be the subcategory whose objects are the union of all anchored Lagrangian branes in $M$ or $\overline{M}$ which correspond to objects of $\EuA$. 

There is an obvious inclusion
\begin{equation} \Aut^{\sigma,gr}_{ev}(M) \hookrightarrow \Aut^{gr}_{ev}\left(M \coprod \overline{M}\right).\end{equation}
Therefore, we have a map
\begin{equation} \Gamma \to \mathsf{StrAut}_{\l\r}\left( \EuA \coprod \overline{\EuA} \right)/G_{ev},\end{equation}
by Lemma \ref{lem:AutFukr} (as the image of $\Gamma$ acts freely on the objects of $\EuA \coprod \overline{\EuA}$).
We also have the isomorphism
\begin{align}
c_{\l\r}^{\overline{M}}: \overline{\EuA} &\xrightarrow{\simeq} \EuA^{op},
\end{align}
commuting with leftwards and rightwards shift functors, from Lemma \ref{lem:oppfuk}.
Thus we obtain a homomorphism
\begin{equation} \Gamma \to \mathsf{StrAut}_{\l\r}\left(\EuA \coprod \EuA^{op} \right)/G_{ev}.\end{equation}
In order to prove that this homomorphism gives a signed group action, it remains to check that
\begin{equation} \gamma \circ op = op \circ \gamma.\end{equation}
This follows easily from a diagram-chase, using the naturality of the construction of the Fukaya category, and of the isomorphisms $c_{\l\r}^M$ and $c_{\l\r}^{\overline{M}}$, with respect to the action of graded symplectomorphisms. 
\end{proof}

\begin{rmk}
Lemma \ref{lem:AutsigmaFukr} is closely related to the results of \cite{Fukaya2017a}, which consider the case of $\Z/2$ acting on $M$ by an anti-symplectic involution, and $\EuA$ the subcategory consisting of the fixed locus of the involution.
\end{rmk}

Let us consider the closed--open map $\EuC \EuO \circ i$ in the situation of Lemma \ref{lem:AutsigmaFukr}. 
We define an action of $\Gamma$ on $H^\bullet(M)$, by setting
\begin{equation} \gamma \cdot \alpha := (-1)^{\sigma(\gamma)} \cdot \left(\gamma^{-1}\right)^* \alpha.\end{equation} 
By Lemma \ref{lem:AutsigmaFukr} together with \eqref{eqn:strautrlacts}, we have an action of $\Gamma$ on $CC^\bullet_{\l\r}(\EuA)$.

\begin{lem}
\label{lem:coisigns}
We have
\begin{equation}\EuC \EuO \circ i (\gamma \cdot \alpha) = \gamma \cdot \EuC\EuO \circ i(\alpha).\end{equation}
In other words, $\EuC\EuO \circ i$ is strictly $\Gamma$-equivariant.
\end{lem}
\begin{proof}
We established this result for even $\gamma$ (i.e., symplectomorphisms) in \S \ref{subsec:coisigns}. 
The proof when $\gamma$ is odd is analogous to the proof of Lemma \ref{lem:AutsigmaFukr}, with one exception: the orientation of $\lambda(T_qS_r)$ changes sign under complex conjugation of the domain, hence the identification $\lambda(T_q S_r) \simeq \Ts_2$ changes sign. 
This is cancelled by the addition of the sign $(-1)^{\sigma(\gamma)}$ in the definition of the action of $\Gamma$ on $H^\bullet(M)$.
\end{proof}

\subsection{Signed group actions on the relative Fukaya category}

We now consider signs in the relative Fukaya category $\fuk(X,D)$ (or $\fuk(\Nef)$), which is a $\G$-graded deformation of $\fuk(X \setminus D)$ over $R$. 
The sign arguments we gave in the construction of the affine Fukaya category apply word-for-word to the construction of the relative Fukaya category, with one exception: \eqref{eqn:mugradings} is replaced by
\begin{equation}
\label{eqn:mugradingsrel}
 g_0 + \deg(\nov^u) = 2-s+\sum_{j=1}^k g_j 
\end{equation}
where $u \in H_2(X,X \setminus D)$ is the homology class of the pseudoholomorphic disc, and $\nov^u \in R$ is the corresponding monomial in the coefficient ring. 
This follows from Lemma \ref{lem:masgrad}.

Now the coefficient ring is concentrated in even degree, so the addition of the term $\deg(\nov^u)$ does not change the equation modulo $2$: it follows that the definition of $\fuk_\r(X,D)$ and $\fuk_\l(X,D)$ can proceed as before. 
However, the extension of the definition of $\fuk_{\l\r}(X\setminus D)$ (see \S \ref{subsec:fukshifts}) to the relative Fukaya category acquires a new wrinkle. 
Recall that the definition of $\fuk_{\l\r}(X \setminus D)$ depended on a choice of morphism of grading data $p: \G \to \Z/4$: and the definition of $\fuk_{\l\r}(X \setminus D)$ relied on \eqref{eqn:mugradings} holding modulo $4$. 
In the relative Fukaya category, the two sides of \eqref{eqn:mugradings} differ by $p(\deg(\nov^u))$, by \eqref{eqn:mugradingsrel}. 
It follows that when we glue $o_{y_1},\ldots,o_{y_k}$ onto $D_u$, we get an orientation operator that lives in
\begin{equation} 
o_{\l\r}(T_{y_0(0)}L_0^\#,T_{y_0(1)}L_1^\#,p(y_0)+p(\deg(\nov^u))).
\end{equation}
We must identify this with the orientation operator in degree $p(y_0)$ modulo $4$. 
If $p(\deg(\nov^u)) \equiv 2 \text{ (mod 4)}$, then there are two ways to do this: via $s_\r^2$ or via $s_\l^2$, and these two identifications differ by a sign.

We define an involution of $R$ by
\begin{align}
i_p: R &\to R,\\
i_p(r) &:= (-1)^{p(\deg(r))/2} \cdot r.
\end{align}

Then the preceding arguments yield:

\begin{lem}
\label{lem:oprelfuk}
Given a morphism $p: \G \to \Z/4$, there is an induced $R$-linear isomorphism
\begin{equation} \fuk_\r(X,D) \simeq i_p^*\fuk_\l(X,D)\end{equation}
which extends the isomorphism
\begin{equation} \fuk_\r(X \setminus D) \simeq \fuk_{\l\r}(X \setminus D) \simeq \fuk_\l(X \setminus D)\end{equation}
from \S \ref{subsec:fukshifts}.
\end{lem}

\begin{rmk}
It is explained in \cite[Remark 5.2.4]{Wehrheim2015} that (translating into our language) there is an isomorphism of cochain complexes
\[ hom^*_{\fuk_\r(X,D)}(L_0,L_1) \simeq i_p^* hom^*_{\fuk_{\r,w_2(TX)}(\overline{X},\overline{D})^{op}}(L_1,L_0).\]
This is related to Lemma \ref{lem:oprelfuk} by the observation that
\[ hom^*_{\fuk_{\r,w_2(TX)}(\overline{X},\overline{D})^{op}}(L_1,L_0) \simeq hom^*_{\fuk_{\l,w_2(TX)}(X,D)}(L_0,L_1).\]
Note that our assumption that the morphism $p: \G \to \Z/4$ exists implies that $w_2(TX) = 0$ (compare Remark \ref{rmk:w2twist}).
\end{rmk}

We can now give the analogue of Lemma \ref{lem:AutsigmaFukr}, using the terminology from \S \ref{subsec:ainfversgp}:

\begin{lem}
\label{lem:autsigrel}
Suppose that $(X,D)$ is equipped with a morphism $p: \G \to \Z/4$. Let $\EuA \subset \fuk_{\l\r}(X \setminus D)$ be a full subcategory which is closed under shifts. 
Let $(\Gamma,\sigma)$ act on $X$, so that:
\begin{itemize}
\item $\Gamma$ preserves $D$ as a set.
\item The induced action of $\Gamma$ on $\G$ respects $p$.
\item The action of $\Gamma$ on the set of unanchored branes underlying $Ob(\EuA)$ is free.
\end{itemize} 
We have an action of $\Gamma$ on $\G$, and a simultaneous action of $\Gamma$ on $R$ by
\[ \gamma \cdot r^u := i_p^{\sigma(\gamma)} \left(r^{\gamma \cdot u}\right)\]
where 
\[ \gamma \cdot u := (-1)^{\sigma(\gamma)} \gamma_* u.\]
Then for appropriate choice of perturbation data in $\fuk_\r(X,D)$, the corresponding full subcategory $\EuA_R \subset \fuk_\r(X,D)$ is a $(\Gamma,\sigma)$-equivariant deformation of $\EuA \subset \fuk_{\l\r}(X \setminus D) \simeq \fuk_\r(X \setminus D)$ over $R$, relative to the simultaneous action of $\Gamma$ on $\G$ and $R$.
\end{lem}

\section{Fixing the mirror map}
\label{sec:hodgems}

The main theme of this paper is to establish conditions under which there exists a mirror map $\Psi$, for which homological mirror symmetry holds. 
In this section we address the question of the uniqueness of the mirror map $\Psi$. 
The basic idea is to use the Hodge-theoretic characterization of the mirror map (for which see, e.g., \cite{coxkatz}): but some subtleties are introduced by the fact that we consider relative K\"{a}hler forms rather than K\"{a}hler forms.

For the purposes of this section, we will use $\Bbbk = \C$.

\subsection{Gradings and coefficients}

Let $(X,D,\omega)$ be a relative K\"{a}hler manifold,  $\G$ the associated grading datum, and $\Nef \subset Nef(X,D)$ a nice cone. 
Assume that $c_1(TX) = 0$, so $X$ admits a holomorphic volume form $\eta$, whose restriction to $X \setminus D$ determines a squared phase map which splits the grading datum (see Lemma \ref{lem:c1tors}). 
Thus we have a canonical identification of the grading datum $\{\Z \to G \to \Z/2\}$ with
\[
\begin{array}{ccccc}
\Z & \to & \Z \oplus H_1(X\setminus D) & \to & \Z/2 \\
j & \mapsto & j \oplus 0 & & \\
&& j \oplus h & \mapsto & [j].
\end{array}\]

We will abbreviate $R := R(\Nef)$. 
We denote the point in $\overline{\cM}_{K\ddot{a}h}(X,D,\Nef) := \spec(R)$ corresponding to the toric maximal ideal $\fm$ by $0$.

We now consider a unital commutative $\G$-graded $R$-algebra $B_R$: by the morphism of grading data $\G \to \Z$ induced by $\Omega$, we can regard $B_R$ as a $\Z$-graded $R$-algebra. 
Note that $R$ is in degree $0$, as follows from the fact that the holomorphic volume form $\eta$ does not have a pole along $D_p$. 
We assume that $B_R$ is concentrated in non-negative degrees and connected, i.e., the degree-$0$ part is spanned by $R \cdot 1$.
We will consider the scheme $Y_R := \proj(B_R)$. 
It is of course a scheme over $\overline{\cM}_{K\ddot{a}h}(X,D,\Nef)$.

\subsection{Hodge-theoretic mirror symmetry}

We denote the formal disc by $\Delta := \spec(\C\power{q})$, and the formal punctured disc by $\Delta^* := \spec(\C\laurent{q})$. 

\begin{defn}
A \emph{disc map} is a morphism of $\C$-schemes
\begin{align*}
\dm: \Delta & \to \overline{\cM}_{K\ddot{a}h}(X,D,\Nef),
\end{align*}
sending $0\mapsto 0$, such that $\dm|_{\Delta^*}$ factors through $\cM_{K\ddot{a}h}(X,D,\Nef)$, and such that $\dm^*$ is adically continuous.
It is equivalent to an adically continuous $\C$-algebra homomorphism $\dm^*: R \to \C\power{q}$, such that $(d^*)^{-1}(q \cdot \C\power{q}) = \fm$, and such that its composition with the inclusion $\C\power{q} \hookrightarrow \C\laurent{q}$ factorizes through $\Rpunc$:
\[ \xymatrix{R \ar[r]^{\dm^*} \hookar[d] & \C\power{q} \hookar[d] \\
	\Rpunc \ar@{-->}[r] & \C\laurent{q}.}\]
\end{defn}

We will also denote $\dm|_{\Delta^*}$ by $\dm$, and the corresponding map $\Rpunc \to \C\laurent{q}$ by $\dm^*$; we hope no confusion will result.

\begin{defn}
Let $\val: \C\laurent{q}^\times \to \Z$ denote the valuation. 
If $\dm$ is a disc map, we define $\val(\dm) \in H^2(X,X \setminus D)$ by
\[ \val(\dm)(u) := \val(\dm^*(\nov^u)) \quad\text{ for all $u \in H_2(X,X \setminus D)$}.\]
Observe that this is well-defined (i.e., $\dm^*(\nov^u) \in \C\laurent{q}^\times$ for all $u$), because each $\nov^u$ is invertible in $\Rpunc$. 
Observe furthermore that $\val(\dm)$ lies in the interior of $\Nef$, because $\dm^*(\fm) \subset q \cdot \C\power{q}$.
\end{defn}

\begin{rmk}
Geometrically, $\val(\dm)(u_p)$ is equal to the order of tangency of the disc map $\dm$ with the divisor $\{\nov_p = 0\}$.
\end{rmk}

\begin{lem}
\label{lem:leadingflat}
If $\Psi^* \in \Aut(R)$ then $\val(\Psi (\dm)) = \val(\dm)$.
\end{lem}
\begin{proof}
Because $R$ is nice, we can apply Lemma \ref{lem:typeAaut}: so 
\[ \dm^*(\Psi^*(\nov_p)) = \dm^*(\nov_p \cdot \psi_p) = \dm^*(\nov_p) \cdot \dm^*(\psi_p).\]
Now $\psi_p \in R_0^\times \implies \dm^*(\psi_p) \in \C\power{q}^\times$, so $\val(\dm^*(\psi_p)) = 0$. 
Thus $\val(\Psi (\dm))(u_p) = \val(\dm)(u_p)$ for all $p$: since the classes $u_p$ span $H_2(X,X \setminus D)$, this proves the result.
\end{proof}

\begin{defn}
We consider the $A$-model VSHS over $\Rpunc$: $\EuH^A(X,D,\Nef):= H^\bullet(X;\Rpunc)\power{u} [n]$, equipped with the connection
\[ \nabla_{\frac{\partial}{\partial \nov_p}}(\alpha) := \frac{\partial \alpha}{\partial \nov_p} - \nov_p^{-1}u^{-1} PD(D_p) \star \alpha.\]
If $\dm$ is a disc map then we have the pullback $\dm^*\EuH^A(X,D,\Nef)$, which is a VSHS over $\C\laurent{q}$. 
Explicitly,
\[ \dm^*\EuH^A(X,D,\Nef) := H^\bullet(X;\C\laurent{q}) \power{u}[n] \qquad \text{(here, `$[n]$' denotes a shift),}\]
equipped with the pullback connection
\[ \nabla_{ \frac{\partial}{\partial q}}(\alpha) := \frac{\partial \alpha}{\partial q} - u^{-1}\sum_p \frac{\partial \phi_p/\partial q}{\phi_p} PD(D_p) \star \alpha,\]
where $\phi_p := \dm^*(\nov_p)$.
\end{defn}

\begin{rmk}
Note that the connection $\nabla$ is singular along the divisors $\{\nov_p = 0\}$. 
This explains why we require a disc map to send the formal punctured disc to the complement of these divisors: if we didn't, the pullback connection would not be defined.
\end{rmk}

\begin{defn}
If $\dm$ is a disc map, we will denote by $Y_{\dm}$ the fibre of $Y_R$ over $\dm$. 
It is clearly a projective and connected $\C \laurent{q}$-scheme, by construction. 
If it is furthermore smooth, we can define the $B$-model VSHS $\EuH^B(Y_{\dm})$ over $\C\laurent{q}$, in the sense of \cite[\S 3.2]{Ganatra2015}.
\end{defn}

We suppose that $Y_\dm$ is smooth for all disc maps $\dm$ such that $\val(\dm)$ lies in the interior of $\Nef$.

\begin{defn}
\label{defn:hodgemirrmap}
In this situation, we say that $\Psi^* \in \Aut(R)$ is a \emph{Hodge-theoretic mirror map} if 
\begin{itemize}
\item it induces the identity map $\fm/\fm^2 \to \fm/\fm^2$ (i.e., $d\Psi(0) = \id$), and
\item for all disc maps $\dm$, we have $\dm^*\EuH^A(X,D,\Nef) \simeq \EuH^B(Y_{ \Psi(\dm)})$.
\end{itemize}
We define the \emph{closed part} to be the restriction $\Psi^*_{cl}: R_{cl} \to R_{cl}$.
\end{defn}

\subsection{Flat coordinates}

\begin{defn}
We say that a disc map $\dm$ is \emph{flat} for $(X,D,\Nef)$ if $\dm^*\EuH^A(X,D,\Nef)$ is a Hodge--Tate VSHS in the sense of \cite[Definition 2.13]{Ganatra2015}, and the uniformizer $q \in \C\laurent{q}$ is a canonical coordinate in the sense of \cite[Definition 2.18]{Ganatra2015}. 
Similarly, we say that a disc map $\dm$ is \emph{flat} for $Y_R$ if $Y_\dm$ is smooth, $\EuH^B(Y_{\dm})$ is Hodge--Tate, and $q$ is a canonical coordinate for it. 
\end{defn}

Observe that if $\Psi$ is a Hodge-theoretic mirror map, then a disc map $\dm$ is flat for $(X,D,\Nef)$ if and only if $\Psi( \dm)$ is flat for $Y_R$. 

\begin{lem}
\label{lem:flatun}
A disc map $\dm$ is flat for $(X,D,\Nef)$ if and only if 
\[ \dm^*(\nov^u) = c(u) \cdot q^{\val(\dm)(u)} \]
for all $\nov^u \in R_{cl}$, for some group homomorphism $c: H_2(X) \to \C^*$.
\end{lem}
\begin{proof}
By a minor extension of the argument given in \cite{Ganatra2015}, $q$ is a canonical coordinate for $\dm^*\EuH^A(X,D,\Nef)$ if and only if $\nabla_{q\partial_q}\nabla_{q \partial_q}(e)$ is of degree $\ge 4$, which is equivalent to
\[ q\frac{\partial}{\partial q} \left(\sum_p q\frac{\partial \phi_p/\partial q}{\phi_p} PD(D_p) \right) = 0.\]
If we evaluate this cohomology class on $u \in H_2(X)$ we obtain 
\begin{align}
q \frac{\partial}{\partial q} \left( \sum_p q\frac{\partial \phi_p/\partial q}{\phi_p} (u \cdot D_p) \right) &= 0 \\
\implies   \sum_p q\frac{\partial \phi_p/\partial q}{\phi_p} (u \cdot D_p)  &= k(u), \text{ a constant dependent on $u$} \\
\implies q\frac{\partial \left(\dm^*\left(\nov^u\right)\right)}{\partial q} &= k(u) \cdot \dm^*\left(\nov^u\right) \\
\implies \dm^*(\nov^u) &= c(u) \cdot q^{k(u)}.\label{eqn:flatdet}
\end{align}
It follows that \eqref{eqn:flatdet} holds for all $\nov^u \in R_{cl}$, by Remark \ref{rmk:usualnov}. 
It is clear that $c$ defines a group homomorphism, since $\dm^*$ is an algebra homomorphism, and that $k = \val(\dm)$.  
The converse follows similarly.
\end{proof}

\begin{thm}
\label{thm:mirrcan}
If $\Psi_1$ and $\Psi_2$ are Hodge-theoretic mirror maps, we have
\[ \left(\Psi_1^*\right)_{cl} = \left(\Psi_2^*\right)_{cl}.\]
Thus, if there exists at least one Hodge-theoretic mirror map, then there is a unique map $\Psi^*_{cl} \in \Aut(R_{cl})$ which is the restriction of any Hodge-theoretic mirror map to the closed part.
\end{thm}
\begin{proof}
Let $\Omega_A$ denote the set of maps $\beta: H_2(X) \to \Z$ which are equal to the composition
\[ H_2(X) \to H_2(X,X \setminus D) \xrightarrow{\val(\dm)} \Z\]
for some disc map $\dm$ which is flat for $(X,D,\Nef)$. 
We have a map
\begin{align}
\label{eqn:iotaar} \iota_A: R_{cl} & \to \C\laurent{q}^{\oplus \Omega_A} \\
\iota_A & := \bigoplus_{\beta \in \Omega_A} \dm^*_{\beta}|_{R_{cl}},
\end{align}
where $\dm_{\beta}$ is a flat disc map with $\val(\dm_{\beta}) = \beta$, and $c(u) \equiv 1$ in accordance with Lemma \ref{lem:flatun} (which shows that these properties characterize $\dm^*_{\beta}|_{R_{cl}}$ uniquely). 

We make a similar construction on the $B$-side (observe that the analogue of Lemma \ref{lem:flatun} holds on the $B$-model also, by our assumption that Hodge-theoretic mirror symmetry holds). 
Because $\Psi_i$ are Hodge-theoretic mirror maps, we have
\[ \dm \mbox{ is flat for $(X,D,\Nef)$ } \iff \Psi_i(\dm) \mbox{ is flat for $Y_R$}\]
for $i=1,2$. 
Furthermore, because $\Psi_i$ induce the identity map on $\fm/\fm^2$, $c \equiv 1$ for $\dm$ is equivalent to $c \equiv 1$ for $\Psi_i (\dm)$. 
Combining this with Lemma \ref{lem:leadingflat}, we have a canonical identification $\Omega_A \simeq \Omega_B$, and a commutative diagram
\[ \xymatrix{ R_{cl} \dar_{\iota_A} & R_{cl} \dar^{\iota_B} \ar[l]^{\left(\Psi^*_i\right)_{cl}}  \\
\C\laurent{q}^{\oplus \Omega_A} \ar@{=}[r] & \C\laurent{q}^{\oplus \Omega_B}.}\]
for $i=1,2$. The result now follows from the fact that $\iota_A$ (and hence also $\iota_B$) is injective, by the following Lemma \ref{lem:iotainj}.
\end{proof}

\begin{lem}
\label{lem:iotainj}
The map 
\[ \iota_A: R_{cl}  \to \C\laurent{q}^{\oplus \Omega_A} \]
of \eqref{eqn:iotaar} is injective.
\end{lem}
\begin{proof}
Suppose that
\[ f := \sum_{u \in NE(\Nef)_{cl}} a_u \nov^u\]
is in the kernel of the map $\iota_A$. 
We will show that $a_u = 0$ for each $u \in NE(\Nef)_{cl}$. 
We start by observing that for each such $u$, there exists some $\beta_u$ in $\Nef^\circ \cap H^2(X,X \setminus D)$, so that 
\[ \{v \in NE(\Nef): \beta_u(v) = \beta_u(u)\} = \{u\}.\]
To prove this, first observe that the corresponding statement is true for an open dense set of $\beta \in \Nef^\circ$, because the hyperplane $\{v: \beta(v) = \beta(u)\}$ intersects $NE(\Nef)_\R$ in a compact set; hence it is true for some $\beta_u \in \Nef^\circ \cap H^2(X,X \setminus D;\Q)$, because $\Nef^\circ$ is open; hence it is true for some $\beta_u \in \Nef^\circ \cap H^2(X,X \setminus D)$ by rescaling by a positive integer.

There is a corresponding flat disc map $\dm_u$ with $c \equiv 1$, given by
\[ \dm^*_u\left(\nov^{u'}\right) := q^{\beta_u(u')}\]
for all $u'$.
Since $f$ is in the kernel of $\iota_A$, we have $\dm^*_u(f) = 0$. 
It follows that the $q^{\beta_u(u)}$ coefficient of $\dm_{\beta_u}(f)$ vanishes: but this coefficient is precisely $a_u$ by construction, so $a_u = 0$. 
The same argument applies to all $u \in NE(\Nef)_{cl}$, so $f=0$.
\end{proof}

\subsection{Base changing to the Novikov field}

Theorem \ref{thm:mirrcan} gives a partial characterization of the Hodge-theoretic mirror map $\Psi^*$.
In this section we would like to explain why this partial characterization is satisfactory as long as we are willing to work over the Novikov field.

Let $\BbK$ be a $\C$-algebra, and
\[ \dm^*: \Rpunc \to \BbK\]
a unital $\C$-algebra homomorphism. 
So if $\EuA$ is an $\Rpunc$-linear $A_\infty$ algebra, we can define the fibre $\BbK$-linear $A_\infty$ algebra $\EuA_\dm$ as in Definition \ref{defn:fibainf}.
 
Recall that $\Rpunc_{cl} \subset \Rpunc$ is the part of degree $0 \in H_1(X \setminus D)$ (see Definition \ref{defn:rcl}).
If $\dm^*: \Rpunc \to \BbK$ is a $\C$-algebra homomorphism, we denote $\dm^*_{cl} := \dm^*|_{\Rpunc_{cl}}: \Rpunc_{cl} \to \BbK$. 

\begin{lem}
\label{lem:divK}
Suppose that $\EuA$ is a $\G$-graded $\Rpunc$-linear $A_\infty$ algebra, we have unital $\C$-algebra homomorphisms
\[ \dm^*,\dmm^*: \Rpunc \to \BbK\]
such that $\dm^*_{cl} = \dmm^*_{cl}$, and the abelian group of units $\BbK^{\times}$ is divisible. 
Then we have an isomorphism of $\BbK$-linear $A_\infty$ algebras,
\[ \EuA_\dm  \simeq \EuA_\dmm.\]
\end{lem}
\begin{proof}
Observe that $\dm$ and $\dmm$ together define a homomorphism of abelian groups,
\begin{align}
\dmm^{-1}\dm: H_2(X,X \setminus D) & \to \BbK^\times,\\
\dmm^{-1}\dm(u) & := \dmm(\nov^{-u})\cdot \dm(\nov^u).
\end{align}
Furthermore, $\dmm^{-1}\dm$ vanishes on the kernel of the map $H_2(X,X \setminus D) \to H_1(X \setminus D)$, by the hypothesis that $\dm_{cl} = \dmm_{cl}$. 
Therefore, by the condition of divisibility, we have a factorization
\[ \xymatrix{H_2(X,X \setminus D) \rar \drar_{\dmm^{-1}\dm} & H_1(X \setminus D) \ar@{-->}[d]_{\dmmm} \\
		& \BbK^\times.}\]
As a result, we have
\begin{equation}
\label{eqn:Xiups}
 \dmm(r) \dmmm(|r|) = \dm(r) 
\end{equation}
for all $r \in \Rpunc$ pure of degree $|r| \in H_1(X \setminus D)$.

We use this to define an isomorphism 
\begin{align}
\EuA_\dm& \xrightarrow{\simeq} \EuA_\dmm \\
a \otimes k & \mapsto a \otimes \dmmm(|a|) \cdot k,
\end{align}
for $a$ pure of degree $|a| \in H_1(X \setminus D)$. 
One easily verifies that this map is well-defined using \eqref{eqn:Xiups}; it is also an isomorphism of $\BbK$-linear $A_\infty$ algebras, using the fact that $\EuA$ is $\G$-graded, so the $A_\infty$ products on $\EuA$ respect the $H_1(X \setminus D)$-grading.
\end{proof}

\begin{example}
The example we have in mind (in light of \S \ref{subsec:relabsol}) is $\BbK = \Lambda$, the Novikov field. 
$\Lambda$ is algebraically closed (see \cite[Lemma A.1]{Fukaya2010c}), so $\Lambda^\times$ is divisible.
\end{example}

\begin{example}
Suppose $\dm,\dmm$ are disc maps with $\dm_{cl} = \dmm_{cl}$. 
Then it need not be true that $Y_{\dm} \simeq Y_{\dmm}$: however after a further base change via $\C\laurent{q} \hookrightarrow \Lambda$, the schemes become isomorphic, as their homogeneous coordinate rings are isomorphic by Lemma \ref{lem:divK}.
\end{example}

\begin{rmk}
Lemma \ref{lem:divK} can not be applied directly to $\BbK = \C \laurent{q}$, since its group of units is not divisible. 
However it is easy to prove a variation in which the conclusion $\EuA_\dm \simeq \EuA_\dmm$ holds after passing to some finite cover $\C\laurent{q^{1/j}}$.
\end{rmk}

\subsection{Homological mirror symmetry}
\label{subsec:hms}

In this section we outline a general approach to proving instances of homological mirror symmetry, using the results of this paper. 
The approach assumes that the absolute and relative Fukaya categories exist and have certain properties. 
We will draw attention to the properties we are using as we go along. 
After sketching the approach in the abstract, we will explain how it was implemented in the case of Greene--Plesser mirror pairs \cite{SS}.

We first introduce the notion of a \emph{Novikov disc map}, which is the same as a disc map except $\C\laurents{q}$ is replaced with the Novikov field $\Lambda$ in the definition. We consider the following situation:

\begin{sit}
\label{sit:hms}
We have
\begin{enumerate}
\item \label{it:1} $\Nef$ is nice and Calabi--Yau, and $[\omega]$ is contained in its interior.
\item \label{it:2} $Y_R$ is as before, and $Y_{\dm}$ is a projective, connected $\Lambda$-scheme  which is smooth and has trivial relative canonical sheaf for all Novikov disc maps $\dm$.
\item \label{it:3} $\EuB$ is a $\G$-graded $R$-linear $A_\infty$ category, such that there is a $\Lambda$-linear $A_\infty$ quasi-equivalence
\[ D^\pi\left(\EuB_\dm\right) \simeq \dbdg Coh\left(Y_{\dm}\right) \]
for all Novikov disc maps $\dm$. `$D^\pi$' denotes the split-closed derived $A_\infty$ category.
\item \label{it:4} $\EuA \subset \fuk(X,D,\Nef)$ is a full subcategory of the relative Fukaya category, and there is a $\G$-graded $R$-linear (possibly curved) $A_\infty$ functor
\begin{equation}
\label{eqn:hmsiso}
 \Psi^* \EuB \dashrightarrow \EuA
 \end{equation}
for some $\Psi^* \in \Aut(R)$, such that $\Psi^*: \fm/\fm^2 \to \fm/\fm^2$ is the identity. 
The functor has curvature of order $\fm$, it is bijective on objects, and its linear term is an isomorphism (not just a quasi-isomorphism) on each morphism space.
\end{enumerate}
\end{sit}

Our main results, Theorems \ref{thm:1} and \ref{thm:2}, give criteria for verifying the main hypothesis \eqref{it:4}. 
In particular, the $A_\infty$ functors produced by our deformation theory arguments are of the kind appearing in \eqref{eqn:hmsiso}, cf. Definition \ref{defn:vers1int}. 

\begin{rmk}
Note that, by Lemma \ref{lem:curvmc}, the curved functor \eqref{eqn:hmsiso} induces an uncurved functor
\[ \Psi^* \EuB^\bc \dashrightarrow \EuA^\bc.\]
Since $\EuB$ is uncurved, we have a strict embedding $\EuB \hookrightarrow \EuB^\bc$ by equipping each object with the zero bounding cochain. Composing, we obtain an uncurved $A_\infty$ functor
\[ \Psi^* \EuB \dashrightarrow \EuA^\bc.\]
Our assumption that the linear term of \eqref{eqn:hmsiso} is an isomorphism ensures that this functor is a quasi-embedding (i.e., cohomologically full and faithful), even after base-change. 
In particular, combining with Hypothesis \eqref{it:3} we obtain an uncurved quasi-embedding
\begin{equation}
\label{eqn:corhms} 
\dbdg Coh(Y_{\Psi(d)}) \hookrightarrow D^\pi \fuk(X,D,\Nef)_d^\bc
\end{equation}
for any Novikov disc map $d$. 
These uncurved quasi-embeddings are the main inputs for Theorems \ref{thm:sitgiveshms} and \ref{thm:hmshodge} below.
\end{rmk}

Let us assume that $\fuk(X,\omega)^\bc$ is a $\Lambda$-linear, $\Z$-graded $A_\infty$ category (see \S \ref{subsec:relabsol}). 
We denote its split-closed derived $A_\infty$ category by $D^\pi \fuk(X,\omega)^\bc$.

\begin{defn}
Let $(X,\omega)$ be a Calabi--Yau symplectic manifold.
Let $Y_\omega$ be a scheme over $\Lambda$, and $\dbdg Coh(Y_\omega)$ a DG enhancement of its bounded derived category of coherent sheaves.
We say that $(X,\omega)$ and $Y_\omega$ are \emph{homologically mirror} if there is a quasi-equivalence of $\Lambda$-linear, $\Z$-graded $A_\infty$ categories
\[ D^\pi\fuk(X,\omega)^\bc \simeq \dbdg Coh(Y_\omega).\]
\end{defn}

\begin{thm}\label{thm:sitgiveshms}
In Situation \ref{sit:hms}, let $\dm := \dm([\omega])$ be the $\Lambda$-point of $\overline{\cM}_{K\ddot{a}h}(X,D,\Nef)$ from Lemma \ref{lem:RLamb}. 
Then $(X,\omega)$ and $Y_{\Psi(\dm)}$ are homologically mirror.
\end{thm}
\begin{proof}
This is the main result of \cite{Perutz2015}. 
The result relies on the relative Fukaya category having certain properties laid out in \cite[\S 2]{Perutz2015}, and being related to the absolute Fukaya category as in \cite[Conjecture 1.8]{Perutz2015}.
The proof also uses an additional geometric condition on $Y_\dm$ called \emph{maximal unipotence}, but it is explained in \cite{Ganatra2016} that this condition can be avoided. 
The requisite properties of the relative Fukaya category will be proved for a particular version of the relative Fukaya category in \cite{Perutz2015a}.
\end{proof}

\begin{thm}
\label{thm:hmshodge}
In Situation \ref{sit:hms}, $\Psi$ is a Hodge-theoretic mirror map for $(X,D,\Nef)$ and $Y_R$, in the sense of Definition \ref{defn:hodgemirrmap}.
\end{thm}
\begin{proof}
The proof of Theorem \ref{thm:sitgiveshms} shows that there is a quasi-equivalence
\[ \dbdg Coh(Y_{\Psi(d)}) \simeq D^\pi\left(\fuk(X,D,\Nef)^\bc_{\dm}\right)\]
for any disc map $\dm$. 
This implies the result, by \cite[Theorem A]{Ganatra2015}. 
The result relies on the relative Fukaya category having certain properties laid out in \cite[\S 4]{Ganatra2015}, which will be proved for a particular version of the relative Fukaya category in \cite{Perutz2015a,Ganatra2015a}. 
It also relies on a certain conjecture on the $B$-side, saying that the cyclic formality map respects Hodge structures.
The proof again uses maximal unipotence of $Y_\dm$, but again it can be avoided by \cite{Ganatra2016}.
\end{proof}

\begin{rmk}
Observe that the $\Lambda$-scheme $Y_{\Psi(\dm)}$ only depends (up to isomorphism) on $(\Psi(\dm))_{cl} = \Psi_{cl}(\dm_{cl})$. 
$\Psi_{cl}$ is uniquely determined and equal to the Hodge-theoretic mirror map, by Theorems \ref{thm:hmshodge} and \ref{thm:mirrcan}, and $\dm_{cl}$ is uniquely determined by the cohomology class $[\omega] \in H^2(X;\R)$.
\end{rmk}

We now explain roughly how Situation \ref{sit:hms} was established for Greene--Plesser mirror pairs in \cite{SS}, leading to a proof of homological mirror symmetry in that case via Theorem \ref{thm:sitgiveshms}.
\begin{itemize}
\item Hypothesis \eqref{it:1} was verified using the results of Section \ref{sec:geomset}. 
\item Hypothesis \eqref{it:2} was verified in [\emph{op. cit.}, Proposition 4.4]: it was shown that when $val(d)$ lies in the interior of an appropriate cone $\Nef$, the tropicalization of $Y_d$ satisfies a condition similar to `tropical smoothness', which guarantees smoothness of $Y_d$.
\item The category $\EuB$ was constructed in [\emph{op. cit.}, Section 4] by using the homological perturbation lemma to construct a minimal model for a subcategory of a certain $R$-linear category of graded matrix factorizations. After base-changing to the field $\Lambda$, this quasi-embeds in a $\Lambda$-linear category of graded matrix factorizations, which is quasi-equivalent to $\dbdg (Y_d)$ by Orlov's theorem \cite{Orlov2009}. 
One verifies that the image of the base-change $\EuB_d$ split-generates, yielding Hypothesis \eqref{it:3}.
\item The Lagrangian objects of $\EuA$ are constructed as lifts of a single immersed Lagrangian under a branched cover of the kind considered in \S \ref{subsec:brcov} (see [\emph{op. cit.}, Section 3.2]). 
The endomorphism algebra of this immersed Lagrangian in the affine Fukaya category downstairs was computed in \cite{Sheridan:pants}, and its first-order deformation classes were computed in \cite{Sheridan:CY}. 
The results of \S \ref{subsec:cobr} allow one to deduce the endomorphism algebra and first-order deformation classes upstairs, and hence to verify the hypotheses of the versality criterion Theorem \ref{thm:2} (with an appropriate signed group action). 
It then suffices to identify $\EuA$ and $\EuB$ modulo $\fm$, and identify their first-order deformation classes, in order to verify Hypothesis \eqref{it:4}.
\end{itemize}

\bibliographystyle{amsalpha}
\bibliography{mybib}

\end{document}